\pdfoutput=1
\documentclass[11pt]{amsart}

\usepackage{mathrsfs}
\usepackage{amssymb}
\usepackage{amsmath}
\usepackage{amsfonts}
\usepackage{mathtools}
\usepackage{geometry}
\usepackage[english]{babel}
\usepackage[algosection, ruled, vlined, linesnumbered]{algorithm2e}
\usepackage{enumitem}
\usepackage{tikz}
\usepackage[font={footnotesize}]{caption}
\usepackage{xcolor}
\usepackage{xparse}
\usepackage{bbm}
\usepackage{trimclip}
\usepackage{float}
\usepackage[LGR,OT1]{fontenc}
\usepackage{csquotes}
\usepackage[e]{esvect}
\usepackage{upgreek}
\usepackage{chngcntr}
\usepackage{datatool}
\usepackage{lastpage}
\usepackage{microtype}
\usepackage[foot]{amsaddr}
\usepackage{xpatch}
\usepackage[
bookmarksopen=true,
bookmarksopenlevel=1,
colorlinks=true,
linkcolor=darkblue,
linktoc=page,
citecolor=darkblue,
]{hyperref}

\usetikzlibrary{math, calc}

\makeatletter
\newcommand{\nAlph}[1]{\@Alph{#1}}
\newcommand{\nalph}[1]{\@alph{#1}}
\makeatother
\newcommand{\ngreekname}[1]{\ifcase#1\or alpha\or beta\or gamma\or delta\or varepsilon\or zeta\or eta\or vartheta\or iota\or kappa\or lambda\or mu\or nu\or xi\or omicron\or pi\or rho\or sigma\or tau\or upsilon\or phi\or chi\or psi\or omega\else \@ctrerr \fi}
\newcommand{\nGreekname}[1]{\ifcase#1\or Alpha\or Beta\or Gamma\or Delta\or Theta\or Lambda\or Pi\or Sigma\or Upsilon\or Phi\or Psi\or Omega\else \@ctrerr \fi}
\newcommand{\ngreek}[1]{\expandafter\csname\ngreekname{#1}\endcsname}
\newcommand{\nGreek}[1]{\expandafter\csname\nGreekname{#1}\endcsname}

\newcommand{\safeEdef}[2]{
	\ifcsname#1\endcsname
		\errmessage{Command \csname#1\endcsname already defined}
	\else
		\expandafter\edef\csname#1\endcsname{#2}
	\fi
}

\let\hbar\undefined
\newcount\tmp
\loop
\safeEdef{b\nAlph{\tmp}}{\noexpand\mathbb{\nAlph{\tmp}}}
\safeEdef{c\nAlph{\tmp}}{\noexpand\mathcal{\nAlph{\tmp}}}
\safeEdef{cc\nAlph{\tmp}}{\noexpand\mathscr{\nAlph{\tmp}}}
\safeEdef{\nalph{\tmp}hat}{\noexpand\hat{\nalph{\tmp}}}
\safeEdef{\nAlph{\tmp}hat}{\noexpand\hat{\nAlph{\tmp}}}
\safeEdef{\nalph{\tmp}bar}{\noexpand\overline{\nalph{\tmp}}}
\safeEdef{\nAlph{\tmp}bar}{\noexpand\overline{\nAlph{\tmp}}}
\safeEdef{\nalph{\tmp}tilde}{\noexpand\tilde{\nalph{\tmp}}}
\safeEdef{\nAlph{\tmp}tilde}{\noexpand\tilde{\nAlph{\tmp}}}
\safeEdef{\nalph{\tmp}vec}{\noexpand\vec{\nalph{\tmp}}}
\safeEdef{bf\nalph{\tmp}}{\noexpand\mathbf{\nalph{\tmp}}}
\safeEdef{frak\nalph{\tmp}}{\noexpand\mathfrak{\nalph{\tmp}}}
\safeEdef{frak\nAlph{\tmp}}{\noexpand\mathfrak{\nAlph{\tmp}}}
\safeEdef{txt\nAlph{\tmp}}{\noexpand\text{\nAlph{\tmp}}}
\safeEdef{txt\nalph{\tmp}}{\noexpand\text{\nalph{\tmp}}}
\advance\tmp by 1
\ifnum\tmp<27
\repeat

\newcount\tmp
\loop
\safeEdef{\ngreekname{\tmp}hat}{\noexpand\hat{\ngreek{\tmp}}}
\safeEdef{\ngreekname{\tmp}bar}{\noexpand\overline{\ngreek{\tmp}}}
\safeEdef{\ngreekname{\tmp}tilde}{\noexpand\tilde{\ngreek{\tmp}}}
\advance\tmp by 1
\ifnum\tmp<25
\repeat

\newcount\tmp
\loop
\safeEdef{\nGreekname{\tmp}hat}{\noexpand\hat{\nGreek{\tmp}}}
\safeEdef{\nGreekname{\tmp}bar}{\noexpand\overline{\nGreek{\tmp}}}
\safeEdef{\nGreekname{\tmp}tilde}{\noexpand\tilde{\nGreek{\tmp}}}
\advance\tmp by 1
\ifnum\tmp<12
\repeat

\newcommand{\eps}{\varepsilon}

\makeatletter
\NewDocumentCommand{\operator}{%
	m %
	m %
	m %
	m %
}{%
	\expandafter\def\csname #1\endcsname{\expandafter\@ifstar{\csname #1@star\endcsname}{\csname #1@nostar\endcsname}}
	\expandafter\NewDocumentCommand\csname #1@star\endcsname{ O{#2} m }{##1\mathopen{}\left#3##2\right#4\mathclose{}}
	\expandafter\NewDocumentCommand\csname #1@nostar\endcsname{ O{} O{#2} m }{##2\mathopen##1#3##3\mathclose##1#4}
}
\makeatother

\operator{pr}{\bP}{[}{]}
\operator{ex}{\bE}{[}{]}
\operator{set}{}{\{}{\}}

\makeatletter
\NewDocumentCommand{\coperator}{%
	m %
	m %
	m %
	m %
	m %
}{%
	\expandafter\def\csname #1\endcsname{\expandafter\@ifstar{\csname #1@star\endcsname}{\csname #1@nostar\endcsname}}
	\expandafter\NewDocumentCommand\csname #1@star\endcsname{ O{#2} m m }{##1\mathopen{}\left#3##2\mathrel{}\middle#4\mathrel{}##3\right#5\mathclose{}}
	\expandafter\NewDocumentCommand\csname #1@nostar\endcsname{ O{} O{#2} m m }{##2\mathopen##1#3##3\mathrel##1#4##4\mathclose##1#5}
}
\makeatother

\coperator{cpr}{\bP}{[}{|}{]}
\coperator{cex}{\bE}{[}{|}{]}
\DeclareMathDelimiter{\colonDelimiter}{\mathord}{operators}{"3A}{operators}{"3A}
\coperator{cset}{}{\{}{\colonDelimiter}{\}}
\coperator{ctup}{}{\lparen}{\colonDelimiter}{\rparen}

\DeclarePairedDelimiter{\paren}{\lparen}{\rparen}
\let\brack\undefined
\DeclarePairedDelimiter{\brack}{\lbrack}{\rbrack}
\DeclarePairedDelimiter{\abs}{\lvert}{\rvert}
\DeclarePairedDelimiter{\floor}{\lfloor}{\rfloor}
\DeclarePairedDelimiter{\ceil}{\lceil}{\rceil}

\newcommand{\rightharpoonupline}{\mathchoice%
	{\clipbox{{0.0\width} {0.3\height} {0.7\width} {-0.425\height}}{$\scriptstyle\rightharpoonup$}}
	{\clipbox{{0.0\width} {0.3\height} {0.7\width} {-0.425\height}}{$\scriptstyle\rightharpoonup$}}
	{\clipbox{{0.0\width} {0.4\height} {0.7\width} {-0.425\height}}{$\scriptscriptstyle\rightharpoonup$}}
	{\clipbox{{0.0\width} {0.4\height} {0.7\width} {-0.425\height}}{$\scriptscriptstyle\rightharpoonup$}}
}

\newcommand{\rightharpoonupend}{\mathchoice%
	{\clipbox{{.675\width} {0.3\height} 0pt {-0.425\height}}{$\scriptstyle\rightharpoonup$}}
	{\clipbox{{.675\width} {0.3\height} 0pt {-0.425\height}}{$\scriptstyle\rightharpoonup$}}
	{\clipbox{{.675\width} {0.4\height} 0pt {-0.425\height}}{$\scriptscriptstyle\rightharpoonup$}}
	{\clipbox{{.675\width} {0.4\height} 0pt {-0.425\height}}{$\scriptscriptstyle\rightharpoonup$}}
}
\makeatletter
\newcommand{\overrightharpoonup}[1]{\mathchoice%
	{\vbox{\m@th\ialign{##\cr$\displaystyle\hbox{$\displaystyle\rightharpoonupline$}\mkern-1mu\cleaders\hbox{$\displaystyle\mkern-2mu\rightharpoonupline$}\hfill\mkern-2mu\rightharpoonupend$\cr\noalign{\nointerlineskip\vspace{-0pt}}$\displaystyle #1$\cr}}}
	{\vbox{\m@th\ialign{##\cr$\textstyle\hbox{$\textstyle\rightharpoonupline$}\mkern-1mu\cleaders\hbox{$\textstyle\mkern-2mu\rightharpoonupline$}\hfill\mkern-2mu\rightharpoonupend$\cr\noalign{\nointerlineskip\vspace{-0pt}}$\textstyle #1$\cr}}}
	{\vbox{\m@th\ialign{##\cr$\scriptstyle\hbox{$\scriptstyle\rightharpoonupline$}\mkern-1mu\cleaders\hbox{$\scriptstyle\mkern-2mu\rightharpoonupline$}\hfill\mkern-2mu\rightharpoonupend$\cr\noalign{\nointerlineskip\vspace{-0pt}}$\scriptstyle #1$\cr}}}
	{\vbox{\m@th\ialign{##\cr$\scriptscriptstyle\hbox{$\scriptscriptstyle\rightharpoonupline$}\mkern-1mu\cleaders\hbox{$\scriptscriptstyle\mkern-2mu\rightharpoonupline$}\hfill\mkern-2mu\rightharpoonupend$\cr\noalign{\nointerlineskip\vspace{-0pt}}$\scriptscriptstyle #1$\cr}}}
}
\makeatother

\newcommand{\ind}{\mathbbm{1}}
\newcommand{\stand}{\text{ and }}

\newcommand{\qtand}{\quad\text{and}\quad}
\newcommand{\qtor}{\quad\text{or}\quad}
\newcommand{\stforall}{\text{ for all }}
\newcommand{\stforsome}{\text{ for some }}
\operator{pri}{\bP_i}{[}{]}
\operator{exi}{\bE_i}{[}{]}
\newcommand{\Xleq}{\leq_{\cX}}
\newcommand{\Xeq}{=_{\cX}}
\newcommand{\Xgeq}{\geq_{\cX}}
\renewcommand{\theta}{\vartheta}

\def\phi{\varphi}
\newcommand{\injection}{\hookrightarrow}
\newcommand{\bijection}{\mathchoice
	{\xrightarrow{\smash{\raisebox{-3pt}{$\scriptstyle\sim$}}}}%
	{\xrightarrow{\smash{\raisebox{-3pt}{$\scriptstyle\sim$}}}}%
	{\xrightarrow{\smash{\raisebox{-2.5pt}{$\scriptscriptstyle\sim$}}}}%
	{\xrightarrow{\smash{\raisebox{-2.5pt}{$\scriptscriptstyle\sim$}}}}%
}
\DeclareMathOperator{\Aut}{Aut}
\DeclareMathOperator{\aut}{aut}

\makeatletter
\newcommand{\@pom}[2]{\phantom{#1-}\llap{\smash{\ooalign{$#1-$\cr\hidewidth\rotatebox[origin=c]{45}{$#1+$}\hidewidth\cr}}}}
\newcommand{\pom}{{\mathbin{\mathpalette\@pom\relax}}}
\makeatother
\newcommand{\R}{\text{\normalfont\ttfamily r}}
\newcommand{\cleq}{\preccurlyeq}
\DeclareMathOperator*{\argmax}{arg\,max}

\newcommand{\comp}{{\mathsf{c}}}
\makeatletter
	\newcommand{\restr}{\@ifstar\restr@star\restr@nostar}
	\newcommand{\restr@star}[2]{\left.#1\right|_{#2}}
	\newcommand{\restr@nostar}[3][]{#2#1|_{#3}}
\makeatother
\newcommand{\om}{\mathord{-}}

\definecolor{darkblue}{rgb}{0,0,0.5}
\definecolor{darkgreen}{rgb}{0,0.5,0}

\newtheorem{theorem}[algocf]{Theorem}

\newtheorem{lemma}[algocf]{Lemma}
\newtheorem{observation}[algocf]{Observation}

\newtheorem{conjecture}[algocf]{Conjecture}

\theoremstyle{definition}

\numberwithin{equation}{section}

\allowdisplaybreaks

\DTLnewdb{realRV}
\DTLnewdb{RV}
\DTLnewdb{time}
\DTLnewdb{stoppingtime}
\DTLnewdb{trajectory}
\DTLnewdb{real}
\DTLnewdb{construction}

\DTLnewdb{realRVII}
\DTLnewdb{RVII}
\DTLnewdb{timeII}
\DTLnewdb{stoppingtimeII}
\DTLnewdb{trajectoryII}
\DTLnewdb{realII}
\DTLnewdb{constructionII}
\newcommand{\gladd}[3]{\DTLnewrow{#1}\DTLnewdbentry{#1}{description}{#2}\DTLnewdbentry{#1}{content}{#3}}

\newcommand{\glprint}[1]{%
	\DTLifdbempty{#1}{}{
		\DTLsort{description}{#1}%
		\begin{itemize}[label={},leftmargin=0pt]
			\DTLforeach*{#1}{\theContent=content}{\item \theContent}
	\end{itemize}}
}

\newif\ifsymbols
\symbolstrue

\newcommand{\sectn}[1]{{\centering\textsc{#1}\par\vspace*{2.3ex plus .2ex}}}

\ifsymbols
\AtEndDocument{\newpage\thispagestyle{empty}\sectn{List of symbols in Sections~\ref{section: heuristics}--\ref{section: proof} and Appendices~\ref{appendix: copies}--\ref{appendix: strictly balanced}}
	\noindent\begin{minipage}[t]{0.5\textwidth}
		\textbf{Real-valued random variables}
		\glprint{realRV}
		\vspace{1cm}
		\textbf{Steps}
		\glprint{time}
	\end{minipage}\hfill
	\begin{minipage}[t]{0.45\textwidth}
		\textbf{Trajectories}
		\glprint{trajectory}
		\vspace{1cm}
		\textbf{Constants}
		\glprint{real}
	\end{minipage}

	\vspace{1cm}
	\noindent\textbf{Random sets}
	\glprint{RV}
	\vspace{1cm}
	\textbf{Stopping times}
	\glprint{stoppingtime}
	\vspace{1cm}
	\newpage\thispagestyle{empty}
	\noindent\textbf{Sets and Tuples}
	\glprint{construction}
	\vspace*{3.5ex plus 1ex minus .2ex}
	\sectn{List of symbols in Sections~\ref{section: sparse}--\ref{section: isolation argument} and Appendix~\ref{appendix: cherries}}
	\noindent
	\noindent\begin{minipage}[t]{0.5\textwidth}
		\textbf{Real-valued random variables}
		\glprint{realRVII}
		\vspace{1cm}
		\textbf{Steps}
		\glprint{timeII}
	\end{minipage}\hfill
	\begin{minipage}[t]{0.45\textwidth}
		\textbf{Trajectories}
		\glprint{trajectoryII}
		\vspace{1cm}
		\textbf{Constants}
		\glprint{realII}
	\end{minipage}

	\vspace{1cm}
	\noindent\textbf{Stopping times}
	\glprint{stoppingtimeII}
	\vspace{1cm}
	\textbf{Sets and Tuples}
	\glprint{constructionII}}
\fi

\newif\ifimages
\imagestrue

\title[The hypergraph removal process]{The hypergraph removal process}

\author[F.~Joos]{Felix Joos}
\author[M.~K\"uhn]{Marcus K\"uhn}

\address{Institut f\"ur Informatik, Universit\"at Heidelberg, Germany}
\email{[joos, kuehn]@informatik.uni-heidelberg.de}

\date{August 03, 2025}

\thanks{The research leading to these results was supported by the Deutsche Forschungsgemeinschaft (DFG, German Research Foundation) -- 428212407}

\begin{document}

	\begin{abstract}
		Let~$k\geq 2$ and fix a~$k$-uniform hypergraph~$\cF$.
		Consider the random process that, starting from a~$k$-uniform hypergraph~$\cH$ on~$n$ vertices, repeatedly deletes the edges of a copy of~$\cF$ chosen uniformly at random and terminates when no copies of~$\cF$ remain.
		Let~$R(\cH,\cF)$ denote the number of edges that are left after termination.
		We show that~$R(\cH,\cF)=n^{k-1/\rho\pm o(1)}$, where $\rho:=(\abs{E(\cF)}-1)/(\abs{V(\cF)}-k)$, holds with high probability provided that~$\cF$ is strictly~$k$-balanced and~$\cH$ is sufficiently dense with pseudorandom properties.
		Since we may in particular choose~$\cF$ and~$\cH$ to be complete graphs, this confirms the major folklore conjecture in the area in a very strong form.
	\end{abstract}

	\maketitle

	\section{Introduction}\label{section: introduction}
	
	Let~$\cF$ be a~$k$-uniform hypergraph where~$k\geq 2$.
	Consider the following two simple random processes for generating~$\cF$-free hypergraphs that were proposed by Bollob\'as and Erd\H{o}s at the \enquote{Quo Vadis, Graph Theory?} conference in 1990.
	Starting with an empty hypergraph on~$n$ vertices, the~\emph{$\cF$-free process} iteratively proceeds as follows.
	Among all vertex sets of size~$k$ that were not previously added and that do not form the edge set of a copy of~$\cF$ with previously added edges, a vertex set is chosen uniformly at random and added as an edge.
	The process terminates when no such vertex sets remain.
	Conversely, starting with a complete~$k$-uniform hypergraph on~$n$ vertices, the~\emph{$\cF$-removal process} iteratively removes all edges of a copy of~$\cF$ chosen uniformly at random among all remaining copies of~$\cF$ until no copies are left.
	
	Besides generating hypergraphs without copies of~$\cF$, the~$\cF$-removal process also yields maximal packings of edge-disjoint copies of~$\cF$ and is furthermore a special case of the random greedy hypergraph matching algorithm.
	Indeed, assuming that~$\cH$ is the complete graph at the start of the process, consider the~$\abs{E(\cF)}$-uniform hypergraph~$\cH^*$ with vertex set~$E(\cH)$ whose edges are the edge sets of the copies of~$\cF$ in~$\cH$.
	Then, the random greedy hypergraph matching algorithm in~$\cH^*$ that builds a matching by iteratively adding an edge chosen uniformly at random among all edges that are disjoint from all previously selected edges directly corresponds to the~$\cF$-removal process in the sense that it generates the same structures using equivalent objects.
	Specifically, in this correspondence the selected edges are simply the edge sets of the chosen copies.
	Many variations and special cases of the random greedy hypergraph matching algorithm have been investigated, see for example~\cite{AKS:97,BB:19,BFL:15,GJKKL:24,G:97,RT:96,S:95,W:99}.
	
	Such random processes are easy to formulate, in many cases however, a precise analysis is challenging.
	The central questions often concern structural properties that typically, that is, with high probability (with probability tending to $1$ as $n\to \infty$), hold for the objects generated at termination.
	In particular, concerning the~$\cF$-free and~$\cF$-removal process, one may ask for asymptotic estimates for the number of edges or equivalently the number of iterations of the algorithm.
	For the~$\cF$-free process on~$n$ vertices, we use~$F_n(\cF)$ to denote the (random) final number of edges present after termination and for the~$\cF$-removal process, we use~$R_n(\cF)$.
	It is interesting to compare the history of the analysis of both processes in detail, see Section~\ref{subsection: comparison}.
		
	For the special case of the~$K_3$-free process, that is, where~$\cF$ is a triangle, Fiz Pontiveros, Griffiths and Morris~\cite{PGM:20} and independently Bohman and Keevash~\cite{BK:21} famously proved that typically~$F_n(K_3)=(\frac{1}{2\sqrt{2}}\pm o(1))(\log n)^{1/2}n^{3/2}$ (after Bohman determined the correct order of magnitude~\cite{B:09}, answering a question of Spencer~\cite{S:95b}).
	For the general case, a lower bound for~$F_n(\cF)$ that holds with high probability is available whenever~$\cF$ comes from a large class of graphs or hypergraphs~\cite{BB:16,BK:10}.
	At least for graphs, this lower bound is conjectured to be tight up to constant factors~\cite{BK:10}, however in general, the best upper bounds that are known to hold with high probability differ from this lower bound by logarithmic factors~\cite{KOT:16}.
	Estimates for~$F_n(\cF)$ that are tight up to constant factors exist for a few specific choices of~$\cF$, see~\cite{P:11,P:14a,P:14b,W:14a,W:14b}.
	
	For the~$\cF$-removal process, already getting close to the order of magnitude of~$R_n(K_3)$ turned out to be challenging.
	First, Spencer~\cite{S:95} as well as R\"odl and Thoma~\cite{RT:96} proved that~$R_n(K_3)=o(n^2)$ typically holds and Grable~\cite{G:97} improved this to $R_n(K_3)\leq n^{11/6+o(1)}$.
	Following these attempts to determine~$R_n(K_3)$, Spencer conjectured that typically $R_n(K_3)=n^{3/2\pm o(1)}$ holds and offered~\$200 for a resolution~\cite{G:97,W:99}.
	The breakthrough here happened when Bohman, Frieze and Lubetzky proved Spencer's conjecture~\cite{BFL:15}.
	Beyond the triangle, no results are known that give bounds that are somewhat close to the correct order of magnitude of~$R_n(\cF)$ for any other $\cF$;
	in fact, obtaining asymptotic estimates for~$R_n(K_4)$ is considered a central open problem in the area.
	(One reason for why this is a difficult problem may be that the technical complexity of the approach taken by Bohman, Frieze and Lubetzky to settle the triangle case seems to explode even for $\cF=K_4$.)
	Following the same heuristic as for the triangle, Bennett and Bohman~\cite{BB:19} state the following more general \enquote{folklore} conjecture predicting~$R_n(\cF)$ whenever~$\cF$ is the~$k$-uniform complete hypergraph~$K^{(k)}_\ell$ on~$\ell$ vertices.
	\begin{conjecture}[{\cite[Conjecture~1.2]{BB:19}}]\label{conjecture}
		Let~$2\leq k<\ell$.
		Then, for all~$\eps>0$, there exists~$n_0\geq 0$ such that for all~$n\geq n_0$, with high probability,
		\begin{equation*}
			n^{k-\frac{\ell-k}{\binom{\ell}{k}-1}-\eps}\leq R_n(K^{(k)}_\ell)\leq n^{k-\frac{\ell-k}{\binom{\ell}{k}-1}+\eps}.
		\end{equation*}
	\end{conjecture}
	
	Our main result confirms Conjecture~\ref{conjecture}.
	In fact, we prove a significantly stronger result.
	For a~$k$-uniform hypergraph~$\cF$, using~$v(\cF)$ to denote the number of vertices of~$\cF$ and~$e(\cF)$ to denote the number of edges of~$\cF$, the~\emph{$k$-density} of~$\cF$ is~$\rho_\cF:=(e(\cF)-1)/(v(\cF)-k)$ if~$v(\cF)\geq k+1$.
	As in~\cite{BK:10}, we say that~$\cF$ is \emph{strictly~$k$-balanced} if~$\cF$ has at least three edges and satisfies~$\rho_\cG<\rho_\cF$ for all proper subgraphs~$\cG$ of~$\cF$ that have at least two edges.
	Note that~$K_\ell^{(k)}$ is strictly~$k$-balanced for all~$2\leq k<\ell$.
	The following is a corollary of our main result (Theorem~\ref{theorem: pseudorandom}).
	 
	\begin{theorem}\label{theorem: main}
		Let~$k\geq 2$ and consider a strictly~$k$-balanced~$k$-uniform hypergraph~$\cF$ with~$k$-density~$\rho$.
		Then, for all~$\eps>0$, there exists~$n_0\geq 0$ such that for all~$n\geq n_0$, with probability at least~$1-\exp(-(\log n)^{5/4})$, we have
		\begin{equation*}
			n^{k-1/\rho-\eps}\leq R_n(\cF)\leq n^{k-1/\rho+\eps}.
		\end{equation*}
	\end{theorem}
	
	Observe that complete (hyper)graphs exhibit a very high degree of symmetry while most strictly~$k$-balanced hypergraphs have locally and globally essentially no symmetries.
	This complicates the analysis and requires us to dedicate substantial parts of the proof to dealing with the extension from cliques to general strictly~$k$-balanced hypergraphs.
	
	Furthermore, our analysis allows starting at any pseudorandom hypergraph,
	which may be a useful scenario for applications.
	In more detail, 
	given a~$k$-uniform hypergraph~$\cH$, we consider the~$\cF$-removal process \emph{starting at~$\cH$} that, now starting with~$\cH$ instead of~$K^{(k)}_n$, again iteratively removes all edges of a copy of~$\cF$ chosen uniformly at random among all remaining copies of~$\cF$ until no copies are left.
	For a~$k$-uniform hypergraph~$\cH$, we use~$R(\cH,\cF)$ to denote the final number of edges of the~$\cF$-removal process starting at~$\cH$.
	
	To formally describe the pseudorandomness we require for our theorem, we introduce the following definitions.
	A~\emph{$k$-graph} is a~$k$-uniform hypergraph and a~$k$-uniform \emph{template} or~\emph{$k$-template} is a pair~$(\cA,I)$ where~$\cA$ is a~$k$-graph and where~$I\subseteq V(\cA)$.
	The \emph{density~$\rho_{\cA,I}$}\gladd{real}{rhoAI}{$\rho_{\cA,I}=\frac{\abs{\cA}-\abs{\cA[I]}}{\abs{V_\cA}-\abs{I}}$} of~$(\cA,I)$ is~$(e(\cA)-e(\cA[I]))/(v(\cA)-\abs{I})$ if~$V(\cA)\neq I$ and~$0$ otherwise where we use~$\cA[I]$ to denote the subgraph of~$\cA$ induced by~$I$.
	A template~$(\cB,J)$ is a \emph{subtemplate} of~$(\cA,I)$ if~$\cB\subseteq \cA$ and~$J=I$.
	We write~$(\cB,J)\subseteq (\cA,I)$ to mean that~$(\cB,J)$ is a subtemplate of~$(\cA,I)$.
	The template~$(\cA,I)$ is \emph{strictly balanced} if~$\rho_{\cB,I}<\rho_{\cA,I}$ holds for all~$(\cB,I)\subseteq (\cA,I)$ with~$V_\cB\neq I$ and~$\cB\neq\cA$.
	Note that for a~$k$-graph~$\cA$ with~$v(\cA)\geq k+1$, the~$k$-density of~$\cA$ is the density of the templates~$(\cA,e)$ with~$e\in E(\cA)$ and that if~$\cA$ has at least three edges, then~$\cA$ is strictly~$k$-balanced if and only if~$(\cA,e)$ is strictly balanced for all~$e\in\cA$.
	For~$0< \eps,\delta< 1$ and~$\rho\geq 1/k$, we say that a~$k$-graph~$\cH$ on~$n$ vertices with~$\theta n^k/k!$ edges is~\emph{$(\eps,\delta,\rho)$-pseudorandom} if for all strictly balanced~$k$-templates~$(\cA,I)$ with~$v(\cA)\leq 1/\eps$ and all injections~$\psi\colon I\to V(\cH)$, the number~$\Phi$ of injections~$\phi\colon V(\cA)\to V(\cH)$ with~$\restr{\phi}{I}=\psi$ and~$\phi(e)\in E(\cH)$ for all~$e\in E(\cA)\setminus E(\cA[I])$ satisfies the properties~\ref{item: pseudorandom I}--\ref{item: pseudorandom IV} below.
	Here, we set~$\phihat:=n^{v(\cA)-v(\cA[I])}\theta^{e(\cA)-e(\cA[I])}$ and~$\zeta:=n^{\delta}/(n\theta^\rho)^{1/2}$.
	\begin{enumerate}[label=(P\arabic*)]
		\item\label{item: pseudorandom I} If~$\rho_{\cA,I}\leq \rho$, then~$\Phi=(1\pm \zeta)\phihat$;
		\item If~$\phihat\geq \zeta^{-\delta^{2/3}}$, then~$\Phi=(1\pm \zeta^\delta)\phihat$;
		\item If~$1\leq \phihat\leq \zeta^{-\delta^{2/3}}$, then~$\Phi=(1\pm (\log n)^{3(v(\cA)-v(\cA[I]))/2}\phihat^{-\delta^{1/2}})\phihat$.
		\item\label{item: pseudorandom IV} If~$\phihat\leq 1$, then~$\Phi\leq (\log n)^{3(v(\cA)-v(\cA[I]))/2}$.
	\end{enumerate}
	We remark that for all~$k\geq 2$ and~$0<\eps,\delta<1$ and~$\rho\geq 1/k$ where~$\delta$ is sufficiently small in terms of~$1/k$ and~$\eps$, the~$k$-uniform binomial random graph on~$n$ vertices where all vertex sets of size~$k$ are edges independently with probability~$p\geq n^{-1/\rho+\delta^{1/2}}$ is~$(\eps,\delta,\rho)$-pseudorandom with high probability.
	Indeed, Chernoff's inequality (see Lemma~\ref{lemma: chernoff}) guarantees~$\theta\geq n^{-1/\rho+3\delta}$ with high probability, sufficient lower tail bounds follow from Janson's inequality (see~\cite[Theorem~1]{J:90}) and for the upper tails, one may apply~\cite[Corollary~4.1]{JR:04}.

	We are now ready to state our main theorem.
	
	\begin{theorem}\label{theorem: pseudorandom}
		Let~$k\geq 2$ and consider a strictly~$k$-balanced~$k$-graph~$\cF$ with~$k$-density~$\rho$.
		Then, for all~$\eps>0$, there exists~$\delta_0>0$ such that for all~$0<\delta<\delta_0$, there exists~$n_0\geq 0$ such that for all~$n\geq n_0$, the following holds.
		If~$\cH$ is a~$(\eps^{20},\delta,\rho)$-pseudorandom~$k$-graph on~$n$ vertices with~$e(\cH)\geq n^{k-1/\rho+\eps^5}$, then, with probability at least~$1-\exp(-(\log n)^{5/4})$, we have
		\begin{equation*}
			n^{k-1/\rho-\eps}\leq R(\cH,\cF)\leq n^{k-1/\rho+\eps}.
		\end{equation*}
	\end{theorem}

	We prove the upper bound in Theorem~\ref{theorem: pseudorandom} in a slightly more general setting in the sense that we only require a weaker notion of balancedness.
	We say that a~$k$-graph~$\cF$ is~\emph{$k$-balanced} if~$\cF$ has at least one edge and satisfies~$\rho_\cG\leq\rho_\cH$ for all subgraphs~$\cG$ of~$\cH$ on at least~$k+1$ vertices.
	
	\begin{theorem}\label{theorem: only upper bound}
		Let~$k\geq 2$ and consider a~$k$-balanced~$k$-graph~$\cF$ with~$k$-density~$\rho$.
		Then, for all~$\eps>0$, there exists~$\delta_0>0$ such that for all~$0<\delta<\delta_0$, there exists~$n_0\geq 0$ such that for all~$n\geq n_0$, the following holds.
		If~$\cH$ is a~$(\eps^{20},\delta,\rho)$-pseudorandom~$k$-graph on~$n$ vertices with~$e(\cH)\geq n^{k-1/\rho+\eps^5}$, then, with probability at least~$1-\exp(-(\log n)^{5/4})$, we have
		\begin{equation*}
			R(\cH,\cF)\leq n^{k-1/\rho+\eps}.
		\end{equation*}
	\end{theorem}

	As part of our proof for Theorem~\ref{theorem: pseudorandom}, we obtain another theorem which describes the behavior of the~$\cF$-removal process starting at~$\cH$ for comparatively sparse~$\cH$ which complements Theorem~\ref{theorem: pseudorandom}.
	To formally describe the slightly different setup for this theorem in the sparse setting, we introduce the following definitions.
	For~$s,c\geq 0$, we say that a~$k$-graph~$\cH$ with~$\theta n^k/k!$ edges is~\emph{$(s,c)$-bounded} if for all strictly balanced templates~$(\cA,I)$ with~$v(\cA)\leq s$, all injections~$\psi\colon I\to V(\cH)$ and~$\phihat:=n^{v(\cA)-\abs{I}}\theta^{e(\cA)-e(\cA[I])}$, the number of injections~$\phi\colon V(\cA)\to V(\cH)$ with~$\restr{\phi}{I}=\psi$ and~$\phi(e)\in\cH$ for all~$e\in\cA$ with~$e\not\subseteq I$ is at most~$c\cdot\max\set{1,\phihat}$.
	We say that~$\cH$ is~\emph{$\cF$-populated} if all edges of~$\cH$ are edges of at least two copies of~$\cF$ in~$\cH$.

	\begin{theorem}\label{theorem: sparse}
		Let~$k\geq 2$ and suppose that~$\cF$ is a strictly~$k$-balanced~$k$-graph on~$m$ vertices with~$k$-density~$\rho$.
		For all~$\eps>0$, there exists~$n_0$ such that for all~$n\geq n_0$ and all~$(4m,n^{\eps^4})$-bounded and~$\cF$-populated~$k$-graphs~$\cH$ on~$n$ vertices with~$n^{k-1/\rho-\eps^4}\leq e(\cH)\leq n^{k-1/\rho+\eps^4}$, with probability at least~$1-\exp(-n^{1/4})$, we have
		\begin{equation*}
			R(\cH,\cF)\geq n^{k-1/\rho-\eps}.
		\end{equation*}
	\end{theorem}

	Recall that by definition,~$\cF$ is strictly~$k$-balanced if and only if~$\cF$ has at least three edges and satisfies~$\rho_{\cG}<\rho_\cF$ for all proper subgraphs~$\cG$ of~$\cF$ that have at least two edges. 
	Hence, Theorems~\ref{theorem: pseudorandom} and~\ref{theorem: sparse} do not cover the case where $e(\cF)=2$,
	but it is possible to also obtain a similar statement for this case.
	If~$\cF$ is a matching (of size $2$), then~$\cF$ has~$k$-density~$1/k$, so in this case the lower bounds in these theorems is always true (if we round down) and hence we ignore this case.
	For the case where~$\cF$ has exactly two edges but is not a matching, we obtain the following two theorems.
	For~$1\leq k'\leq k$, we say that~$\cH$ is~\emph{$k'$-populated} if all sets~$U\subseteq V(\cH)$ with~$\abs{U}=k'$ are contained in at least two edges of~$\cH$.
	
	\begin{theorem}\label{theorem: cherries}
		Let~$k\geq 2$ and consider a~$k$-graph~$\cF$ with~$k$-density~$\rho$ that is not a matching, has exactly two edges and no isolated vertices.
		Then, for all~$\eps>0$, there exists~$\delta_0>0$ such that for all~$0<\delta<\delta_0$, there exists~$n_0\geq 0$ such that for all~$n\geq n_0$, the following holds.
		If~$\cH$ is a~$(\eps^{20},\delta,\rho)$-pseudorandom~$k$-graph on~$n$ vertices with~$e(\cH)\geq n^{k-1/\rho+\eps^5}$, then, with probability at least~$1-\exp(-(\log n)^{5/4})$, we have
		\begin{equation*}
			n^{k-1/\rho-\eps}\leq R(\cH,\cF)\leq n^{k-1/\rho+\eps}.
		\end{equation*}
	\end{theorem}
	
	\begin{theorem}\label{theorem: sparse cherries}
		Let~$k\geq 2$ and suppose that~$\cF$ is a~$k$-graph with~$k$-density~$\rho$ that is not a matching, has exactly two edges and no isolated vertices.
		Let~$k':=\abs{e\cap f}$ where~$e$ and~$f$ denote the edges of~$\cF$.
		For all~$\eps>0$, there exists~$n_0$ such that for all~$n\geq n_0$ and all~$(4m,n^{\eps^4})$-bounded and~$k'$-populated~$k$-graphs~$\cH$ on~$n$ vertices with~$n^{k-1/\rho-\eps^4}\leq e(\cH)\leq n^{k-1/\rho+\eps^4}$, with probability at least~$1-\exp(-n^{1/4})$, we have
		\begin{equation*}
			R(\cH,\cF)\geq n^{k-1/\rho-\eps}.
		\end{equation*}
	\end{theorem}

	\subsection{The history of the~\texorpdfstring{$\cF$}{F}-free and the~\texorpdfstring{$\cF$}{F}-removal process}\label{subsection: comparison}
	Modern research concerning the~$\cF$-free process began in 1992 when Ruci\'nski and Wormald~\cite{RW:92} answered a question of Erd\H{o}s regarding the~$\cF$-free process where~$\cF$ is a ($2$-uniform) star.
	Concerning triangles Spencer~\cite{S:95b} conjectured in~1995 that with high probability, the~$K_3$-free process terminates with~$\Theta((\log n)^{1/2}n^{3/2})$ edges.
	This is the behavior one would expect when assuming that edges present in a hypergraph generated during the~$\cF$-free process are essentially distributed as if they were included independently with an appropriate probability.
	We discuss this heuristic in more detail at the end in Section~\ref{section: concluding remarks}.
	
	The~$K_3$-free process as well as the variation of this process where not only triangles but all cycles of odd length are forbidden was investigated by Erd\H{o}s, Suen and Winkler~\cite{ESW:95}.
	Their result yields upper and lower bounds for~$F_n(K_3)$ that hold with high probability and are tight up to a~$\log n$ factor.
	Bollob\'as and Riordan~\cite{BR:00} obtained analogous bounds for~$F_n(\cF)$ if~$\cF$ is a complete graph or cycle on four vertices.
	In 2001, Osthus and Taraz~\cite{OT:01} generalized these results to all strictly~$2$-balanced graphs thus providing estimates for this large class of graphs that are tight up to logarithmic factors.
	
	Guided by similar intuition as above, for the~$K_3$-removal process, Spencer conjectured that with high probability, this process also terminates with~$n^{3/2\pm o(1)}$ edges (see~\cite{G:97,W:99}).
	More generally, a special case of a conjecture of Alon, Kim and Spencer~\cite{AKS:97} about hypergraph matchings predicts~$n^{k-1/\rho_\cF\pm o(1)}$ as the expected value of~$R_n(\cF)$ where~$\rho_\cF$ denotes the~$k$-density of~$\cF$.
	Concerning estimates available around 2001 however, the situation for the~$\cF$-removal process was very different compared to the~$\cF$-free process.
	Only upper bounds for~$R_n(\cF)$ that do not match the order of magnitude of~$R_n(\cF)$ were known, namely~$n^{11/6+o(1)}$ for~$R_n(K_3)$ due to Grable~\cite{G:97} and, as a consequence of a result about the random greedy hypergraph matching algorithm due to Wormald~\cite{W:99}, $n^{k-1/(9e(\cF)^2-9e(\cF)+3)+o(1)}$ for the general case.
	Intuitively, perhaps one reason that complicates the analysis of the~$\cF$-removal process compared to the~$\cF$-free process is the fact that to arrive at roughly~$n^{3/2}$ edges, the~$K_3$-free process needs to run for roughly~$n^{3/2}$ iterations while the~$K_3$-removal process requires~$(1-o(1))n^2/6$ iterations.
	
	It is worth mentioning that to obtain the general upper bound, Wormald introduced a new approach known as \emph{differential equation method} that relies on closely following the evolution of carefully chosen key quantities throughout the process. This technique turned out to be a very valuable for later improvements in the area.
	
	Using such an approach Bohman~\cite{B:09} was able to prove estimates for~$F_n(K_3)$ that are tight up to constant factors thereby confirming the aforementioned conjecture of Spencer.
	Shortly after this, Bohman and Keevash~\cite{BK:10}, again using similar techniques, obtained new lower bounds for~$F_n(\cF)$  if~$\cF$ is a strictly~$2$-balanced graph and they conjecture that these bounds are tight up to constant factors.
	In the following years, these developments led to further progress for specific choices of~$\cF$ due to Picollelli~\cite{P:11,P:14a,P:14b} as well as Warnke~\cite{W:14a,W:14b}.
	Eventually, by considering the random greedy independent set algorithm in hypergraphs, Bohman and Bennett~\cite{BB:16} extended the lower bound to the hypergraph setting.
	Generalizing the results of Osthus and Taraz~\cite{OT:01}, upper bounds in the hypergraph setting were obtained by K\"uhn, Osthus and Taylor~\cite{KOT:16}.
	
	In contrast, concerning the~$\cF$-removal process, even with these new techniques available, there were no improvements until~2015.
	Only using a refined version of the differential equation method that exploits self-correcting behavior of key quantities to improve the precision of the analysis, Bohman, Frieze and Lubetzky~\cite{BFL:15} were able to confirm Spencer's conjecture for the~$K_3$-removal process and show that with high probability,~$R_n(K_3)=n^{3/2\pm o(1)}$.
	This refined version is known as \emph{critical interval method}, for other examples, see~\cite{BFL:18,BK:21,BP:12,PGM:20,H:24,TWZ:07}. 
	Using such an approach often requires an even more careful choice of key quantities to be able to rely on self-correcting behavior as some quantities may disturb the behavior of others.
	Indeed, for their analysis Bohman, Frieze and Lubetzky give explicit constructions of very specific substructures which they count.
	These substructures and their explicit descriptions are tailored towards the triangle case and it remained unclear how to generalize these structures that are already complicated for the triangle case.
	
	Investigating again the random greedy hypergraph matching algorithm, but without similarly sophisticated substructures, Bohman and Bennett~\cite{BB:19} showed that with high probability,~$R_n(\cF)\leq n^{k-1/(2e(\cF)-2)+o(1)}$.
	This upper bound improves on Wormald's previous result, and for hypergraph matchings takes the analysis to a natural barrier, but still has not the correct order of magnitude; without the appropriate substructures, it seems impossible to rely on self-correcting behavior to the same extent that was necessary to determine the order of magnitude of~$R_n(K_3)$.
	
	In a landmark result Fiz Pontiveros, Griffiths and Morris~\cite{PGM:20} and independently Bohman and Keevash~\cite{BK:21} asymptotically determined the typically encountered final number of edges in the triangle-free process with the correct constant factor, that is, they showed that typically, the final number of edges is~$(\frac{1}{2\sqrt{2}}\pm o(1))(\log n)^{1/2}n^{3/2}$.
	Furthermore, together with bounds for the independence number of the eventually generated graph, for large~$t$, this yields an improved lower bound for the Ramsey numbers~$R(3,t)$.
	These results also rely on the exploitation of self-correcting behavior by considering carefully chosen key quantities, which further highlights the power of this technique.
	
	For our proof, we also take such an approach.
	To overcome the seemingly exploding complexity of the necessary substructures, even when generalizing the approach of Bohman, Frieze and Lubetzky to the case where~$\cF=K_4$, instead of giving explicit constructions, we develop an implicit way of selecting the appropriate key quantities.
	This forces us to argue without explicit knowledge of the structures that we investigate which makes the nature of our proof significantly more abstract.
	One may argue that this implicit choice is the main step for the proof of Theorem~\ref{theorem: only upper bound} for cliques.
	For general strictly $k$-balanced hypergraphs, we introduce a symmetrization approach as a further crucial ingredient for our proof.

	\section{Outline of the proof}\label{section: outline}

	To determine when the~$\cF$-removal process terminates, we crucially rely on closely tracking the evolution of the numbers of occurrences of certain key substructures within the random hypergraphs generated by the iterated removal of the randomly chosen copies of~$\cF$.
	We do this essentially all the way until the point where we would expect no more remaining copies and this tracking constitutes the heart of our proof.
	The main obstacle here lies in selecting appropriate substructures that allow us to carry out such an analysis with sufficient precision for the necessary number of steps.
	When we finally arrive at a step where typically only few copies remain, the structural insights that the knowledge of these key quantities provide allow us to apply Theorem~\ref{theorem: sparse} or Theorem~\ref{theorem: sparse cherries} to show that then, the~$\cF$-removal process typically quickly terminates such that the overall runtime is as expected.
	The proof of Theorem~\ref{theorem: sparse} and Theorem~\ref{theorem: sparse cherries} relies on an argument that is separate from the analysis of the algorithm up to the point where typically only few copies remain and we present it at the end of the paper starting in Section~\ref{section: sparse}.
	
	The number of copies of $\cF$ still present in~$\cH$ is one obvious example for one of the aforementioned key quantities that is crucial for understanding the behavior and following the evolution of the process.
	We employ supermartingale concentration techniques to show that the random processes given by the key quantities that we select typically closely follow a deterministic trajectory that we deduce from heuristic considerations.
	Such an approach resembles the \emph{differential equation method} introduced by Wormald~\cite{W:99}.
	To maintain precise control over the key random processes in the sense that we can still guarantee that expected one-step changes are as suggested by intuition, we exploit a phenomenon that can be described as a self-correcting behavior certain key quantities inherently exhibit.
	Furthermore, we require precise estimates also for the quantities that determine the one-step changes of the key random processes, which often forces us to enlarge our collection.

	More specifically, let~$\cH^*$ denote the~$e(\cF)$-uniform hypergraph where the edges present at some step~$i$ form the vertex set of~$\cH^*$ and where the edge sets of copies of~$\cF$ present at step~$i$ are the edges of~$\cH^*$.
	Let~$H^*$ denote the number of edges of~$\cH^*$, that is, the number of present copies of~$\cF$.
	Let~$\frakF(0),\frakF(1),\ldots$ to denote the natural filtration associated with the~$\cF$-removal process and consider the following example.
	Assuming that for all distinct edges~$e$ and~$f$, the number of copies of~$\cF$ that contain both~$e$ and~$f$ is negligible compared to the degrees~$d_{\cH^*}(e)$ and~$d_{\cH^*}(f)$, in expectation, the one-step change~$\Delta H^*$ of the number of present copies when transitioning to the next step is
	\begin{equation}\label{equation: outline number of copies}
		\cex{\Delta H^*}{\frakF(i)}
		\approx-\sum_{\cF'\in E(\cH^*)}\sum_{e\in E(\cF')}\frac{d_{\cH^*}(e)}{H^*}.
	\end{equation}
	Note that here, the larger~$H^*$, the larger the expected decrease (we divide by~$H^*$, but the remaining copies are counted by both, the number of summands in the outer sum and the degrees).
	When considering the one-step changes of a process that measures the deviation of the number of remaining copies from an appropriate deterministic prediction, this causes a drift that, in expectation, steers the number of copies towards the prediction.
	Exploiting such self-correcting behavior turns out to be crucial for a precise analysis of the process.
	This leads to an approach often called \emph{critical-interval method}.
	Earlier applications of such an approach can be found in~\cite{BFL:18,BFL:15,BK:21,BP:12,PGM:20,TWZ:07}.
	
	Another important observation is that~\eqref{equation: outline number of copies} introduces the degrees~$d_{\cH^*}(e)$ of remaining edges~$e$ as further crucial quantities whose evolution we wish to follow using supermartingale concentration.
	As such an edge~$e$ itself could be removed during the next removal of a copy of~$\cF$, it is more convenient to instead consider the degree~$d_{\cH^*}'(e)$ of~$e$ in the hypergraph~$\cH^*_e$ obtained from~$\cH^*$ by adding~$e$ as a vertex and the edge sets of all copies~$\cF'$ of~$\cF$ where all edges~$f\in E(\cF')\setminus\set{e}$ are present as edges.
	Note that if~$e\in E(\cH^*)$, then~$\cH^*_e=\cH^*$ and~$d_{\cH^*}'(e)=d_{\cH^*}(e)$.
	Since we again aim to rely on supermartingale concentration, for a remaining edge~$e$, we are again interested in the one-step change~$\Delta d_{\cH^*}'(e)$ of~$d_{\cH^*}'(e)$ when transitioning to the next step.
	
	Similarly as above, we estimate
	\begin{equation*}
		\cex{\Delta d_{\cH^*}'(e)}{\frakF(i)}
		\approx-\sum_{\cF'\in E(\cH^*_e)\colon e\in E(\cF')}\sum_{f\in E(\cF')\setminus \set{e}}\frac{d_{\cH^*}(f)}{H^*}.
	\end{equation*}
	Since the degrees of remaining edges are included in our collection of key quantities, we have estimates available for the degrees that we could use to approximate the expected one-step changes of the degrees.
	This is a valid approach that leads to a natural barrier in the analysis, see~\cite{BB:19,BFL:18}.
	However, due to undesirable accumulation of estimation errors, such an approach is insufficient for an analysis up to the point where we may apply Theorem~\ref{theorem: sparse} or Theorem~\ref{theorem: sparse cherries}.
	
	Consider the following idea to circumvent this issue.
	If precise estimates for the number~$\Phi_e$ of substructures within~$\cH^*_e$ that consist of two copies of~$\cF$ that share an edge~$e'\neq e$ and where one copy contains~$e$ were available, we could rely on the identity
	\begin{equation*}
		\cex{\Delta d_{\cH^*}'(e)}{\frakF(i)}=-\frac{\Phi_e}{H^*}.
	\end{equation*}
	However, if we now add the random variables~$\Phi_e$ to our collection of tracked key quantities, we essentially only shifted the problem to determining the one-step changes of these new random variables and similarly iterating the extension of the collection by adding further key quantities that count substructures consisting of more and more copies of~$\cF$ overlapping at edges quickly becomes unsustainable as the collection becomes too large.
	
	The very high-level approach described so far, including the separation into an analysis of the early evolution and an analysis of the late evolution of the process, is essentially the same as in the analysis of the case where~$\cF$ is a triangle~\cite{BFL:15}.
	Consequently, the same obstacle mentioned above is encountered.
	To remedy this issue, Bohman, Frieze and Lubetzky~\cite{BFL:15} carefully control the extension of the collection of key quantities manually by giving explicit descriptions of the elements of a suitably chosen collection of structures of overlapping triangles using sequences of the symbols~$0$,~$1$ and~$\texttt{e}$.
	This collection is chosen roughly based on the above idea and its size grows with~$1/\eps$ to allow for sufficiently precise estimates, but at the cost of some however negligible precision, the collection is still sufficiently small to allow an analysis of the evolution of all the relevant random variables.
	
	Explicitly describing the relevant substructures that facilitate such an analysis seems practically infeasible for hypergraphs or even graphs larger than the triangle.
	Instead, we implicitly choose our collection as a with respect to inclusion minimal collection of substructures that is closed under certain carefully chosen substructure transformations, where intuitively we still follow the above idea of considering substructures of overlapping copies.
	With this definition, we need to rely on a density argument to see that this even yields a finite collection.
	While the size of our collection size grows with~$1/\eps$, we show that it is not too large and that, by choice of the transformations, it allows a precise analysis of the evolution of all key quantities related to the substructures in the collection.
	Due to the implicit nature of our collection, we have to make our arguments without concrete knowledge of the structures we consider and all properties need to be deduced from the minimality of the collection as a collection that is closed under the aforementioned transformations.
	This often makes our arguments substantially more abstract.
	For example, for the analysis of the triangle case in~\cite{BFL:15}, substructures called \emph{fans} in~\cite{BFL:15} that essentially correspond to graphs that for some~$\ell\geq 1$ consist of vertices~$u,v_1,\ldots,v_\ell$ and the edges~$\set{u,v_i}$ and~$\set{v_{j},v_{j+1}}$ where~$1\leq i\leq\ell$ and~$1\leq j\leq\ell-1$ play a key role.
	In our more general analysis, we instead work with maximizers of density based optimization problems that we consider without concrete knowledge of their structure.
	
	A further obstacle that we overcome in our analysis is related to a possible lack of symmetry of~$\cF$ compared to a triangle.
	The structure of two overlapping copies of~$\cF$ depends not only on the size of the overlap but also on the specific choice of the shared part.
	This can cause transformations to switch between different non-interchangeable choices within copies of~$\cF$, which complicates the crucial part of the argument where estimation errors need to be calibrated such that the self-correcting behavior of the random processes remains mostly undisturbed by other quantities that also occur in the expressions for the expected one-step changes.
	We overcome this by considering our random processes in groups to restore symmetry in the sense that whenever we apply transformations to all members of a group simultaneously, we remain in a situation where all non-interchangeable choices within copies of~$\cF$ are represented if this was previously the case.
	
	Finally, as mentioned above, to complete our argumentation it remains to prove Theorems~\ref{theorem: sparse} and~\ref{theorem: sparse cherries}.
	In our significantly more general setting, adapting the argument presented in~\cite{BFL:15} to obtain a similar statement for the triangle case requires additional insights for a sufficient understanding of the structure of the random hypergraphs typically encountered around the time when we would typically expect the process to terminate.
	While in the triangle case certain configurations formed by overlapping copies of~$\cF$ are impossible as the triangle is simply too small to allow such overlaps of distinct copies, arguments bounding the numbers of such configurations are non-trivial for larger hypergraphs or even graphs.

\section{Organization of the paper}
	Theorem~\ref{theorem: main} is an immediate consequence of Theorem~\ref{theorem: pseudorandom}.
	Furthermore, the upper bounds in Theorems~\ref{theorem: pseudorandom} and~\ref{theorem: cherries} follow from Theorem~\ref{theorem: only upper bound}.
	In the first part of our paper, our goal is to analyze the removal process for a sufficient number of steps to see that with high probability, the process eventually generates a~$k$-graph that is sufficiently sparse to confirm Theorem~\ref{theorem: only upper bound} and that satisfies the properties necessary for an application of Theorem~\ref{theorem: sparse} or Theorem~\ref{theorem: sparse cherries} that then establishes the lower bound in Theorem~\ref{theorem: pseudorandom} or Theorem~\ref{theorem: cherries}.
	Subsequently, in the second part, we prove Theorems~\ref{theorem: sparse} and~\ref{theorem: sparse cherries}.

	As mentioned in Section~\ref{section: outline}, our precise analysis of the process consists of closely tracking the evolution of the number of occurrences of certain key substructures within the random~$k$-graphs generated by the process.
	We present this core of our proof as two closely related instances of a supermartingale concentration argument.
	Section~\ref{section: chains} is dedicated to implicitly defining our carefully selected substructures and obtaining key insights concerning singular such substructures.
	In Section~\ref{section: branching families}, we adjust our point of view and consider these structures in groups to establish symmetry, which is crucial for the careful calibration of estimation error needed to exploit self-correcting behavior.
	In Section~\ref{section: proof}, we show that the theorems in Section~\ref{section: introduction} are essentially immediate consequences of the more technical insights gained in Sections~\ref{section: chains} and~\ref{section: branching families}.

	As preparation for the argumentation in Sections~\ref{section: chains} and~\ref{section: branching families}, we first proceed as follows.
	After collecting some general notation that we use throughout the paper in Section~\ref{section: notation}, we introduce the setup for the first part of the paper and formally state the goal for this part in Section~\ref{section: process}.
	Then, in Section~\ref{section: heuristics}, we describe the heuristics that lead to our choices of deterministic trajectories that we expect key quantities to follow.
	Furthermore, towards the end of Section~\ref{section: heuristics}, we formally describe how introducing appropriate stopping times allows us to present the aforementioned two instances mostly separately.
	As final preparations for Sections~\ref{section: chains} and \ref{section: branching families}, in Section~\ref{section: stopping times} we subsequently introduce notation and terminology specific to our situation, we define key stopping times and we gather some statements concerning key quantities defined up to this point.
	
	For the second part of the paper, where we prove Theorems~\ref{theorem: sparse} and~\ref{theorem: sparse cherries}, we first describe the setup for this part in Section~\ref{section: sparse}.
	In Section~\ref{section: gluing}, we further investigate the structure of the hypergraphs generated towards the expected end of the process to deduce the necessary bounds that we subsequently rely on.
	Then, in Section~\ref{section: counting}, we present an extended tracking argument for the number of remaining copies which serves as further preparation for the arguments in Section~\ref{section: isolation argument} where we finally show that typically, sufficiently many edges remain when the process terminates.

	\section{Notation}\label{section: notation}
	
	For sets~$A,B$, we write~$\phi\colon A\injection B$ for an injective function~$\phi$ from~$A$ to~$B$ and we write~$\phi\colon A\bijection B$ for a bijection from~$A$ to~$B$.
	For integers~$i,j$, we set~$i\wedge j:=\min\set{i,j}$ and~$i\vee j:=\max\set{i,j}$.
	We use~$\binom{A}{i}$ to denote the set of all~$i$-sets~$B\subseteq A$, that is all sets~$B\subseteq A$ with~$\abs{B}=i$.
	We write~$\alpha\pm\eps=\beta\pm\delta$ to mean that~$[\alpha-\eps,\alpha+\eps]\subseteq [\beta-\delta,\beta+\delta]$.
	We occasionally only write~$\alpha$ instead of~$\floor{\alpha}$ or~$\ceil{\alpha}$ when the rounding is not important.
	A~$k$-graph is a~$k$-uniform hypergraph.
	Let~$\cH$ denote a~$k$-graph.
	We write~$V(\cH)$ or~$V_\cH$ for the vertex set of~$\cH$ and~$E(\cH)$ for the edge set of~$\cH$.
	We often simply write~$\cH$ instead of~$E(\cH)$.
	We set~$v(\cH):=\abs{V_\cH}$ and~$e(\cH):=\abs{\cH}$.
	For~$U\subseteq V(\cH)$, we use~$d_\cH(U)$ to denote the \emph{degree} of~$U$ in~$\cH$, that is the number of edges~$e$ of~$\cH$ with~$U\subseteq e$ and for~$v\in V(\cH)$, we set~$d_\cH(v):=d_\cH(\set{v})$.
	For~$U\subseteq V_\cH$, we write~$\cH[U]$ for the subgraph of~$\cH$ induced by~$U$, that is, the subgraph with vertex set~$U$ and edge set~$\cset{ e\in\cH }{ e\subseteq U }$ and we use~$\cH-U$ to denote the~$k$-graph~$\cH[V_\cH\setminus U]$.
	For~$k$-graphs~$\cH_1,\cH_2$, we write~$\cH_1\subseteq\cH_2$ to mean that~$\cH_1$ is a subgraph of~$\cH_2$ and we write~$\cH_1\subsetneq\cH_2$ to mean that~$\cH_1$ is a proper subgraph of~$\cH_2$.
	We write~$\cH_1+\cH_2$ for the~$k$-graph with vertex set~$V(\cH_1)\cup V(\cH_2)$ and edge set~$\cH_1\cup\cH_2$.
	For an event~$\cE$, we use~$\ind_\cE$ to denote the indicator random variable of~$\cE$.
	
	We remark that a list of symbols that we use not just locally but across several sections is provided at the end of the paper for the convenience of the reader.

	\section{Removal process}\label{section: process}
	From now on, until the end of Section~\ref{section: proof}, we focus on the first part.
	To this end, in this section, we describe the removal process that we analyze in the subsequent sections.
	For now, we assume a slightly more general setup similar to the one in Theorem~\ref{theorem: only upper bound}.
	In more detail, let~$k\geq 2$ and fix a~$k$-balanced~$k$-graph~$\cF$ on~$m$\gladd{real}{m}{$m=\abs{V_\cF}$} vertices with~$\abs{\cF}\geq 2$ and~$k$-density~$\rho_\cF$\gladd{real}{rhoF}{$\rho_\cF=\frac{\abs{\cF}-1}{\abs{V_\cF}-k}$}.
	Suppose that~$0<\eps<1$ is sufficiently small in terms of~$1/m$, that~$0<\delta<1$ is sufficiently small in terms of~$\eps$ and that~$n$ is sufficiently large in terms of~$1/\delta$.
	Suppose that~$\cH(0)$ is a~$(\eps^4,\delta,\rho_\cF)$-pseudorandom~$k$-graph on~$n$ vertices where
	\begin{equation*}
		\theta:=\frac{k!\,\abs{\cH(0)}}{n^k}\geq n^{-1/\rho_\cF+\eps}
	\end{equation*}\gladd{real}{theta}{$\theta=k!\,\abs{\cH(0)}/n^k$}%
	and let~$\cH^*(0)$ denote the~$\abs{\cF}$-graph with vertex set~$\cH(0)$ whose edges are the edge sets of copies of~$\cF$ that are subgraphs of~$\cH(0)$.
	Consider the following random process.
	
	\par\bigskip
	\IncMargin{1em}
	\begin{algorithm}[H]
		\SetAlgoLined
		\DontPrintSemicolon
		$i\gets 1$\;
		\While{$\cH^*(i-1)\neq\emptyset$}{
			choose~$\cF_0(i)\in \cH^*(i-1)$ uniformly at random\;
			$\cH^*(i)\gets \cH^*(i-1)-\cF_0(i)$\;
			$i\gets i+1$\;
		}
		\caption{Random~$\cF$-removal}
		\label{algorithm: removal}
	\end{algorithm}
	\par\bigskip
	
	If the process fails to execute step~$i+1$ and instead terminates, that is if~$\cH^*(i)=\emptyset$, then, for~$j\geq i+1$, let~$\cH^*(j):=\cH^*(i)$.
	For~$i\geq 1$, let~$\cH(i)$ denote the~$k$-graph with vertex set~$V_\cH:=V_{\cH(0)}$ and edge set~$V_{\cH^*(i)}$.
	Furthermore, let
	\begin{equation*}
		H^*(i):=\abs{\cH^*(i)}\qtand
		H(i)=\abs{\cH(i)}.
	\end{equation*}\gladd{realRV}{Hstar(i)}{$H^*(i)=\abs{\cH^*(i)}$}\gladd{realRV}{H(i)}{$H(i)=\abs{\cH(i)}$}%
	Let~$\frakF(0),\frakF(1),\ldots$ denote the natural filtration associated with the random process above.
	Finally, define the stopping time
	\begin{equation*}
		\tau_\emptyset:=\min\cset{ i\geq 0 }{ \cH^*(i)=\emptyset }
	\end{equation*}\gladd{stoppingtime}{tauemptyset}{$\tau_\emptyset=\min\cset{ i\geq 0 }{ \cH^*(i)=\emptyset }$}%
	that indicates when Algorithm~\ref{algorithm: removal} terminates in the sense that~$\tau_\emptyset$ is the number of successfully executed steps and hence the number of copies that were removed until termination.

	Since during every successful step of the process exactly~$\abs{\cF}$ edges are removed, an analysis up to step
	\begin{equation*}
		i^\star:=\frac{(\theta-n^{-1/\rho_\cF+\eps})n^k}{\abs{\cF}k!}
	\end{equation*}\gladd{time}{istar}{$i^\star=\frac{(\theta-n^{-1/\rho_\cF+\eps})n^k}{\abs{\cF}k!}$}%
	is sufficient for our purpose.
	Specifically, in Section~\ref{section: proof}, we show that Theorem~\ref{theorem: technical bounds} below holds.
	
	\begin{theorem}\label{theorem: technical bounds}
		With the setup above, the following holds.
		With probability at least~$1-\exp(-(\log n)^{4/3})$, the~$k$-graph~$\cH(i^\star)$ is~$(4m,n^\eps)$-bounded,~$\cF$-populated,~$k'$-populated for all~$1\leq k'\leq k-1/\rho_\cF$ and has~$n^{k-1/\rho_\cF+\eps}/k!$ edges.
	\end{theorem}
	An application of Theorem~\ref{theorem: technical bounds} with~$\eps^5$ playing the role of~$\eps$ immediately yields Theorem~\ref{theorem: only upper bound} and hence the upper bounds in Theorems~\ref{theorem: pseudorandom} and~\ref{theorem: cherries}.
	Additionally, in combination with Theorems~\ref{theorem: sparse} and~\ref{theorem: sparse cherries}, such an application of Theorem~\ref{theorem: technical bounds} also yields the lower bounds in Theorems~\ref{theorem: pseudorandom} and~\ref{theorem: cherries}.
	Thus, for the first part, it only remains to prove Theorem~\ref{theorem: technical bounds}.

\section{Trajectories}\label{section: heuristics}
	In every step of Algorithm~\ref{algorithm: removal}, exactly~$\abs{\cF}$ edges are removed.
	Hence, if~$0\leq i\leq \tau_\emptyset$, we have
	\begin{equation*}
		H(i)=\frac{\theta n^k}{k!}-\abs{\cF}i.
	\end{equation*}
	The heuristic arguments in this section are based on the assumption that typically, for all~$i\geq 0$, the edge set of~$\cH(i)$ behaves as if it was obtained by including every~$k$-set~$e\subseteq V_{\cH(0)}$ independently at random with probability
	\begin{equation*}
		\phat(i):=\theta-\frac{\abs{\cF}k!\,i}{n^k}.
	\end{equation*}\gladd{trajectory}{phat(i)}{$\phat(i)=\theta-\frac{\abs{\cF}k!\,i}{n^k}$}%
	Note that~$\phat(i)$ is chosen such that when following the probabilistic construction above, the expected number of included edges is essentially the true number of edges in~$\cH(i)$.
	
	Let~$\Aut(\cF)$ denote the set of automorphisms of~$\cF$, that is the set of bijections~$\phi\colon V_\cF\bijection V_\cF$ with~$\phi(e),\phi^{-1}(e)\in\cF$ for all~$e\in\cF$ and let~$\aut(\cF):=\abs{\Aut(\cF)}$.
	Based on the above assumption about the behavior of~$\cH(i)$, we estimate
	\begin{equation*}
		\ex{H^*(i)}
		\approx \frac{n^m\phat(i)^{\abs{\cF}}}{\aut(\cF)}=:\hhat^*(i).
	\end{equation*}\gladd{trajectory}{hhatstar(i)}{$\hhat^*(i)=\frac{n^m\phat(i)^{\abs{\cF}}}{\aut(\cF)}$}%
	As outlined in Section~\ref{section: outline}, our precise analysis of the random removal process essentially consists of proving that the numbers of many carefully chosen additional substructures within~$\cH^*(i)$ are typically concentrated around a deterministic trajectory.
	More specifically, these substructures will be given by embeddings of templates.
	Recall that, as defined in Section~\ref{section: introduction}, a~$k$-template is a pair~$(\cA,I)$ of a~$k$-graph~$\cA$ and a vertex set~$I\subseteq V_\cA$.
	For~$i\geq 0$, a~$k$-template~$(\cA,I)$ and an injection~$\psi\colon I\injection V_{\cH(i)}$, which may be thought of as a partial localization of the template~$(\cA,I)$ within~$\cH(i)$, we are interested in the collection~$\Phi^\sim_{\cA,\psi}(i)$ of embeddings of~$\cA$ into~$\cH(i)$ that extend~$\psi$.
	Formally, we set
	\begin{equation*}
		\Phi^\sim_{\cA,\psi}(i):=\cset{ \phi\colon V_\cA\injection V_{\cH(i)} }{ \restr{\phi}{I}=\psi\stand \phi(e)\in\cH(i)\stforall e\in\cA\setminus\cA[I] }.
	\end{equation*}\gladd{RV}{phiApsisim}{$\Phi^\sim_{\cA,\psi}(i)=\cset{ \phi\colon V_\cA\injection V_{\cH(i)} }{ \restr{\phi}{I}=\psi\stand \phi(e)\in\cH(i)\stforall e\in\cA\setminus\cA[I] }$}%
	For a template~$(\cA,I)$ and~$\psi\colon I\injection V_{\cH(i)}$, we anticipate
	\begin{equation*}
		\ex{\abs{ \Phi^\sim_{\cA,\psi}(i) }}
		\approx n^{\abs{V_\cA}-\abs{I}}\phat(i)^{\abs{\cA}-\abs{\cA[I]}}=:\phihat_{\cA,I}(i).
	\end{equation*}\gladd{trajectory}{phihatAI(i)}{$\phihat_{\cA,I}(i)=n^{\abs{V_\cA}-\abs{I}}\phat(i)^{\abs{\cA}-\abs{\cA[I]}}$}%
	This final estimate is only valid if~$(\cA,I)$ has certain desirable properties that make it well-behaved and that we specify in Section~\ref{section: stopping times}.
	We ensure that all templates where we are interested in precise estimates for the number of embeddings satisfy these properties.
	
	Our organization of the proof that up to step~$i^\star$, key quantities remain close to their trajectory with high probability is as follows.
	In the subsequent sections, we define stopping times~$\tau_{\cH^*},\tau_{\ccB},\tau_{\ccB'},\tau_{\frakC},\tau_{\frakB}$ that measure when key quantities significantly deviate from their trajectory.
	Then, to argue that
	\begin{equation*}
		i^\star<\tau_{\cH^*}\wedge\tau_{\ccB}\wedge\tau_{\ccB'}\wedge\tau_{\frakC}\wedge\tau_{\frakB}=:\tau^\star
	\end{equation*}\gladd{stoppingtime}{taustar}{$\tau^\star=\tau_{\cH^*}\wedge\tau_{\ccB}\wedge\tau_{\ccB'}\wedge\tau_{\frakC}\wedge\tau_{\frakB}$}%
	holds with high probability, we observe that
	\begin{align*}
		\set{\tau^\star\leq i^\star}
		&= \bigcup_{\tau\in\set{\tau_{\cH^*},\tau_{\ccB},\tau_{\ccB'},\tau_{\frakC},\tau_{\frakB}}}\set{\tau\leq \tau^\star\wedge i^\star}
	\end{align*}
	and show that the probabilities for the five events on the right are small.
	For~$\tau\in\set{\tau_{\cH^*},\tau_\ccB,\tau_{\ccB'}}$, a suitable bound for the probability of the corresponding event on the right may be obtained similarly as the analogous statements for the triangle case in~\cite{BFL:15}  by employing standard critical interval arguments.
	New ideas that allow us to carry out an analysis of the hypergraph removal process in great generality are required for suitable bounds for the two remaining events, that is when~$\tau\in\set{\tau_{\frakC},\tau_{\frakB}}$.
	We dedicate Sections~\ref{section: chains} and~\ref{section: branching families} to bounding the probabilities of these two events.
	Note that in fact, each of these five events occurs if and only if the corresponding inequality holds with equality.
	
\section{Template embeddings and key stopping times}\label{section: stopping times}
	
	We introduce the following conventions and notations to simplify notation.
	In general, if~$X(0),X(1),\ldots$ is a sequence of numbers or random variables and~$i\geq 0$, we define~$\Delta X(i):=X(i+1)-X(i)$.
	To refer to a previously defined~$X(i)$, we often only write~$X$ to mean~$X(i)$, so for example when we only write~$H^*$, this is meant to be replaced with~$H^*(i)$.
	Note that this introduces no ambiguity concerning~$V_\cH$ since~$V_{\cH(i)}$ is the same for all~$i\geq 0$.
	For an event~$\cE$, a random variable~$X$ and~$i\geq 0$, we define~$\pri{\cE}:=\cpr{\cE}{\frakF(i)}$ and~$\exi{X}:=\cex{X}{\frakF(i)}$.
	We write~$X=_{\cE}Y$ for two expressions~$X$ and~$Y$ and an event~$\cE$, to express the statement that~$X$ and~$Y$ represent (possibly constant) random variables that are equal whenever~$\cE$ occurs, or equivalently, to express that~$X\cdot\ind_{\cE}=Y\cdot\ind_{\cE}$.
	Similarly, we write~$X\leq_{\cE}Y$ to mean~$X\cdot\ind_{\cE}\leq Y\cdot\ind_{\cE}$ and~$X\geq_{\cE}Y$ to mean~$X\cdot\ind_{\cE}\geq Y\cdot\ind_{\cE}$.

	Extending the terminology concerning templates that we introduce in Section~\ref{section: introduction}, 
	we say that a template~$(\cA,I)$ is a \emph{copy} of a template~$(\cB,J)$ if there exists a bijection~$\phi\colon V_\cA\bijection V_\cB$ with~$\phi(e)\in\cB$ for all~$e\in\cA$,~$\phi^{-1}(e)\in\cA$ for all~$e\in\cB$ and~$\phi(I)=J$.
	We say that~$(\cA,I)$ is \emph{balanced} if~$\rho_{\cB,I}\leq \rho_{\cA,I}$ for all~$(\cB,I)\subseteq(\cA,I)$.
	Note that a~$k$-graph~$\cG$ is~$k$-balanced if and only if~$(\cG,e)$ is balanced for all~$e\in\cG$.
	For a template~$(\cA,I)$,~$\psi\colon I\injection V_\cH$ and~$i\geq 0$ let~$\Phi_{\cA,\psi}(i):=\abs{\Phi^\sim_{\cA,\psi}}$\gladd{realRV}{PhiApsi}{$\Phi_{\cA,\psi}(i)=\abs{\Phi^\sim_{\cA,\psi}(i)}$}.

	The definition of the stopping times mentioned in Section~\ref{section: heuristics} depend on what it means to deviate significantly from a corresponding trajectory.
	The formal definition relies on appropriately chosen error terms that we define for the key quantities that we wish to track and that quantify the maximum deviation from the trajectory that we allow.
	Many of these error terms are expressed in terms of~$\delta$ and~$\zeta(i)$, where for~$i\geq 0$, we set
	\begin{equation*}
		\zeta(i):=\frac{n^{\eps^2}}{n^{1/2} \phat^{\rho_\cF/2}}.
	\end{equation*}\gladd{trajectory}{zeta(i)}{$\zeta(i)=\frac{n^{\eps^2}}{n^{1/2} \phat(i)^{\rho_\cF/2}}$}%
	For~$\alpha\geq 0$ and a template~$(\cA,I)$ let
	\begin{equation*}
		i^\alpha_{\cA,I}:=\min\cset{i\geq 0}{\phihat_{\cA,I}\leq \zeta^{-\alpha}},
	\end{equation*}\gladd{time}{iAIalpha}{$i^\alpha_{\cA,I}=\min\cset{i\geq 0}{\phihat_{\cA,I}(i)\leq \zeta(i)^{-\alpha}}$}%
	where we set~$\min\emptyset:=\infty$.
	Note in particular, that~$i^0_{\cA,I}=\min\cset{i\geq 0}{\phihat_{\cA,I}\leq 1}$.
	
	We consider the families of templates
	\begin{equation*}
		\begin{aligned}
			\ccF&:=\cset{ (\cF,f) }{ f\in\cF },\\
			\ccB&:=\cset{ (\cA,I) }{\text{$(\cA,I)$ is a balanced~$k$-template with~$\abs{V_\cA}\leq 1/\eps^4$ and~$i_{\cA,I}^{\delta^{1/2}}\geq 1$}},\\
			\ccB'&:=\cset{ (\cA,I) }{\text{$(\cA,I)$ is a strictly balanced~$k$-template with~$\abs{V_\cA}\leq 1/\eps^4$ and~$i_{\cA,I}^{0}\geq 1$}}.
		\end{aligned}
	\end{equation*}\gladd{construction}{B}{$\ccF=\cset{ (\cF,f) }{ f\in\cF }$}\gladd{construction}{B}{$\ccB=\cset{ (\cA,I) }{\text{$(\cA,I)$ is a balanced~$k$-template with~$\abs{V_\cA}\leq 1/\eps^4$,~$i_{\cA,I}^{\delta^{1/2}}\geq 1$}}$}\gladd{construction}{Bprime}{$\ccB'=\cset{ (\cA,I) }{\text{$(\cA,I)$ is a strictly balanced~$k$-template with~$\abs{V_\cA}\leq 1/\eps^4$,~$i_{\cA,I}^{0}\geq 1$}}$}%
	For~$x\geq 0$, let
	\begin{equation*}
		\alpha_x:=2^{x+1}-2
	\end{equation*}
	\gladd{real}{alphax}{$\alpha_x=2^{x+1}-2$}%
	and let~$\alpha_{\cA,I}:=\alpha_{\abs{V_\cA}-\abs{I}}$\gladd{real}{alphaAI}{$\alpha_{\cA,I}=\alpha_{\abs{V_\cA}-\abs{I}}$}.
	In the following observation, we briefly state the properties that motivate the choice of~$\alpha_x$ and that we rely on for arguments further below.
	\begin{observation}\label{observation: choice of alpha}
		Let~$x,y\geq 0$ and~$z\geq 1$.
		Then,
		\begin{equation*}
			2\alpha_x+2\leq \alpha_{x+1},\quad
			\alpha_x+\alpha_y\leq \alpha_{x+y},\quad
			\alpha_z\geq 2.
		\end{equation*}
	\end{observation}
	We define the stopping times
	\begin{align*}
		\tau_{\cH^*}&:=\min\cset{i\geq 0 }{ H^*\neq(1\pm \zeta^{1+\eps^3})\hhat^*},\\
		\tau_\ccF&:=\min\cset{i\geq 0 }{ \Phi_{\cF,\psi}\neq(1\pm \delta^{-1}\zeta)\phihat_{\cF,f}\stforsome (\cF,f)\in\ccF,\psi\colon f\injection V_\cH},\\
		\tau_{\ccB}&:=\min\set*{\setlength\arraycolsep{0pt}\begin{array}{ll}
				i\geq 0 :~&\Phi_{\cA,\psi}\neq (1\pm \zeta^\delta)\phihat_{\cA,I}\stand i\leq i^{\delta^{1/2}}_{\cA,I}\\&\quad\text{for some }(\cA,I)\in\ccB,\psi\colon I\injection V_\cH
		\end{array}},\\
		\tau_{\ccB'}&:=\min\set*{\setlength\arraycolsep{0pt}\begin{array}{ll}
				i\geq 0 :~&\Phi_{\cA,\psi}\neq (1\pm (\log n)^{\alpha_{\cA,I}} \phihat_{\cA,I}^{-\delta^{1/2}})\phihat_{\cA,I}\stand i^{\delta^{1/2}}_{\cA,I}\leq i\leq i^0_{\cA,I}\\&\quad\text{for some } (\cA,I)\in\ccB',\psi\colon I\injection V_\cH
		\end{array}}.
	\end{align*}%
	\gladd{stoppingtime}{tauF}{$\tau_\ccF=\min\cset{i\geq 0 }{ \Phi_{\cF,\psi}\neq(1\pm \delta^{-1}\zeta)\phihat_{\cF,f}\stforsome (\cF,f)\in\ccF,\psi\colon f\injection V_\cH}$}%
	\gladd{stoppingtime}{tauB}{$\tau_{\ccB}=\min\cset{i\geq 0}{\Phi_{\cA,\psi}\neq (1\pm \zeta^\delta)\phihat_{\cA,I}\stand i\leq i^{\delta^{1/2}}_{\cA,I}\text{for some }(\cA,I)\in\ccB,\psi\colon I\injection V_\cH}$}%
	\gladd{stoppingtime}{tauBprime}{$\tau_{\ccB'}=\min\set*{\setlength\arraycolsep{0pt}\begin{array}{ll}
				i\geq 0 :~&\Phi_{\cA,\psi}\neq (1\pm (\log n)^{\alpha_{\cA,I}} \phihat_{\cA,I}^{-\delta^{1/2}})\phihat_{\cA,I}\stand i^{\delta^{1/2}}_{\cA,I}\leq i\leq i^0_{\cA,I}\\&\quad\text{for some } (\cA,I)\in\ccB',\psi\colon I\injection V_\cH
		\end{array}}$}%
	\gladd{stoppingtime}{tauHstar}{$\tau_{\cH^*}=\min\cset{i\geq 0 }{ H^*\neq(1\pm \zeta^{1+\eps^3})\hhat^*}$}%
	Three of these stopping times are mentioned in Section~\ref{section: heuristics}.
	Since the precise definition of the other two stopping times~$\tau_{\frakC}$ and~$\tau_\frakB$ is not always relevant, we occasionally only work with the simpler stopping time~$\tau_{\ccF}$ that satisfies~$\tau_{\ccF}\geq \tau_{\frakC}$ and we define
	\begin{equation}\label{equation: definition of tautilde star}
		\tautilde^\star:=\tau_{\cH^*}\wedge\tau_{\ccB}\wedge \tau_{\ccB'}\wedge \tau_\ccF\geq \tau^\star.
	\end{equation}\gladd{stoppingtime}{tautildestar}{$\tautilde^\star=\tau_{\cH^*}\wedge\tau_{\ccB}\wedge \tau_{\ccB'}\wedge \tau_\ccF$}%
	Observe that the relative error~$\zeta^{1+\eps^3}$ that we allow for~$H^*$ is significantly smaller than the relative error~$\delta^{-1}\zeta$ that we allow for~$\Phi_{\cF,f}$ where~$f\in\cF$.
	Furthermore, the relative error~$\zeta^\delta$ that we use for the number of embeddings~$\Phi_{\cA,\psi}$ corresponding to a balanced extension~$(\cA,I)\in\ccB$ and~$\psi\colon I\injection V(\cH)$ is significantly larger than these two previous error terms.
	However, it is at most~$n^{-\delta^2}$, reflecting the fact that we still expect tight concentration around the corresponding trajectory provided that we can still expect~$\Phi_{\cA,\psi}$ to be sufficiently large in the sense that we are not beyond step~$i_{\cA,I}^{\delta^{1/2}}$.
	Finally, concerning the fourth stopping time, we are only interested in the further evolution of the number of embeddings beyond step~$i_{\cA,I}^{\delta^{1/2}}$, but still at most up to step~$i_{\cA,I}^0$, if~$(\cA,I)$ is strictly balanced.
	For this further evolution, our relative error term is essentially potentially as large as~$(\log n)^{\alpha_{\cA,I}}$.
	Note that all error terms are sensible in the sense that at least in the very beginning, before the removal of any copy, the corresponding random variables are within the margin of error as implied by Lemma~\ref{lemma: initially good}.
	Before we turn to this lemma and its proof, we first state two useful Lemmas.
	Lemma~\ref{lemma: extension trajectory split} formulates a convenient fact concerning the trajectories corresponding to the numbers of embeddings of templates that we use below without explicitly referencing it.
	In Lemma~\ref{lemma: extension is strictly balanced}, we consider a construction of strictly balanced templates within~$k$-graphs.
	It is convenient to have Lemma~\ref{lemma: extension is strictly balanced} available for the proof of Lemma~\ref{lemma: initially good} and furthermore, the simple construction plays a crucial role in Section~\ref{section: chains}.
	Overall, the verification in Lemma~\ref{lemma: initially good} that the initial conditions are suitable and the following lemmas in Sections~\ref{subsection: key quantities}--\ref{subsection: degrees} play mostly an auxiliary role and the proofs rely on standard arguments and are not important for understanding the setup and argumentation in Sections~\ref{section: chains} and~\ref{section: branching families} where we turn to the new ideas that allow us to analyze the~$\cF$-removal process in great generality.
	Hence, if the desire is to focus on these new contributions, one may skip these results and continue reading at the beginning of Section~\ref{subsection: auxiliary control} where we make some final remarks concerning the overall setup as preparation for Sections~\ref{section: chains} and~\ref{section: branching families}.

	\begin{lemma}\label{lemma: extension trajectory split}
		Let~$i\geq 0$.
		Suppose that~$(\cA,I)$ is a template and let~$I\subseteq U\subseteq V_\cA$.
		Then,~$\phihat_{\cA,I}=\phihat_{\cA,U}\cdot\phihat_{\cA[U],I}$.
	\end{lemma}
	
	\begin{proof}
		We have
		\begin{equation*}
			\phihat_{\cA,I}
			=n^{\abs{V_\cA}-\abs{I}}\phat^{\abs{\cA}-\abs{\cA[I]}}
			=n^{\abs{V_\cA}-\abs{U}}\phat^{\abs{\cA}-\abs{\cA[U]}}n^{\abs{U}-\abs{I}}\phat^{\abs{\cA[U]}-\abs{\cA[I]}}
			=\phihat_{\cA,U}\phihat_{\cA[U],I},
		\end{equation*}
		which completes the proof.
	\end{proof}
	
	\begin{lemma}\label{lemma: extension is strictly balanced}
		Suppose that~$\cA$ is a~$k$-graph and let~$\alpha\geq 0$ and~$U\subseteq V_\cA$.
		Suppose that among all subsets~$U\subseteq I'\subsetneq V_\cA$ with~$\rho_{\cA,I'}\leq \alpha$, the set~$I$ has maximal size.
		Then, the template~$(\cA,I)$ is strictly balanced.
	\end{lemma}
	\begin{proof}
		Let~$(\cB,I)\subseteq (\cA,I)$ with~$I\neq V_\cB$ and~$\cB\neq\cA$.
		We show that~$\rho_{\cB,I}<\rho_{\cA,I}$.
		We may assume that~$\cB$ is an induced subgraph of~$\cA$ and then we have~$I\subsetneq V_\cB\subsetneq V_\cA$.
		By choice of~$I$, we obtain~$\rho_{\cA,V_\cB}>\alpha\geq \rho_{\cA,I}$ and hence
		\begin{equation*}
			\rho_{\cB,I}
			=\frac{\rho_{\cA,I}(\abs{V_\cA}-\abs{I})-\rho_{\cA,V_\cB}(\abs{V_\cA}-\abs{V_\cB})}{\abs{V_\cB}-\abs{I}}
			<\frac{\rho_{\cA,I}(\abs{V_\cA}-\abs{I})-\rho_{\cA,I}(\abs{V_\cA}-\abs{V_\cB})}{\abs{V_\cB}-\abs{I}}
			=\rho_{\cA,I},
		\end{equation*}
		which completes the proof.
	\end{proof}

	\begin{lemma}\label{lemma: initially good}
		Let~$i:=0$.
		Suppose that~$(\cA,I)$ is a~$k$-template with~$\abs{V_\cA}\leq 1/\eps^4$ and let~$\psi\colon I\injection V_\cH$.
		Then, the following holds.
		\begin{enumerate}[label=\textup{(\roman*)}]
			\item If~$\rho_{\cB,I}\leq \rho_\cF$ for all~$(\cB,I)\subseteq(\cA,I)$, then~$\Phi_{\cA,\psi}=(1\pm \zeta^{1+2\eps^3})\phihat_{\cA,I}$.
			\item $H^*=(1\pm \zeta^{1+2\eps^3})\hhat^*$.
			\item If~$(\cA,I)\in\ccB$, then~$\Phi_{\cA,\psi}=(1\pm \zeta^{\delta+\delta^2})\phihat_{\cA,I}$.
			\item If~$(\cA,I)\in\ccB'$ and~$i_{\cA,I}^{\delta^{1/2}}=0$, then~$\Phi_{\cA,\psi}=(1\pm (\log n)^{\alpha_{\cA,I}-1/2}\phihat_{\cA,I}^{-\delta^{1/2}})\phihat_{\cA,I}$.
		\end{enumerate}
	\end{lemma}
	\begin{proof}
		We obtain~(ii) as an immediate consequence of~(i) and we show that~(i),~(iii) and~(iv) follow from the~$(\eps^4,\delta,\rho_\cF)$-pseudorandomness of~$\cH$.
		More specifically, while~(iv) is a direct consequence of the pseudorandomness, for~(i) and~(iii), we deconstruct~$(\cA,I)$ into a series of strictly balanced templates to employ the pseudorandomness.
		Note that in the definition of~$(\eps^4,\delta,\rho_\cF)$-pseudorandomness, the fraction~$\zeta_0:=n^\delta/(n\theta^{\rho_\cF})^{1/2}$ played the role of~$\zeta$ in the definition, however, here we have~$\zeta=\zeta(0)=n^{\eps^2}/(n\theta^{\rho_\cF})^{1/2}=n^{\eps^2}\zeta_0/n^\delta$.
		Choosing a larger~$\zeta$ here and in the definitions of the key stopping times gives us additional room for errors that we exploit in the proof.
		In detail, we prove the four statements as follows.

		\begin{enumerate}[label=\textup{(\roman*)}, topsep=0pt, wide]
			\item Suppose that~$\rho_{\cB,I}\leq\rho_\cF$ holds for all~$(\cB,I)\subseteq (\cA,I)$.
			We use induction on~$\abs{V_\cA}-\abs{I}$ to show that
			\begin{equation}\label{equation: initial sparse templates induction}
				\Phi_{\cA,\psi}=(1\pm 2(\abs{V_\cA}-\abs{I})\zeta^{1+3\eps^3})\phihat_{\cA,I}.
			\end{equation}
			Since~$\abs{V_\cA}\leq 1/\eps^4$, this is sufficient.
			
			Let us proceed with the proof by induction.
			If~$\abs{V_\cA}-\abs{I}=0$, then~$\Phi_{\cA,I}=1=\phihat_{\cA,I}$.
			Let~$\ell\geq 1$ and suppose that\eqref{equation: initial sparse templates induction} holds if~$\abs{V_\cA}-\abs{I}\leq \ell-1$.
			Suppose that~$\abs{V_\cA}-\abs{I}=\ell$.
			Suppose that among all subsets~$I\subseteq U'\subsetneq V_\cA$ with~$\rho_{\cA,U'}\leq \rho_\cF$, the set~$U$ has maximal size.
			By Lemma~\ref{lemma: extension is strictly balanced}, the extension~$(\cA,U)$ is strictly balanced.
			We have
			\begin{equation}\label{equation: initial sparse templates split}
				\Phi_{\cA,\psi}=\sum_{\phi\in\Phi_{\cA[U],\psi}^\sim} \Phi_{\cA,\phi}.
			\end{equation}
			We use the estimate for~$\Phi_{\cA[U],\psi}$ provided by the induction hypothesis and for~$\phi\in\Phi_{\cA[U],\psi}^\sim$, we estimate~$\Phi_{\cA,\phi}$ using the pseudorandomness of~$\cH$.
			
			Let us turn to the details.
			The template~$(\cA,U)$ is strictly balanced and satisfies~$\rho_{\cA,U}\leq\rho_\cF$, so since~$\cH$ is~$(\eps^4,\delta,\rho_\cF)$-pseudorandom, for all~$\phi\in\Phi_{\cA[U],\psi}^\sim$, we have
			\begin{equation*}
				\Phi_{\cA,\phi}
				=(1\pm\zeta_0)\phihat_{\cA,U}
				=\paren[\bigg]{1\pm \frac{n^{\delta}}{n^{\eps^2}}\zeta}\phihat_{\cA,U}
				=(1\pm \zeta^{1+3\eps^3})\phihat_{\cA,U}.
			\end{equation*}
			Since by induction hypothesis, we have~$\Phi_{\cA[U],\psi}=(1\pm 2(\abs{U}-\abs{I})\zeta^{1+3\eps^3})\phihat_{\cA[U],I}$, returning to~\eqref{equation: initial sparse templates split}, we conclude that
			\begin{equation*}
				\Phi_{\cA,\psi}
				=(1\pm 2(\abs{U}-\abs{I})\zeta^{1+3\eps^3})\phihat_{\cA[U],I}\cdot (1\pm \zeta^{1+3\eps^3})\phihat_{\cA,U}
				=(1\pm 2(\abs{V_\cA}-\abs{I})\zeta^{1+3\eps^3})\phihat_{\cA,I},
			\end{equation*}
			which completes the proof of~(i).
			
			\item This is a consequence of~(i) and the fact that~$\cF$ is~$k$-balanced.
			To see this, we argue as follows.
			Fix~$f\in\cF$ and let~$\psi\colon \emptyset\to V_\cH$.
			Then, we have
			\begin{equation*}
				\begin{aligned}
					H^*
					&=\frac{\Phi_{\cF,\psi}}{\aut(\cF)}
					=\frac{1}{\aut(\cF)}\sum_{\phi\in \Phi^\sim_{\cF[f],\psi}}\Phi_{\cF,\phi}
					=(1\pm \zeta^{1+2\eps^3})\frac{\phihat_{\cF,f}\cdot\Phi_{\cF[f],\psi}}{\aut(\cF)}
					=(1\pm \zeta^{1+2\eps^3})\frac{\phihat_{\cF,f}\cdot k!\,H}{\aut(\cF)}\\
					&=(1\pm \zeta^{1+2\eps^3})\frac{\phihat_{\cF,f}\cdot\theta n^k}{\aut(\cF)}
					=(1\pm \zeta^{1+2\eps^3})\hhat^*,
				\end{aligned}
			\end{equation*}
			which completes the proof of~(ii).
			
			\item Suppose that~$(\cA,I)$ is balanced and that~$\phihat_{\cA,I}\geq \zeta^{-\delta^{4/7}(\abs{V_\cA}-\abs{I})}$.
			We argue similarly as in the proof of~(i) and use induction on~$\abs{V_\cA}-\abs{I}$ to show that
			\begin{equation}\label{equation: initial balanced templates induction}
				\Phi_{\cA,\psi}=(1\pm 2(\abs{V_\cA}-\abs{I})\zeta^{\delta+2\delta^2})\phihat_{\cA,I}.
			\end{equation}
			Since~$\abs{V_\cA}\leq 1/\eps^4$, this is sufficient.
			
			Let us proceed with the proof by induction.
			If~$\abs{V_\cA}-\abs{I}=0$, then~$\Phi_{\cA,I}=1=\phihat_{\cA,I}$.
			Let~$\ell\geq 1$ and suppose that~\eqref{equation: initial balanced templates induction} holds if~$\abs{V_\cA}-\abs{I}\leq \ell-1$.
			Suppose that~$\abs{V_\cA}-\abs{I}=\ell$.
			Suppose that among all subsets~$I\subseteq U'\subsetneq V_\cA$ with~$\rho_{\cA,U'}\leq \rho_{\cA,I}$, the set~$U$ has maximal size.
			By Lemma~\ref{lemma: extension is strictly balanced}, the extension~$(\cA,U)$ is strictly balanced.
			Due to~$\theta\geq n^{-1/\rho_\cF+\eps}$, we have
			\begin{equation}\label{equation: initial lower trajectory bound}
				\begin{aligned}
					\phihat_{\cA,I}
					&\geq \zeta^{-\delta^{4/7}(\abs{V_\cA}-\abs{I})}
					\geq\paren[\bigg]{\frac{n^{\eps^2}}{n^\delta}}^{-\delta^{4/7}}\zeta_0^{-\delta^{4/7}}
					> \paren[\bigg]{\frac{n^{\eps\rho_\cF/2}}{n^{\delta}}}^{-\delta^{4/7}/2}\zeta_0^{-\delta^{4/7}}\\
					&\geq \paren[\bigg]{\frac{n^{1/2}\theta^{\rho_\cF/2}}{n^\delta}}^{-\delta^{4/7}/2}\zeta_0^{-\delta^{4/7}}
					=\zeta_0^{-\delta^{4/7}/2}
					\geq\zeta_0^{-\delta^{3/5}}.
				\end{aligned}
			\end{equation}
			Hence, if~$U=I$, then, since~$(\cA,U)$ is strictly balanced and since~$\zeta^{\delta+\delta^2}\geq \zeta_0^\delta$, the desired estimate follows from the fact that~$\cH$ is~$(\eps^4,\delta,\rho_\cF)$-pseudorandom.
			Thus, we may assume that~$U\neq I$.
			We have
			\begin{equation*}
				\rho_{\cA[U],I}
				=\frac{\rho_{\cA,I}(\abs{V_\cA}-\abs{I})-\rho_{\cA,U}(\abs{V_\cA}-\abs{U})}{\abs{U}-\abs{I}}
				\geq \frac{\rho_{\cA,I}(\abs{V_\cA}-\abs{I})-\rho_{\cA,I}(\abs{V_\cA}-\abs{U})}{\abs{U}-\abs{I}}
				=\rho_{\cA,I}.
			\end{equation*}
			Hence, since~$(\cA,I)$ is balanced, the template~$(\cA[U],I)$ has density~$\rho_{\cA[U],I}=\rho_{\cA,I}$ and is also balanced.
			Additionally, we have
			\begin{equation*}
				\phihat_{\cA[U],I}
				=\phihat_{\cA,I}^{(\abs{U}-\abs{I})/(\abs{V_\cA}-\abs{I})}
				\geq \zeta^{-\delta^{4/7}(\abs{U}-\abs{I})}
			\end{equation*}
			and~\eqref{equation: initial lower trajectory bound} entails
			\begin{equation*}
				\phihat_{\cA,U}
				\geq \phihat_{\cA,I}^{(\abs{V_\cA}-\abs{U})/(\abs{V_\cA}-\abs{I})}
				\geq \phihat_{\cA,I}^{\eps^4}
				\geq 	\zeta_0^{-\delta^{2/3}}.
			\end{equation*}
			We have
			\begin{equation}\label{equation: initial balanced templates split}
				\Phi_{\cA,\psi}=\sum_{\phi\in\Phi_{\cA[U],\psi}^\sim} \Phi_{\cA,\phi}.
			\end{equation}
			We use the estimate for~$\Phi_{\cA[U],\psi}$ provided by the induction hypothesis and for~$\phi\in\Phi_{\cA[U],\psi}^\sim$, we estimate~$\Phi_{\cA,\phi}$ using the pseudorandomness of~$\cH$.
			
			Let us turn to the details.
			The template~$(\cA,U)$ is strictly balanced and we have~$\phihat_{\cA,U}\geq \zeta_0^{-\delta^{2/3}}$, so since~$\cH$ is~$(\eps^4,\delta,\rho_\cF)$-pseudorandom, for all~$\phi\in\Phi_{\cA[U],\psi}$, we obtain
			\begin{equation*}
				\Phi_{\cA,\phi}
				=\paren[\bigg]{1\pm\paren[\bigg]{\frac{n^\delta}{n^{\eps^2}}\zeta}^\delta}\phihat_{\cA,U}
				=(1\pm \zeta^{\delta+2\delta^2})\phihat_{\cA,U}.
			\end{equation*}
			Furthermore, the template~$(\cA[U],I)$ is balanced and we have~$\phihat_{\cA[U],I}\geq \zeta^{-\delta^{4/7}(\abs{U}-\abs{I})}$, so by induction hypothesis, we obtain
			\begin{equation*}
				\Phi_{\cA,\psi}=(1\pm 2(\abs{U}-\abs{I})\zeta^{\delta+2\delta^2})\phihat_{\cA[U],I}.
			\end{equation*}
			Returning to~\eqref{equation: initial balanced templates split}, we conclude that
			\begin{equation*}
				\Phi_{\cA,\psi}
				=(1\pm 2(\abs{U}-\abs{I})\zeta^{\delta+2\delta^2})\phihat_{\cA[U],I}\cdot (1\pm \zeta^{\delta+2\delta^2})\phihat_{\cA,U}
				=(1\pm 2(\abs{V_\cA}-\abs{I})\zeta^{\delta+2\delta^2})\phihat_{\cA,I},
			\end{equation*}
			which completes the proof of~(iii).

			\item Suppose that~$(\cA,I)\in\ccB'$ and~$i_{\cA,I}^{\delta^{1/2}}=0$.
			We may assume that~$I\neq V_\cA$.
			If~$\phihat_{\cA,I}\geq \zeta_0^{-\delta^{2/3}}$, then since~$\cH$ is~$(\eps^4,\delta,\rho_\cF)$-pseudorandom, using~$\phihat_{\cA,I}\leq \zeta^{-\delta^{1/2}}$, we have
			\begin{equation*}
				\Phi_{\cA,I}
				=(1\pm\zeta_0^\delta)\phihat_{\cA,I}
				=(1\pm\zeta^\delta)\phihat_{\cA,I}
				=(1\pm \phihat_{\cA,I}^{-\delta^{1/2}})\phihat_{\cA,I}
				=(1\pm (\log n)^{\alpha_{\cA,I}-1/2}\phihat_{\cA,I}^{-\delta^{1/2}})\phihat_{\cA,I}.
			\end{equation*}
			If~$\phihat_{\cA,I}\leq \zeta_0^{-\delta^{2/3}}$, then again since~$\cH$ is~$(\eps^4,\delta,\rho_\cF)$-pseudorandom, we obtain
			\begin{equation*}
				\Phi_{\cA,I}
				=(1\pm (\log n)^{3(\abs{V_\cA}-\abs{I})/2}\phihat_{\cA,I}^{-\delta^{1/2}})\phihat_{\cA,I}
				=(1\pm (\log n)^{\alpha_{\cA,I}-1/2}\phihat_{\cA,I}^{-\delta^{1/2}})\phihat_{\cA,I},
			\end{equation*}
			which completes the proof of~(iv).\qedhere
		\end{enumerate}
	\end{proof}
	
	\subsection{Auxiliary results about key quantities}\label{subsection: key quantities}
	We gather some statements concerning the key quantities defined up to this point.
	Lemmas~\ref{lemma: bounds of phat}--~\ref{lemma: zeta and H} provide useful bounds concerning~$\phat$,~$\zeta$ and~$H$.
	
	\begin{lemma}\label{lemma: bounds of phat}
		Let~$0\leq i\leq i^\star$.
		Then~$n^{-1/\rho_\cF+\eps}\leq \phat\leq 1$.
	\end{lemma}
	
	\begin{proof}
		We obviously have~$\phat\leq \theta\leq 1$ and furthermore~$\phat\geq \phat(i^\star)=n^{-1/\rho_\cF+\eps}$.
	\end{proof}
	
	\begin{lemma}\label{lemma: bounds of delta phat}
		Let~$0\leq i\leq i^\star$.
		Then,~$\phat(i+1)\geq (1-n^{-\eps^2})\phat$.
	\end{lemma}
	
	\begin{proof}
		Lemma~\ref{lemma: bounds of phat} implies
		\begin{equation*}
			\phat(i+1)
			=\paren[\bigg]{1-\frac{\abs{\cF}k!}{n^k\phat}}\phat
			\geq \paren[\bigg]{ 1-\frac{2\abs{\cF}k!}{n^\eps} }\phat
			\geq (1-n^{-\eps^2})\phat,
		\end{equation*}
		which completes the proof.
	\end{proof}

	\begin{lemma}\label{lemma: edges of H}
		Let~$0\leq i\leq i^\star$ and~$\cX:=\set{i\leq\tau_\emptyset}$.
		Then,~$H\Xeq n^k\phat/k!$.
	\end{lemma}
	
	\begin{proof}
		We have
		\begin{equation*}
			H
			\Xeq \frac{\theta n^k}{k!}-\abs{\cF}i
			=\frac{n^k\phat}{k!},
		\end{equation*}
		which completes the proof.
	\end{proof}
	
	\begin{lemma}\label{lemma: bounds of zeta}
		Let~$0\leq i\leq i^\star$.
		Then,~$n^{-1/2+\eps^2}\leq \zeta\leq n^{-\eps^2}$.
	\end{lemma}
	
	\begin{proof}
		Indeed, using Lemma~\ref{lemma: bounds of phat}, we obtain
		\begin{equation*}
			n^{-1/2+\eps^2}
			\leq \frac{n^{\eps^2}}{n^{1/2}\phat^{\rho_\cF/2}}
			=\zeta
			\leq \frac{n^{\eps^2}}{n^{1/2}\phat(i^\star)^{\rho_\cF/2}}
			=\frac{n^{\eps^2}}{n^{1/2}n^{(-1+\eps\rho_\cF)/2}}
			=\frac{n^{\eps^2}}{n^{\eps\rho_\cF/2}}
			\leq n^{-\eps^2},
		\end{equation*}
		which completes the proof.
	\end{proof}
	
	\begin{lemma}\label{lemma: zeta and H}
		Let~$0\leq i\leq i^\star$ and~$\cX:=\set{i\leq\tau_\emptyset}$.
		Then,~$1/H\Xleq k!/(n\phat^{\rho_\cF})\leq \zeta^{2+2\eps^2}$.
	\end{lemma}
	
	\begin{proof}
		Lemma~\ref{lemma: edges of H} together with Lemma~\ref{lemma: bounds of phat} entails
		\begin{equation*}
			\frac{1}{H}
			\Xeq \frac{k!}{n^k\phat}
			\leq \frac{k!}{(n\phat^{\rho_\cF})^k}
			\leq \frac{k!}{n\phat^{\rho_\cF}}^.
		\end{equation*}
		Furthermore, using Lemma~\ref{lemma: bounds of zeta}, we obtain
		\begin{equation*}
			\frac{k!}{n\phat^{\rho_\cF}}
			\leq \frac{n^{\eps^2}}{n\phat^{\rho_\cF}}
			= n^{-\eps^2}\zeta^2
			\leq \zeta^{2+2\eps^2},
		\end{equation*}
		which completes the proof.
	\end{proof}
	
	\subsection{Deterministic changes}
	Next, we gather bounds mostly concerning the behavior of deterministic trajectories and their one-step changes.
	To this end, we state the following consequence of Taylor's theorem.
	
	\begin{lemma}[Taylor's theorem]\label{lemma: taylor}
		Let~$a<x<x+1<b$ and suppose~$f\colon (a,b)\rightarrow\bR$ is twice continuously differentiable.
		Then,
		\begin{equation*}
			f(x+1)=f(x)+f'(x)\pm \max_{\xi\in [x,x+1]}\abs{f''(\xi)}.
		\end{equation*}
	\end{lemma}

	\begin{observation}\label{observation: derivative phihat}
		Extend~$\phat$ and~$\phihat_{\cA,I}$ to continuous trajectories defined on the whole interval~$[0,i^\star+1]$ using the same expressions as above.
		Then, for~$x\in [0,i^\star+1]$,
		\begin{equation*}
			\begin{gathered}
				\phihat_{\cA,I}'(x)=-(\abs{\cA}-\abs{\cA[I]})\frac{\abs{\cF}k!\,\phihat_{\cA,I}(x)}{n^k\phat(x)},\\
				\phihat_{\cA,I}''(x)=(\abs{\cA}-\abs{\cA[I]})(\abs{\cA}-\abs{\cA[I]}-1)\frac{\abs{\cF}^2(k!)^2\phihat_{\cA,I}(x)}{n^{2k}\phat(x)^2}.
			\end{gathered}
		\end{equation*} 
	\end{observation}
	
	\begin{lemma}\label{lemma: delta phihat}
		Let~$0\leq i\leq i^\star$ and~$\cX:=\set{i\leq \tau_\emptyset}$.
		Suppose that~$(\cA,I)$ is a template.
		Then,
		\begin{equation*}
			\Delta \phihat_{\cA,I}\Xeq -(\abs{\cA}-\abs{\cA[I]})\frac{\abs{\cF}\phihat_{\cA,I}}{H}\pm \frac{\zeta^{2+\eps^2}\phihat_{\cA,I}}{H}.
		\end{equation*}
	\end{lemma}
	
	\begin{proof}
		This is a consequence of Taylor's theorem.
		In detail, we argue as follows.
		
		Together with Observation~\ref{observation: derivative phihat}, Lemma~\ref{lemma: taylor} yields
		\begin{equation*}
			\Delta\phihat_{\cA,I}
			=-(\abs{\cA}-\abs{\cA[I]})\frac{\abs{\cF}k!\,\phihat_{\cA,I}}{n^k\phat}\pm \max_{x\in[i,i+1]}(\abs{\cA}-\abs{\cA[I]})(\abs{\cA}-\abs{\cA[I]}-1)\frac{\abs{\cF}^2(k!)^2\phihat_{\cA,I}(x)}{n^{2k}\phat(x)^2}.
		\end{equation*}
		We investigate the first term and the maximum separately.
		Using Lemma~\ref{lemma: edges of H}, we have
		\begin{equation*}
			-(\abs{\cA}-\abs{\cA[I]})\frac{\abs{\cF}k!\,\phihat_{\cA,I}}{n^k\phat}
			\Xeq-(\abs{\cA}-\abs{\cA[I]})\frac{\abs{\cF}\phihat_{\cA,I}}{H}.
		\end{equation*}
		If~$\phihat_{\cA,I}(x)/\phat(x)^2$ is not decreasing in~$x$ for~$x\in[i,i+1]$, then~$\abs{\cA}-\abs{\cA[I]}=0$ or~$\abs{\cA}-\abs{\cA[I]}-1=0$.
		Hence, Lemma~\ref{lemma: edges of H} together with Lemma~\ref{lemma: zeta and H} yields
		\begin{align*}
			\max_{x\in[i,i+1]}(\abs{\cA}-\abs{\cA[I]})(\abs{\cA}-\abs{\cA[I]}-1)\frac{\abs{\cF}^2(k!)^2\phihat_{\cA,I}(x)}{n^{2k}\phat(x)^2}
			&\leq \frac{\abs{\cA}^2\abs{\cF}^2(k!)^2\phihat_{\cA,I}}{n^{2k}\phat^2}
			\Xeq \frac{\abs{\cA}^2\abs{\cF}^2\phihat_{\cA,I}}{H^2}\\
			&\leq \frac{\zeta^{2+2\eps^2}\abs{\cA}^2\abs{\cF}^2\phihat_{\cA,I}}{H}
			\leq \frac{\zeta^{2+\eps^2}\phihat_{\cA,I}}{H},
		\end{align*}
		which completes the proof.
	\end{proof}
	
	\begin{lemma}\label{lemma: trajectory at cutoff}
		Let~$\alpha\geq 0$.
		Suppose that~$(\cA,I)$ is a template with~$\abs{V_\cA}\leq 1/\eps^4$ and~$i^\alpha_{\cA,I}\geq 1$.
		Let~$0\leq i\leq i^\alpha_{\cA,I}$.
		Then,~$\phihat_{\cA,I}\geq (1-n^{-\eps^3})\zeta^{-\alpha}$.
	\end{lemma}
	
	\begin{proof}
		For~$j\geq 0$, let~$\psihat_{\cA,I}(j):=\zeta(j)^\alpha\phihat_{\cA,I}(j)$.
		It suffices to show that~$\psihat_{\cA,I}\geq (1-n^{-\eps})$.
		Note that~$\psihat_{\cA,I}(j)\geq 1$ for all~$0\leq j\leq i^\alpha_{\cA,I}-1$.
		If~$\abs{\cA}-\abs{\cA[I]}-\alpha\rho_\cF/2\leq 0$, then~$\psihat_{\cA,I}\geq \psihat_{\cA,I}(0)\geq 1$.
		Otherwise, from Lemma~\ref{lemma: bounds of delta phat}, we obtain
		\begin{equation*}
			\psihat_{\cA,I}
			\geq\psihat_{\cA,I}(i^\alpha_{\cA,I})
			\geq (1-n^{-\eps^2})^{\abs{\cA}}\psihat_{\cA,I}(i^\alpha_{\cA,I}-1)
			\geq (1-n^{-\eps^2})^{\abs{\cA}}
			\geq (1-n^{-\eps^3}),
		\end{equation*}
		which completes the proof.
	\end{proof}
	
	\begin{lemma}\label{lemma: weak strictly balanced bound}
		Suppose that~$(\cA,I)$ is a strictly balanced template with~$\abs{V_\cA}\leq 1/\eps^4$.
		Let~$i\geq 0$ and~$\cX:=\set{i<\tau_\ccB\wedge \tau_{\ccB'}}$.
		Let~$\psi\colon V_\cA\injection V_\cH$.
		Then,~$\Phi_{\cA,\psi}\Xleq (1+\log n)^{\alpha_{\cA,I}}(1\vee \phihat_{\cA,I})$.
	\end{lemma}
	
	\begin{proof}
		We may assume that~$I\neq V_\cA$.
		If~$i_{\cA,I}^0=0$, then~$\phihat_{\cA,I}(0)\leq 1$ and thus, since~$\cH$ is~$(\eps^4,\delta,\rho_\cF)$-pseudorandom, we have
		\begin{equation*}
			\Phi_{\cA,I}
			\leq\Phi_{\cA,I}(0)
			\leq (\log n)^{3(\abs{V_\cA}-\abs{I})/2}
			\leq (1+\log n)^{\alpha_{\cA,I}}.
		\end{equation*}
		Hence, we may also assume that~$(\cA,I)\in\ccB'$.
		If~$i\geq i_{\cA,I}^{0}$, then Lemma~\ref{lemma: trajectory at cutoff} entails
		\begin{equation*}
			\Phi_{\cA,I}
			\leq\Phi_{\cA,I}(i_{\cA,I}^{0})
			\Xleq (1+(\log n)^{\alpha_{\cA,I}}\phihat_{\cA,I}(i_{\cA,I}^0)^{-\delta^{1/2}})\phihat_{\cA,I}(i_{\cA,I}^0)
			\leq (1+\log n)^{\alpha_{\cA,I}},
		\end{equation*}
		so we may additionally assume that~$i< i_{\cA,I}^0$.
		If~$i\geq i_{\cA,I}^{\delta^{1/2}}$, then
		\begin{equation*}
			\Phi_{\cA,I}
			\Xleq (1+(\log n)^{\alpha_{\cA,I}}\phihat_{\cA,I}^{-\delta^{1/2}})\phihat_{\cA,I}
			\leq (1+\log n)^{\alpha_{\cA,I}}\phihat_{\cA,I}.
		\end{equation*}
		Hence, we may also additionally assume that~$i< i_{\cA,I}^{\delta^{1/2}}$ and thus in particular~$(\cA,I)\in\ccB$.
		Then,
		\begin{equation*}
			\Phi_{\cA,I}
			\Xleq (1+\zeta^\delta)\phihat_{\cA,I}
			\leq (1+\log n)^{\alpha_{\cA,I}}\phihat_{\cA,I},
		\end{equation*}
		which completes the proof.
	\end{proof}
	
	\subsection{Control over templates}
	Here, we present three statements that show that control over the numbers of balanced templates and strictly balanced templates also provides some control over the number of certain templates that are not necessarily balanced.
	Lemma~\ref{lemma: arbitrary embedding} may be interpreted as a generalization of~\cite[Corollary~3.3]{BFL:15} and, with respect to the main part of the analysis, plays a similar auxiliary role.
	
	\begin{lemma}\label{lemma: arbitrary embedding}
		Let~$i\geq 0$ and let~$\cX:=\set{i<\tau_\ccB\wedge\tau_{\ccB'}}$.
		Suppose that~$(\cA,I)$ is a template with~$\abs{V_\cA}\leq 1/\eps^4$ and let~$\psi\colon I\injection V_\cH$.
		Then, the following holds.
		\begin{enumerate}[label=\textup{(\roman*)}]
			\item\label{item: arbitrary non-vanishing embedding} If~$\phihat_{\cB,I}\geq 1$ for all~$(\cB,I)\subseteq (\cA,I)$, then~$\Phi_{\cA,\psi}\Xleq (1+\log n)^{\alpha_{\cA,I}}\phihat_{\cA,I}$;
			\item\label{item: arbitrary vanishing embedding} If~$\phihat_{\cA,J}\leq 1$ for all~$I\subseteq J\subseteq V_\cA$, then~$\Phi_{\cA,\psi}\Xleq (1+\log n)^{\alpha_{\cA,I}}$.
		\end{enumerate}
	\end{lemma}
	
	\begin{proof}
		We use induction on~$\abs{V_\cA}-\abs{I}$ to show that~\ref{item: arbitrary non-vanishing embedding} and~\ref{item: arbitrary vanishing embedding} hold.
		If~$\abs{V_\cA}-\abs{I}=0$, then~$\Phi_{\cA,\psi}=1=\phihat_{\cA,I}$ and hence~\ref{item: arbitrary non-vanishing embedding} and~\ref{item: arbitrary vanishing embedding} are true.
		
		Let~$\ell\geq 1$ and suppose that both statements hold if~$\abs{V_\cA}-\abs{I}\leq \ell-1$.
		Suppose that~$\abs{V_\cA}-\abs{I}=\ell$.
		First, suppose that there is an isolated vertex~$v\notin I$ in~$\cA$.
		If~$\phihat_{\cB,I}\geq 1$ for all~$(\cB,I)\subseteq (\cA,I)$, using the induction hypothesis, we obtain
		\begin{equation*}
			\Phi_{\cA,\psi}
			= (n-\abs{V_\cA}+1)\cdot  \Phi_{\cA-\set{v},\psi}
			\Xleq (1+\log n)^{\alpha_{\cA-\set{v},I}}\phihat_{\cA,I}
			\leq (1+\log n)^{\alpha_{\cA,I}}\phihat_{\cA,I},
		\end{equation*}
		so~\ref{item: arbitrary non-vanishing embedding} holds.
		Furthermore, we have~$\phihat_{\cA,V_\cA\setminus \set{v}}=n>1$, so~\ref{item: arbitrary vanishing embedding} is vacuously true.
		
		Hence, now suppose that there is no isolated vertex~$v\notin I$ in~$\cA$.
		Let~$I\subseteq U\subseteq V_\cA$ such that~$\rho_{\cA[U],I}$ is maximal and subject to this, that~$\abs{U}$ is minimal.
		Then,~$(\cA[U],I)$ is strictly balanced.
		Furthermore, since there are no isolated vertices~$v\notin I$ in~$\cA$, we have~$\rho_{\cA[U],I}\geq \rho_{\cA,I}>0$ by choice of~$U$ and hence~$U\neq I$.
		Note that
		\begin{equation}\label{equation: embedding in two steps}
			\Phi_{\cA,\psi}=\sum_{\phi\in \Phi_{\cA[U],\psi}^\sim} \Phi_{\cA,\phi}.
		\end{equation}
		To obtain~\ref{item: arbitrary non-vanishing embedding} and~\ref{item: arbitrary vanishing embedding}, we use the strict balancedness of~$(\cA[U],I)$ to bound~$\Phi_{\cA[U],\psi}$ and the induction hypothesis to bound~$\Phi_{\cA,\phi}$ for all~$\phi\in \Phi_{\cA[U],\psi}^\sim$.
		
		In more detail, for~\ref{item: arbitrary non-vanishing embedding} we argue as follows.
		Suppose that~$\phihat_{\cB,I}\geq 1$ for all~$(\cB,I)\subseteq (\cA,I)$.
		For all~$(\cB,U)\subseteq (\cA,U)$ and~$\cB':=\cB+\cA[U]$, we have~$\rho_{\cB',I}\leq \rho_{\cA[U],I}$ by choice of~$U$.
		Thus, since~$\cB'[U]=\cA[U]$ and~$\cB'[I]=\cA[I]$, we obtain
		\begin{align*}
			\phihat_{\cB,U}
			&=\phihat_{\cB',U}
			=n^{\abs{V_{\cB'}}-\abs{U}}\phat^{\abs{\cB'}-\abs{\cB'[I]}-(\abs{\cA[U]}-\abs{\cA[I]})}
			=n^{\abs{V_{\cB'}}-\abs{U}}\phat^{\rho_{\cB',I}(\abs{V_{\cB'}}-\abs{I})-\rho_{\cA[U],I}(\abs{U}-\abs{I})}\\
			&\geq n^{\abs{V_{\cB'}}-\abs{U}}\phat^{\rho_{\cB',I}(\abs{V_{\cB'}}-\abs{U})}
			=\phihat_{\cB',I}^{(\abs{V_{\cB'}}-\abs{U})/(\abs{V_{\cB'}}-\abs{I})}
			\geq 1.
		\end{align*} 
		Hence, for all~$\phi\in\Phi_{\cA[U],\psi}^\sim$, by induction hypothesis,
		\begin{equation}\label{equation: arbitrary embedding i second step}
			\Phi_{\cA,\phi}\Xleq (1+\log n)^{\alpha_{\cA,U}}\phihat_{\cA,U}.
		\end{equation}
		The template~$(\cA[U],I)\subseteq (\cA,I)$ is strictly balanced.
		Furthermore, since we suppose that~$\phihat_{\cB,I}\geq 1$ for all~$(\cB,I)\subseteq (\cA,I)$, we have~$\phihat_{\cA[U],I}\geq 1$.
		Thus, Lemma~\ref{lemma: weak strictly balanced bound} entails
		\begin{equation}\label{equation: arbitrary embedding i first step}
			\Phi_{\cA[U],\psi}\Xleq (1+\log n)^{\alpha_{\cA[U],I}}\phihat_{\cA[U],I}.
		\end{equation}
		Combining~\eqref{equation: arbitrary embedding i second step} and~\eqref{equation: arbitrary embedding i first step} with~\eqref{equation: embedding in two steps}, we obtain
		\begin{equation*}
			\Phi_{\cA,\psi}
			\Xleq (1+\log n)^{\alpha_{\cA[U],I}}\phihat_{\cA[U],I}\cdot (1+\log n)^{\alpha_{\cA,U}}\phihat_{\cA,U}.
		\end{equation*}
		Hence, employing Observation~\ref{observation: choice of alpha} as well as Lemma~\ref{lemma: extension trajectory split} yields
		\begin{equation*}
			\Phi_{\cA,\psi}
			\Xleq(1+\log n)^{\alpha_{\cA,I}}\phihat_{\cA,I}
		\end{equation*}
		and thus shows that~\ref{item: arbitrary non-vanishing embedding} holds.
		Recall that, as mentioned above, when we use the fact expressed in Lemma~\ref{lemma: extension trajectory split}, we will not always explicitly reference this lemma.
		
		Let us turn to~\ref{item: arbitrary vanishing embedding}.
		Now, no longer suppose that necessarily~$\phihat_{\cB,I}\geq 1$ for all~$(\cB,I)\subseteq (\cA,I)$ and instead suppose that~$\phihat_{\cA,J}\leq 1$ for all~$I\subseteq J\subseteq V_\cA$.
		Then, in particular~$\phihat_{\cA,J}\leq 1$ for all~$U\subseteq J\subseteq V_\cA$.
		Hence, for all~$\phi\in\Phi_{\cA[U],\psi}^\sim$, by induction hypothesis,
		\begin{equation}\label{equation: arbitrary embedding ii second step}
			\Phi_{\cA,\phi}\Xleq (1+\log n)^{\alpha_{\cA,U}}.
		\end{equation}
		The template~$(\cA[U],I)\subseteq (\cA,I)$ is strictly balanced.
		Furthermore, since we suppose that~$\phihat_{\cA,J}\leq 1$ for all~$I\subseteq J\subseteq V_\cA$, so in particular~$\phihat_{\cA,I}\leq 1$, and since~$\rho_{\cA,I}\leq \rho_{\cA[U],I}$ by choice of~$U$, we obtain
		\begin{equation*}
			\phihat_{\cA[U],I}
			\leq n^{\abs{U}-\abs{I}}\phat^{\rho_{\cA,I}(\abs{U}-\abs{I})}
			= \phihat_{\cA,I}^{(\abs{U}-\abs{I})/(\abs{V_\cA}-\abs{I})}
			\leq 1.
		\end{equation*}
		Hence, Lemma~\ref{lemma: weak strictly balanced bound} entails
		\begin{equation}\label{equation: arbitrary embedding ii first step}
			\Phi_{\cA[U],\psi}
			\Xleq (1+\log n)^{\alpha_{\cA[U],I}}.
		\end{equation}
		Similarly as above, combining~\eqref{equation: arbitrary embedding ii second step} and~\eqref{equation: arbitrary embedding ii first step} with~\eqref{equation: embedding in two steps} and employing Observation~\ref{observation: choice of alpha} yields
		\begin{equation*}
			\Phi_{\cA,\psi}
			\Xleq (1+\log n)^{\alpha_{\cA[U],I}}\cdot (1+\log n)^{\alpha_{\cA,U}}
			\leq(1+\log n)^{\alpha_{\cA,I}}
		\end{equation*}
		and hence shows that~\ref{item: arbitrary vanishing embedding} holds.
	\end{proof}
	
	\begin{lemma}\label{lemma: arbitrary embedding split}
		Let~$i\geq 0$ and~$\cX:=\set{i<\tau_\ccB\wedge\tau_{\ccB'} }$.
		Suppose that~$(\cA,I)$ is a template with~$\abs{V_\cA}\leq 1/\eps^4$, let~$I\subsetneq J\subseteq V_\cA$ and from all subtemplates~$(\cB',I)\subseteq(\cA,I)$ with~$J\subseteq V_{\cB'}$, choose~$(\cB,I)$ such that~$\phihat_{\cB,I}$ is minimal.
		Let~$\psi\colon J\injection V_\cH$.
		Then,~$\Phi_{\cA,\psi}\Xleq (1+\log n)^{\alpha_{\cA,J}}\phihat_{\cA,I}/\phihat_{\cB,I}$.
	\end{lemma}

	\begin{proof}
		Since~$\abs{\cA[V_\cB]}-\abs{\cA[I]}\geq \abs{\cB}-\abs{\cB[I]}$ entails~$\phihat_{\cA[V_\cB],I}\leq \phihat_{\cB,I}$, we may assume that~$\cB$ is an induced subgraph of~$\cA$.
		Indeed, by choice of~$\cB$, we obtain~$\phihat_{\cA[V_\cB],I}= \phihat_{\cB,I}$, so we may replace~$\cB$ with~$\cA[V_\cB]$ since the statement we wish to obtain only depends on~$\phihat_{\cB,I}$.
		Note that
		\begin{equation}\label{equation: splitting at smallest trajectory}
			\Phi_{\cA,\psi}=\sum_{\phi\in \Phi^\sim_{\cB,\psi}}\Phi_{\cA,\phi}.
		\end{equation}
		We use Lemma~\ref{lemma: arbitrary embedding} to bound~$\Phi_{\cB,\psi}$ and~$\Phi_{\cA,\phi}$ for all~$\phi\in \Phi^\sim_{\cB,\psi}$.
		
		In more detail, we argue as follows.
		Let~$\phi\in \Phi^\sim_{\cB,\psi}$ and consider a subtemplate~$(\cC,V_\cB)\subseteq (\cA,V_\cB)$.
		Then, for~$\cC':=\cC+\cA[V_\cB]$, we have~$\phihat_{\cB,I}\leq \phihat_{\cC',I}$ by choice of~$(\cB,I)$ and hence
		\begin{equation*}
			\phihat_{\cC,V_\cB}
			=\phihat_{\cC',V_\cB}
			=\frac{\phihat_{\cC',I}}{\phihat_{\cB,I}}
			\geq 1.
		\end{equation*}
		Thus, Lemma~\ref{lemma: arbitrary embedding}~\ref{item: arbitrary non-vanishing embedding} entails
		\begin{equation}\label{equation: arbitrary embedding split second step}
			\Phi_{\cA,\phi}\Xleq (1+\log n)^{\alpha_{\cA,V_\cB}}\phihat_{\cA,V_\cB}.
		\end{equation}
		Next, in order so bound~$\Phi_{\cB,\psi}$, suppose that~$J\subseteq J'\subseteq V_\cB$.
		Then,~$\phihat_{\cB,I}\leq \phihat_{\cB[J'],I}$ by choice of~$(\cB,I)$ and hence
		\begin{equation*}
			\phihat_{\cB,J'}
			=\frac{\phihat_{\cB,I}}{\phihat_{\cB[J'],I}}
			\leq 1.
		\end{equation*}
		Thus, Lemma~\ref{lemma: arbitrary embedding}~\ref{item: arbitrary vanishing embedding} entails
		\begin{equation}\label{equation: arbitrary embedding split first step}
			\Phi_{\cB,\psi}\Xleq (1+\log n)^{\alpha_{\cB,J}}.
		\end{equation}
		
		Since~$\cB$ is an induced subgraph of~$\cA$, combining~\eqref{equation: arbitrary embedding split second step} and~\eqref{equation: arbitrary embedding split first step} with~\eqref{equation: splitting at smallest trajectory} and employing Observation~\ref{observation: choice of alpha} yields
		\begin{equation*}
			\Phi_{\cA,\psi}
			\Xleq (1+\log n)^{\alpha_{\cB,J}}\cdot (1+\log n)^{\alpha_{\cA,V_\cB}}\phihat_{\cA,V_\cB}
			\leq (1+\log n)^{\alpha_{\cA,J}}\phihat_{\cA,V_\cB}
			=(1+\log n)^{\alpha_{\cA,J}}\frac{\phihat_{\cA,I}}{\phihat_{\cB,I}},
		\end{equation*}
		which completes the proof.
	\end{proof}
	
	\begin{lemma}\label{lemma: loss at one edge}
		Let~$i\geq 0$ and~$\cX:=\set{i<\tau_\ccB\wedge\tau_{\ccB'}}$.
		Suppose that~$(\cA,I)$ is a~$k$-template with~$\abs{V_\cA}\leq 1/\eps^4$ and let~$\psi\colon I\injection V_\cH$. Let~$e\in\cA\setminus\cA[I]$ and from all subtemplates~$(\cB',I)\subseteq (\cA,I)$ with~$e\in\cB'$, choose~$(\cB,I)$ such that~$\phihat_{\cB,I}$ is minimal.
		Then,
		\begin{equation*}
			\abs{\cset{ \phi\in\Phi^\sim_{\cA,\psi} }{ \phi(e)\in \cF_0(i+1) }} \Xleq 2k!\abs{\cF}(\log n)^{\alpha_{\cA,I\cup e}}\frac{\phihat_{\cA,I}}{\phihat_{\cB,I}}.
		\end{equation*}
	\end{lemma}
	
	\begin{proof}
		Note that
		\begin{equation*}
			\abs{\cset{ \phi\in\Phi^\sim_{\cA,\psi} }{ \phi(e)\in \cF_0(i+1) }}\leq \sum_{f\in\cF_0(i+1)}\abs{\cset{ \phi\in\Phi^\sim_{\cA,\psi} }{ \phi(e)=f }},
		\end{equation*}
		so it suffices to obtain
		\begin{equation*}
			\abs{\cset{ \phi\in\Phi^\sim_{\cA,\psi} }{ \phi(e)=f }}\Xleq 2k!\,(\log n)^{\alpha_{\cA,I\cup e}}\frac{\phihat_{\cA,I}}{\phihat_{\cB,I}}.
		\end{equation*}
		for all~$f\in\cH(0)$.
		This is a consequence of Lemma~\ref{lemma: arbitrary embedding split}
		
		In detail, we argue as follows.
		Fix~$f\in \cH(0)$.
		We have
		\begin{equation}\label{equation: embeddings using f}
			\abs{\cset{ \phi\in\Phi^\sim_{\cA,\psi} }{ \phi(e)=f }}
			\leq \sum_{\psi'\colon I\cup e\injection \psi(I)\cup f\colon \restr{\psi'}{I}=\psi} \Phi_{\cA,\psi'}.
		\end{equation}
		For~$\psi'\colon I\cup e\injection \psi(I)\cup f$, Lemma~\ref{lemma: arbitrary embedding split} entails
		\begin{equation*}
			\Phi_{\cA,\psi'}
			\Xleq (1+\log n)^{\alpha_{\cA,I\cup e}} \frac{\phihat_{\cA,I}}{\phihat_{\cB,I}}
			\leq 2(\log n)^{\alpha_{\cA,I\cup e}} \frac{\phihat_{\cA,I}}{\phihat_{\cB,I}}.
		\end{equation*}
		Combining this upper bound with~\eqref{equation: embeddings using f} completes the proof.
	\end{proof}
	
	\subsection{Degrees}\label{subsection: degrees}
	The numbers of embeddings of the templates~$(\cF,f)$ where~$f\in\cF$ play a special role since they are closely related to the degrees in~$\cH^*$. 
	
	\begin{lemma}\label{lemma: star degrees}
		Let~$i\geq 0$ and~$e\in\cH$.
		Then,
		\begin{equation*}
			d_{\cH^*}(e)=\frac{\sum_{f\in\cF}\sum_{\psi\colon f\bijection e} \Phi_{\cF,\psi}}{\aut(\cF)}.
		\end{equation*}
	\end{lemma}
	
	\begin{proof}
		Let~$\psi_0\colon\emptyset\to V_\cH$.
		We have
		\begin{equation*}
			d_{\cH^*}(e)
			=\frac{\abs{\cset{\phi\in \Phi^\sim_{\cF,\psi_0}}{ e\in\phi(\cF) }}}{\aut(\cF)}
			=\frac{\sum_{f\in\cF}\sum_{\psi\colon f\bijection e}\Phi_{\cF,\psi}}{\aut(\cF)},
		\end{equation*}
		which completes the proof.

	\end{proof}
	
	\begin{lemma}\label{lemma: star codegrees}
		Let~$0\leq i\leq i^\star$ and~$\cX:=\set{ i<\tau_\ccB\wedge\tau_{\ccB'} }$.
		Consider distinct~$e_1,e_2\in\cH$ and~$f\in\cF$.
		Then,~$d_{\cH^*}(e_1,e_2)\Xleq \zeta^{2+\eps^2}\phihat_{\cF,f}$.
	\end{lemma}
	\begin{proof}
		We have
		\begin{equation*}
			d_{\cH^*}(e_1,e_2)\leq \sum_{f_1,f_2\in \cF} \sum_{\psi\colon f_1\cup f_2\bijection e_1\cup e_2} \Phi_{\cF,\psi}.
		\end{equation*}
		Fix distinct~$f_1,f_2\in\cF$,~$J:=f_1\cup f_2$ and~$\psi\colon J\bijection e_1\cup e_2$.
		We obtain a suitable upper bound for~$\Phi_{\cF,\psi}$ from Lemma~\ref{lemma: arbitrary embedding split} as follows.
		
		Since~$(\cF,f_1)$ is balanced, for all~$(\cA,f_1)\subseteq(\cF,f_1)$ with~$J\subseteq V_\cA$, we have~$\rho_{\cA,f_1}\leq \rho_\cF$ and hence using Lemma~\ref{lemma: bounds of phat}, we obtain
		\begin{equation*}
			\phihat_{\cA,f_1}
			=(n\phat^{\rho_{\cA,f_1}})^{\abs{V_\cA}-k}
			\geq (n\phat^{\rho_{\cF}})^{\abs{V_\cA}-k}
			\geq n\phat^{\rho_\cF}.
		\end{equation*}
		Thus, Lemma~\ref{lemma: arbitrary embedding split} together with Lemma~\ref{lemma: bounds of zeta} entails
		\begin{equation*}
			\Phi_{\cF,\psi}
			\Xleq n^{\eps^2}\frac{\phihat_{\cF,f_1}}{n\phat^{\rho_\cF}}
			= n^{-\eps^2}\zeta^2\phihat_{\cF,f_1}
			\leq n^{-\eps^2/2}\frac{\zeta^{2}\phihat_{\cF,f_1}}{\abs{\cF}^2 (2k)!}
			\leq \frac{\zeta^{2+\eps^2}\phihat_{\cF,f_1}}{\abs{\cF}^2 (2k)!},
		\end{equation*}
		which completes the proof.
	\end{proof}
	
	\subsection{Concentration of key quantities}\label{subsection: auxiliary control}
	Overall our proof relies on showing that key quantities that are crucial for our precise analysis of the process are typically concentrated around a deterministic trajectory.
	Establishing concentration for any of these quantities relies on the assumption that the other key quantities behave as expected.
	More specifically, for certain collections of key quantities, we show that it is unlikely that a key quantity from this collection is the first among all key quantities to significantly deviate from its corresponding trajectory as long as only steps~$0\leq i\leq i^\star$ are considered.
	Before we turn to the core of our argument that allows us to analyze the removal process in our very general setting, we end this section with Lemma~\ref{lemma: auxiliary control} below that provides such statements for three collections of key quantities that correspond to stopping times defined above.
	Recall that as defined~\eqref{equation: definition of tautilde star}, the stopping time in~$\tautilde^\star$ is the minimum of the four stopping times introduced in Section~\ref{section: stopping times}.
	
	\begin{lemma}\label{lemma: auxiliary control}
		\begin{enumerate}[label=\textup{(\roman*)}]
			\item\label{item: control copies} $\pr{\tau_{\cH^*}\leq\tautilde^\star\wedge i^\star}\leq \exp(-n^{\eps^2})$.
			\item\label{item: control balanced} $\pr{\tau_{\ccB}\leq \tautilde^\star \wedge i^\star}\leq \exp(-n^{\delta^2})$.
			\item\label{item: control strictly balanced} $\pr{\tau_{\ccB'}\leq\tautilde^\star \wedge i^\star}\leq \exp(-(\log n)^{3/2})$.
		\end{enumerate}
	\end{lemma}
	
	The three parts of Lemma~\ref{lemma: auxiliary control} can be proven by standard applications of the critical interval method.
	Essentially, the argumentation for the analogous statements in the triangle case, see~\cite[Sections~2 and~3]{BFL:15}, can be adapted to the more general setting without encountering any major obstacles.
	We remark that for Lemma~\ref{lemma: auxiliary control}~\ref{item: control copies}, similarly as in~\cite{BFL:15}, it is crucial to exploit that if for some~$i\geq 0$, the hypergraph~$\cH^*$ is approximately vertex-regular and has negligible~$2$-degrees, we may approximate
	\begin{equation*}
		\exi{\Delta H^*}
		\approx -\frac{1}{H^*}\sum_{\cF'\in\cH^*}\sum_{e\in\cF'}d_{\cH^*}(e)
		=-\frac{1}{H^*}\sum_{e\in\cH^*} d_{\cH^*}(e)^2
		\approx -\frac{1}{H^*}\frac{\paren{\sum_{e\in\cH^*} d_{\cH^*}(e)}^2}{H}
		=-\frac{\abs{\cF}^2 H^*}{H}.
	\end{equation*}
	Formally, one may rely on the following simple Lemma from~\cite{BB:19} which we also apply further below.
	
	\begin{lemma}[{\cite[Lemma~3.1]{BB:19}}]\label{lemma: averaging}
		Let~$a,a_1,\ldots,a_n$ and~$b,b_1,\ldots,b_n$ such that~$\abs{a_i-a}\leq \alpha$ and~$\abs{b_i-b}\leq \beta$ for all~$i,j\in[n]$.
		Then,
		\begin{equation*}
			\sum_{1\leq i\leq n}a_ib_i
			= \frac{1}{n}\paren[\Big]{\sum_{1\leq i\leq n}a_i}\paren[\Big]{\sum_{1\leq i\leq n}b_i} \pm 2\alpha\beta n.
		\end{equation*}
	\end{lemma}
	\begin{proof}
		Note that
		\begin{equation*}
			\sum_{1\leq i\leq n}a_ib_i-\frac{1}{n}\paren[\Big]{\sum_{1\leq i\leq n}a_i}\paren[\Big]{\sum_{1\leq i\leq n}b_i}
			=\sum_{1\leq i\leq n} (a_i-a)(b_i-b)-\frac{1}{n}\paren[\Big]{\sum_{1\leq i\leq n}(a_i-a)}\paren[\Big]{\sum_{1\leq i\leq n}(b_i-b)}.
		\end{equation*}
		By the triangle inequality, we have
		\begin{equation*}
			\abs[\Big]{\sum_{1\leq i\leq n} (a_i-a)(b_i-b)}\leq \alpha\beta n
			\qtand
			\abs[\Big]{\paren[\Big]{\sum_{1\leq i\leq n}(a_i-a)}\paren[\Big]{\sum_{1\leq i\leq n}(b_i-b)}}\leq \alpha\beta n^2,
		\end{equation*}
		so the statement follows.
	\end{proof}
	
	Furthermore, when adapting the arguments from the triangle case, Lemma~\ref{lemma: star codegrees} replaces the trivial upper bound on the~$2$-degrees in~$\cH^*$ (given two edges, there is at most one triangle containing both).
	For completeness, we provide proofs for the three parts of Lemma~\ref{lemma: auxiliary control} in Appendices~\ref{appendix: copies}--\ref{appendix: strictly balanced}.

\section{Chains}\label{section: chains}
	
	Our precise analysis of the hypergraph removal process crucially relies on precise estimates for the random variables~$\Phi_{\cF,\psi}$ where~$\psi\colon f\injection V_\cH$ for some~$f\in\cF$ that essentially correspond to the degrees in the random~$\abs{\cF}$-graph~$\cH^*$ (see Lemma~\ref{lemma: star degrees}).
	More precisely, Lemma~\ref{lemma: auxiliary control} provides estimates for key quantities at step~$i$ that hold with high probability only while~$i<\tau_\ccF$.
	To complete our argument based on stopping times as outlined at the end of Section~\ref{section: heuristics}, we need to show that this typically holds if~$i\leq i^\star$ provided that the key quantities analyzed in these previous sections behaved as expected up to this step.
	
	The desire to control these numbers of embeddings motivates the definition of a collection~$\frakC$ of carefully chosen templates that includes the templates~$(\cF,f)\in\ccF$.
	Before providing formal definitions of the concepts involved in the definitions of these templates in Section~\ref{subsection: formal chains}, we first give some motivation and intuition where we omit some details.
	
	We obtain the aforementioned templates from structures that we call \emph{chains} and remark that in~\cite{BFL:15}, substructures playing a similar role for the special case where~$\cF$ is a triangle are called \emph{ladders}.
	Similarly as in~\cite{BFL:15}, our choice of chains is based on the following idea.
	For a chain template~$(\cC,I)$,~$\psi\colon I\rightarrow V_\cH$ and~$e\in\cC\setminus\cC[I]$, to estimate the number of embeddings~$\varphi\in \Phi^\sim_{\cC,\psi}$ lost due to~$\varphi(e)\notin\cH(i+1)$, for an edge~$f\in\cF$ and a bijection~$\beta\colon f\bijection e$, we are interested in the number
	\begin{equation}\label{equation: true loss at one edge motivation}
		\sum_{\phi\in\Phi_{\cC,\psi}^\sim}\Phi_{\cF,\phi\circ\beta}
	\end{equation}
	Simply obtaining an estimate for this number based on our estimates for~$\Phi_{\cC,\psi}$ and~$\Phi_{\cF,\psi'}$ where~$\psi'\colon f\injection V_\cH$ would lead to an undesirable accumulation of errors.
	Instead, to achieve more precision that in the end allows us to closely follow the evolution of key quantities for a sufficient number of steps, the initial idea might be to include a chain in our collection~$\frakC$ that provides a template~$(\cC_+,I)$ where~$\cC_+$ is, in an intuitive sense, an \emph{extension} of~$\cC$ obtained from~$\cC$ by gluing a copy~$\cF'$ of~$\cF$ onto~$\cC$ such that for all~$v\in f$, the vertex~$v$ is identified with~$\beta(v)$ while no other vertices outside~$e$ and~$f$ are identified with one another. 
	Then, we could simply consider~$\Phi_{\cC_+,\psi}$.
	However, iterating this unrestricted extension approach yields a growing collection of chains that quickly becomes uncontrollable.
	To prevent this, we introduce another chain transformation that we call \emph{reduction} that is meant to counterbalance the extension steps by potentially removing vertices from chains that grow due to extension such that in the end, up to being copies of one another, we only need a finite collection of chains.
	In particular, we are interested in a transformation~$\cC''$ of~$\cC_+$ that we call \emph{branching} of~$\cC$ and that is obtained by combining an extension operation with a reduction operation. Formally, we define~$\cC''$ to be a suitable induced subgraph of~$\cC_+$.
	If for the vertex set~$V''$ of the branching~$\cC''$, the embeddings of the template~$(\cC_+,V'')$ can be controlled based on our estimates for embeddings of balanced templates, then it could appear sensible to approximate the number in~\eqref{equation: true loss at one edge motivation} as
	\begin{equation}\label{equation: branching without support}
		\sum_{\phi\in\Phi_{\cC,\psi}^\sim}\Phi_{\cF,\phi\circ\beta}
		\approx\sum_{\phi\in\Phi_{\cC'',\psi}^\sim}\Phi_{\cC_+,\phi}.
	\end{equation}
	Recall that our motivation was to analyze the one-step changes of~$\Phi_{\cC,\psi}$ and that our goal is to exploit the self-correcting behavior of this number of embeddings in the following sense: If there are more embeddings than expected, then it is more likely that embeddings get destroyed hence providing a self-correcting drift (and similarly if there are fewer embeddings than expected).
	With the expression in~\eqref{equation: branching without support} based only on the branching, this is hard to exploit directly since there is no explicit dependence on~$\Phi_{\cC,\psi}$.
	To remedy this, we define another chain, which we call \emph{support}, that is obtained from the branching through another transformation, which we call \emph{truncation}.
	During truncation, we remove what remains of the vertices that were added when the copy~$\cF'$ was glued onto~$\cC$ and we choose the branching such that this truncation can be undone by again gluing the copy~$\cF'$ onto the support.
	This yields an induced subgraph~$\cC'$ of~$\cC$ which only depends on~$e$ and the original chain.
	We ensure that for the vertex set~$V'$ of the support, the embeddings of the template~$(\cC,V')$ can be controlled based on our estimates for embeddings of balanced templates.
	Then, Lemma~\ref{lemma: averaging} allows us approximate the number in~\eqref{equation: true loss at one edge motivation} as
	\begin{equation}\label{equation: approximating one step motivation}
		\begin{aligned}
			\sum_{\phi\in\Phi_{\cC,\psi}^\sim}\Phi_{\cF,\phi\circ\beta}
			&=\sum_{\psi'\in \Phi^\sim_{\cC',\psi}} \Phi_{\cF,\psi'\circ\beta} \Phi_{\cC,\psi'}\\
			&\approx \frac{\paren[\big]{\sum_{ \psi'\in\Phi_{\cC',\psi}^\sim }\Phi_{\cF,\psi'\circ\beta} }\paren[\big]{\sum_{ \psi'\in\Phi_{\cC',\psi}^\sim }\Phi_{\cC,\psi'} }}{\Phi_{\cC',\psi}}
			\approx\frac{\Phi_{\cC'',\psi} }{\Phi_{\cC',\psi}}\Phi_{\cC,\psi}.
		\end{aligned}
	\end{equation}
	The choice for our collection~$\frakC$ of chains is motivated by the fact that for such an argument,~$\frakC$ needs to be closed under taking branchings and supports of chains contained in~$\frakC$.
	
	In Section~\ref{subsection: formal chains}, we formally define the terms \emph{chain}, \emph{extension}, \emph{truncation}, \emph{reduction}, \emph{branching} and \emph{support} and we fix our collection~$\frakC$.
	In Section~\ref{subsection: averaging}, we turn the motivation outlined here into formal arguments to obtain a version of~\eqref{equation: approximating one step motivation} with quantified errors.
	Our arguments that rely on the self-correcting behavior require a careful choice of error terms as well as a consideration of chains in groups that we call \emph{branching families} to exploit symmetry that we discuss in Section~\ref{subsection: error parameter}.
	While we defer the analysis of branching families to Section~\ref{section: branching families}, we define them in Section~\ref{subsection: tracking chains} and subsequently use them in a supermartingale argument based on the insight from Section~\ref{subsection: averaging} that ensures that the embeddings of chains are typically concentrated as desired.

	\subsection{Formal definition}\label{subsection: formal chains}

	Consider a sequence~$A=\cA_1,\ldots,\cA_\ell$ of~$k$-graphs where~$\ell\geq 0$ and for~$0\leq i\leq \ell$ define~$q_i:=1+\sum_{1\leq j\leq i}(\abs{\cA_j}-1)$.
	We say that~$A$ is a \emph{loose path} starting at a~$k$-set~$I$ if there exists an ordering~$e_1,\ldots,e_{q_\ell}$ of~$\cA_1+\ldots+\cA_\ell$ such that~$e_1=I$ and such that~$\cA_i=\set{ e_{q_{i-1}},\ldots, e_{q_i} }$ for all~$1\leq i\leq \ell$.
	We call~$A$ \emph{vertex-separated} if~$V_{\cA_1+\ldots+\cA_{i-1}}\cap V_{\cA_i+\ldots+\cA_\ell}=e_{q_{i-1}}$ for all~$2\leq i\leq\ell$. 
	
	A triple~$\frakc=(F,V,I)$ where~$F=\cF_1,\ldots,\cF_\ell$ with~$\ell\geq 0$ is a \emph{chain} if~$F$ is the empty sequence and~$V=I$ is a~$k$-set or if~$F$ is a vertex-separated loose path of copies of~$\cF$ starting at~$I$ such that~$I\subseteq V\subseteq V_{\cF_1+\ldots+\cF_\ell}\subseteq V_\cF\cup\bN$.
	The choice of~$\bN$ here is essentially arbitrary and only serves to provide some infinite set of potential vertices, which is convenient when we want to consider the set of all chains.
	The \emph{chain template} given by~$\frakc$ is the template~$(\cC_{\frakc},I)$ where~$\cC_\frakc$ is the~$k$-graph with vertex set~$I$ and edge set~$\set{I}$ if~$\ell=0$ and where~$\cC_{\frakc}=(\cF_1+\ldots+\cF_\ell)[V]$\gladd{construction}{Cc}{$\cC_{(\cF_1,\ldots,\cF_\ell,V,I)}=\begin{cases}
			\cI	&\text{if~$\ell=0$};\\
			(\cF_1+\ldots+\cF_\ell)[V] &\text{otherwise},
		\end{cases}$ where~$V_\cI=I$,~$E(\cI)=\set{I}$} otherwise.
	
	We now formally define the three basic transformations of chains mentioned in the beginning of this section:~\emph{extension}, \emph{truncation} and~\emph{reduction}.
	
	For all~$\beta\colon f\bijection e$ where~$f\in\cF$ and~$e\in\cC_{\frakc}\setminus\cC_\frakc[I]$ such that~$e\notin\cF_i$ for all~$1\leq i\leq \ell-1$, fix an arbitrary copy~$\cF_{\frakc}^{\beta}$ of~$\cF$ with vertex set~$V_{\frakc}^{\beta}\subseteq e\cup \bN$ such that the following holds
	\begin{enumerate}[label=\textup{(\roman*)}]
		\item $e\in\cF_{\frakc}^{\beta}$;
		\item $V_{\cF_1+\ldots+\cF_\ell}\cap V_{\frakc}^{\beta}=e$;
		\item there exists a bijection~$\beta'\colon V_{\cF}\bijection V_{\frakc}^{\beta}$ with~$\beta'(f')\in\cF_{\frakc}^{\beta}$ for all~$f'\in\cF$ and~$\restr{\beta'}{f}=\beta$;
		\item $V_{\frakc}^{\beta_1}\cap V_{\frakc}^{\beta_2}=e $ for all distinct~$\beta_1\colon f_1\bijection e$ and~$\beta_2\colon f_2\bijection e$ with~$f_1,f_2\in\cF$.
	\end{enumerate}
	The~\emph{$\beta$-extension} of~$\frakc$ is the chain~$\frakc|\beta:=(F',V',I)$\gladd{construction}{c|beta}{$(\cF_1,\ldots,\cF_\ell,V,I)|\beta=(\cF_1,\ldots,\cF_\ell,\cF_\frakc^\beta, V\cup V_{\cF_\frakc^\beta},I)$} where
	\begin{equation*}
		F':=\cF_1,\ldots,\cF_\ell,\cF_{\frakc}^{\beta}
		\qtand
		V':=V\cup V_{\frakc}^{\beta}.
	\end{equation*}
	
	For~$0\leq \ell'\leq \ell$, the~\emph{$\ell'$-truncation} of~$\frakc$ is the chain~$\frakc|\ell':=(F',V',I)$\gladd{construction}{c|ell}{$(\cF_1,\ldots,\cF_\ell,V,I)|\ell'=(\cF_1,\ldots,\cF_{\ell'},V\cap V(\cF_1+\ldots+\cF_{\ell'}),I)$} where~$F'$ is the empty sequence and~$V'=I$ if~$\ell'=0$ and where
	\begin{equation*}
		F':=\cF_1,\ldots,\cF_{\ell'}
		\qtand
		V':=V\cap V_{\cF_1+\ldots+\cF_{\ell'}}
	\end{equation*}
	otherwise.
	For convenience, we set~$\frakc|\om:=\frakc|\ell-1$\gladd{construction}{c|-}{$(\cF_1,\ldots,\cF_\ell,V,I)|\om=(\cF_1,\ldots,\cF_\ell,V,I)|\ell-1$} if~$\ell\geq 1$.
	
	If~$\ell=0$, let~$W_\frakc:=V$.
	If~$\ell\geq 1$, then, among the vertex sets~$W$ with~$(V_{\cF_1}\cup V_{\cF_\ell})\cap V\subseteq W\subsetneq V$ and~$\rho_{ \cC_{\frakc},W }\leq \rho_\cF+\eps^2$, choose~$W_{\frakc}$ such that~$\abs{W_{\frakc}}$ is maximal if such a vertex set exists and choose~$W_{\frakc}=V$ otherwise\gladd{construction}{Wc}{$W_{(\cF_1,\ldots,\cF_\ell,V,I)}=\begin{cases}
			V &\parbox{6cm}{if~$\ell=0$ or~$\rho_{\cC_{(\cF_1,\ldots,\cF_\ell,V,I)},W}>\rho_\cF+\eps^2$\\for all~$(V_{\cF_1}\cup V_{\cF_\ell})\cap V\subseteq W\subsetneq V$;}\\
			\displaystyle\argmax_{\substack{(V_{\cF_1}\cup V_{\cF_\ell})\cap V\subseteq W\subsetneq V\colon\\ \rho_{\cC_{(\cF_1,\ldots,\cF_\ell,V,I)},W}\leq\rho_\cF+\eps^2 }} \abs{W}&\text{otherwise}	
		\end{cases}$}.
	The \emph{reduction} of~$\frakc$ is the chain~$\frakc|\R$ inductively defined as follows.
	If~$W_{\frakc}=V$, then~$\frakc|\R:=\frakc$.
	If~$W_{\frakc}\neq V$, then~$\frakc|\R:=(F,W_{\frakc},I)|\R$\gladd{construction}{c|r}{$(F,V,I)|\R=\begin{cases}
			(F,V,I) &\text{if~$W_{(F,V,I)} = V$;}\\
			(F,W_{(\cF_1,\ldots,\cF_\ell,V,I)},I)|\R &\text{otherwise}
		\end{cases}$}.
	It is easy to see that this indeed provides a well-defined reduction for all chains.
	Crucially, Lemma~\ref{lemma: extension is strictly balanced} guarantees that each reduction step corresponds to a strictly balanced extension in the sense that if~$W_{\frakc}\neq V$, then~$(\cC_{\frakc},W_{\frakc})$ is strictly balanced.

	\ifimages
	\begin{figure}
		\includegraphics{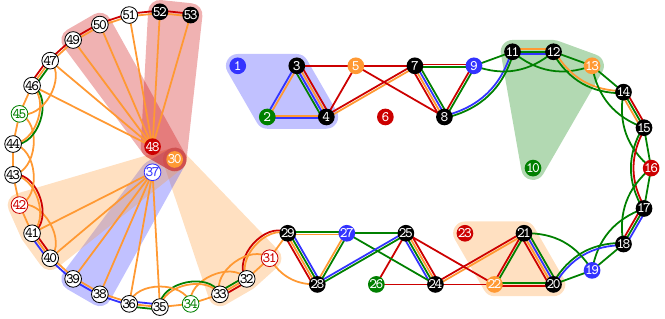}
		\includegraphics{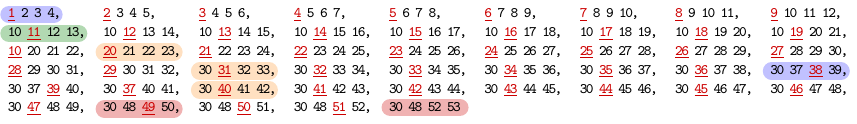}
		\caption{A~$3$-uniform chain template~$(\cC,I)$ for the special case where~$\cF=K^{(3)}_4$.
		The chain template is given by a chain~$\frakc=(F,V,I)$ where~$I=\protect\set{1,2}$ and where~$F=\cF_1,\ldots,\cF_{50}$ is a sequence of~$50$ copies of~$\cF$ whose vertices are elements of~$\protect\set{1,\ldots,54}$.
		The vertex sets of these copies are listed below the visualization of the chain template and for each copy~$\cF_i$ with~$1\leq i\leq 49$, the unique vertex of~$\cF_i$ that is not a vertex of~$\cF_{i+1}$ underlined.
		Instead of drawing the edges of~$\cC$, we instead draw edges of the links of selected colored, that is red, green blue or orange, vertices.
		Here, the \emph{link} of a vertex~$u\in V_\cC$ is the~$2$-graph with vertex set~$V_\cC$ where~$\protect\set{v,w}$ is an edge if~$\protect\set{u,v,w}\in \cC$.
		To distinguish more clearly between edges of~$\cC$ and edges of the links, here we call edges of~$\cC$ \emph{faces}. 
		For every face~$f\in\cC$, there exists a colored vertex~$v\in f$ such that~$f\setminus\protect\set{v}$ is one of the edges of the link of~$v$ that is drawn in the same color as~$v$.
		Hence, for a vertex~$u$, incident faces are represented either by incident edges of a link of another vertex or as edges that have the same color as~$u$.
		Not all edges of the link of a colored vertex are drawn.
		Every face is represented by exactly one drawn edge, so in particular, the number of faces is the number of drawn edges.
		Exactly two vertices of every copy in~$F$ are colored.
		Furthermore, the drawn edges are selected such that every copy~$\cF'$ in~$F$ corresponds to a monochromatic triangle together with a vertex of the same color in the following sense: the vertex together with the vertices of the triangle forms the vertex set of~$\cF'$ and the edges of the triangle together with an edge that has the same color as the unique colored vertex of the triangle represent the faces of~$\cF'$.
		Selected copies in~$F$ are highlighted using a colored background.\\
		Suppose that~$\eps^2=1/10$.
		Then~$W_{\frakc}=\protect\set{1,\ldots,30,48,52,53}$ and the vertices outside this set are highlighted.
		Note that for the chain~$\frakc':=(F',V',I)$ with~$F'=\cF_1,\ldots,\cF_{49}$ and~$V'=\protect\set{1,\ldots,52}$, the reduction operation is trivial in the sense that~$\frakc'|\R=\frakc'$ due to~$W_{\frakc'}=V'$.
		Hence, an extension that transforms~$\frakc'$ into~$\frakc$ transforms a chain where reduction is trivial into a chain where this is not the case. 
		}
	\end{figure}
	\fi
	
	With these transformations, we can now formally define \emph{branching} and \emph{support}.
	Let~$\beta\colon f\bijection e$ where~$f\in\cF$ and~$e\in\cC_{\frakc}\setminus\cC_\frakc[I]$ and suppose that~$\ell'\geq 0$ is minimal such that~$e\in\cC_{\frakc|\ell'}$.
	We say that~$\frakc|[\beta]:=\frakc|\ell'|\beta|\R$\gladd{construction}{c|[beta]}{$\frakc|[\beta]=\frakc|\ell'|\beta|\R$, where~$\ell'\geq 0$ is minimal such that the image of~$\beta$ is an edge of~$\cC_{\frakc|\ell'}$} is the~\emph{$\beta$-branching} of~$\frakc$ and that the chain~$\frakc|e:=\frakc|[\beta]|\om$\gladd{construction}{c|e}{$\frakc|e=\frakc|[\beta]|\om$}, which only depends on~$e$ and~$\frakc$, is the~\emph{$e$-support} in~$\frakc$.
	
	Suppose that~$U\subseteq V$.
	For~$\psi\colon U\injection V_\cH$ and~$i\geq 0$, we set~$\Phi_{\frakc,\psi}^\sim(i):=\Phi_{\cC_\frakc,\psi}^\sim$\gladd{RV}{Phicpsitilde}{$\Phi_{\frakc,\psi}^\sim(i)=\Phi_{\cC_\frakc,\psi}^\sim(i)$} and~$\Phi_{\frakc,\psi}(i):=\abs{\Phi_{\frakc,\psi}^\sim}$\gladd{realRV}{Phicpsi}{$\Phi_{\frakc,\psi}(i)=\abs{\Phi_{\frakc,\psi}^\sim(i)}$}.
	Furthermore, we set~$\phihat_{\frakc,U}:=\phihat_{\cC_\frakc,U}$.
	
	Finally, we choose the collection of chains~$\frakc=(F,V,I)$ where we are interested in~$\Phi_{\frakc,\psi}$ for~$\psi\colon I\injection V_\cH$.
	We call a collection~$\frakC'$ of chains \emph{admissible} if it satisfies the following properties.
	\begin{enumerate}[label=(\roman*)]
		\item $(\cF,V_\cF,f)\in\frakC'$ for all~$f\in\cF$.
		\item For all~$\frakc=(F,V,I)\in \frakC'$ where~$F$ has length~$\ell$ and all~$1\leq\ell'\leq\ell$, we have~$\frakc|\ell'\in\frakC'$.
		\item For all~$\frakc=(F,V,I)\in \frakC'$ where~$F=\cF_1,\ldots,\cF_\ell$, and all~$\beta\colon f\injection e$ where~$f\in\cF$ and~$e\in\cC_{\frakc}\setminus\cC_\frakc[I]$ such that~$e\notin \cF_i$ for all~$1\leq i\leq \ell-1$, we have~$\frakc|\beta|\R\in\frakC'$.
	\end{enumerate}
	Every arbitrary intersection of admissible collections of chains is also admissible.
	Hence, there exists an admissible collection of chains that is minimal with respect to inclusion.
	We choose the collection~$\frakC$ of chains~$\frakc=(F,V,I)$ with~$F=\cF_1,\ldots,\cF_\ell$ where we are interested in~$\Phi_{\frakc,\psi}$ for all~$\psi\colon I\rightarrow V_\cH$ and~$i\geq 0$ as this minimal admissible collection.
	For our arguments, it is crucial that when considering the chains~$\frakc=(F,V,I)\in\frakC$, the template~$(\cC,I)$ is not too large and that we do not end up with too many random processes~$\Phi_{\frakc,\psi}(0),\Phi_{\frakc,\psi}(1),\ldots$ where~$\psi\colon I\injection V_\cH$ (note that we enforce no bound for the length of the sequence~$F$).
	Lemma~\ref{lemma: chain size} below provides suitable bounds for the sizes of the vertex set~$V$ which in turn yields a suitable bound for the number of such random processes (see Lemma~\ref{lemma: finite chain collection}).
	Lemmas~\ref{lemma: chains are reduced} and \ref{lemma: chains are not trivial} state simple useful properties of chains~$\frakc\in\frakC$ that are almost immediate from the definition of~$\frakC$.

	\begin{lemma}\label{lemma: ladder subextension density}
		Suppose that~$\frakc=(F,V,I)$ is a chain and let~$(\cA,I)\subseteq(\cC_\frakc,I)$.
		Then,~$\rho_{\cA,I}\leq\rho_\cF$.
	\end{lemma}
	\begin{proof}
		We may assume that~$F$ has length~$\ell\geq 1$ and that~$\cA$ is an induced subgraph of~$\cC_\frakc$.
		Suppose that~$F=\cF_1,\ldots,\cF_\ell$.
		For~$1\leq i\leq\ell$, let~$V_i:=V\cap V_{\cF_i}$.
		Let~$f_1:=I$ and for~$2\leq i\leq\ell$, let~$f_i\in\cF_{i-1}\cap \cF_{i}$.
		For~$1\leq i\leq\ell$, let~$U_i:=(V_\cA\cup f_i)\cap V_i$ and~$\cA_i:=\cF_i[U_i]$.
		Note that since~$(\cF_i,f_i)$ is balanced, we have~$\rho_{\cA_i,f_i}\leq \rho_\cF$.
		Since~$F$ is a vertex-separated loose path, we have
		\begin{equation*}
			V_\cA\setminus I
			=\bigcup_{1\leq i\leq \ell} U_{i}\setminus f_i.
		\end{equation*}
		and~$(U_i\setminus f_i)\cap (U_j\setminus f_j)=\emptyset$ for all~$1\leq i<j\leq \ell$.
		This entails~$\abs{V_\cA}-\abs{I}=\sum_{1\leq i\leq \ell}\abs{U_i}-\abs{f_i}$.
		Furthermore,
		\begin{equation*}
			\cA\setminus\cA[I]
			=\bigcup_{1\leq i\leq \ell} \cF_i[V_\cA\cap V_i]\setminus\cF_i[f_i]
			\subseteq \bigcup_{1\leq i\leq \ell} \cA_i\setminus\cA_i[f_i].
		\end{equation*}
		Similarly as above, since~$(\cA_i\setminus \cA_i[f_i])\cap (\cA_j\setminus\cA_j[f_j])=\emptyset$ for all~$1\leq i<j\leq\ell$, this entails~$\abs{\cA}-\abs{\cA[I]}\leq \sum_{1\leq i\leq \ell} \abs{\cA_i}-\abs{\cA_i[f_i]}$.
		Thus, we obtain
		\begin{equation*}
			\rho_{\cA,I}
			\leq\frac{\sum_{1\leq i\leq \ell} \abs{\cA_i}-\abs{\cA_i[f_i]}}{\sum_{1\leq i\leq \ell}\abs{U_i}-\abs{f_i}}
			=\frac{\sum_{1\leq i\leq \ell} \rho_{\cA_i,f_i}(\abs{U_i}-\abs{f_i})}{\sum_{1\leq i\leq \ell}\abs{U_i}-\abs{f_i}}
			\leq \rho_\cF,
		\end{equation*}
		which completes the proof.
	\end{proof}
	
	\begin{lemma}\label{lemma: too large is not reduced}
		Suppose that~$\frakc=(F,V,I)$ is a chain with~$\abs{V}\geq 1/\eps^3$.
		Then,~$W_{\frakc}\neq V$.
	\end{lemma}
	\begin{proof}
		Suppose that~$F=\cF_1,\ldots,\cF_\ell$.
		We show that for~$W:=V_{\cF_1+\cF_\ell}$, as a consequence of Lemma~\ref{lemma: ladder subextension density}, we have~$\rho_{\cC_{\frakc},W}\leq \rho_\cF+\eps^2$.
		Then, we obtain~$W_{\frakc}\neq V$ by choice of~$W_{\frakc}$.
		
		Let us turn to the details.
		We have
		\begin{equation*}
			\abs{\cC_\frakc}-\abs{\cC_\frakc[W]}\leq \abs{\cC_\frakc}-\abs{\cC_\frakc[I]}
		\end{equation*}
		and
		\begin{equation*}
			\abs{V}-\abs{W}\geq \abs{V}-\abs{I}-2m
			=\paren[\bigg]{1-\frac{2m}{\abs{V}-k}}(\abs{V}-\abs{I})
			\geq \paren[\bigg]{1-\frac{2m}{\frac{1}{\eps^3}-k}}(\abs{V}-\abs{I})
			\geq \frac{\abs{V}-\abs{I}}{1+\frac{\eps^2}{\rho_\cF}}.
		\end{equation*}
		With Lemma~\ref{lemma: ladder subextension density}, this yields
		\begin{equation*}
			\rho_{\cC_{\frakc},W}
			\leq \paren[\bigg]{1+\frac{\eps^2}{\rho_\cF}}\frac{\abs{\cC_\frakc}-\abs{\cC_\frakc[I]}}{\abs{V}-\abs{I}}
			\leq \paren[\bigg]{1+\frac{\eps^2}{\rho_\cF}}\rho_\cF
			=\rho_\cF+\eps^2,
		\end{equation*}
		which completes the proof.
	\end{proof}
	
	\begin{lemma}\label{lemma: chain size}
		Let~$(F,V,I)\in\frakC$.
		Then,~$\abs{V}\leq 1/\eps^3$.
	\end{lemma}
	\begin{proof}
		Consider the collection~$\frakC'$ of all chains~$(F,V,I)$ with~$\abs{V}\leq 1/\eps^3$.
		As a consequence of Lemma~\ref{lemma: too large is not reduced}, this collection is admissible, so we have~$\frakC\subseteq \frakC'$.
	\end{proof}

	\begin{lemma}\label{lemma: chains are reduced}
		Let~$\frakc\in\frakC$.
		Then,~$\frakc=\frakc|\R$.
	\end{lemma}
	\begin{proof}
		Consider the collection~$\frakC'$ of all chains~$\frakc$ with~$\frakc=\frakc|\R$.
		By choice of~$\frakC$, if~$\frakC'$ is admissible, then~$\frakC\subseteq\frakC'$, so it suffices to show that~$\frakC'$ is admissible.
		
		For all~$f\in\cF$ and~$\frakc:=(\cF,V_\cF,f)$, we have~$\frakc=\frakc|\R$.
		Consider~$\frakc=(F,V,I)\in\frakC'$ where~$F$ has length~$\ell$ and let~$1\leq \ell'\leq\ell$.
		Suppose that~$\frakc'=(F',V',I)=\frakc|\ell'$, let~$V'':=V_{\cF_{\ell'+1}+\ldots+\cF_\ell}\cap V$ and let~$\cC:=\cC_\frakc$ and~$\cC':=\cC_{\frakc'}$.
		Since for all~$(V_{\cF_1}\cup V_{\cF_{\ell'}})\cap V\subseteq W\subseteq V'$, we have~$\rho_{\cC',W}=\rho_{\cC,W\cup V''}$, from~$W_\frakc=V$, we obtain~$W_{\frakc'}=V'$.
		Hence, we have~$\frakc'|\R=\frakc'$ and thus~$\frakc'\in\frakC'$.
		Finally, since for all chains~$\frakc$, we have~$\frakc|\R=\frakc|\R|\R$, we conclude that~$\frakC'$ is admissible.
	\end{proof}

	\begin{lemma}\label{lemma: chains are not trivial}
		Let~$\frakc=(F,V,I)\in\frakC$ where~$F=\cF_1,\ldots,\cF_\ell$.
		Then,~$\abs{\cC_\frakc\setminus\cC_\frakc[I]}\geq\abs{\cF}-1$ and hence~$\ell\geq 1$.
	\end{lemma}
	\begin{proof}
		Consider the collection~$\frakC'$ of all chains~$(F,V,I)$ where~$F'=\cF_1,\ldots,\cF_\ell$ for some~$\ell\geq 1$ such that~$V_{\cF_1}\subseteq V$.
		For all~$\frakc'=(F',V',I')\in\frakC'$, we have~$\abs{\cC_{\frakc'}\setminus\cC_{\frakc'}[I']}\geq\abs{\cF}-1$.
		Furthermore,~$\frakC'$ is admissible, so we have~$\frakC\subseteq \frakC'$.
	\end{proof}

	\subsection{Branching and support}\label{subsection: averaging}
	In this section, we follow the argumentation in the beginning of Section~\ref{section: chains} to obtain Lemma~\ref{lemma: ladder averaging} where use the branching and support constructions to estimate the expected number of embeddings of a chain template lost when removing the next randomly chosen copy of~$\cF$.
	As preparation for the proof of Lemma~\ref{lemma: ladder averaging} we first consider templates~$(\cC_\frakc,V_{\cC_{\frakc'}})$ that correspond to truncation and reduction transformations introduced above in the sense that~$\frakc'$ is the transformation of the chain~$\frakc$.
	For such templates, we show that we can control the number of embeddings based on control over balanced extensions (see Lemma~\ref{lemma: reverse contraction}, Lemma~\ref{lemma: reverse reduction} and Lemma~\ref{lemma: reverse branching}).
	To this end, we first state Lemma~\ref{lemma: embeddings avoiding} that quantifies the number of embeddings that avoid a given small subset of~$V_\cH$, which will be helpful in the following situations.
	Suppose that~$(\cA,I)$ is a template and that~$J\subseteq I$ is a subset such that for all~$e\in\cA$ with~$e\cap J\neq\emptyset$, we have~$e\in\cA[I]$ and suppose that~$\psi\colon I\injection V_\cH$.
	Let~$\psi':=\restr{\psi}{I\setminus J}$.
	Then, the number~$\Phi_{\cA,\psi}$ of embeddings of~$(\cA,I)$ that extend~$\psi$ is equal to the number of embeddings~$\phi\in\Phi_{\cA-J,\psi'}^\sim$ of~$(\cA-J,I\setminus J)$ that extend~$\psi'$ and additionally avoid~$\psi(J)$ in the sense that~$\phi(V_{\cA-J})\cap \psi(J)=\emptyset$.
	We introduce the following notation.
	For a template~$(\cA,I)$,~$\psi\colon I\injection V_\cH$ and~$W\subseteq V_\cH\setminus \psi(I)$, let
	\begin{equation*}
		\Phi_{\cA,I}^{\sim,W}:=\cset{ \phi\in\Phi_{\cA,I}^\sim }{ \phi(V_\cA)\cap W=\emptyset }\qtand
		\Phi_{\cA,I}^W:=\abs{\Phi_{\cA,I}^{\sim,W}}.
	\end{equation*}\gladd{RV}{PhiAItildeW}{$\Phi_{\cA,I}^{\sim,W}(i)=\cset{ \phi\in\Phi_{\cA,I}^\sim(i) }{ \phi(V_\cA)\cap W=\emptyset }$}\gladd{realRV}{PhiAIW}{$\Phi_{\cA,I}^W(i)=\abs{\Phi_{\cA,I}^{\sim,W}(i)}$}

	\begin{lemma}\label{lemma: embeddings avoiding}
		Let~$0\leq i\leq i^\star$ and~$\cX:=\set{ i<\tau_\ccB\wedge\tau_{\ccB'} }$.
		Suppose that~$(\cA,I)$ is a template with~$\abs{V_\cA}\leq 1/\eps^4$ and~$\rho_{\cB,I}\leq\rho_\cF+\eps^2$ for all~$(\cB,I)\subseteq(\cA,I)$.
		Let~$\psi\colon I\injection V_\cH$ and~$W\subseteq V_\cH\setminus \psi(I)$ with~$\abs{W}\leq 1/\eps^3$.
		Then,
		\begin{equation*}
			\Phi_{\cA,\psi}-\Phi_{\cA,\psi}^W\Xleq \zeta^{3/2} \phihat_{\cA,I}.
		\end{equation*}
	\end{lemma}
	\begin{proof}
		For~$v\in V_\cA\setminus I$ and~$w\in W$, let~$\psi_v^w\colon I\cup\set{v}\injection \psi(I)\cup\set{w}$ with~$\restr{\psi_v^w}{I}=\psi$.
		We have
		\begin{equation*}
			\Phi_{\cA,\psi}-\Phi_{\cA,\psi}^W
			\leq \sum_{v\in V_\cA\setminus I}\sum_{w\in W}\abs{\cset{ \phi\in\Phi_{\cA,\psi}^\sim }{ \phi(v)=w }}
			= \sum_{v\in V_\cA\setminus I}\sum_{w\in W}\Phi_{\cA,\psi_v^w}.
		\end{equation*}
		Hence, it suffices to show that for all~$v\in V_\cA\setminus I$ and~$w\in W$, we have~$\Phi_{\cA,\psi_v^w}\leq \zeta^{5/3}\phihat_{\cA,I}$.
		We show that this is a consequence of Lemma~\ref{lemma: arbitrary embedding split}.
		
		To this end, suppose that~$v\in V_\cA\setminus I$ and~$w\in W$.
		For all subtemplates~$(\cB,I)\subseteq(\cA,I)$ with~$v\in V_\cB$, using the fact that~$\zeta^{-1}\leq n^{1/2}\phat^{\rho_\cF/2}$ and Lemma~\ref{lemma: bounds of phat}, we have
		\begin{equation*}
			\phihat_{\cB,I}
			=(n\phat^{\rho_{\cB,I}})^{\abs{V_\cB}-\abs{I}}
			\geq (n\phat^{\rho_{\cF}+\eps^2})^{\abs{V_\cB}-\abs{I}}
			\geq n\phat^{\rho_{\cF}+\eps^2}
			\geq (n\phat^{\rho_\cF+8\eps^2})^{1/8}\zeta^{-7/4}
			\geq \zeta^{-7/4}.
		\end{equation*}
		Thus, Lemma~\ref{lemma: arbitrary embedding split} entails
		\begin{equation*}
			\Phi_{\cA,\psi_v^w}
			\Xleq (1+\log n)^{\alpha_{\cA,I\cup\set{v}}}\zeta^{7/4}\phihat_{\cA,I}
			\leq \zeta^{5/3}\phihat_{\cA,I},
		\end{equation*}
		which completes the proof.
	\end{proof}
	
	\begin{lemma}\label{lemma: connecting edges}
		Suppose that~$\cA_1,\cA_2$ is a subsequence of a vertex-separated loose path.
		Then, there exist edges~$e_1\in\cA_1$ and~$e_2\in\cA_{2}$ with~$V_{\cA_1}\cap V_{\cA_{2}}\subseteq e_1\cap e_{2}$.
	\end{lemma}
	\begin{proof}
		Consider a vertex-separated loose path~$B=\cB_1,\ldots,\cB_{\ell}$ that has~$\cA_1,\cA_2$ as a subsequence.
		Let~$1\leq i<j\leq \ell$ such that~$\cB_{i}=\cA_1$ and~$\cB_{j}=\cA_{2}$.
		Let~$e_{1}$ denote the unique edge in~$\cB_{i}\cap \cB_{i+1}$ and let~$e_{2}$ denote the unique edge in~$\cB_{j-1}\cap \cB_{j}$.
		Then,
		\begin{equation*}
			V_{\cA_1}\cap V_{\cA_{2}}
			=V_{\cB_{i}}\cap V_{\cB_{j}}
			\subseteq V_{\cB_1+\ldots+\cB_{i}}\cap V_{\cB_{i+1}+\ldots+\cB_{\ell}}
			=e_{1}
		\end{equation*}
		and similarly
		\begin{equation*}
			V_{\cA_1}\cap V_{\cA_{2}}
			=V_{\cB_{i}}\cap V_{\cB_{j}}
			\subseteq V_{\cB_1+\ldots+\cB_{j-1}}\cap V_{\cB_{j}+\ldots+\cB_{\ell}}
			=e_{2},
		\end{equation*}
		which completes the proof.
	\end{proof}
	
	\begin{lemma}\label{lemma: joins are small}
		Suppose that~$\cF_1,\cF_2$ is a subsequence of a vertex-separated loose path of copies of~$\cF$.
		Let~$I:=V_{\cF_1}\cap V_{\cF_{2}}$.
		Then,~$\abs{I}=k$ or~$\abs{I}\leq k-1/\rho_\cF$.
		Hence, if~$I\subseteq V_\cA$ for some~$k$-graph~$\cA$ that has exactly one edge and no isolated vertices, then~$\rho_{\cA,I}\leq \rho_\cF$.
	\end{lemma}
	\begin{proof}
		If~$\cF_1\cap \cF_{2}\neq\emptyset$, then~$\abs{I}=k$ and hence the statement follows.
		Thus, we may assume that~$\cF_1\cap \cF_{2}=\emptyset$.
		Consider a vertex-separated loose path~$F'=\cF_1',\ldots,\cF_{\ell}'$ of copies of~$\cF$ that has~$\cF_1,\cF_2$ as a subsequence.
		Let~$1\leq i<j\leq \ell$ such that~$\cF_{i}'=\cF_1$ and~$\cF_{j}'=\cF_{2}$.
		Since~$\cF_1\cap\cF_{2}=\emptyset$, we have~$j\geq i+2$.
		Choose~$f_-,f_+\in\cF_{i+1}'$ such that~$f_-$ is the unique edge in~$\cF_{i}'\cap \cF_{i+1}'$ and such that~$f_+$ is the unique edge in~$\cF_{i+1}'\cap \cF_{i+2}'$.
		Then,
		\begin{equation*}
			V_{\cF_1}\cap V_{\cF_{2}}
			=V_{\cF_{i}'}\cap V_{\cF_{j}'}
			\subseteq V_{\cF_1'+\ldots+\cF_{i}'}\cap V_{\cF_{i+1}'+\ldots+\cF_{\ell}'}
			=f_{-}
		\end{equation*}
		and similarly
		\begin{equation*}
			V_{\cF_1}\cap V_{\cF_{2}}
			=V_{\cF_{i}'}\cap V_{\cF_{j}'}
			\subseteq V_{\cF_1'+\ldots+\cF_{i+1}'}\cap V_{\cF_{i+2}'+\ldots+\cF_{\ell}'}
			=f_{+}.
		\end{equation*}
		Hence,~$V_{\cF_1}\cap V_{\cF_{2}}\subseteq f_-\cap f_+$.
		Thus, it suffices to show that~$\abs{f_-\cap f_+}\leq k-1/\rho_\cF$.
		This follows from the fact that~$(\cF_{i+1}',f_-)$ is balanced.
		To see this, consider the template~$(\cF_{i+1}'[f_-\cup f_+],f_-)$.
		Then,
		\begin{equation*}
			\rho_\cF
			\geq \rho_{\cF_{i+1}'[f_-\cup f_+],f_-}
			\geq \frac{1}{\abs{f_-\cup f_+}-\abs{f_-}}
			= \frac{1}{k-\abs{f_-\cap f_+}}
		\end{equation*}
		and hence~$\abs{f_-\cap f_+}\leq k-1/\rho_\cF$.
	\end{proof}

	\begin{lemma}\label{lemma: subextension of cF density}
		Suppose that~$\cF_1,\cF_2$ is a subsequence of a vertex-separated loose path of copies of~$\cF$.
		Suppose that~$\cA$ is a subgraph of~$\cF_1$ or~$\cF_2$.
		Let~$I:=V_\cA\cap V_{\cF_1}\cap V_{\cF_2}$.
		Then,~$\rho_{\cA,I}\leq \rho_\cF$.
	\end{lemma}
	\begin{proof}
		Since~$\cF_2,\cF_1$ is also a subsequence of a vertex-separated loose path of copies of~$\cF$, we may assume that~$\cA$ is a subgraph of~$\cF_1$.
		Furthermore, we may assume that~$\cA$ is an induced subgraph of~$\cF_1$.
		By Lemma~\ref{lemma: connecting edges}, we may fix an edge~$f_1\in\cF_1$ with~$V_{\cF_1}\cap V_{\cF_2}\subseteq f_1$.
		If~$f_1\not\subseteq V_\cA$, then~$\cA[I]=\emptyset$ and thus, using the fact that~$(\cF_1,f_1)$ is balanced, we obtain
		\begin{equation*}
			\begin{aligned}
				\abs{\cA}-\abs{\cA[I]}
				&=\abs{\cF_1[V_\cA]}
				\leq \abs{\cF_1[V_\cA\cup f_1]}-\abs{\cF_1[f_1]}
				=\rho_{\cF_1[V_\cA\cup f_1],f_1}(\abs{V_\cA\cup f_1}-\abs{f_1})\\
				&\leq \rho_{\cF}(\abs{V_\cA\cup f_1}-\abs{f_1})
				=\rho_{\cF}\abs{V_\cA\setminus f_1}
				\leq \rho_{\cF}(\abs{V_\cA}-\abs{I}).
			\end{aligned}
		\end{equation*}
		If~$f_1\subseteq V_\cA$, then~$I=V_{\cF_1}\cap V_{\cF_2}$, so using the fact that~$(\cF_1,f_1)$ is balanced and Lemma~\ref{lemma: joins are small}, we obtain
		\begin{equation*}
			\begin{aligned}
				\abs{\cA}-\abs{\cA[I]}
				&= \abs{\cA}-\abs{\cA[f_1]}+\abs{\cF_1[f_1]}-\abs{\cF_1[I]}
				=\rho_{\cA,f_1}(\abs{V_\cA}-\abs{f_1})+\rho_{\cF_1[f_1],I}(\abs{f_1}-\abs{I})\\
				&\leq \rho_{\cF}(\abs{V_\cA}-\abs{f_1})+\rho_{\cF}(\abs{f_1}-\abs{I})
				=\rho_{\cF}(\abs{V_\cA}-\abs{I}),
			\end{aligned}
		\end{equation*}
		which completes the proof.
	\end{proof}
	
	\begin{lemma}\label{lemma: subextension of cF}
		Let~$0\leq i\leq i^\star$ and~$\cX:=\set{i<\tau_{\ccB}}$.
		Suppose that~$\cF_1,\cF_2$ is a subsequence of a vertex-separated loose path of copies of~$\cF$.
		Suppose that~$\cA$ is a subgraph of~$\cF_1$ or~$\cF_2$.
		Let~$I:=V_\cA\cap V_{\cF_1}\cap V_{\cF_2}$.
		Let~$\psi\colon I\injection V_\cH$.
		Then,
		\begin{equation*}
			\Phi_{\cA,\psi}\Xeq (1\pm \eps^{-1}\zeta^\delta)\phihat_{\cA,I}.
		\end{equation*}
	\end{lemma}
	
	\begin{proof}
		We use induction on~$\abs{V_\cA}-\abs{I}$ to show that
		\begin{equation}\label{lemma: stronger subextension of cF}
			\Phi_{\cA,\psi}\Xeq(1\pm 2(\abs{V_\cA}-\abs{ I})\zeta^\delta)\phihat_{\cA,I}.
		\end{equation}
		If~$\abs{V_\cA}-\abs{I}=0$, then~$\Phi_{\cA,\psi}=1=\phihat_{\cA,I}$.
		Let~$\ell\geq 1$ and suppose that~\eqref{lemma: stronger subextension of cF} holds if~$\abs{V_\cA}-\abs{I}\leq \ell-1$.
		Suppose that~$\abs{V_\cA}-\abs{I}=\ell$.
		From Lemma~\ref{lemma: subextension of cF density}, we obtain~$\rho_{\cA,I}\leq\rho_\cF$.
		Suppose that among all subsets~$I\subseteq U'\subsetneq V_\cA$ with~$\rho_{\cA,U'}\leq \rho_\cF$, the set~$U$ has maximal size.
		By Lemma~\ref{lemma: extension is strictly balanced}, the extension~$(\cA,U)$ is balanced.
		We have
		\begin{equation}\label{equation: remnant of F split}
			\Phi_{\cA,\psi}=\sum_{\phi\in\Phi_{\cA[U],\psi}^{\sim}}\Phi_{\cA,\phi}.
		\end{equation}
		We use the estimate for~$\Phi_{\cA[U],\psi}$  provided by the induction hypothesis and for~$\phi\in\Phi_{\cA[U],\psi}^{\sim}$, we estimate~$\Phi_{\cA,\phi}$ using the balancedness of~$(\cA,U)$ to conclude that~$\Phi_{\cA,\psi}$ is bounded as desired.
		
		Let us turn to the details.
		Since~$\zeta^{-2}\leq n\phat^{\rho_\cF}$, for all~$j\leq i$, we have
		\begin{equation*}
			\phihat_{\cA,U}(j)
			=(n\phat(j)^{\rho_{\cA,U}})^{\abs{V_\cA}-\abs{U}}
			\geq (n\phat^{\rho_{\cF}})^{\abs{V_\cA}-\abs{U}}
			\geq \zeta^{-2(\abs{V_\cA}-\abs{U})}
			\geq \zeta^{-2}
			> \zeta^{-\delta^{1/2}}.
		\end{equation*}
		Hence~$i< i_{\cA,U}^{\delta^{1/2}}$, and thus for all~$\phi\in\Phi^\sim_{\cA[U],\psi}$, we have~$\Phi_{\cA,\phi}\Xeq (1\pm\zeta^\delta)\phihat_{\cA,U}$.
		Since by induction hypothesis, we have~$\Phi_{\cA[U],\psi}=(1\pm 2(\abs{U}-\abs{I})\zeta^\delta)\phihat_{\cA[U],I}$, returning to~\eqref{equation: remnant of F split}, we conclude that
		\begin{align*}
			\Phi_{\cA,\psi}
			=(1\pm 2(\abs{U}-\abs{I})\zeta^\delta)(1\pm \zeta^\delta)\phihat_{\cA,U}\phihat_{\cA[U],I}
			=(1\pm 2(\abs{V_\cA}-\abs{I})\zeta^\delta)\phihat_{\cA,I},
		\end{align*}
		which completes the proof.
	\end{proof}
	
	\begin{lemma}\label{lemma: reverse contraction}
		Let~$0\leq i\leq i^\star$ and~$\cX:=\set{i<\tau_{\ccB}\wedge\tau_{\ccB'}}$.
		Suppose that~$\frakc=(F,V,I)\in\frakC$ is a chain where~$F$ has length~$\ell$.
		Let~$0\leq \ell'\leq\ell$ and suppose that~$(F',V',I)=\frakc|\ell'$.
		Let~$\psi\colon V'\injection V_\cH$.
		Then,~$\Phi_{\frakc,\psi}\Xeq(1\pm \eps^{-5k}\zeta^\delta)\phihat_{\frakc,V'}$.
	\end{lemma}
	
	\begin{proof}
		For~$0\leq\ell_0\leq\ell$, let
		\begin{equation*}
			g_{\ell_0}:=\abs{\cset{ \ell_0\leq\ell_1\leq\ell-1 }{ \cC_{\frakc|\ell_1}\neq\cC_{\frakc|\ell_1+1} }}.
		\end{equation*}
		We use induction on~$\ell-\ell'$ to show that
		\begin{equation}\label{equation: stronger reverse contraction}
			\Phi_{\frakc,\psi}\Xeq (1\pm 4g_{\ell'}\eps^{-1}\zeta^\delta)\phihat_{\cC,V'}.
		\end{equation}
		By Lemma~\ref{lemma: chain size}, we have~$\abs{V}\leq \eps^{-3}$, hence~$\abs{\cC_\frakc}\leq \eps^{-3k}$ and thus~$g_{\ell'}\leq \eps^{-3}+\eps^{-3k}\leq\eps^{-4k}$, so it suffices to obtain~\eqref{equation: stronger reverse contraction}.
		
		Let us proceed with the proof by induction.
		If~$\ell-\ell'=0$, then~$\Phi_{\frakc,\psi}=1=\phihat_{\cC,V'}$.
		Let~$q\geq 1$ and suppose that~\eqref{equation: stronger reverse contraction} holds whenever~$\ell-\ell'\leq q-1$.
		Suppose that~$\ell-\ell'=q$.
		Suppose that~$\frakc'=(F',V',I)=\frakc|\ell'$ and~$\frakc''=(F'',V'',I)=\frakc|\ell'+1$.
		If~$\cC_{\frakc''}=\cC_{\frakc'}$, then~\eqref{equation: stronger reverse contraction} follows by induction hypothesis, so we may assume~$\cC_{\frakc''}\neq\cC_{\frakc'}$ and hence~$g_{\ell'+1}=g_{\ell'}-1$.
		We have
		\begin{equation}\label{equation: reverse contraction split}
			\Phi_{\frakc,\psi}=\sum_{\phi\in\Phi_{\frakc'',\psi}^\sim}\Phi_{\frakc,\phi}.
		\end{equation}
		We use Lemma~\ref{lemma: subextension of cF} to estimate~$\Phi_{\frakc'',\psi}$ and for~$\phi\in\Phi_{\frakc'',\psi}^\sim$, we use the estimate for~$\Phi_{\frakc,\phi}$ provided by the induction hypothesis to conclude that~$\Phi_{\frakc,\psi}$ can be estimated as desired.
		
		Let us turn to the details.
		Let~$\cA:=\cF_{\ell'+1}[V\cap V_{\cF_{\ell'+1}}]$ and~$J:=V_\cA\cap V_{\cF_{\ell'}}$.
		Note that~$\phihat_{\cA,J}=\phihat_{\cC'',V'}$.
		Lemma~\ref{lemma: subextension of cF density} allows us to apply Lemma~\ref{lemma: embeddings avoiding} such that using Lemma~\ref{lemma: subextension of cF}, we obtain
		\begin{equation*}
			\Phi_{\frakc'',\psi}
			= \Phi_{\cA,\restr{\psi}{J}}^{\psi(V'\setminus J)}
			\Xeq \Phi_{\cA,\restr{\psi}{J}}\pm \zeta^{3/2}\phihat_{\cA,J}
			\Xeq (1\pm 2\eps^{-1}\zeta^\delta)\phihat_{\cA,J}
			=(1\pm 2\eps^{-1}\zeta^\delta)\phihat_{\cC'',V'}.
		\end{equation*}
		Furthermore, by induction hypothesis, for all~$\phi\in\Phi_{\frakc'',\psi}^\sim$, we have
		\begin{equation*}
			\Phi_{\frakc,\phi}
			\Xeq (1\pm 4g_{\ell'+1}\eps^{-1}\zeta^\delta)\phihat_{\cC,V''}
			= (1\pm 4(g_{\ell'}-1)\eps^{-1}\zeta^\delta)\phihat_{\cC,V''}.
		\end{equation*}
		Thus, returning to~\eqref{equation: reverse contraction split}, we conclude that
		\begin{equation*}
			\Phi_{\frakc,\psi}
			\Xeq (1\pm 2\eps^{-1}\zeta^\delta)\phihat_{\cC'',V'}\cdot (1\pm 4(g_{\ell'}-1)\eps^{-1}\zeta^\delta)\phihat_{\cC,V''}
			=(1\pm 4g_{\ell'}\eps^{-1}\zeta^\delta)\phihat_{\cC,V'},
		\end{equation*}
		which completes the proof.
	\end{proof}
	
	\begin{lemma}\label{lemma: reverse reduction}
		Let~$0\leq i\leq i^\star$ and~$\cX:=\set{i<\tau_{\ccB}}$.
		Suppose that~$\frakc$ is the~$\beta$-extension of a chain in~$\frakC$ for some~$\beta$ and let~$(F',V',I)=\frakc|\R$.
		Let~$\psi\colon V'\injection V_\cH$.
		Then,~$\Phi_{\frakc,\psi}=(1\pm \eps^{-4}\zeta^\delta)\phihat_{\frakc,V'}$.
	\end{lemma}
	
	\begin{proof}
		Suppose that~$\frakc=(F,V,I)$ where~$F=\cF_1,\ldots,\cF_\ell$.
		By definition of~$\frakc':=(F',V',I)$, there exists a sequence of chains~$\frakc=(F,V_0,I),\ldots,(F,V_t,I)=\frakc'$ with~$V_0\supseteq\ldots\supseteq V_t$ such that for all~$1\leq s\leq t$, the set~$V_s$ is a subset of~$V_{s-1}$ of maximal size chosen from all subsets~$(V_{\cF_1}\cup V_{\cF_\ell})\cap V_{s-1}\subseteq W\subsetneq V_{s-1}$ with~$\rho_{\cC_{(F,V_{s-1},I)},W}\leq \rho_\cF+\eps^2$.
		
		For~$0\leq s\leq t$, let~$\cC_s:=\cC_{(F,V_{s},I)}$.
		Using induction on~$s$, we show that for all~$0\leq s\leq t$ and~$\psi_s\colon V_s\injection V_\cH$, we have
		\begin{equation}\label{equation: stronger reverse reduction}
			\Phi_{\frakc,\psi_s}\Xeq(1\pm 2s\zeta^\delta)\phihat_{\frakc,V_s}.
		\end{equation}
		By Lemma~\ref{lemma: chain size}, we have~$\abs{V}\leq 2\eps^{-3}$ and hence~$2t\leq \eps^{-4}$, so this is sufficient.
		
		Let us proceed with the proof by induction.
		If~$s=0$, then, for all~$\psi_s\colon V_s\injection V_\cH$, we have~$\Phi_{\frakc,\psi_s}=1=\phihat_{\frakc,V_s}$.
		Let~$q\geq 1$ and suppose that~\eqref{equation: stronger reverse reduction} holds whenever~$s\leq q-1$.
		Suppose that~$s=q$ and let~$\psi_s\colon V_s\injection V_\cH$.
		We have
		\begin{equation}\label{equation: reverse reduction split}
			\Phi_{\frakc,\psi_s}
			=\sum_{\phi\in\Phi^\sim_{\cC_{s-1},\psi_s}} \Phi_{\frakc,\phi}.
		\end{equation}
		By Lemma~\ref{lemma: extension is strictly balanced}, the extension~$(\cC_{s-1},V_{s})$ is balanced, so we may estimate~$\Phi_{\cC_{s-1},\psi_s}$ based on balancedness, while for~$\phi\in\Phi^\sim_{\cC_{s-1},\psi_s}$ the induction hypothesis provides an estimate for~$\Phi_{\frakc,\phi}$.
		
		Let us turn to the details.
		Using the fact that~$\zeta^{-1}\leq n^{1/2}\phat^{\rho_\cF/2}$ and Lemma~\ref{lemma: bounds of phat}, for all~$j\leq i$, we obtain
		\begin{equation*}
			\begin{aligned}
				\phihat_{\cC_{s-1},V_s}(j)
				&=(n\phat(j)^{\rho_{\cC_{s-1},V_s}})^{\abs{V_{s-1}}-\abs{V_s}}
				\geq (n\phat^{\rho_{\cF}+\eps^2})^{\abs{V_{s-1}}-\abs{V_s}}
				\geq n\phat^{\rho_{\cF}+\eps^2}
				\geq (n\phat^{\rho_{\cF}+2\eps^2})^{1/2}\zeta^{-1}\\
				&\geq \zeta^{-1}
				> \zeta^{-\delta^{1/2}}.
			\end{aligned}
		\end{equation*}
		Hence~$i< i_{\cC_{s-1},V_s}^{\delta^{1/2}}$ and thus~$\Phi_{\cC_{s-1},\psi_s}\Xeq (1\pm \zeta^\delta)\phihat_{\cC_{s-1},V_s}$.
		Furthermore, for all~$\phi\in\Phi^\sim_{\cC_{s-1},\psi_s}$, by induction hypothesis we have~$\Phi_{\frakc,\phi}\Xeq (1\pm 2(s-1)\zeta^\delta)\phihat_{\frakc,V_{s-1}}$, so returning to~\eqref{equation: reverse reduction split}, we conclude that
		\begin{equation*}
			\Phi_{\frakc,\psi_s}
			\Xeq (1\pm \zeta^\delta)\phihat_{\cC_{s-1},V_s}\cdot (1\pm 2(s-1)\zeta^\delta)\phihat_{\frakc,V_{s-1}}
			=(1\pm 2s\zeta^\delta)\phihat_{\frakc,V_{s}},
		\end{equation*}
		which completes the proof.
	\end{proof}

	\begin{lemma}\label{lemma: partial reverse reduction density}
		Suppose that~$\frakc$ is the~$\beta$-extension of a chain in~$\frakC$ for some~$\beta$ and let~$(F',V',I)=\frakc|\R$.
		Let~$(\cA,V')\subseteq (\cC_\frakc,V')$
		Then,~$\rho_{\cA,V'}\leq \rho_\cF+\eps^2$.
	\end{lemma}
	
	\begin{proof}
		Suppose that~$\frakc=(F,V,I)$ where~$F=\cF_1,\ldots,\cF_\ell$.
		By definition of~$\frakc':=(F',V',I)$, there exists a sequence of chains~$\frakc=(F,V_0,I),\ldots,(F,V_t,I)=\frakc'$ with~$V_0\supseteq\ldots\supseteq V_t$ such that for all~$1\leq s\leq t$, the set~$V_s$ is a subset of~$V_{s-1}$ of maximal size chosen from all subsets~$(V_{\cF_1}\cup V_{\cF_\ell})\cap V_{s-1}\subseteq W\subsetneq V_{s-1}$ with~$\rho_{\cC_{(F,V_{s-1},I)},W}\leq \rho_\cF+\eps^2$.
		
		For~$0\leq s\leq t$, let~$\cC_s:=\cC_{(F,V_s,I)}$ and~$\cA_s:=\cA[V_s\cap V_\cA]$ and for~$0\leq s\leq t-1$, let~$\cA_s':=\cA_s+\cC_{s+1}$.
		For~$0\leq s\leq t-1$, consider the extensions~$(\cA_s,V_{\cA_{s+1}})$ and~$(\cA_s',V_{s+1})$.
		We have
		\begin{equation*}
			V_{\cA_s}\setminus V_{\cA_{s+1}}
			=(V_s \cap V_\cA)\setminus(V_{s+1}\cap V_\cA)
			=(V_s\cap V_\cA)\setminus V_{s+1}
			=(V_{\cA_s}\cup V_{s+1})\setminus V_{s+1}
			=V_{\cA_s'}\setminus V_{s+1}
		\end{equation*}
		and hence~$\abs{V_{\cA_s}}-\abs{V_{\cA_{s+1}}}=\abs{V_{\cA_s'}}-\abs{V_{s+1}}$.
		Furthermore, we have~$\cC_{s+1}\cap \cA_s= \cA_s[V_{\cA_{s+1}}]$ and~$\cA_s[V_{\cA_{s+1}}]\cup \cC_{s+1}=\cA_s'[V_{s+1}]$, hence
		\begin{equation*}
			\cA_s\setminus \cA_s[V_{\cA_{s+1}}]
			=\cA_s\setminus (\cA_s[V_{\cA_{s+1}}]\cup \cC_{s+1})
			=\cA_s'\setminus (\cA_s[V_{\cA_{s+1}}]\cup \cC_{s+1})
			=\cA_s'\setminus \cA_s'[V_{s+1}]
		\end{equation*}
		and thus~$\abs{\cA_s}-\abs{\cA_s[V_{\cA_{s+1}}]}=\abs{\cA_s'}-\abs{\cA_s'[V_{s+1}]}$.
		In particular, this yields~$\rho_{\cA_s,V_{\cA_{s+1}}}=\rho_{\cA_s',V_{s+1}}$.
		Since~$\cA\subseteq\cC_\frakc$ implies~$\cA_s\subseteq \cC_\frakc[V_s]=\cC_s$, we have~$(\cA_{s}',V_{s+1})\subseteq (\cC_s,V_{s+1})$.
		Using Lemma~\ref{lemma: extension is strictly balanced}, this entails
		\begin{equation*}
			\rho_{\cA_s,V_{\cA_{s+1}}}
			=\rho_{\cA_s',V_{s+1}}
			\leq \rho_{\cC_s,V_{s+1}}
			\leq \rho_\cF+\eps^2.
		\end{equation*}
		We conclude that
		\begin{align*}
			\abs{\cA}-\abs{\cA[V']}
			&=\sum_{s=0}^{t-1} \abs{\cA_{s}}-\abs{\cA_{s+1}}
			=\sum_{s=0}^{t-1} \rho_{\cA_s,V_{\cA_{s+1}}}(\abs{V_{\cA_s}}-\abs{V_{\cA_{s+1}}})\\
			&\leq (\rho_{\cF}+\eps^2)\sum_{s=0}^{t-1} \abs{V_{\cA_s}}-\abs{V_{\cA_{s+1}}}
			=(\rho_{\cF}+\eps^2)(\abs{V_\cA}-\abs{V'}),
		\end{align*}
		which completes the proof.		
	\end{proof}
	
	\begin{lemma}\label{lemma: reverse branching}
		Let~$0\leq i\leq i^\star$ and~$\cX:=\set{ i<\tau_{\ccB}\wedge\tau_{\ccB'} }$.
		Let~$\frakc=(F,V,I)\in\frakC$ and~$e\in\cC_\frakc\setminus\cC_\frakc[I]$.
		Suppose that~$\frakc'=(F',V',I)=\frakc|e$.
		Let~$\psi\colon V'\injection V_\cH$.
		Then,~$\Phi_{\frakc,\psi}=(1\pm \eps^{-6k}\zeta^\delta)\phihat_{\frakc,V'}$.
	\end{lemma}
	
	\begin{proof}
		Consider an arbitrary~$\beta\colon f\bijection e$ where~$f\in\cF$.
		Suppose that~$\frakc''=(F'',V'',I)=\frakc|[\beta]$.
		Furthermore, suppose that~$\ell\geq 1$ is minimal with~$e\in\cC_{\frakc|\ell}$ and suppose that
		\begin{gather*}
			\frakc'_+=(F'_+,V'_+,I):=\frakc|{\ell},\quad
			\frakc''_+=(F''_+,V''_+,I):=\frakc|\ell|\beta.
		\end{gather*}
		We have
		\begin{equation}\label{equation: reverse branching split}
			\Phi_{\frakc,\psi}
			=\sum_{\phi\in\Phi_{\frakc'_+,\psi}^\sim} \Phi_{\frakc,\phi}.
		\end{equation}
		We use Lemma~\ref{lemma: reverse reduction} to estimate~$\Phi_{\frakc'_+,\psi}$ and for~$\phi\in\Phi_{\frakc'_+,\psi}^\sim$, we use Lemma~\ref{lemma: reverse contraction} to estimate~$\Phi_{\frakc,\phi}$.
		
		Let us turn to the details.
		First, consider~$\Phi_{\frakc'_+,\psi}$.
		Choose an arbitrary injection~$\psi'\colon V''\injection V_\cH$ with~$\restr{\psi'}{V'}=\psi$.
		With Lemma~\ref{lemma: reverse reduction}, since~$\phihat_{\frakc''_+,V''}=\phihat_{\frakc'_+,V'}$, we obtain
		\begin{equation*}
			\begin{aligned}
				\Phi_{\frakc'_+,\psi}
				&=\Phi_{\frakc''_+,\psi'}+\Phi_{\frakc'_+,\psi}-\Phi_{\frakc'_+,\psi}^{\psi'(V''\setminus V')}
				\Xeq (1\pm\eps^{-4}\zeta^\delta)\phihat_{\frakc''_+,V''}+\Phi_{\frakc'_+,\psi}-\Phi_{\frakc'_+,\psi}^{\psi'(V''\setminus V')}\\
				&= (1\pm\eps^{-4}\zeta^\delta)\phihat_{\frakc'_+,V'}+\Phi_{\frakc'_+,\psi}-\Phi_{\frakc'_+,\psi}^{\psi'(V''\setminus V')}.
			\end{aligned}
		\end{equation*}
		To bound~$\Phi_{\frakc'_+,\psi}-\Phi_{\frakc'_+,\psi}^{\psi'(V''\setminus V')}$, we employ Lemma~\ref{lemma: embeddings avoiding} which we may apply as a consequence of Lemma~\ref{lemma: partial reverse reduction density}.
		To this end, recall that in Section~\ref{subsection: formal chains}, to define the~$\beta$-extension of~$\frakc$, we fixed a copy~$\cF_{\frakc}^\beta$ of~$\cF$.
		For all~$(\cA,V')\subseteq (\cC_{\frakc'_+},V')$ and~$\cA':=\cA+\cF_{\frakc}^\beta$, the template~$(\cA',V'')$ is a subtemplate of~$(\cC_{\frakc''_+},V'')$ and we have~$\rho_{\cA,V'}=\rho_{\cA',V''}$, so Lemma~\ref{lemma: partial reverse reduction density} entails~$\rho_{\cA,V'}\leq \rho_{\cF}+\eps^2$.
		Hence, we may apply Lemma~\ref{lemma: embeddings avoiding} to obtain
		\begin{equation*}
			\Phi_{\frakc'_+,\psi}-\Phi_{\frakc'_+,\psi}^{\psi'(V''\setminus V')}\Xleq \zeta^{3/2}\phihat_{\frakc'_+,V'}.
		\end{equation*}
		Thus,
		\begin{equation*}
			\Phi_{\frakc'_+,\psi}
			\Xeq (1\pm\eps^{-5}\zeta^\delta)\phihat_{\frakc'_+,V'}.
		\end{equation*}
		
		Next, fix~$\phi\in\Phi_{\frakc'_+,\psi}^\sim$ and consider~$\Phi_{\frakc,\phi}$.
		Then, Lemma~\ref{lemma: reverse contraction} entails
		\begin{equation*}
			\Phi_{\frakc,\phi}\Xeq (1\pm\eps^{-5k}\zeta^\delta)\phihat_{\frakc,V'_+}.
		\end{equation*}
		Thus, returning to~\eqref{equation: reverse branching split}, we conclude that
		\begin{equation*}
			\Phi_{\frakc,\psi}
			\Xeq (1\pm\eps^{-5}\zeta^\delta)\phihat_{\frakc'_+,V'}\cdot (1\pm\eps^{-5k}\zeta^\delta)\phihat_{\frakc,V'_+}
			=(1\pm\eps^{-6k}\zeta^\delta)\phihat_{\frakc,V'},
		\end{equation*}
		which completes the proof.
	\end{proof}
	
	\begin{lemma}\label{lemma: ladder averaging}
		Let~$\frakc=(F,V,I)\in\frakC$ and let~$e\in \cC_\frakc\setminus\cC_\frakc[I]$.
		Let~$0\leq i\leq i^\star$ and
		\begin{equation*}
			\cX:=\set{i<\tau_{\cH^*}\wedge \tau_{\ccF}\wedge \tau_{\ccB}\wedge \tau_{\ccB'} }\cap \set{ \Phi_{\frakc,\psi}\leq 2\phihat_{\frakc,I} }\cap \set{ \Phi_{\frakc|e,\psi}\leq 2\phihat_{\frakc|e,I} }.
		\end{equation*}
		Then,
		\begin{equation*}
			\exi{\abs{\cset{ \phi\in\Phi^\sim_{\frakc,\psi} }{ \phi(e)\in \cF_0(i+1) }}}
			\Xeq \paren[\bigg]{\sum_{f\in\cF}\sum_{\beta\colon f\bijection e}\frac{\Phi_{\frakc|[\beta],\psi}\Phi_{\frakc,\psi} }{\aut(\cF)H^*\Phi_{\frakc|e,\psi}}}\pm \zeta^{1+\delta/2}\frac{\phihat_{\frakc,I}}{H}.
		\end{equation*}
	\end{lemma}
	
	\begin{proof}
		Lemma~\ref{lemma: star degrees} entails
		\begin{equation}\label{equation: chain loss at one edge using degree}
			\exi{\abs{\cset{ \phi\in\Phi^\sim_{\frakc,\psi} }{ \phi(e)\in \cF_0(i+1) }}}
			= \sum_{\phi\in\Phi^\sim_{\frakc,\psi} } \frac{d_{\cH^*}(\phi(e))}{H^*}
			=\sum_{f\in\cF}\sum_{\beta\colon f\bijection e}\frac{\sum_{\phi\in\Phi^\sim_{\frakc,\psi} } \Phi_{\cF,\phi\circ\beta}}{\aut(\cF)H^*}.
		\end{equation}
		Suppose that~$\frakc|e=(F',V',I)$.
		For~$f\in\cF$ and~$\beta\colon f\bijection e$, using Lemma~\ref{lemma: reverse branching} and the fact that~$\Phi_{\cF,\phi\circ\beta}\Xeq (1\pm \delta^{-1}\zeta)\phihat_{\cF,f}$ holds for all~$\phi\in\Phi_{\frakc,\psi}^\sim$, Lemma~\ref{lemma: averaging} yields
		\begin{equation*}
			\begin{aligned}
				\sum_{\phi\in\Phi^\sim_{\frakc,\psi} } \Phi_{\cF,\phi\circ\beta}
				&= \sum_{\phi\in\Phi^\sim_{\frakc|e,\psi} } \Phi_{\cF,\phi\circ\beta}\Phi_{\frakc,\phi}\\
				&\Xeq \frac{1}{\Phi_{\frakc|e,\psi}}\paren[\Big]{\sum_{\phi\in\Phi^\sim_{\frakc|e,\psi} } \Phi_{\cF,\phi\circ\beta}}\paren[\Big]{\sum_{\phi\in\Phi^\sim_{\frakc|e,\psi} }\Phi_{\frakc,\phi}}\pm \delta^{-2}\zeta^{1+\delta}\phihat_{\cF,f}\phihat_{\frakc,V'}\Phi_{\frakc|e,\psi}\\
				&=\frac{\Phi_{\frakc,\psi}}{\Phi_{\frakc|e,\psi}}\paren[\Big]{\sum_{\phi\in\Phi^\sim_{\frakc|e,\psi} } \Phi_{\cF,\phi\circ\beta}}\pm \delta^{-2}\zeta^{1+\delta}\phihat_{\cF,f}\phihat_{\frakc,V'}\Phi_{\frakc|e,\psi}.
			\end{aligned}
		\end{equation*}
		Since Lemma~\ref{lemma: embeddings avoiding} entails
		\begin{equation*}
			\sum_{\phi\in\Phi^\sim_{\frakc|e,\psi}}\Phi_{\cF,\phi\circ\beta}
			\Xeq \sum_{\phi\in\Phi^\sim_{\frakc|e,\psi}}(\Phi_{\cF,\phi\circ\beta}^{\phi(V'\setminus e)}\pm\zeta^{3/2}\phihat_{\cF,f})
			=\Phi_{\frakc|[\beta],\psi}\pm\zeta^{3/2}\phihat_{\cF,f}\Phi_{\frakc|e,\psi},
		\end{equation*}
		we conclude that
		\begin{equation*}
			\begin{aligned}
				\sum_{\phi\in\Phi^\sim_{\frakc,\psi} } \Phi_{\cF,\psi\circ\beta}
				&\Xeq\frac{\Phi_{\frakc,\psi}\Phi_{\frakc|[\beta],\psi}}{\Phi_{\frakc|e,\psi}}\pm \delta^{-2}\zeta^{1+\delta}\phihat_{\cF,f}\phihat_{\frakc,V'}\Phi_{\frakc|e,\psi}\pm \zeta^{3/2}\phihat_{\cF,f}\Phi_{\frakc,\psi}\\
				&\Xeq \frac{\Phi_{\frakc,\psi}\Phi_{\frakc|[\beta],\psi}}{\Phi_{\frakc|e,\psi}}\pm \delta^{-3}\zeta^{1+\delta}\phihat_{\cF,f}\phihat_{\frakc,V'}\phihat_{\frakc|e,I}\pm \zeta^{4/3}\phihat_{\cF,f}\phihat_{\frakc,I}\\
				&= \frac{\Phi_{\frakc,\psi}\Phi_{\frakc|[\beta],\psi}}{\Phi_{\frakc|e,\psi}}\pm \delta^{-4}\zeta^{1+\delta}\phihat_{\cF,f}\phihat_{\frakc,I}.
			\end{aligned}
		\end{equation*}
		Combining this with~\eqref{equation: chain loss at one edge using degree}, we obtain
		\begin{equation*}
			\exi{\abs{\cset{ \phi\in\Phi^\sim_{\frakc,\psi} }{ \phi(e)\in \cF_0(i+1) }}}
			\Xeq \paren[\bigg]{\sum_{f\in\cF}\sum_{\beta\colon f\bijection e}\frac{\Phi_{\frakc|[\beta],\psi}\Phi_{\frakc,\psi} }{\aut(\cF)H^*\Phi_{\frakc|e,\psi}}}\pm \frac{\abs{\cF}k!\, \zeta^{1+\delta}\phihat_{\cF,f}\phihat_{\frakc,I}}{\delta^4\aut(\cF)H^*}.
		\end{equation*}
		Since Lemma~\ref{lemma: edges of H} yields
		\begin{equation*}
			\frac{\abs{\cF}k!\, \zeta^{1+\delta}\phihat_{\cF,f}\phihat_{\frakc,I}}{\delta^4\aut(\cF)H^*}
			\Xleq\frac{\abs{\cF}k!\, \zeta^{1+\delta}\phihat_{\cF,f}\phihat_{\frakc,I}}{\delta^5\aut(\cF)\hhat^*}
			= \frac{\abs{\cF}k!\, \zeta^{1+\delta}\phihat_{\frakc,I}}{\delta^5 n^k\phat}
			\Xleq \zeta^{1+\delta/2}\frac{\phihat_{\frakc,I}}{H},
		\end{equation*}
		this completes the proof.
	\end{proof}

	\subsection{Tracking chains}\label{subsection: tracking chains}
	Suppose that~$0\leq i\leq i^\star$, consider a chain~$\frakc=(F,V,I)\in\frakC$ with~$F=\cF_1,\ldots,\cF_\ell$ and let~$\psi\colon I\injection V_\cH$.
	We do not directly show that the number of embeddings~$\Phi_{\frakc,\psi}$ is typically close to a deterministic trajectory.
	Instead, we define
	\begin{equation*}
		\cG_\frakc:=\cF_\ell[V\cap V_{\cF_\ell}]\qtand
		J_\frakc:=\begin{cases}
			I &\text{if~$\ell=1$};\\
			V_{\cF_{\ell-1}}\cap V_{\cG_\frakc} &\text{if~$\ell\geq 2$}
		\end{cases}
	\end{equation*}\gladd{construction}{Gc}{$\cG_{(\cF_1,\ldots,\cF_\ell,V,I)}=\cF_\ell[V\cap V(\cF_\ell)]$}\gladd{construction}{Jc}{$J_\frakc=\begin{cases}
		I &\text{if~$\ell=1$};\\
		V_{\cF_{\ell-1}}\cap V_{\cG_\frakc} &\text{if~$\ell\geq 2$}
	\end{cases}$}%
	and show that~$\Phi_{\frakc,\psi}$ is typically close to~$\phihat_{\cG_\frakc,J_\frakc}\Phi_{\frakc|\om,\psi}$ which given~$\Phi_{\frakc|\om,\psi}$ is the random quantity our deterministic heuristic estimates for embeddings suggest for
	\begin{equation*}
		\sum_{\phi\in \Phi_{\frakc|\om,\psi}^\sim}\Phi_{\cG_\frakc,\restr{\phi}{J_\frakc}}\approx \Phi_{\frakc,\psi}.
	\end{equation*}
	To this end, let
	\begin{equation*}
		\Phihat_{\frakc,\psi}(i):=\phihat_{\cG_\frakc,J_\frakc}\Phi_{\cC_{\frakc|\om},\psi}\qtand
		X_{\frakc,\psi}(i):=\Phi_{\cC_\frakc,\psi}-\Phihat_{\frakc,\psi}.
	\end{equation*}\gladd{realRV}{Phihatcpsi}{$\Phihat_{\frakc,\psi}(i)=\phihat_{\cG_\frakc,J_\frakc}(i)\Phi_{\cC_{\frakc|\om},\psi}(i)$}\gladd{realRV}{Xcpsi}{$X_{\frakc,\psi}(i)=\Phi_{\cC_\frakc,\psi}(i)-\Phihat_{\frakc,\psi}(i)$}
	Our analysis of~$\Phi_{\frakc,\psi}$ crucially relies on Lemma~\ref{lemma: ladder averaging}.
	There, a sum of numbers of embeddings of branchings of~$\frakc$ is a key quantity which motivates the following definition.
	For~$e\in\cC_\frakc\setminus\cC_\frakc[I]$, the~\emph{$e$-branching family} of~$\frakc$ is\gladd{construction}{Bce}{$\frakB_\frakc^e=\cset{ \frakb }{ \text{$\frakb$ is the~$\beta$-branching of~$\frakc$ for some~$\beta\colon f\bijection e$ where~$f\in\cF$} }$}
	\begin{equation*}
		\frakB_\frakc^e:=\cset{ \frakb }{ \text{$\frakb$ is the~$\beta$-branching of~$\frakc$ for some~$\beta\colon f\bijection e$ where~$f\in\cF$} }.
	\end{equation*}
	We define the stopping times
	\begin{equation*}
		\begin{aligned}
			\tau_{\frakC}&:=\min\cset{
				i\geq 0 }{\Phi_{\frakc,\psi}\neq \Phihat_{\frakc,\psi}\pm \delta^{-1}\zeta\phihat_{\frakc,I}\text{ for some } \frakc=(F,V,I)\in\frakC, \psi\colon I\injection V_\cH},\\
			\tautilde_{\frakB}&:=\min\set*{\setlength\arraycolsep{0pt}\begin{array}{ll}
					i\geq 0 :~&\sum_{\frakb\in \frakB_{\frakc}^e}\Phi_{\frakb,\psi}\neq \sum_{\frakb\in \frakB_{\frakc}^e}\Phihat_{\frakb,\psi}\pm \delta^{-1/2}\zeta\phihat_{\frakb,I}\\&\quad\text{for some } \frakc=(F,V,I)\in\frakC, e\in\cC_\frakc\setminus\cC_\frakc[I], \psi\colon I\injection V_\cH
			\end{array}}.
		\end{aligned}
	\end{equation*}\gladd{stoppingtime}{tauC}{$\tau_{\frakC}=\min\cset{
			i\geq 0 }{\Phi_{\frakc,\psi}\neq \Phihat_{\frakc,\psi}\pm \delta^{-1}\zeta\phihat_{\frakc,I}\text{ for some } \frakc=(F,V,I)\in\frakC, \psi\colon I\injection V_\cH}$}\gladd{stoppingtime}{tautildeB}{$\tautilde_{\frakB}=\min\set*{\setlength\arraycolsep{0pt}\begin{array}{ll}
				i\geq 0 :~&\sum_{\frakb\in \frakB_{\frakc}^e}\Phi_{\frakb,\psi}\neq \sum_{\frakb\in \frakB_{\frakc}^e}\Phihat_{\frakb,\psi}\pm \delta^{-1/2}\zeta\phihat_{\frakb,I}\\&\quad\text{for some } \frakc=(F,V,I)\in\frakC, e\in\cC_\frakc\setminus\cC_\frakc[I], \psi\colon I\injection V_\cH
		\end{array}}$}
	The stopping time~$\tau_\frakC$ is the fourth stopping time mentioned in Section~\ref{section: heuristics}.
	Similarly as with the introduction of the stopping time~$\tau_\ccF\geq \tau_{\frakC}$ in Section~\ref{section: stopping times}, the precise definition of~$\tau_\frakB$ is not relevant in this section, so we instead work with the stopping time~$\tautilde_\frakB$ that satisfies~$\tautilde_\frakB\geq \tau_\frakB$.
	We set
	\begin{equation*}
		\tautilde_\frakC^\star:=\tau_{\cH^*}\wedge \tau_\ccB\wedge\tau_{\ccB'}\wedge\tau_\frakC\wedge \tautilde_{\frakB}\geq \tau^\star.
	\end{equation*}\gladd{stoppingtime}{tautildeCstar}{$\tautilde_\frakC^\star=\tau_{\cH^*}\wedge \tau_\ccB\wedge\tau_{\ccB'}\wedge\tau_\frakC\wedge \tautilde_{\frakB}$}
	We remark that whenever the aforementioned numbers of embeddings are close to their corresponding random trajectories, they are also close to a corresponding deterministic trajectory in the following sense.
	\begin{lemma}\label{lemma: ladders deterministic trajectory}
		Let~$i\geq 0$ and~$\cX:=\set{i<\tau_{\frakC}}$.
		Let~$\frakc=(F,V,I)\in\frakC$ and~$\psi\colon I\injection V_\cH$.
		Then,~$\Phi_{\frakc,\psi}\Xeq (1\pm \delta^5)\phihat_{\frakc,I}$.
	\end{lemma}
	\begin{proof}
		Similarly as in the proof of Lemma~\ref{lemma: reverse contraction}, for every chain~$\frakc'=(F',V',I)$ where~$F'$ has length~$\ell'$, let
		\begin{equation*}
			g_{\frakc'}:=\abs{ \cset{ 0\leq\ell''\leq\ell'-1 }{ \cC_{\frakc'|\ell''}\neq\cC_{\frakc'|\ell''+1} } }.
		\end{equation*}
		Suppose that~$F=\cF_1,\ldots,\cF_\ell$.
		We use induction on~$\ell$ to show that
		\begin{equation}\label{equation: ladders deterministic stronger}
			\Phi_{\frakc,\psi}\Xeq (1\pm g_\frakc\delta^{-1}\zeta)\phihat_{\frakc,I}.
		\end{equation}
		By Lemma~\ref{lemma: chain size}, we have~$\abs{V}\leq \eps^{-3}$, hence~$\abs{\cC_\frakc}\leq \eps^{-3k}$ and thus~$g_\frakc\leq\eps^{-3}+\eps^{-3k}$, so this is sufficient.
		
		Let us proceed with the proof by induction.
		If~$\ell=1$, then~$g_\frakc=1$ by Lemma~\ref{lemma: chains are not trivial} and we have~$\Phi_{\frakc,\psi}\Xeq (1\pm\delta^{-1}\zeta)\phihat_{\frakc,I}$.
		Let~$q\geq 2$ and suppose that~\eqref{equation: ladders deterministic stronger} holds if~$\ell\leq q-1$.
		Suppose that~$\ell=q$.
		If~$\cC_{\frakc|\om}=\cC_\frakc$, then~\eqref{equation: ladders deterministic stronger} follows by induction hypothesis, so we may assume~$\cC_{\frakc|\om}\neq\cC_\frakc$ and hence~$g_{\frakc|\om}=g_{\frakc}-1$.
		Then, by induction hypothesis we have
		\begin{equation*}
			\Phi_{\frakc|\om,\psi}
			\Xeq (1\pm g_{\frakc|\om}\delta^{-1}\zeta)\phihat_{\frakc|\om,I}
			=(1\pm (g_{\frakc}-1)\delta^{-1}\zeta)\phihat_{\frakc|\om,I}.
		\end{equation*}
		Since~$\phihat_{\cG_\frakc,J_\frakc}\phihat_{\frakc|\om,I}=\phihat_{\frakc,I}$, this yields
		\begin{equation*}
			\begin{aligned}
				\Phi_{\frakc,\psi}
				&\Xeq \phihat_{\cG_\frakc,J_\frakc}\Phi_{\frakc|\om,\psi}\pm \delta^{-1}\zeta\phihat_{\frakc,I}
				\Xeq (1\pm (g_\frakc-1)\delta^{-1}\zeta)\phihat_{\frakc,I}\pm \delta^{-1}\zeta\phihat_{\frakc,I}\\
				&=(1\pm g_\frakc\delta^{-1}\zeta)\phihat_{\frakc,I},
			\end{aligned}
		\end{equation*}
		which completes the proof.
	\end{proof}
	
	In this section, we show that the probability that~$\tau_\frakC\leq \tautilde_\frakC^\star\wedge i^\star$ is small.
	The collection~$\frakC$ is infinite, however, Lemma~\ref{lemma: finite chain collection} shows that it suffices to consider a collection of chains of size at most~$1/\delta$.
	By relying on a union bound argument, this allows us to essentially only consider one fixed chain~$\frakc=(F,V,I)\in\frakC$.
	\begin{lemma}\label{lemma: finite chain collection}
		There exists a collection~$\frakC_0\subseteq\frakC$ with~$\abs{\frakC_0}\leq 1/\delta$ such that for all~$\frakc=(F,V,I)\in\frakC$, there exists a chain~$\frakc_0=(F_0,V_0,I_0)\in\frakC_0$ such that~$(\cC_{\frakc_0},I_0)$ is a copy of~$(\cC_\frakc,I)$ while~$(\cC_{\frakc_0|\om},I_0)$ is a copy of~$(\cC_{\frakc|\om},I)$.
	\end{lemma}
	\begin{proof}
		Consider the set~$\ccT$ of all templates~$(\cA,I)$ where~$V_\cA\subseteq\set{1,\ldots,1/\eps^3}$.
		By Lemma~\ref{lemma: chain size}, for all~$\frakc=(F,V,I)\in\frakC$, we may choose a template~$\cT_\frakc\in\ccT$ that is a copy of~$(\cC_\frakc,I)$.
		Let~$\ccT_2:=\cset{ (\cT_\frakc,\cT_{\frakc|\om}) }{ \frakc\in\frakC }\subseteq \ccT^2$ and for every pair~$\ccP\in\ccT_2$, choose a chain~$\frakc_\ccP\in\frakC$ with~$\ccP=(\cT_{\frakc_\ccP},\cT_{\frakc_\ccP|\om})$.
		Then,~$\cset{ \frakc_\ccP }{ \ccP\in\ccT_2 }$ is a collection as desired.
	\end{proof}
	\begin{observation}\label{observation: ladder individual}
		Suppose that~$\frakC_0\subseteq\frakC$ is a collection of chains as in Lemma~\ref{lemma: finite chain collection}.
		For~$\frakc=(F,V,I)\in\frakC$ and~$\psi\colon I\injection V_\cH$, let
		\begin{equation*}
			\tau_{\frakc,\psi}:=\min\cset{ i\geq 0 }{ \Phi_{\frakc,\psi}\neq \Phihat_{\frakc,\psi}\pm \delta^{-1}\zeta\phihat_{\frakc,I} }.
		\end{equation*}
		Then,
		\begin{equation*}
			\pr{\tau_\frakC\leq \tautilde^\star_\frakC\wedge i^\star }
			\leq \sum_{\frakc=(F,V,I)\in\frakC_0,\psi\colon I\injection V_\cH\colon} \pr{ \tau_{\frakc,\psi}\leq \tautilde^\star_\frakC\wedge i^\star }.
		\end{equation*}
	\end{observation}
	
	Hence, fix~$\frakc=(F,V,I)\in\frakC$ where~$F=\cF_1,\ldots,\cF_\ell$ and furthermore fix~$\psi\colon I\injection V_\cH$.
	Note that by Lemma~\ref{lemma: chains are not trivial}, we have~$\cC_\frakc\setminus\cC_\frakc[I]\neq\emptyset$ and~$\ell\geq 1$.
	For~$i\geq 0$, let~$\xi_1(i)$ denote the corresponding absolute error appearing in the definition of~$\tau_\frakC$ and consider a slightly smaller error term~$\xi_0(i)$, that is let
	\begin{equation*}
		\xi_1(i):=\delta^{-1} \zeta\phihat_{\frakc,I}\qtand
		\xi_0(i):=(1-\delta) \xi_1(i)
	\end{equation*}
	and define the stopping time
	\begin{equation*}
		\tau:=\min\cset{i\geq 0}{ \Phi_{\frakc,\psi}\neq\Phihat_{\frakc,\psi}\pm \xi_1 }.
	\end{equation*}
	Our goal is now to show that~$\Phi_{\frakc,\psi}$ is typically in the interval~$I_1(i):=[\Phihat_{\frakc,\psi}-\xi_1, \Phihat_{\frakc,\psi}+\xi_1]$ as long as other key quantities are as predicted.
	More formally, our goal is to show that the probability that~$\tau\leq \tautilde_{\frakC}^\star\wedge i^\star$ is sufficiently small.
	Define the \enquote{critical} intervals
	\begin{equation*}
		I^-(i):=[\Phihat_{\frakc,\psi}-\xi_1,\Phihat_{\frakc,\psi}-\xi_0],\quad
		I^+(i):=[\Phihat_{\frakc,\psi}+\xi_0,\Phihat_{\frakc,\psi}+\xi_1].
	\end{equation*}
	As long as~$\Phi_{\frakc,\psi}$ is not close to the boundary of~$I_1$ in the sense that~$\Phi_{\frakc,\psi}$ is in the interval~$I_0(i):=[\Phihat_{\frakc,\psi}-\xi_0, \Phihat_{\frakc,\psi}+\xi_0]$, within the next few steps~$i$, there is no danger that~$\Phi_{\frakc,\psi}$ could be outside~$I_1$ provided that we chose~$\xi_1$ to be sufficiently large compared to~$\xi_0$.
	The situation only becomes \enquote{critical} when~$\Phi_{\frakc,\psi}$ is outside~$I_0$, that is when~$\Phi_{\frakc,\psi}$ enters the critical interval~$I^-$ or~$I^+$.
	Exploiting the fact that whenever this is the case, the process exhibits self-correcting behavior in the sense that whenever this is the case, in expectation~$\Phi_{\frakc,\psi}$ returns to values close to~$\Phihat_{\frakc,\psi}$, we show that it is unlikely that~$\Phi_{\frakc,\psi}$ ever fully crosses one of the critical intervals.
	Since, as we formally show later,~$\Phi_{\frakc,\psi}$ cannot jump over one of the critical intervals in one step, it suffices to restrict our attention to the behavior of~$\Phi_{\frakc,\psi}$ inside the critical intervals.
	
	For~$\pom\in\set{-,+}$, consider the random variable
	\begin{equation*}
		Y^\pom(i):=\pom X_{\frakc,\psi}-\xi_1
	\end{equation*}
	that measures by how much~$\Phi_{\frakc,\psi}$ exceeds the permitted deviation~$\xi_1$ from~$\Phihat_{\frakc,\psi}$.
	Our goal is to show that~$Y^\pom$ is non-positive whenever~$i\leq\tautilde_\frakC^\star$.
	To show that this is the case, for all~$i_0\geq 0$, we consider an auxiliary random process~$Z^\pom_{i_0}(i_0),Z^\pom_{i_0}(i_0+1),\ldots$ that follows the evolution of~$Y^\pom(i_0),Y^\pom(i_0+1),\ldots$ as long as the situation is relevant for our analysis, that is until~$\Phi_{\frakc,\psi}$ has left the critical interval~$I^\pom$ or until we are at step~$\tautilde_\frakC^\star\wedge i^\star$.
	In these cases, that is when~$Z^\pom_{i_0}$ no longer follows~$Y^\pom$, we simply define the auxiliary process to remain constant.
	Note in particular, that if a deviation of~$\Phi_{\frakc,\psi}$ from~$\Phihat_{\frakc,\psi}$ beyond~$\xi_1$ caused the auxiliary process to no longer follows~$Y^\pom$, then the value of the auxiliary process at step~$i^\star$ indicates this since the relevant value~$Y^\pom(\tautilde_\frakC^\star\wedge i^\star)$ is the last value captured.
	Formally, for~$i_0\geq 0$, we define the stopping time
	\begin{equation*}
		\tau^\pom_{i_0}:=\min\cset{i\geq i_0}{ \Phi_{\frakc,\psi}\notin I^\pom }
	\end{equation*}
	that measures when, starting at step~$i_0$, the random variable~$\Phi_{\frakc,\psi}$ is first outside the critical interval~$I^\pom$.
	Note that if~$\Phi_{\frakc,\psi}(i_0)\notin I^\pom$, then~$\tau^\pom_{i_0}=i_0$.
	For~$i\geq i_0$, let
	\begin{equation*}
		Z^\pom_{i_0}(i):=Y^\pom(i_0\vee (i\wedge \tau^\pom_{i_0}\wedge \tautilde_\frakC^\star\wedge i^\star)).
	\end{equation*}
	In fact, for our analysis it suffices to consider only the evolution of~$Z^\pom_{\sigma^\pom}(\sigma^\pom),Z^\pom_{\sigma^\pom}(\sigma^\pom+1),\ldots$ where
	\begin{equation*}
		\sigma^\pom:=\min\cset{ j\geq 0 }{ \pom X_{\frakc,\psi}\geq \xi_0 \stforall j\leq i<\tautilde_\frakC^\star\wedge i^\star }\leq \tautilde_\frakC^\star\wedge i^\star
	\end{equation*}
	is the last step at which~$\Phi_{\frakc,\psi}$ entered the critical interval~$I^\pom$ before step~$\tautilde_\frakC^\star\wedge i^\star$.
	Indeed, if~$\tau\leq \tautilde_{\frakC}^\star\wedge i^\star$, then, for some~$\pom\in\set{+,-}$, we have~$\Phi_{\frakc,\psi}\in I^\pom$ for all~$\sigma^\pom\leq i<\tautilde_{\frakC}^\star\wedge i^\star$, hence~$\tau^\pom_{\sigma^\pom}=\tautilde_{\frakC}^\star\wedge i^\star$ and thus~$Z^\pom_{\sigma^\pom}(i^\star)=Y^\pom(\tautilde_{\frakC}^\star\wedge i^\star)=Y^\pom(\tau)>0$.
	This reasoning leads to the following observation.
	\begin{observation}\label{observation: ladder critical times}
		$\set{\tau\leq \tautilde_\frakC^\star\wedge i^\star}\subseteq \set{ Z^-_{\sigma^-}(i^\star)>0 }\cup\set{ Z^+_{\sigma^+}(i^\star)>0 }$.
	\end{observation}
	
	We use Freedman's inequality for supermartingales below to show that the probabilities of the events on the right in Observation~\ref{observation: ladder critical times} are sufficiently small.
	\begin{lemma}[Freedman's inequality for supermartingales~\cite{freedman:75}]\label{lemma: freedman}
		Suppose that~$X(0),X(1),\ldots$ is a supermartingale with respect to a filtration~$\frakX(0),\frakX(1),\ldots$ such that~$\abs{X(i+1)-X(i)}\leq a$ for all~$i\geq 0$ and~$\sum_{i\geq 0}\cex{\abs{X(i+1)-X(i)}}{\frakX(i)}\leq b$.
		Then, for all~$t>0$,
		\begin{equation*}
			\pr{X(i)\geq X(0)+t\stforsome i\geq 0}\leq \exp\paren[\bigg]{-\frac{t^2}{2a(t+b)}}.
		\end{equation*}
	\end{lemma}
	
	We dedicate Sections~\ref{subsubsection: chain trend} and~\ref{subsubsection: chain boundedness} to proving that the auxiliary random processes satisfy the conditions that are necessary for an application of Lemma~\ref{lemma: freedman}.
	The application itself is the topic of Section~\ref{subsubsection: chain concentration}.
	
	\subsubsection{Trend}\label{subsubsection: chain trend}
	Here, we prove that for all~$\pom\in\set{-,+}$ and~$i_0\geq 0$, the expected one-step changes of the process~$Z_{i_0}^\pom(i_0),Z_{i_0}^\pom(i_0+1),\ldots$ are non-positive.
	In Lemma~\ref{lemma: delta phihat G J xi ladder}, we estimate the one-step changes of the error term that we use in this section.
	Then in Lemma~\ref{lemma: ladder deviation trend}, we state a precise estimate for the expected one-step change of the random process~$X_{\frakc,\psi}(0),X_{\frakc,\psi}(1),\ldots$ that measures the deviations from the random trajectory given by~$\Phihat_{\frakc,\psi}(0),\Phihat_{\frakc,\psi}(1),\ldots$.
	To obtain this precise estimate, which is the key argument in this section, we crucially rely on Lemma~\ref{lemma: ladder averaging} and the even more precise control over branching families that we have in step~$i$ whenever~$i<\tau_\frakB$.
	Assuming such control over branching families in our arguments here serves to shift the main arguments based on the exploitation of self-correcting behavior to a slightly different setting, namely from individual chains to families, which turns out to be crucial for our argumentation (see Section~\ref{section: branching families}).
	At the end of this section, we combine the previously collected estimates to conclude that~$Z_{i_0}^\pom(i_0),Z_{i_0}^\pom(i_0+1),\ldots$ is indeed a supermartingale for all~$\pom\in\set{-,+}$ and~$i_0\geq 0$ (see Lemma~\ref{lemma: ladder trend}).
	
	\begin{observation}\label{observation: xi chain}
		Extend~$\phat$ and~$\xi_1$ to continuous trajectories defined on the whole interval~$[0,i^\star+1]$ using the same expression as above.
		Then, for~$x\in[0,i^\star+1]$,
		\begin{equation*}
			\begin{gathered}
				\xi_1'(x)=-\paren[\bigg]{\abs{\cC_\frakc}-1-\frac{\rho_\cF}{2}}\frac{\abs{\cF}k!\,\xi_1(x)}{n^k\phat(x)},\\
				\xi_1''(x)=-\paren[\bigg]{\abs{\cC_\frakc}-1-\frac{\rho_\cF}{2}}\paren[\bigg]{\abs{\cC_\frakc}-2-\frac{\rho_\cF}{2}}\frac{\abs{\cF}^2(k!)^2\xi_1(x)}{n^{2k}\phat(x)^2}.
			\end{gathered}
		\end{equation*}
	\end{observation}
	
	\begin{lemma}\label{lemma: delta phihat G J xi ladder}
		Let~$0\leq i\leq i^\star$ and~$\cX:=\set{i\leq\tau_\emptyset}$.
		Then,
		\begin{equation*}
			\Delta\xi_1 \Xeq -\paren[\bigg]{\abs{\cC_\frakc}-1-\frac{\rho_{\cF}}{2}}\frac{\abs{\cF}\xi_1}{H}\pm \frac{\zeta^{2}\xi_1}{H}.
		\end{equation*}
	\end{lemma}
	\begin{proof}
		This is a consequence of Taylor's theorem.
		In detail, we argue as follows.
		
		Together with Observation~\ref{observation: xi chain} and Lemma~\ref{lemma: chain size}, Lemma~\ref{lemma: taylor} yields
		\begin{equation*}
			\Delta\xi_1 = -\paren[\bigg]{\abs{\cC_\frakc}-1-\frac{\rho_\cF}{2}}\frac{\abs{\cF}k!\,\xi_1}{n^k\phat}\pm \max_{x\in[i,i+1]}\frac{\xi_1(x)}{\delta n^{2k}\phat(x)^2}
		\end{equation*}
		We investigate the first term and the maximum separately.
		Using Lemma~\ref{lemma: edges of H}, we have
		\begin{equation*}
			-\paren[\bigg]{\abs{\cC_\frakc}-1-\frac{\rho_\cF}{2}}\frac{\abs{\cF}k!\,\xi_1}{n^k\phat}
			\Xeq-\paren[\bigg]{\abs{\cC_\frakc}-1-\frac{\rho_\cF}{2}}\frac{\abs{\cF}\xi_1}{H}.
		\end{equation*}
		Furthermore, using Lemma~\ref{lemma: bounds of delta phat}, Lemma~\ref{lemma: edges of H} and Lemma~\ref{lemma: zeta and H} yields
		\begin{equation*}
			\max_{x\in[i,i+1]}\frac{\xi_1(x)}{\delta n^{2k}\phat(x)^2}
			\leq \frac{\xi_1}{\delta n^{2k}\phat(i+1)^2}
			\leq \frac{\xi_1}{\delta^2 n^{2k}\phat^2}
			\Xleq \frac{\xi_1}{\delta^2 H^2}
			\Xleq \frac{\zeta^{2+2\eps^2}\xi_1}{\delta^2 H}
			\leq \frac{\zeta^{2+\eps^2}\xi_1}{H}.
		\end{equation*}
		Thus we obtain the desired expression for~$\Delta\xi_1$.
	\end{proof}
	
	\begin{lemma}\label{lemma: phihat G J next}
		For all~$0\leq i\leq i^\star$, we have
		\begin{equation*}
			\phihat_{\cG_\frakc,J_\frakc}(i+1)=(1\pm \zeta^{2})\phihat_{\cG_\frakc,J_\frakc}.
		\end{equation*}
	\end{lemma}
	\begin{proof}
		This follows from Lemma~\ref{lemma: delta phihat} and Lemma~\ref{lemma: zeta and H}.
	\end{proof}
	
	In the next lemma, we state the expression for the expected one-step change~$\exi{\Delta X_{\frakc,\psi}}$ that we subsequently use to obtain the desired supermartingale property.
	In the proof, ignoring error terms, we essentially argue as follows.
	We have
	\begin{equation}\label{equation: expected chain deviation change}
		\begin{aligned}
			\exi{\Delta X_{\frakc,\psi}}
			&=\exi{\Delta \Phi_{\frakc,\psi}}-\exi{\Delta(\phihat_{\cG_\frakc,J_\frakc}\Phi_{\frakc|-,\psi})}\\
			&=\exi{\Delta \Phi_{\frakc,\psi}}-(\Delta\phihat_{\cG_\frakc,J_\frakc})\Phi_{\frakc|-,\psi}-\phihat_{\cG_\frakc,J_\frakc}(i+1)\exi{\Delta \Phi_{\frakc|-,\psi}}.
		\end{aligned}
	\end{equation}
	Since~$\phihat_{\cG_\frakc,J_\frakc}(i+1)\approx \phihat_{\cG_\frakc,J_\frakc}$, this yields
	\begin{equation}\label{equation: heuristic chain change}
		\exi{\Delta X_{\frakc,\psi}}
		\approx\exi{\Delta \Phi_{\frakc,\psi}}-(\Delta\phihat_{\cG_\frakc,J_\frakc})\Phi_{\frakc|-,\psi}-\phihat_{\cG_\frakc,J_\frakc}\exi{\Delta \Phi_{\frakc|-,\psi}}.
	\end{equation}
	Contributions to~$\Delta\Phi_{\frakc,\psi}$ come from the loss of edges~$\phi(e)$ where~$\phi\in\Phi^\sim_{\frakc,\psi}$ and~$e\in\cC_\frakc\setminus\cC_\frakc[I]$.
	Note that if~$e\in\cC_{\frakc|-}$, then for this loss of~$\phi(e)$, there is a corresponding contribution to~$\Delta\Phi_{\frakc|-,\psi}$.
	Otherwise, there is no corresponding contribution to~$\Delta\Phi_{\frakc|-,\psi}$, however, we find a corresponding contribution in
	\begin{equation*}
		(\Delta\phihat_{\cG_\frakc,J_\frakc})\Phi_{\frakc|-,\psi}
		\approx -\abs{\cG_\frakc\setminus\set{J_\frakc}}\frac{\abs{\cF}}{H}\phihat_{\cG_\frakc,J_\frakc}\Phi_{\frakc|-,\psi}
		=-\abs{\cG_\frakc\setminus\set{J_\frakc}}\frac{\abs{\cF}}{H}\Phihat_{\frakc,\psi}.
	\end{equation*}
	With this in mind, relying on Lemma~\ref{lemma: ladder averaging} and Lemma~\ref{lemma: edges of H}, for~$f\in\cF$, we estimate
	\begin{equation*}
		\begin{aligned}
			\exi{\Delta \Phi_{\frakc,\psi}}
			&\approx-\sum_{e\in \cC_{\frakc}\setminus\cC_\frakc[I]}\sum_{f\in\cF}\sum_{\beta\colon f\bijection e}\frac{\Phi_{\frakc|[\beta],\psi}\Phi_{\frakc,\psi}}{\aut(\cF)H^*\Phi_{\frakc|e,\psi}}\\
			&\approx -\sum_{e\in \cC_{\frakc|-}\setminus\cC_{\frakc|-}[I]}\sum_{f\in\cF}\sum_{\beta\colon f\bijection e}\frac{\phihat_{\cF,f}\Phi_{\frakc,\psi}}{\aut(\cF)\hhat^*}-\sum_{e\in \cG_\frakc\setminus\set{J_\frakc}}\sum_{f\in\cF}\sum_{\beta\colon f\bijection e}\frac{\Phi_{\frakc,\psi}\Phi_{\frakc|[\beta],\psi}}{\aut(\cF)\hhat^*\Phi_{\frakc|e,\psi}}\\
			&\approx -\frac{(\abs{\cC_{\frakc|-}}-1)\abs{\cF}}{H}\Phi_{\frakc,\psi}-\sum_{e\in \cG_\frakc\setminus\set{J_\frakc}}\sum_{f\in\cF}\sum_{\beta\colon f\bijection e}\frac{\Phi_{\frakc,\psi}\Phi_{\frakc|[\beta],\psi}}{k!\,H\phihat_{\cF,f}\Phi_{\frakc|e,\psi}}
		\end{aligned}
	\end{equation*}
	and similarly
	\begin{equation*}
			\exi{\Delta \Phi_{\frakc,\psi}}
			\approx-\sum_{e\in \cC_{\frakc|-}\setminus\cC_{\frakc|-}[I]}\sum_{f\in\cF}\sum_{\beta\colon f\bijection e}\frac{\Phi_{\frakc|[\beta],\psi}\Phi_{\frakc|-,\psi}}{\aut(\cF)H^*\Phi_{\frakc|e,\psi}}
			\approx -\frac{(\abs{\cC_{\frakc|-}}-1)\abs{\cF}}{H}\Phi_{\frakc|-,\psi}.
	\end{equation*}
	Combining the previous three estimates with~\eqref{equation: heuristic chain change}, we obtain
	\begin{equation*}
		\exi{\Delta X_{\frakc,\psi}}
		\approx -\frac{(\abs{\cC_{\frakc|\om}}-1)\abs{\cF}}{H}X_{\frakc,\psi}-\sum_{e\in\cG_\frakc\setminus\set{J_\frakc}}\sum_{f\in\cF}\sum_{\beta\colon f\bijection e}\paren[\bigg]{\frac{\Phi_{\frakc,\psi}\Phi_{\frakc|[\beta],\psi}}{k!\,H\phihat_{\cF,f}\Phi_{\frakc|e,\psi}}-\frac{\Phihat_{\frakc,\psi}}{k!\,H}}.
	\end{equation*}
	Let us investigate the innermost sum on the right.
	The branchings of~$\frakc$ are two extension transformations away from the chain~$\frakc|-$ that appears in the corresponding contributions.
	As our chain tracking only compares chains that are one extension step apart, we introduce the chain~$\frakc$ itself to compare the contributions in the sense that for~$e\in\cG_\frakc\setminus\set{J_\frakc}$,~$f\in\cF$ and~$\beta\colon f\bijection e$, we write
	\begin{equation*}
		\begin{aligned}
			\frac{\Phi_{\frakc,\psi}\Phi_{\frakc|[\beta],\psi}}{k!\,H\phihat_{\cF,f}\Phi_{\frakc|e,\psi}}-\frac{\Phihat_{\frakc,\psi}}{k!\,H}
			&=\frac{\Phi_{\frakc,\psi}X_{\frakc|[\beta],\psi}}{k!\,H\phihat_{\cF,f}\Phi_{\frakc|e,\psi}}+\frac{\Phi_{\frakc,\psi}}{k!\,H}-\frac{\Phihat_{\frakc,\psi}}{k!\,H}
			=\frac{\Phi_{\frakc,\psi}}{k!\,H\phihat_{\cF,f}\Phi_{\frakc|e,\psi}}X_{\frakc|[\beta],\psi}+\frac{1}{k!\,H}X_{\frakc,\psi}\\
			&\approx \frac{\phihat_{\frakc,I}}{k!\,H\phihat_{\frakc|[\beta],I}}X_{\frakc|[\beta],\psi}+\frac{1}{k!\,H}X_{\frakc,\psi}
		\end{aligned}
	\end{equation*}
	Overall, this  leads to
	\begin{equation*}
		\begin{aligned}
			\exi{\Delta X_{\frakc,\psi}}
			&\approx-\frac{(\abs{\cC_{\frakc|\om}}-1)\abs{\cF}}{H}X_{\frakc,\psi}-\frac{\abs{\cG_\frakc\setminus\set{J_\frakc}}\abs{\cF}}{H}X_{\frakc,\psi}-\sum_{e\in\cG_\frakc\setminus\set{J_\frakc}}\sum_{\frakb\in\frakB_\frakc^e} \frac{\phihat_{\frakc,I}}{k!\,H\phihat_{\frakb,I}}X_{\frakb,\psi}\\
			&=-\frac{(\abs{\cC_\frakc}-1)\abs{\cF}}{H}X_{\frakc,\psi}-\sum_{e\in\cG_\frakc\setminus\set{J_\frakc}}\sum_{\frakb\in\frakB_\frakc^e} \frac{\phihat_{\frakc,I}}{k!\,H\phihat_{\frakb,I}}X_{\frakb,\psi}.
		\end{aligned}
	\end{equation*}

	\begin{lemma}\label{lemma: ladder deviation trend}
		Let~$0\leq i\leq i^\star$ and~$\cX:=\set{i<\tau_{\cH^*}\wedge \tau_\ccB\wedge \tau_{\ccB'}\wedge \tau_\frakC }$.
		Then,
		\begin{equation*}
			\exi{\Delta X_{\frakc,\psi}}\Xeq -\frac{(\abs{\cC_\frakc}-1)\abs{\cF}}{H}X_{\frakc,\psi}-\paren[\bigg]{\sum_{e\in\cG_\frakc\setminus\set{J_\frakc}}\sum_{\frakb\in \frakB_\frakc^e}\frac{\phihat_{\frakc,I}}{k!\,H\phihat_{\frakb,I}}X_{\frakb,\psi}}\pm\delta^2\frac{\xi_1}{H}.
		\end{equation*}
	\end{lemma}
	
	\begin{proof}
		Similarly as in~\eqref{equation: expected chain deviation change}, we have
		\begin{equation*}
			\Delta X_{\frakc,\psi}=(\Delta\Phi_{\frakc,\psi})-\phihat_{\cG_\frakc,J_\frakc}(i+1)(\Delta\Phi_{\frakc|\om,\psi})-(\Delta\phihat_{\cG_\frakc,J_\frakc})\Phi_{\frakc|\om,\psi}.
		\end{equation*}
		By Lemma~\ref{lemma: delta phihat} and Lemma~\ref{lemma: phihat G J next}, this entails
		\begin{equation*}
			\begin{aligned}
				\Delta X_{\frakc,\psi}
				&=(\Delta\Phi_{\frakc,\psi})-(1\pm\zeta^{2})\phihat_{\cG_\frakc,J_\frakc}(\Delta\Phi_{\frakc|\om,\psi})+(1\pm\zeta^{2})\frac{\abs{\cF}\abs{\cG_\frakc\setminus\set{J_\frakc}}}{H}\phihat_{\cG_\frakc,J_\frakc}\Phi_{\frakc|\om,\psi}\\
				&=(\Delta\Phi_{\frakc,\psi})-\phihat_{\cG_\frakc,J_\frakc}(\Delta\Phi_{\frakc|\om,\psi})+\frac{\abs{\cF}\abs{\cG_\frakc\setminus\set{J_\frakc}}}{H}\Phihat_{\frakc,\psi}\pm\zeta^{2}\phihat_{\cG_\frakc,J_\frakc}(\Delta\Phi_{\frakc|\om,\psi})\\&\hphantom{=}\mathrel{}\quad\pm \zeta^{3/2}\frac{\phihat_{\cG_\frakc,J_\frakc}\Phi_{\frakc|\om,\psi}}{H}.
			\end{aligned}
		\end{equation*}
		Since by Lemma~\ref{lemma: ladders deterministic trajectory} we have~$\Phi_{\frakc|\om,\psi}\Xeq (1\pm \delta^5)\phihat_{\frakc|\om,I}$, this yields
		\begin{equation}\label{equation: base expression for X}
			\Delta X_{\frakc,\psi}
			\Xeq(\Delta\Phi_{\frakc,\psi})-\phihat_{\cG_\frakc,J_\frakc}(\Delta\Phi_{\frakc|\om,\psi})+\frac{\abs{\cF}\abs{\cG_\frakc\setminus\set{J_\frakc}}}{H}\Phihat_{\frakc,\psi}\pm\zeta^{2}\phihat_{\cG_\frakc,J_\frakc}(\Delta\Phi_{\frakc|\om,\psi})\pm \zeta^{4/3}\frac{\phihat_{\frakc,I}}{H}.
		\end{equation}
		Using Lemma~\ref{lemma: ladder averaging}, we obtain
		\begin{equation}\label{equation: expected c change}
			\begin{aligned}
				\exi{\Delta\Phi_{\frakc,\psi}}
				&=-\sum_{e\in\cC_\frakc\setminus \cC_\frakc[I]} \exi{\abs{\cset{ \phi\in\Phi_{\frakc,\psi}^\sim }{ \phi(e)\in \cF_0(i+1) }}}\\
				&\Xeq-\paren[\bigg]{\sum_{e\in\cC_\frakc\setminus \cC_\frakc[I]}\sum_{f\in\cF}\sum_{\beta\colon f\bijection e}\frac{\Phi_{\frakc|[\beta],\psi}\Phi_{\frakc,\psi}}{\aut(\cF)H^*\Phi_{\frakc|e,\psi}}}\pm \zeta^{1+\delta/3}\frac{\phihat_{\frakc,I}}{H}.
			\end{aligned}
		\end{equation}
		Note that for~$e\in\cC_{\frakc|\om}$,~$f\in\cF$ and~$\beta\colon f\bijection e$, we have~$\frakc|\om|e=\frakc|e$ and~$\frakc|\mathord{-}|[\beta]=\frakc|[\beta]$.
		Hence, again using Lemma~\ref{lemma: ladder averaging}, we similarly obtain
		\begin{equation}\label{equation: expected c- change}
			\begin{aligned}
				\exi{\Delta\Phi_{\frakc|\om,\psi}}
				&=-\sum_{e\in\cC_{\frakc|\om}\setminus \cC_{\frakc|\om}[I]} \exi{\abs{\cset{ \phi\in\Phi_{\frakc|\om,\psi}^\sim }{ \phi(e)\in \cF_0(i+1) }}}\\
				&\Xeq-\paren[\bigg]{\sum_{e\in\cC_{\frakc|\om}\setminus \cC_{\frakc|\om}[I]}\sum_{f\in\cF}\sum_{\beta\colon f\bijection e}\frac{\Phi_{\frakc|[\beta],\psi}\Phi_{\frakc|\om,\psi}}{\aut(\cF)H^*\Phi_{\frakc|e,\psi}}}\pm \zeta^{1+\delta/3}\frac{\phihat_{\frakc|\om,I}}{H}.
			\end{aligned}
		\end{equation}
		Furthermore, since by Lemma~\ref{lemma: ladders deterministic trajectory} we have~$\Phi_{\frakc|[\beta],\psi}\Xeq (1\pm\delta^5)\phihat_{\frakc|[\beta],I}$,~$\Phi_{\frakc|\om,\psi}\Xeq (1\pm\delta^5)\phihat_{\frakc|\om,I}$ and~$\Phi_{\frakc|e,\psi}\Xeq (1\pm\delta^5)\phihat_{\frakc|e,I}$, using Lemma~\ref{lemma: edges of H}, for~$f\in\cF$, this yields
		\begin{equation}\label{equation: absolute expected c- change}
			\abs{\exi{\Delta\Phi_{\frakc|\om,\psi}}}
			\Xleq \frac{2\abs{\cC_\frakc}\abs{\cF}k!\,\phihat_{\cF,f}\phihat_{\frakc|\om,I}}{\aut(\cF)H^*}+\zeta^{1+\delta/3}\frac{\phihat_{\frakc|\om,I}}{H}
			\Xleq \frac{3\abs{\cC_\frakc}\abs{\cF}\phihat_{\frakc|\om,I}}{H}.
		\end{equation}
		From~\eqref{equation: base expression for X}, using~\eqref{equation: expected c change} and~\eqref{equation: expected c- change} as well as the fact that~$\cC_\frakc\setminus\cC_\frakc[I]=(\cC_{\frakc|\om}\setminus\cC_{\frakc|\om}[I])\cup (\cG_\frakc\setminus\set{J_\frakc})$, we obtain
		\begin{equation*}
			\begin{aligned}
				\exi{\Delta X_{\frakc,\psi}}
				&\Xeq -\paren[\bigg]{\sum_{e\in\cC_{\frakc|\om}\setminus \cC_{\frakc|\om}[I]}\sum_{f\in\cF}\sum_{\beta\colon f\bijection e}\frac{\Phi_{\frakc|[\beta],\psi}}{\aut(\cF)H^*\Phi_{\frakc|e,\psi}}}(\Phi_{\frakc,\psi}-\Phihat_{\frakc,\psi})\\
				&\hphantom{=}\mathrel{}\quad -\sum_{e\in\cG_\frakc\setminus \set{J_\frakc}}\paren[\bigg]{\paren[\bigg]{ \sum_{f\in\cF}\sum_{\beta\colon f\bijection e}\frac{\Phi_{\frakc,\psi}}{\aut(\cF)H^*\Phi_{\frakc|e,\psi}}\Phi_{\frakc|[\beta],\psi}} -\frac{\abs{\cF}}{H}\Phihat_{\frakc,\psi} } \\&\hphantom{=}\mathrel{}\quad\pm \zeta^2\phihat_{\cG_\frakc,J_\frakc}\exi{\Delta \Phi_{\frakc|\om,\psi}}\pm \zeta^{1+\delta/4}\frac{\phihat_{\frakc,I}}{H}.
			\end{aligned}
		\end{equation*}
		Due to~\eqref{equation: absolute expected c- change}, this yields
		\begin{equation}\label{equation: chain deviation two terms}
			\begin{aligned}
				\exi{\Delta X_{\frakc,\psi}}
				&\Xeq -\paren[\bigg]{\sum_{e\in\cC_{\frakc|\om}\setminus \cC_{\frakc|\om}[I]}\sum_{f\in\cF}\sum_{\beta\colon f\bijection e}\frac{\Phi_{\frakc|[\beta],\psi}}{\aut(\cF)H^*\Phi_{\frakc|e,\psi}}}(\Phi_{\frakc,\psi}-\Phihat_{\frakc,\psi})\\
				&\hphantom{=}\mathrel{}\quad -\sum_{e\in\cG_\frakc\setminus J_\frakc}\paren[\bigg]{\paren[\bigg]{ \sum_{f\in\cF}\sum_{\beta\colon f\bijection e}\frac{\Phi_{\frakc,\psi}}{\aut(\cF)H^*\Phi_{\frakc|e,\psi}}\Phi_{\frakc|[\beta],\psi}} -\frac{\abs{\cF}}{H}\Phihat_{\frakc,\psi} } \pm \zeta^{1+\delta/5}\frac{\phihat_{\frakc,I}}{H}.
			\end{aligned}
		\end{equation}
		We investigate the first two terms of the sum on the right side separately.
		
		First, note that for all~$e\in\cC_{\frakc|\om}\setminus\cC_{\frakc|\om}[I]$, using Lemma~\ref{lemma: ladders deterministic trajectory} and Lemma~\ref{lemma: edges of H}, we obtain
		\begin{align*}
			\sum_{f\in\cF}\sum_{\beta\colon f\bijection e}\frac{ \Phi_{\frakc|[\beta],\psi}}{\aut(\cF)H^*\Phi_{\frakc|e,\psi}}
			&\Xeq (1\pm 4\delta^5)\frac{\abs{\cF}k!\, \phihat_{\cF_\frakc^\beta,e} }{\aut(\cF)H^*}
			\Xeq (1\pm 5\delta^5)\frac{\abs{\cF}}{H}.
		\end{align*}
		Thus, for the first term, using~$X_{\frakc,\psi}\Xleq \xi_1$ and Lemma~\ref{lemma: chain size}, we obtain
		\begin{equation}\label{equation: chain deviation first term}
			\begin{aligned}
				-\paren[\bigg]{\sum_{e\in\cC_{\frakc|\om}\setminus \cC_{\frakc|\om}[I]}\sum_{f\in\cF}\sum_{\beta\colon f\bijection e}\frac{\Phi_{\frakc|[\beta],\psi}}{\aut(\cF)H^*\Phi_{\frakc|e,\psi}}}(\Phi_{\frakc,\psi}-\Phihat_{\frakc,\psi})
				&\Xeq -(1\pm 5\delta^5)\frac{(\abs{\cC_{\frakc|\om}}-1)\abs{\cF}}{H}X_{\frakc,\psi}\\
				&\Xeq-\frac{(\abs{\cC_{\frakc|\om}}-1)\abs{\cF}}{H}X_{\frakc,\psi}\pm \delta^4\frac{\xi_1}{H}.
			\end{aligned}
		\end{equation}
		
		Let us consider the second term.
		For all~$e\in\cG_\frakc\setminus\set{J_\frakc}$, using the fact that for all~$f,f'\in\cF$ and~$\beta\colon f\bijection e$, we have~$\Phihat_{\frakc|[\beta],\psi}=\phihat_{\cF_\frakc^\beta,e}\Phi_{\frakc|[\beta]|-,\psi}=\phihat_{\cF,f'}\Phi_{\frakc|e,\psi}$, we obtain 
		\begin{equation}\label{equation: second term base}
			\begin{aligned}
				\paren[\bigg]{\sum_{f\in\cF}\sum_{\beta\colon f\bijection e}\frac{\Phi_{\frakc,\psi}}{\aut(\cF)H^*\Phi_{\frakc|e,\psi}}\Phi_{\frakc|[\beta],\psi}}-\frac{\abs{\cF}}{H}\Phihat_{\frakc,\psi}\hspace{-8cm}\\
				&=\paren[\bigg]{\sum_{f\in\cF}\sum_{\beta\colon f\bijection e}\frac{\Phi_{\frakc,\psi}}{\aut(\cF)H^*\Phi_{\frakc|e,\psi}}X_{\frakc|[\beta],\psi}}+\paren[\bigg]{\sum_{f\in\cF}\sum_{\beta\colon f\bijection e}\frac{\Phi_{\frakc,\psi}}{\aut(\cF)H^*\Phi_{\frakc|e,\psi}}\Phihat_{\frakc|[\beta],\psi}}-\frac{\abs{\cF}}{H}\Phihat_{\frakc,\psi}\\
				&=\paren[\bigg]{\sum_{f\in\cF}\sum_{\beta\colon f\bijection e}\frac{\Phi_{\frakc,\psi}}{\aut(\cF)H^*\Phi_{\frakc|e,\psi}}X_{\frakc|[\beta],\psi}}+\frac{\abs{\cF}k!\,\phihat_{\cF,f'}}{\aut(\cF)H^*}\Phi_{\frakc,\psi}-\frac{\abs{\cF}}{H}\Phihat_{\frakc,\psi}
			\end{aligned}
		\end{equation}
		Note that from Lemma~\ref{lemma: ladders deterministic trajectory} together with Lemma~\ref{lemma: edges of H}, for all~$f\in\cF$ and~$\beta\colon f\bijection e$, we obtain
		\begin{equation}\label{equation: second term step one}
			\frac{\Phi_{\frakc,\psi}}{\aut(\cF)H^*\Phi_{\frakc|e,\psi}}X_{\frakc|[\beta],\psi}
			\Xeq (1\pm 4\delta^5)\frac{\phihat_{\frakc,I}}{\aut(\cF)H^*\phihat_{\frakc|e,I}}X_{\frakc|[\beta],\psi}
			\Xeq (1\pm 5\delta^5)\frac{\phihat_{\frakc,I}}{k!\,H\phihat_{\frakc|[\beta],I}}X_{\frakc|[\beta],\psi}
		\end{equation}
		and that again Lemma~\ref{lemma: edges of H} together with Lemma~\ref{lemma: ladders deterministic trajectory} yields
		\begin{equation}\label{equation: second term step two}
			\frac{\abs{\cF}k!\,\phihat_{\cF,f'}}{\aut(\cF)H^*}\Phi_{\frakc,\psi}-\frac{\abs{\cF}}{H}\Phihat_{\frakc,\psi}
			\Xeq (1\pm\zeta^{1+\eps^4})\frac{\abs{\cF}}{H}\Phi_{\frakc,\psi}-\frac{\abs{\cF}}{H}\Phihat_{\frakc,\psi}
			\Xeq \frac{\abs{\cF}}{H}X_{\frakc,\psi}\pm\zeta^{1+\eps^5}\frac{\phihat_{\frakc,I}}{H}.
		\end{equation}
		From~\eqref{equation: second term base}, using~\eqref{equation: second term step one} and~\eqref{equation: second term step two} as well as the fact that~$X_{\frakc|[\beta],\psi}\Xleq \delta^{-1}\zeta\phihat_{\frakc|[\beta],I}$, we obtain
		\begin{equation*}
			\begin{aligned}
				\paren[\bigg]{\sum_{f\in\cF}\sum_{\beta\colon f\bijection e}\frac{\Phi_{\frakc,\psi}}{\aut(\cF)H^*\Phi_{\frakc|e,\psi}}\Phi_{\frakc|[\beta],\psi}}-\frac{\abs{\cF}}{H}\Phihat_{\frakc,\psi}\hspace{-4.25cm}\\
				&\Xeq (1\pm 5\delta^{5})\paren[\bigg]{\sum_{f\in\cF}\sum_{\beta\colon f\bijection e}\frac{\phihat_{\frakc,I}}{k!\,H\phihat_{\frakc|[\beta],I}}X_{\frakc|[\beta],\psi}}+\frac{\abs{\cF}}{H}X_{\frakc,\psi}\pm\zeta^{1+\eps^5}\frac{\phihat_{\frakc,I}}{H}\\
				&\Xeq \paren[\bigg]{\sum_{\frakb\in \frakB_\frakc^e}\frac{\phihat_{\frakc,I}}{k!\,H\phihat_{\frakb,I}}X_{\frakb,\psi}}+\frac{\abs{\cF}}{H}X_{\frakc,\psi}\pm \delta^4\frac{\xi_1}{H}.
			\end{aligned}
		\end{equation*}
		Thus, for the second term we have
		\begin{equation}\label{equation: chain deviation second term}
			\begin{aligned}
				-\sum_{e\in\cG_\frakc\setminus J_\frakc}\paren[\bigg]{\paren[\bigg]{ \sum_{f\in\cF}\sum_{\beta\colon f\bijection e}\frac{\Phi_{\frakc,\psi}}{\aut(\cF)H^*\Phi_{\frakc|e,\psi}}\Phi_{\frakc|[\beta],\psi}} -\frac{\abs{\cF}}{H}\Phihat_{\frakc,\psi} }\hspace{-7cm}\\
				&\Xeq -\frac{\abs{\cG_\frakc\setminus \set{J_\frakc}}\abs{\cF}}{H}X_{\frakc,\psi}-\paren[\bigg]{ \sum_{e\in\cG_\frakc\setminus \set{J_\frakc}}\sum_{\frakb\in \frakB_\frakc^e}\frac{\phihat_{\frakc,I}}{k!\,H\phihat_{\frakb,I}}X_{\frakb,\psi} }\pm \delta^3\frac{\xi_1}{H}.
			\end{aligned}
		\end{equation}
		Since~$\abs{\cC_\frakc}=\abs{\cC_{\frakc|\om}}+\abs{\cG_\frakc\setminus\set{J_\frakc}}$, combining~\eqref{equation: chain deviation two terms} with~\eqref{equation: chain deviation first term} and~\eqref{equation: chain deviation second term} completes the proof.
	\end{proof}
	
	\begin{lemma}\label{lemma: ladder trend}
		Let~$0\leq i_0\leq i$ and~$\pom\in\set{-,+}$.
		Then,~$\exi{\Delta Z^\pom_{i_0}}\leq 0$.
	\end{lemma}
	
	\begin{proof}
		Suppose that~$i<i^\star$ and let~$\cX:=\set{i<\tau^\pom_{i_0}\wedge \tautilde_{\frakC}^\star}$.
		We have~$\exi{\Delta Z_{i_0}^\pom}=_{\cX^\comp}0$ and~$\exi{\Delta Z_{i_0}^\pom}\Xeq\exi{\Delta Y^\pom}$, so it suffices to obtain~$\exi{\Delta Y^\pom}\Xleq 0$.
		From Lemma~\ref{lemma: ladder deviation trend}, we obtain
		\begin{equation*}
			\begin{aligned}
				\exi{\Delta(\pom X_{\frakc,\psi})}
				&\Xleq -\frac{(\abs{\cC_\frakc}-1)\abs{\cF}}{H}(\pom X_{\frakc,\psi}) -\paren[\bigg]{ \sum_{e\in\cG_\frakc\setminus\set{J_\frakc}} \sum_{\frakb\in\frakB_\frakc^e} \frac{\phihat_{\frakc,I}}{k!\,H\phihat_{\frakb,I}}(\pom X_{\frakb,\psi}) }+\delta^2\frac{\xi_1}{H}\\
				&\leq -\frac{\abs{\cF}}{H}\paren[\bigg]{ (\abs{\cC_\frakc}-1)(\pom X_{\frakc,\psi}) -\paren[\bigg]{ \sum_{e\in\cG_\frakc\setminus\set{J_\frakc}} \sum_{\frakb\in\frakB_\frakc^e} \frac{\phihat_{\frakc,I}}{\abs{\cF}k!\,\phihat_{\frakb,I}}(\pom X_{\frakb,\psi}) }-\delta^2\xi_1}.
			\end{aligned}
		\end{equation*}
		Note that for all~$e\in\cG_\frakc\setminus\set{J_\frakc}$ and~$\frakb_1,\frakb_2\in\frakB_\frakc^e$, we have~$\phihat_{\frakb_1,I}=\phihat_{\frakb_2,I}$, so we may choose~$\phihat_{\frakc,I}^{e}$ such that~$\phihat_{\frakc,I}^e=\phihat_{\frakb,I}$ for all~$\frakb\in\frakB_\frakc^e$.
		With Lemma~\ref{lemma: chain size}, using that~$\pom X_{\frakc,\psi}\Xgeq (1-\delta)\xi_1$ as well as
		\begin{equation*}
			\abs[\Big]{\sum_{\frakb\in\frakB_\frakc^e} X_{\frakb,\psi}}
			\Xleq \sum_{\frakb\in\frakB_\frakc^e}\delta^{-1/2}\zeta\phihat_{\frakb,\psi},
		\end{equation*}
		we obtain
		\begin{equation*}
			\begin{aligned}
				\exi{\Delta(\pom X_{\frakc,\psi})}
				&\leq -\frac{\abs{\cF}}{H}\paren[\bigg]{ (\abs{\cC_\frakc}-1)(\pom X_{\frakc,\psi}) -\frac{1}{\abs{\cF}k!}\paren[\bigg]{ \sum_{e\in\cG_\frakc\setminus\set{J_\frakc}}\frac{\phihat_{\frakc,I}}{\phihat_{\frakc,I}^e} \sum_{\frakb\in\frakB_\frakc^e} \pom X_{\frakb,\psi} }-\delta^2\xi_1}\\
				&\Xleq -\frac{\abs{\cF}}{H}\paren[\bigg]{ (\abs{\cC_\frakc}-1)(1-\delta)\xi_1 -\frac{1}{\abs{\cF}k!}\paren[\bigg]{ \sum_{e\in\cG_\frakc\setminus\set{J_\frakc}}\frac{\phihat_{\frakc,I}}{\phihat_{\frakc,I}^e} \sum_{\frakb\in\frakB_\frakc^e} \delta^{-1/2}\zeta\phihat_{\frakb,I} }-\delta^2\xi_1}\\
				&\leq -\frac{\abs{\cF}}{H}\paren[\bigg]{ (\abs{\cC_\frakc}-1)\xi_1 -\frac{1}{\abs{\cF}k!}\paren[\bigg]{ \sum_{e\in\cG_\frakc\setminus\set{J_\frakc}} \sum_{\frakb\in\frakB_\frakc^e} \delta^{-1/2}\zeta\phihat_{\frakc,I} }-\eps\xi_1}\\
				&=-\frac{\abs{\cF}}{H}\paren{ (\abs{\cC_\frakc}-1)\xi_1 -\delta^{1/2}\abs{\cG_\frakc\setminus\set{J_\frakc}}\xi_1-\eps\xi_1}
				\leq -\frac{\abs{\cF}}{H}\paren{ (\abs{\cC_\frakc}-1)\xi_1 -\eps^{1/2}\xi_1}.
			\end{aligned}
		\end{equation*}
		Thus, due to Lemma~\ref{lemma: delta phihat G J xi ladder}, we have
		\begin{equation*}
			\exi{\Delta Y^\pom}
			\Xleq -\frac{\abs{\cF}}{H}\paren[\bigg]{ \frac{\rho_\cF}{2}\xi_1 -\eps^{1/3}\xi_1}
			\leq 0,
		\end{equation*}
		which completes the proof.
	\end{proof}
	
	\subsubsection{Boundedness}\label{subsubsection: chain boundedness}
	Here, we first obtain suitable bounds for the absolute one-step changes of the processes~$Y^\pom(0),Y^\pom(1),\ldots$ and~$Z_{i_0}^\pom(i_0),Z_{i_0}^\pom(i_0+1),\ldots$ (see Lemma~\ref{lemma: absolute change ladder not stopped} and Lemma~\ref{lemma: absolute change ladder}) as well as for the expected absolute one-step changes of the second process (see Lemma~\ref{lemma: expected change ladder}).
	
	\begin{lemma}\label{lemma: absolute change chain deviation}
		Let~$0\leq i_0\leq i\leq i^\star$,~$\pom\in\set{-,+}$ and~$\cX:=\set{i<\tau_\ccB\wedge\tau_{\ccB'}\wedge \tau_{\frakC}}$.
		Then,
		\begin{equation*}
			\abs{\Delta X_{\frakc,\psi}}\Xleq n^{\eps^4}\frac{\phihat_{\frakc,I}(i_0)}{n\phat(i_0)^{\rho_\cF}}.
		\end{equation*}
	\end{lemma}
	\begin{proof}
		For all~$(\cA,I)\subseteq(\cC_\frakc,I)$ with~$V_\cA\neq I$, Lemma~\ref{lemma: ladder subextension density} together with Lemma~\ref{lemma: bounds of phat} implies
		\begin{equation*}
			\phihat_{\cA,I}
			\geq (n\phat^{\rho_\cF})^{\abs{\cA}-\abs{\cA[I]}}
			\geq n\phat^{\rho_\cF}.
		\end{equation*}
		Hence, due to Lemma~\ref{lemma: chain size}, Lemma~\ref{lemma: loss at one edge} together with Lemma~\ref{lemma: chains are not trivial} implies
		\begin{equation*}
			\abs{\Delta\Phi_{\frakc,\psi}}
			\Xleq \abs{\cC_\frakc}\cdot 2k!\abs{\cF}(\log n)^{\alpha_{\cC_\frakc,I}}\frac{\phihat_{\frakc,I}}{n\phat^{\rho_\cF}}
			\leq n^{\eps^5}\frac{\phihat_{\frakc,I}}{n\phat^{\rho_\cF}}
			\leq n^{\eps^5}\frac{\phihat_{\frakc,I}(i_0)}{n\phat(i_0)^{\rho_\cF}}.
		\end{equation*}
		Similarly, we obtain
		\begin{equation*}
			\abs{\Delta\Phi_{\frakc|\om,\psi}}\Xleq n^{\eps^5}\frac{\phihat_{\frakc|\om,I}(i_0)}{n\phat(i_0)^{\rho_\cF}}.
		\end{equation*}
		With Lemma~\ref{lemma: delta phihat}, Lemma~\ref{lemma: ladders deterministic trajectory} and Lemma~\ref{lemma: phihat G J next}, using Lemma~\ref{lemma: chains are not trivial}, we conclude that
		\begin{equation*}
			\begin{aligned}
				\abs{\Delta  X_{\frakc,\psi}}
				&\leq \abs{\Delta\Phi_{\frakc,\psi}}+\phihat_{\cG_\frakc,J_\frakc}(i+1)\abs{\Delta \Phi_{\frakc|\om,\psi}}+\abs{\Delta\phihat_{\cG_\frakc,J_\frakc}}\Phi_{\frakc|\om,\psi}\\
				&\leq \abs{\Delta\Phi_{\frakc,\psi}}+2\phihat_{\cG_\frakc,J_\frakc}\abs{\Delta \Phi_{\frakc|\om,\psi}}+2\abs{\cF}^2\frac{\phihat_{\cG_\frakc,J_\frakc}\Phi_{\frakc|\om,\psi}}{H}\\
				&\Xleq n^{\eps^5}\frac{\phihat_{\frakc,I}(i_0)}{n\phat(i_0)^{\rho_\cF}}+2n^{\eps^5}\frac{\phihat_{\cG_\frakc,J_\frakc}\phihat_{\frakc|\om,I}(i_0)}{n\phat(i_0)^{\rho_\cF}}+4\abs{\cF}^2\frac{\phihat_{\frakc,I}}{H}\\
				&\leq n^{\eps^5}\frac{\phihat_{\frakc,I}(i_0)}{n\phat(i_0)^{\rho_\cF}}+2n^{\eps^5}\frac{\phihat_{\frakc,I}(i_0)}{n\phat(i_0)^{\rho_\cF}}+4\abs{\cF}^2\frac{\phihat_{\frakc,I}(i_0)}{H(i_0)}.
			\end{aligned}
		\end{equation*}
		With Lemma~\ref{lemma: zeta and H}, this completes the proof.
	\end{proof}

	\begin{lemma}\label{lemma: absolute change ladder not stopped}
		Let~$0\leq i_0\leq i\leq i^\star$,~$\pom\in\set{-,+}$ and~$\cX:=\set{i<\tau_\ccB\wedge\tau_{\ccB'}\wedge\tau_\frakC}$.
		Then,
		\begin{equation*}
			\abs{\Delta Y^\pom}\leq n^{\eps^3}\frac{\phihat_{\frakc,I}(i_0)}{n\phat(i_0)^{\rho_\cF}}.
		\end{equation*}
	\end{lemma}
	
	\begin{proof}
		Combining Lemma~\ref{lemma: delta phihat G J xi ladder} and Lemma~\ref{lemma: absolute change chain deviation}, using Lemma~\ref{lemma: chains are not trivial}, we obtain
		\begin{equation*}
			\abs{\Delta Y^\pom}
			\leq \abs{\Delta  X_{\frakc,\psi}} + \abs{\Delta\xi_1}
			\leq n^{\eps^4}\frac{\phihat_{\frakc,I}(i_0)}{n\phat(i_0)^{\rho_\cF}}+\frac{\phihat_{\frakc,I}}{H}
			\leq n^{\eps^4}\frac{\phihat_{\frakc,I}(i_0)}{n\phat(i_0)^{\rho_\cF}}+\frac{\phihat_{\frakc,I}(i_0)}{H(i_0)}.
		\end{equation*}
		With Lemma~\ref{lemma: zeta and H}, this completes the proof.
	\end{proof}
	
	\begin{lemma}\label{lemma: absolute change ladder}
		Let~$0\leq i_0\leq i\leq i^\star$ and~$\pom\in\set{-,+}$.
		Then,
		\begin{equation*}
			\abs{\Delta Z^\pom_{i_0}}\leq n^{\eps^3}\frac{\phihat_{\frakc,I}(i_0)}{n\phat(i_0)^{\rho_\cF}}.
		\end{equation*}
	\end{lemma}
	
	\begin{proof}
		This is an immediate consequence of Lemma~\ref{lemma: absolute change ladder not stopped}.
	\end{proof}
	
	\begin{lemma}\label{lemma: expected change chain deviation}
		Let~$0\leq i\leq i^\star$,~$\pom\in\set{-,+}$ and~$\cX:=\set{ i<\tau_{\cH^*}\wedge\tau_\ccB\wedge\tau_{\ccB'}\wedge\tau_{\frakC} }$.
		Then,
		\begin{equation*}
			\exi{\abs{ \Delta X_{\frakc,\psi} }}\Xleq n^{\eps^4}\frac{\phihat_{\frakc,I}}{n^k\phat}.
		\end{equation*}
	\end{lemma}
	\begin{proof}
		With Lemma~\ref{lemma: delta phihat}, Lemma~\ref{lemma: ladders deterministic trajectory} and Lemma~\ref{lemma: phihat G J next}, we obtain
		\begin{align*}
			\exi{\abs{ \Delta X_{\frakc,\psi} }}
			&\leq \exi{\abs{ \Delta\Phi_{\frakc,\psi}}}+\phihat_{\cG_\frakc,J_\frakc}(i+1)\exi{\abs{ \Delta\Phi_{\frakc|\om,\psi}}}+\abs{ \Delta\phihat_{\cG_\frakc,J_\frakc} }\Phi_{\frakc|\om,\psi}\\
			&\leq \exi{\abs{ \Delta\Phi_{\frakc,\psi}}}+2\phihat_{\cG_\frakc,J_\frakc}\exi{\abs{ \Delta\Phi_{\frakc|\om,\psi}}}+2\abs{\cF}^2\frac{\phihat_{\frakc,I}}{H}
		\end{align*}
		Thus, due to Lemma~\ref{lemma: edges of H}, it suffices to obtain
		\begin{equation*}
			\exi{\abs{ \Delta\Phi_{\frakc,\psi}}}\Xleq n^{\eps^5}\frac{\phihat_{\frakc,I}}{n^k\phat}\qtand
			\exi{\abs{ \Delta\Phi_{\frakc|\om,\psi}}}\Xleq n^{\eps^5}\frac{\phihat_{\frakc|\om,I}}{n^k\phat}.
		\end{equation*}
		To this end, for~$e\in \cC_\frakc\setminus\cC_\frakc[I]$, from all subtemplates~$(\cA,I)\subseteq (\cC_\frakc,I)$ with~$e\in\cA$, choose~$(\cA_e,I)$ such that~$\phihat_{\cA_e,I}$ is minimal.
		Furthermore, for every subtemplate~$(\cA,I)\subseteq (\cC_\frakc,I)$, let
		\begin{equation*}
			\Phi_{\cA,\psi}^e:=\abs{\cset{ \phi\in\Phi_{\cA,\psi}^\sim }{ \phi(e)\in\cF_0(i+1) }}.
		\end{equation*}
		Then, due to Lemma~\ref{lemma: chain size}, Lemma~\ref{lemma: loss at one edge} yields
		\begin{equation*}
			\Phi_{\cC_\frakc,\psi}^e\Xleq 2k!\abs{\cF}(\log n)^{\alpha_{\cC_\frakc,I\cup e}}\frac{\phihat_{\frakc,I}}{\phihat_{\cA_e,I}},
		\end{equation*}
		so we obtain
		\begin{equation}\label{equation: chain change}
			\begin{aligned}
				\abs{ \Delta\Phi_{\frakc,\psi}}
				&\leq \sum_{e\in \cC_\frakc\setminus \cC_\frakc[I]} \Phi_{\cC_\frakc,\psi}^e
				=\sum_{e\in \cC_\frakc\setminus \cC_\frakc[I]} \ind_{\set{ \Phi_{\cC_\frakc,\psi}^e\geq 1 }}\Phi_{\cC_\frakc,\psi}^e\\
				&\Xleq 2k!\abs{\cF}(\log n)^{\alpha_{\cC_\frakc,I\cup e}}\phihat_{\frakc,I}\sum_{e\in \cC_\frakc\setminus \cC_\frakc[I]} \frac{\ind_{\set{ \Phi_{\cC_\frakc,\psi}^e\geq 1 }}}{\phihat_{\cA_e,I}}
				\leq n^{\eps^6}\phihat_{\frakc,I}\sum_{e\in \cC_\frakc\setminus \cC_\frakc[I]} \frac{\ind_{\set{ \Phi_{\cC_\frakc,\psi}^e\geq 1 }}}{\phihat_{\cA_e,I}}\\
				&\leq n^{\eps^6}\phihat_{\frakc,I}\sum_{e\in \cC_\frakc\setminus \cC_\frakc[I]} \frac{\ind_{\set{ \Phi_{\cA_e,\psi}^e\geq 1 }}}{\phihat_{\cA_e,I}}
				\leq n^{\eps^6}\phihat_{\frakc,I}\sum_{e\in \cC_\frakc\setminus \cC_\frakc[I]} \sum_{\phi\in\Phi^\sim_{\cA_e,\psi}}\frac{\ind_{\set{ \phi(e)\in \cF_0(i+1)  }}}{\phihat_{\cA_e,I}}.
			\end{aligned}
		\end{equation}
		For all~$e\in\cH$,~$f\in\cF$ and~$\psi'\colon f\bijection e$, we have~$\Phi_{\cF,\psi'}\Xeq (1\pm\delta^{-1}\zeta)\phihat_{\cF,f}$.
		Furthermore, we have~$H^*\Xeq (1\pm \zeta^{1+\eps^3})\hhat^*$.
		Thus, using Lemma~\ref{lemma: star degrees}, for all~$e\in\cC_\frakc\setminus\cC_\frakc[I]$ and~$\phi\in\Phi_{\cA_e,\psi}^e$, we obtain
		\begin{equation*}
			\pri{ \phi(e)\in \cF_0(i+1)  }
			=\frac{d_{\cH^*}(\phi(e))}{H^*}
			\Xleq \frac{2\abs{\cF}k!\,\phihat_{\cF,f}}{H^*}
			\Xleq \frac{4\abs{\cF}k!\,\phihat_{\cF,f}}{\hhat^*}
			\leq \frac{n^{\eps^6}}{n^k\phat}.
		\end{equation*}
		Combining this with~\eqref{equation: chain change} yields
		\begin{equation*}
			\exi{\abs{ \Delta\Phi_{\frakc,\psi}}}
			\Xleq n^{\eps^6}\phihat_{\frakc,I}\sum_{e\in \cC_\frakc\setminus \cC_\frakc[I]} \sum_{\phi\in\Phi^\sim_{\cA_e,\psi}}\frac{\pri{ \phi(e)\in \cF_0(i+1)  }}{\phihat_{\cA_e,I}}
			\Xleq n^{2\eps^6}\frac{\phihat_{\frakc,I}}{n^k\phat}\sum_{e\in \cC_\frakc\setminus \cC_\frakc[I]}\frac{\Phi_{\cA_e,\psi}}{\phihat_{\cA_e,I}}.
		\end{equation*}
		For all~$e\in\cC_\frakc\setminus\cC_\frakc[I]$ and~$(\cB,I)\subseteq(\cA_e,I)\subseteq (\cC_\frakc,I)$, Lemma~\ref{lemma: ladder subextension density} together with Lemma~\ref{lemma: bounds of phat} entails
		\begin{equation*}
			\phihat_{\cB,I}
			=(n\phat^{\rho_{\cB,I}})^{\abs{V_\cB}-\abs{I}}
			\geq (n\phat^{\rho_{\cF}})^{\abs{V_\cB}-\abs{I}}
			\geq 1
		\end{equation*}
		and so Lemma~\ref{lemma: arbitrary embedding} yields
		\begin{equation*}
			\Phi_{\cA_e,I}\Xleq 2(\log n)^{\alpha_{\cA_e,I}}\phihat_{\cA_e,I}\leq n^{\eps^6}\phihat_{\cA_e,I}.
		\end{equation*}
		We conclude that
		\begin{equation*}
			\exi{\abs{ \Delta\Phi_{\frakc,\psi}}}
			\Xleq n^{3\eps^6}\abs{\cC_\frakc\setminus\cC_\frakc[I]}\frac{\phihat_{\frakc,I}}{n^k\phat}
			\leq n^{4\eps^6}\frac{\phihat_{\frakc,I}}{n^k\phat}.
		\end{equation*}
		Similarly, we obtain
		\begin{equation*}
			\exi{\abs{ \Delta\Phi_{\frakc|\om,\psi}}}
			\Xleq n^{4\eps^6}\frac{\phihat_{\frakc|\om,I}}{n^k\phat},
		\end{equation*}
		which completes the proof.
	\end{proof}
	
	\begin{lemma}\label{lemma: expected change ladder}
		Let~$0\leq i_0\leq i^\star$ and~$\pom\in\set{-,+}$.
		Then,~$\sum_{i\geq i_0} \exi{\abs{ \Delta Z^\pom_{i_0} }}\leq n^{\eps^3}\phihat_{\frakc,I}(i_0)$.
	\end{lemma}
	
	\begin{proof}
		Suppose that~$i_0\leq i<i^\star$ and let~$\cX:=\set{i< \tau_{\cH^*}\wedge\tau_\ccB\wedge\tau_{\ccB'}\wedge\tau_{\frakC}}$.
		We have~$\exi{\abs{\Delta Z_{i_0}^\pom}}=_{\cX^\comp}0$ and with Lemma~\ref{lemma: delta phihat G J xi ladder}, Lemma~\ref{lemma: expected change chain deviation} and Lemma~\ref{lemma: edges of H}, using Lemma~\ref{lemma: chains are not trivial}, we obtain
		\begin{align*}
			\exi{\abs{ \Delta Z_{i_0}^\pom }}
			&\leq \exi{\abs{ \Delta Y^\pom }}
			\leq \exi{\abs{ \Delta X_{\frakc,\psi} }}+\abs{\Delta\xi_1}
			\Xleq n^{\eps^4}\frac{\phihat_{\frakc,I}}{n^k\phat}+\frac{\phihat_{\frakc,I}}{H}
			\Xleq n^{\eps^3}\frac{\phihat_{\frakc,I}}{n^k\phat}
			\leq n^{\eps^3}\frac{\phihat_{\frakc,I}(i_0)}{n^k\phat(i_0)}.
		\end{align*}
		Thus,
		\begin{equation*}
			\sum_{i\geq i_0} \exi{\abs{ \Delta Z^\pom_{i_0} }}
			= \sum_{i_0\leq i\leq i^\star-1} \exi{\abs{ \Delta Z^\pom_{i_0} }}
			\leq (i^\star-i_0) \frac{n^{\eps^3} \phihat_{\frakc,I}(i_0)}{n^k\phat(i_0)}.
		\end{equation*}
		Since
		\begin{equation*}
			i^\star-i_0
			\leq \frac{\theta n^k}{\abs{\cF}k!}-i_0
			=\frac{n^k\phat(i_0)}{\abs{\cF}k!}
			\leq n^k\phat(i_0),
		\end{equation*}
		this completes the proof.
	\end{proof}

	\subsubsection{Supermartingale argument}\label{subsubsection: chain concentration}
	In this section, we obtain the final ingredient for our application of Lemma~\ref{lemma: freedman} and subsequently show that the probabilities of the events on the right in Observation~\ref{observation: ladder critical times} are indeed small.
	
	In more detail, we first prove Lemma~\ref{lemma: initial error ladder} that states that for all~$\pom\in\set{-,+}$, at time~$i=\sigma^\pom$ where the process~$\Phi_{\frakc,\psi}(0),\Phi_{\frakc,\psi}(1),\ldots$ just left the non-critical interval between the critical intervals, it cannot have jumped over the critical interval~$I^\pom$.
	Then, we combine this insight with the results form the previous two sections to apply Lemma~\ref{lemma: freedman} in the proof of Lemma~\ref{lemma: control ladder}.
	
	\begin{lemma}\label{lemma: initial error ladder}
		Let~$\pom\in\set{-,+}$.
		Then,~$Z^\pom_{\sigma^\pom}(\sigma^\pom)\leq -\delta^2\xi_1(\sigma^\pom)$.
	\end{lemma}
	
	\begin{proof}
		Together with Lemma~\ref{lemma: ladder subextension density}, Lemma~\ref{lemma: initially good} implies~$\tautilde_{\frakC}^\star\geq 1$ and~$\pom X_{\frakc,\psi}(0)<\xi_0(0)$, so we have~$\sigma^\pom\geq 1$.
		Thus, by definition of~$\sigma^\pom$, for~$i:=\sigma^\pom-1$, we have~$\pom X_{\frakc,\psi}\leq \xi_0$ and thus
		\begin{equation*}
			Z_i^\pom=\pom X_{\frakc,\psi}-\xi_1\leq -\delta \xi_1.
		\end{equation*}
		Furthermore, since~$\sigma^\pom\leq \tau_{\ccB}\wedge \tau_{\ccB'}\wedge\tau_\frakC$, we may apply Lemma~\ref{lemma: absolute change ladder not stopped} to obtain
		\begin{equation*}
			Z_{\sigma^\pom}^\pom(\sigma^\pom)
			=Z_i^\pom +\Delta Y^\pom
			\leq Z_i^\pom + \delta^2 \xi_1
			\leq -\delta \xi_1 + \delta^2 \xi_1
			\leq -\delta^2\xi_1.
		\end{equation*}
		Since Lemma~\ref{lemma: chains are not trivial} entails~$\Delta\xi_1\leq 0$, this completes the proof.
	\end{proof}
	
	\begin{lemma}\label{lemma: control ladder}
		$\pr{\tau_{\frakC}\leq \tautilde_\frakC^\star\wedge i^\star}\leq \exp(-n^{\eps^3})$.
	\end{lemma}
	
	\begin{proof}
		Considering Observation~\ref{observation: ladder individual}, it suffices to show that
		\begin{equation*}
			\pr{\tau\leq \tautilde_\frakC^\star\wedge i^\star}\leq \exp(-n^{2\eps^3}).
		\end{equation*}
		Hence, by Observation~\ref{observation: ladder critical times}, is suffices to show that for~$\pom\in\set{-,+}$, we have
		\begin{equation*}
			\pr{ Z^\pom_{\sigma^\pom}(i^\star)>0 }\leq \exp(-n^{3\eps^3}).
		\end{equation*}
		Due to Lemma~\ref{lemma: initial error ladder}, we have
		\begin{equation*}
			\pr{ Z_{\sigma^\pom}^\pom(i^\star)>0 }
			\leq \pr{ Z_{\sigma^\pom}^\pom(i^\star) -Z_{\sigma^\pom}^\pom(\sigma^\pom) > \delta^2\xi_1(\sigma^\pom) }
			\leq \sum_{0\leq i\leq i^\star}\pr{ Z_{i}^\pom(i^\star) -Z_{i}^\pom > \delta^2\xi_1 }.
		\end{equation*}
		Thus, for~$0\leq i\leq i^\star$, it suffices to obtain
		\begin{equation*}
			\pr{ Z_{i}^\pom(i^\star) -Z_{i}^\pom > \delta^2\xi_1 }\leq \exp(-n^{4\eps^3}).
		\end{equation*}
		We show that this bound is a consequence of We show that this bound is a consequence of Freedman's inequality for supermartingales.
		
		Let us turn to the details.
		Lemma~\ref{lemma: ladder trend} shows that~$Z^\pom_i(i),Z^\pom_i(i+1),\ldots$ is a supermartingale, while Lemma~\ref{lemma: absolute change ladder} provides the bound~$\abs{\Delta Z^\pom_i(j)}\leq n^{\eps^3}\phihat_{\frakc,I}/(n\phat^{\rho_\cF})$ for all~$j\geq i$ and Lemma~\ref{lemma: expected change ladder} provides the bound~$\sum_{j\geq i} \ex[][\bE_j]{ \abs{\Delta Z^\pom_i(j)} }\leq n^{\eps^3}\phihat_{\frakc,I}$.
		Hence, we may apply Lemma~\ref{lemma: freedman} to obtain
		\begin{align*}
			\pr{ Z_{i}^\pom(i^\star) -Z_{i}^\pom > \delta^2\xi_1 }
			&\leq \exp\paren[\bigg]{ -\frac{\delta^4\xi_1^2}{2n^{\eps^3}\frac{\phihat_{\frakc,I}}{n\phat^{\rho_\cF}}(\delta^2\xi_1+n^{\eps^3}\phihat_{\frakc,I})} }
			\leq \exp\paren[\bigg]{ -\frac{\delta^4\xi_1^2n\phat^{\rho_\cF}}{4n^{2\eps^3}\phihat_{\frakc,I}^2}}\\
			&= \exp\paren[\bigg]{ -\frac{\delta^2 n^{2\eps^2}}{4n^{2\eps^3}}}
			\leq \exp(-n^{4\eps^3}),
		\end{align*}
		which completes the proof.
	\end{proof}

	\section{Branching families}\label{section: branching families}
	
	This section is dedicated to introducing and analyzing the special setup based on branching families that we rely on for exploiting the self-correcting behavior of the process.
	Suppose that~$0\leq i\leq i^\star$, consider a chain~$\frakc=(F,V,I)\in\frakC$ and~$\psi\colon I\injection V_\cH$.
	As suggested by our definition of~$\tautilde_\frakB$, we wish to show that~$\sum_{\frakb\in\frakB_\frakc^e}\Phi_{\cC_\frakb,\psi}$ is typically close to~$\sum_{\frakb\in\frakB_\frakc^e}\Phihat_{\frakb,\psi}$, however, instead of choosing~$\delta^{-1/2}\zeta\phihat_{\cC_\frakb,I}$ as the error term that quantifies the deviation that we allow, we use~$\eps^{-\chi_{\frakB_\frakc^e}}\zeta\phihat_{\cC_\frakb,I}$ for a carefully chosen \emph{error parameter}~$\chi_{\frakB_\frakc^e}$ that crucially depends on the branching family~$\frakB_\frakc^e$.
	
	Considering branching families instead of individual chains and using different error terms for different branching families allows us to overcome the following obstacles that we encounter when attempting to exploit self-correcting behavior.
	When we analyze the expected one-step changes of~$\Phi_{\frakc,\psi}$ for a chain~$\frakc=(F,V,I)\in\frakC$ and~$\psi\colon I\injection V_\cH$ using Lemma~\ref{lemma: ladder averaging}, different chains besides~$\frakc$ itself play an important role and their behavior could undermine the self-correcting drift that would naturally steer~$\Phi_{\frakc,\psi}$ closer to the anticipated trajectory whenever it deviates.
	In an attempt to control this we might want to allow only significantly smaller deviations for these other chains such that the self-correcting drift still dominates.
	This approach leads to the desire to implement a hierarchy of error terms such that the error terms of other chains that appear as transformations of~$\frakc$ are negligible.
	If~$\cF$ is not symmetric, on the level of individual chains, necessary negligibility may form cyclic dependencies that make it impossible to find such a hierarchy.
	However, since relevant other chains that appear as transformations always appear in groups, analyzing these groups instead allows us to reduce the aforementioned directed cyclic structures to loops such that on the level of branching families, such a hierarchic approach is feasible.
	
	In Section~\ref{subsection: error parameter}, we discuss the careful choice of error parameters.
	In Section~\ref{subsection: tracking families}, we subsequently employ supermartingale concentration techniques that exploit the self-correcting behavior to show that branching families typically behave as expected such that our dependence on the stopping time~$\tautilde_\frakB$ in Section~\ref{subsection: tracking chains} is justified. 
	
	\subsection{Error parameter}\label{subsection: error parameter}
	This section is dedicated to providing and analyzing appropriate choices for the error parameters mentioned in the beginning of Section~\ref{section: branching families}.
	To this end, we introduce the following concepts.
	For a sequence~$F=\cF_1,\ldots,\cF_\ell$ of copies of~$\cF$, we define
	\begin{equation*}
		\chi_F:=-\eps^{-5k(k+1)}\sum_{1\leq i\leq \ell-1} \eps^{5k\abs{V_{\cF_{i}}\cap V_{\cF_{i+1}}}}.
	\end{equation*}
	For a chain~$\frakc=(F,V,I)$, we say that a subsequence~$F'=\cF_1,\ldots,\cF_\ell$ of~$F$ is~\emph{$\frakc$-sufficient} if~$(\cF_1+\ldots+\cF_\ell)[V]=\cC_\frakc$ and we say that~$F'$ is \emph{minimally~$\frakc$-sufficient} if~$F'$ is~$\frakc$-sufficient while no proper subsequence of~$F'$ is~$\frakc$-sufficient.
	The \emph{error parameter} of~$\frakc$ is
	\begin{equation*}
		\chi_\frakc:=\abs{V}+\min_{ F'\colon \text{$F'$ is minimally~$\frakc$-sufficient}  } \chi_{F'}.
	\end{equation*}
	We observe that for all~$e\in\cC_\frakc\setminus\cC_\frakc[I]$, all error parameters of branchings~$\frakb\in\frakB_\frakc^e$ are equal (see Lemma~\ref{lemma: same error parameter}), which we obtain as a consequence of the following observation.
	
	\begin{observation}\label{observation: sufficient for other branching}
		Suppose that~$\frakc=(F,V,I)$ is a chain and suppose that~$e\in\cC_\frakc\setminus\cC_\frakc[I]$.
		Let~$\frakb,\frakb'\in\frakB_\frakc^e$.
		Suppose that~$\cF_1,\ldots,\cF_\ell$ is~$\frakb$-sufficient and that~$\cF_\ell'$ is the last element in the first component of~$\frakb'$.
		Then,~$\cF_1,\ldots,\cF_{\ell-1},\cF_\ell'$ is~$\frakb'$ sufficient.
	\end{observation}
	
	\begin{lemma}\label{lemma: same error parameter}
		Suppose that~$\frakc=(F,V,I)$ is a chain and suppose that~$e\in\cC_\frakc\setminus\cC_\frakc[I]$.
		Let~$\frakb,\frakb'\in\frakB_\frakc^e$.
		Then,~$\chi_{\frakb}=\chi_{\frakb'}$.
	\end{lemma}
	\begin{proof}
		Suppose that~$F=\cF_1,\ldots,\cF_\ell$ is minimally~$\frakb$-sufficient.
		Due to symmetry, it suffices to show that there exists a minimally~$\frakb'$-sufficient sequence~$F'$ with~$\chi_{F'}= \chi_F$.
		Suppose that~$\cF'$ is the last element in the first component of~$\frakb'$ and let~$F':=\cF_1,\ldots,\cF_{\ell-1},\cF'$.
		By Observation~\ref{observation: sufficient for other branching}, the sequence~$F'$ is~$\frakb'$-sufficient.
		Furthermore, for every~$\frakb'$-sufficient subsequence of~$F'$, replacing the last element with~$\cF_\ell$ yields a subsequence of~$F$ which again by Lemma~\ref{observation: sufficient for other branching} is~$\frakb$-sufficient.
		Hence, since~$F$ is minimally~$\frakb$-sufficient, the sequence~$F'$ is minimally~$\frakb'$-sufficient.
		Furthermore, we have
		\begin{equation*}
			V_{\cF_{\ell-1}}\cap V_{\cF_\ell}
			=V_{\cF_{\ell-1}}\cap e
			=V_{\cF_{\ell-1}}\cap V_{\cF'}
		\end{equation*}
		and thus~$\chi_{F'}= \chi_F$.
	\end{proof}
	For a chain~$\frakc=(F,V,I)$ and~$e\in\cC_\frakc\setminus\cC_\frakc[I]$, this allows us to choose the error parameter~$\chi_{\frakB_\frakc^e}$ of~$\frakB_\frakc^e$ such that~$\chi_{\frakB_\frakc^e}=\chi_\frakb$ for all~$\frakb\in\frakB_\frakc^e$.
	The key property of our error parameters that we formally state in Lemma~\ref{lemma: error parameter property} is that whenever we consider the branching~$\frakb'$ of a branching~$\frakb$ of a chain~$\frakc\in\frakC$, then~$\chi_{\frakb'}\leq\chi_\frakb-1$ or we are in a situation where the branching families of~$\frakc$ and~$\frakb$ are essentially the same.
	
	To formally state the close relationship between branching families that we encounter whenever the branching of a branching has the same error parameter, we introduce the following term.
	For two chains~$\frakc=(F,V,I)$ and~$\frakc'=(F',V',I')$ and edges~$e\in\cC_\frakc\setminus\cC_\frakc[I]$ and~$e'\in\cC_{\frakc'}\setminus\cC_{\frakc'}[I']$, we say that the branching families~$\frakB_\frakc^e$ and~$\frakB_{\frakc'}^{e'}$ are \emph{template equivalent} if there exists a bijection~$\gamma\colon \frakB_{\frakc}^e\bijection \frakB_{\frakc'}^{e'}$ such that for all~$\frakb\in\frakB_{\frakc}^e$, the chain template~$(\cC_{\frakb},I)$ is a copy of~$(\cC_{\gamma(\frakb)},I')$ while~$(\cC_{\frakb|\om},I)$ is a copy of~$(\cC_{\gamma(\frakb)|\om},I')$.
	We encounter such a close relationship between branching families for example when comparing the branching family of a chain and the branching family of the corresponding support (see Lemma~\ref{lemma: branchings are branchings of support}).
	
	To show that we have template equivalence of relevant branching families, we argue based on a refined notion of copy for templates.
	More specifically, for two templates~$(\cA,I)$ and~$(\cB,J)$ and~$a\in\cA$ and~$b\in\cB$, we say that~$(\cB,J)$ is a copy of~$(\cA,I)$ \emph{with~$b$ playing the role of~$a$} if there exists a bijection~$\phi\colon V_\cA\bijection V_\cB$ with~$\phi(e)\in\cB$ for all~$e\in\cA$,~$\phi^{-1}(e)\in\cA$ for all~$e\in\cB$,~$\phi(I)=J$ and~$\phi(a)=b$.
	Lemma~\ref{lemma: branchings are copy invariant} states the connection between this notion of copy and template equivalence that we rely on.
	
	Lemmas~\ref{lemma: copy vertex set intersections}--\ref{lemma: minimally sufficient length} serve as further preparation for the proof of Lemma~\ref{lemma: error parameter property}.

	\begin{lemma}\label{lemma: branchings are branchings of support}
		Suppose that~$\fraks$ is the~$e$-support of a chain~$\frakc$.
		Then,~$\frakB_{\frakc}^e$ and~$\frakB_{\fraks}^{e}$ are template equivalent.
	\end{lemma}
	\begin{proof}
		Suppose that~$\frakc=(F,V,I)$ where~$F$ has length~$\ell$.
		Let~$\beta\colon f\bijection e$ where~$f\in\cF$.
		We have~$\fraks=\frakc|\beta|\R|\om$, so the chain template given by~$\fraks|\beta$ is a copy of the chain template given by~$\frakc|\beta|\R$.
		Since~$\fraks$ is the~$e$-support of~$\frakc$, we have~$\fraks|\beta|\R=\fraks|\beta$.
		Thus, the chain template given by~$\fraks|\beta|\R$ is a copy of the chain template given by~$\frakc|\beta|\R$.
		Furthermore, we additionally have~$\frakc|\beta|\R|\om=\fraks=\fraks|\beta|\R|\om$ so a bijection~$\gamma\colon \frakB_{\frakc}^e\bijection \frakB_{\fraks}^{e}$ as in the definition of template equivalence exists.
	\end{proof}
	
	\begin{lemma}\label{lemma: branchings are copy invariant}
		Suppose that~$\frakc=(F,V,I)$ is the~$e$-support of a chain.
		Suppose that~$\frakc'=(F',V',I')$ is a chain such that for some~$e'\in\cC_{\frakc'}\setminus\cC_{\frakc'}[I']$, the template~$(\cC_{\frakc'},I')$ is a copy of~$(\cC_{\frakc},I)$ with~$e'$ playing the role of~$e$.
		Then,~$\frakB_{\frakc}^e$ and~$\frakB_{\frakc'}^{e'}$ are template equivalent.
	\end{lemma}
	\begin{proof}
		Suppose that~$\phi\colon V\bijection V'$ is a bijection with~$\phi(e)\in \cC_{\frakc'}$ for all~$e\in\cC_\frakc$ and~$\phi^{-1}(e)\in\cC_{\frakc'}$ for all~$e\in\cC_\frakc$,~$\phi(I)=I'$ and~$\phi(e)=e'$.
		Suppose that~$\frakb\in\frakB_\frakc^e$ where~$\frakb=\frakc|\beta|\R$ for some~$\beta\colon f\bijection e$ where~$f\in\cF$.
		Let~$\beta':=\phi\circ\beta$ and~$\frakb':=\frakc'|\beta'|\R$.
		To see that assigning~$\frakb'$ as the image of~$\frakb$ under a map~$\gamma\colon \frakB_\frakc^e\rightarrow \frakB_{\frakc'}^{e'}$ yields a bijection as desired, it suffices to show that~$(\cC_{\frakb},I)$ is a copy of~$(\cC_{\frakb'},I')$ while~$(\cC_{\frakb|\om},I)$ is a copy of~$(\cC_{\frakb'|\om},I')$.
		
		First, observe that there exists a bijection
		\begin{equation*}
			\phi_+\colon V\cup V_{\cF_\frakc^\beta}\bijection V'\cup V_{\cF_{\frakc'}^{\beta'}}
		\end{equation*}
		with~$\restr{\phi_+}{V}=\phi$ such that~$\phi_+(e)\in \cC_{\frakb'}$ for all~$e\in\cC_\frakb$ and~$\phi_+^{-1}(e)\in\cC_{\frakb'}$ for all~$e\in\cC_\frakb$.
		Hence,~$(\cC_{\frakc|\beta},I)$ is a copy of~$(\cC_{\frakc'|\beta'},I')$.
		Since~$\frakc$ is the~$e$-support of a chain, we have~$\frakc|\beta|\R=\frakc|\beta$ and thus~$\frakc'|\beta'|\R=\frakc'|\beta'$; so~$(\cC_{\frakb},I)$ is a copy of~$(\cC_{\frakb'},I')$.
		Furthermore, we obtain~$\frakb|\om=\frakc$ and~$\frakb'|\om=\frakc'$, which completes the proof.
	\end{proof}

	\begin{lemma}\label{lemma: copy vertex set intersections}
		Suppose that~$\frakc=(I,F,V)$ is a chain with~$F=\cF_1,\ldots,\cF_\ell$.
		For~$1\leq i\leq\ell$, let~$V_i:=V_{\cF_i}$.
		Let~$1\leq i\leq i'\leq j'\leq j\leq\ell$.
		Then,~$V_{i}\cap V_{j}\subseteq V_{i'}\cap V_{j'}$.
	\end{lemma}
	
	\begin{proof}
		Since~$F$ is a vertex-separated loose path, we have~$V_i\cap V_j\subseteq V_{i'}\cap V_j$ and~$V_i\cap V_j\subseteq V_{j'}\cap V_j$.
		Thus,
		\begin{equation*}
			V_i\cap V_j
			\subseteq V_{i'}\cap V_{j'}\cap V_j
			\subseteq V_{i'}\cap V_{j'},
		\end{equation*}
		which completes the proof.
	\end{proof}

	\begin{lemma}\label{lemma: reduced copy vertex set intersections}
		Suppose that~$\frakc=(F,V,I)\in\frakC$ is a chain and that~$\cF_1,\ldots,\cF_\ell$ is minimally~$\frakc$-sufficient.
		For~$1\leq i\leq \ell$, let~$V_i:=V_{\cF_i}$.
		Then, for~$1\leq i\leq j\leq \ell$ where~$i\leq j-2$, we have
		\begin{equation*}
			\abs{V_i\cap V_j}\leq \min_{i\leq i'\leq j-1} \abs{V_{i'}\cap V_{i'+1}}-1.
		\end{equation*}
	\end{lemma}
	
	\begin{proof}
		Let~$i+1\leq i''\leq j-1$ such that
		\begin{equation*}
			\min_{i\leq i'\leq j-1} \abs{V_{i'}\cap V_{i'+1}}=\min_{i''-1\leq i'\leq i''} \abs{V_{i'}\cap V_{i'+1}}.
		\end{equation*}
		Then, since Lemma~\ref{lemma: copy vertex set intersections} entails~$\abs{V_{i}\cap V_{j}}\leq \abs{V_{i''-1}\cap V_{i''+1}}$, it suffices to show that
		\begin{equation*}
			\abs{V_{i''-1}\cap V_{i''+1}}\leq \min_{i''-1\leq i'\leq i''} \abs{V_{i'}\cap V_{i'+1}}-1.
		\end{equation*}
		
		To prove this, we use contraposition and argue as follows.
		Suppose now that~$\frakc=(F_0,V,I)$ is a chain and that~$F=\cF_1,\ldots,\cF_\ell$ is minimally~$\frakc$-sufficient.
		Let~$\cC:=\cC_{\frakc}$.
		For~$1\leq i\leq\ell$, let~$V_i:=V_{\cF_i}$.
		Suppose that there exists~$2\leq i'\leq \ell-1$ with
		\begin{equation*}
			\abs{V_{i'-1}\cap V_{i'+1}}\geq \abs{V_{i'-1}\cap V_{i'}}
			\qtor
			\abs{V_{i'-1}\cap V_{i'+1}}\geq \abs{V_{i'}\cap V_{i'+1}}.
		\end{equation*}
		We show that then, for
		\begin{equation*}
			U:=\paren[\Big]{\bigcup_{1\leq i\leq \ell\colon i\neq i'} V_i}\cap V,\quad
			J:=U\cap V_{i'},
		\end{equation*}
		we have~$U\neq V$ and that furthermore, as a consequence of Lemma~\ref{lemma: subextension of cF density}, we have
		\begin{equation*}
			\rho_{\cC,U}=\rho_{\cF_{i'}[V_i\cap V],J}\leq \rho_\cF.
		\end{equation*}
		This implies that~$W_\frakc\neq V$ and hence~$\frakc\neq\frakc|\R$.
		With Lemma~\ref{lemma: chains are reduced}, this yields~$\frakc\notin\frakC$ and thus completes the proof by contraposition.
		
		Let us turn to the details.
		First, note that by choice of~$i'$, Lemma~\ref{lemma: copy vertex set intersections} entails that we have
		\begin{equation*}
			V_{i'-1}\cap V_{i'+1}=V_{i'-1}\cap V_{i'}
			\qtor
			V_{i'-1}\cap V_{i'+1}=V_{i'}\cap V_{i'+1}.
		\end{equation*}
		If~$V_{i'-1}\cap V_{i'+1}=V_{i'-1}\cap V_{i'}$, then
		\begin{equation*}
			V_{i'-1}\cap V_{i'}
			=V_{i'-1}\cap V_{i'+1}\cap V_{i'}
			\subseteq V_{i'}\cap V_{i'+1}.
		\end{equation*}
		Similarly, if~$V_{i'-1}\cap V_{i'+1}=V_{i'}\cap V_{i'+1}$, then
		\begin{equation*}
			V_{i'}\cap V_{i'+1}
			=V_{i'-1}\cap V_{i'+1}\cap V_{i'}
			\subseteq V_{i'-1}\cap V_{i'}.
		\end{equation*}
		Hence, in particular we have
		\begin{equation*}
			V_{i'-1}\cap V_{i'}\subseteq V_{i'}\cap V_{i'+1}
			\qtor
			V_{i'}\cap V_{i'+1}\subseteq V_{i'-1}\cap V_{i'}.
		\end{equation*}
		Since Lemma~\ref{lemma: copy vertex set intersections} implies
		\begin{equation*}
			J
			=\paren[\Big]{ \bigcup_{1\leq i\leq\ell\colon i\neq i'} V_{i'}\cap V_{i} }\cap V
			=((V_{i'-1}\cap V_{i'})\cup (V_{i'}\cap V_{i'+1}))\cap V,
		\end{equation*}
		this yields
		\begin{equation}\label{equation: I comes from only one intersection}
			J=V_{i'-1}\cap V_{i'}\cap V
			\qtor
			J=V_{i'}\cap V_{i'+1}\cap V.
		\end{equation}
		
		To see that~$U\neq V$, we argue as follows.
		Since~$\cF_1,\ldots,\cF_\ell$ is minimally~$\frakc$-sufficient, for
		\begin{equation*}
			\cS:=\cF_1+\ldots+\cF_{\ell}
			\qtand
			\cS_{i'}:=\cF_1+\ldots+\cF_{i'-1}+\cF_{i'+1}+\ldots+\cF_\ell,
		\end{equation*}
		we obtain~$\cS_{i'}[U]\neq\cS[V]$.
		If there exists a vertex~$v\in V\setminus U$, then~$U\neq V$.
		Thus, for our proof that~$U\neq V$, we may assume that there exists an edge~$e\in \cS[V]\setminus \cS_{i'}[U]\subseteq \cF_{i'}[V\cap V_{i'}]\setminus \cS_{i'}[U]$.
		If~$\abs{J}\leq k-1$, then~$\cF_{i'}[J]=\emptyset$ and if~$\abs{J}\geq k$, then, since~$F$ is a subsequence of a vertex-separated loose path, due to~\eqref{equation: I comes from only one intersection}, we have~$\abs{J}=k$ and furthermore~$\cF_{i'}[J]\subseteq \cF_{i'-1}[V\cap V_{i'-1}]$ or~$\cF_{i'}[J]\subseteq \cF_{i'+1}[V\cap V_{i'+1}]$.
		Hence, in any case, we have~$\cF_{i'}[J]\subseteq \cS_{i'}[U]$ and thus~$e\in \cF_{i'}[V\cap V_{i'}]\setminus \cF_{i'}[J]$.
		This implies that there exists
		\begin{equation*}
			v
			\in e\setminus J
			\subseteq (V\cap V_{i'})\setminus J
			=(V\cap V_{i'})\setminus U
			\subseteq V\setminus U,
		\end{equation*}
		so we have~$U\neq V$.
		
		It remains to prove that~$\rho_{\cC,U}\leq \rho_\cF$.	
		To this end, let~$\cA:=\cF_{i'}[V\cap V_{i'}]$ and note that for all~$1\leq i\leq\ell$ with~$i\neq i'$ and~$f\in\cF_i[V\cap V_i]$, we have~$f\subseteq U$ and hence~$f\in \cC[U]$.
		Thus,
		\begin{equation*}
			\cC\setminus\cC[U]
			=\paren[\Big]{ \bigcup_{1\leq i\leq \ell} \cF_i[V\cap V_i] }\setminus\cC[U]
			=\cA\setminus\cC[U]
			=\cA\setminus(\cC[U]\cap \cA)
			=\cA\setminus \cA[J].
		\end{equation*}
		Furthermore, for all~$1\leq i\leq \ell$ with~$i\neq i'$ and~$v\in V\cap V_i$, we have~$v\in U$, so we also have
		\begin{equation*}
			V\setminus U
			=\paren[\Big]{ \bigcup_{1\leq i\leq\ell} V\cap V_i }\setminus U
			=(V\cap V_i)\setminus U
			=(V\cap V_i)\setminus J.
		\end{equation*}
		Thus,~$\rho_{\cC,U}=\rho_{\cA,J}$.
		Hence, since~\eqref{equation: I comes from only one intersection} states that we have~$J=V_\cA\cap V_{i'-1}\cap V_{i'}$ or~$J=V_\cA\cap V_{i'}\cap V_{i'+1}$, Lemma~\ref{lemma: subextension of cF density} entails~$\rho_{\cC,U}\leq \rho_\cF$, which completes the proof as explained above.
	\end{proof}
	
	\begin{lemma}\label{lemma: minimally sufficient length}
		Suppose that~$F=\cF_1,\ldots,\cF_\ell$ is minimally~$\frakc$-sufficient for some chain~$\frakc\in\frakC$.
		Then,~$\ell\leq 1/\eps^{4k}$.
	\end{lemma}
	\begin{proof}
		Suppose that~$\frakc=(V,F,I)$.
		By Lemma~\ref{lemma: chain size}, we have~$\abs{V}\leq \eps^{-3}$, hence~$\abs{\cC_\frakc}\leq \eps^{-3k}$ and thus, since~$F$ is minimally~$\frakc$-sufficient,~$\ell\leq\eps^{-3}+\eps^{-3k}$.
	\end{proof}
	
	\begin{lemma}\label{lemma: error parameter property}
		Suppose that~$\frakc\in\frakC$.
		Let~$e\in\cC_\frakc\setminus\cC_\frakc[I]$ and~$\beta\colon f\bijection e$ where~$f\in\cF$.
		Suppose that~$\frakb$ is the~$\beta$-branching of~$\frakc$.
		Let~$e'\in \cF_\frakc^\beta\setminus\set{e}$ and~$\beta'\colon f'\bijection e'$ where~$f'\in\cF$.
		Suppose that~$\frakb'$ is the~$\beta'$-branching of~$\frakb$.
		Then,
		\begin{equation*}
			\chi_{\frakb'}\leq \chi_\frakb-1
			\qtor
			\chi_{\frakb'}= \chi_\frakb.
		\end{equation*}
		Furthermore, if~$\chi_{\frakb'}= \chi_\frakb$, then~$\frakB_\frakc^e$ and~$\frakB_{\frakb}^{e'}$ are template equivalent.
	\end{lemma}
	
	\begin{proof}
		Suppose that~$\frakb=(F,V,I)$ and~$\frakb'=(F',V',I)$.
		From all minimally~$\frakb$-sufficient sequences, choose~$\cF_1,\ldots,\cF_{\ell-1}$ such that~$\chi_{\cF_1,\ldots,\cF_{\ell-1}}$ is minimal.
		Let~$\cF_\ell:=\cF_{\frakb}^{\beta'}$.
		For~$1\leq i\leq \ell$, let~$V_i:=V_{\cF_i}$.
		Observe that the sequence~$\cF_1,\ldots,\cF_\ell$ is~$\frakb'$-sufficient.
		Consider a minimally~$\frakb'$-sufficient subsequence~$\cF_{i_1},\ldots,\cF_{i_{\ell'}}$ of~$\cF_1,\ldots,\cF_\ell$ with~$i_1=1$.
		Note that~$i_{\ell'}=\ell$.
		To shorten notation, for~$1\leq i,j\leq \ell$, we set
		\begin{equation*}
			f(i,j):=\eps^{5k\abs{V_{i}\cap V_{j}}}.
		\end{equation*}
		For all~$1\leq j\leq \ell'-1$ with~$i_{j+1}=i_j+1$, we vacuously have
		\begin{equation*}
			\sum_{i_j\leq i\leq i_{j+1}-1} f(i,i+1)
			=  f(i_j,i_{j+1})
		\end{equation*}
		and for all~$1\leq j\leq \ell'-2$ with~$i_{j+1}\geq i_j+2$, Lemma~\ref{lemma: reduced copy vertex set intersections} together with Lemma~\ref{lemma: minimally sufficient length} implies
		\begin{align*}
			\sum_{i_j\leq i\leq i_{j+1}-1} f(i,i+1)
			&\leq (i_{j+1}-i_j) \eps^{5k}f(i_j,i_{j+1})
			\leq \eps f(i_j,i_{j+1})\\
			&= f(i_j,i_{j+1})-(1-\eps)f(i_j,i_{j+1})
			\leq f(i_j,i_{j+1})-\frac{\eps^{5k^2}}{2}.
		\end{align*}
		For
		\begin{equation*}
			\Lambda:=\begin{cases}
				1 &\text{if~$i_1,\ldots,i_{\ell'-1}\neq 1,\ldots,\ell-2$};\\
				0 &\text{otherwise},
			\end{cases}
		\end{equation*}
		using that~$V'$ is the disjoint union of~$V'\setminus V$ and~$V\cap V'=V\setminus (V\setminus V')$, this yields
		\begin{align*}
			\chi_{\frakb'}
			&\leq\abs{V'}-\eps^{-5k(k+1)}\paren[\Big]{\sum_{1\leq j\leq \ell'-2} f(i_j,i_{j+1})}-\eps^{-5k(k+1)}f(i_{\ell'-1},\ell)\\
			&\leq \abs{V'}-\eps^{-5k(k+1)}\paren[\bigg]{\frac{\eps^{5k^2}}{2}\Lambda+\sum_{1\leq j\leq \ell'-2}\sum_{i_j\leq i\leq i_{j+1}-1} f(i,i+1)}-\eps^{-5k(k+1)}f(i_{\ell'-1},\ell)\\
			&= \abs{V'}-\eps^{-5k(k+1)}\paren[\Big]{\sum_{1\leq i\leq i_{\ell'-1}-1} f(i,i+1)}-\frac{\eps^{-5k}}{2}\Lambda-\eps^{-5k(k+1)}f(i_{\ell'-1},\ell)\\
			&= \abs{V}+m-k-\abs{V\setminus V'}-\eps^{-5k(k+1)}\paren[\Big]{\sum_{1\leq i\leq \ell-2} 
				f(i,i+1)}\\&\hphantom{=}\mathrel{}\quad+\eps^{-5k(k+1)}\paren[\Big]{\sum_{i_{\ell'-1}\leq i\leq \ell-2} f(i,i+1)}-\frac{\eps^{-5k}}{2}\Lambda-\eps^{-5k(k+1)}f(i_{\ell'-1},\ell)\\
			&=\chi_{\frakb}+m-k-\abs{V\setminus V'} +\eps^{-5k(k+1)}\paren[\Big]{\sum_{i_{\ell'-1}\leq i\leq \ell-2} f(i,i+1)}\\&\hphantom{=}\mathrel{}\quad-\eps^{-5k(k+1)}f(i_{\ell'-1},\ell)-\frac{\eps^{-5k}}{2}\Lambda.
		\end{align*}
		Note that~$i_{\ell'-1}\leq \ell-3$,~$i_{\ell'-1}= \ell-2$ or~$i_{\ell'-1}=\ell-1$.
		We investigate the three cases separately.
		
		First, suppose that~$i_{\ell'-1}\leq \ell-3$.
		Then using Lemma~\ref{lemma: reduced copy vertex set intersections}, Lemma~\ref{lemma: minimally sufficient length} and Lemma~\ref{lemma: copy vertex set intersections}, we obtain
		\begin{align*}
			\sum_{i_{\ell'-1}\leq i\leq \ell-2} f(i,i+1)
			&\leq (\ell-i_{\ell'-1}-1) \eps^{5k}f(i_{\ell'-1},\ell-1)
			\leq \eps f(i_{\ell'-1},\ell-1)\\
			&=f(i_{\ell'-1},\ell-1)-(1-\eps)f(i_{\ell'-1},\ell-1)
			\leq f(i_{\ell'-1},\ell-1)-\frac{\eps^{5k^2}}{2}\\
			&\leq f(i_{\ell'-1},\ell)-\frac{\eps^{5k^2}}{2}.
		\end{align*}
		Hence, if~$i_{\ell'-1}\leq \ell-3$, then
		\begin{equation*}
			\chi_{\frakb'}
			\leq\chi_{\frakb}+m-k-\abs{V\setminus V'} -\frac{\eps^{-5k}}{2}
			\leq \chi_{\frakb}-1.
		\end{equation*}
		
		Next, suppose that~$i_{\ell'-1}= \ell-2$.
		If
		\begin{equation*}
			\abs{V_{i_{\ell'-1}}\cap V_{\ell-1}}\geq \abs{V_{i_{\ell'-1}}\cap V_{\ell}}+1,
		\end{equation*}
		then
		\begin{align*}
			\sum_{i_{\ell'-1}\leq i\leq \ell-2} f(i,i+1)
			&=f(i_{\ell'-1},\ell-1)
			\leq \eps^{5k}f(i_{\ell'-1},\ell)\\
			&=f(i_{\ell'-1},\ell)-(1-\eps^{5k})f(i_{\ell'-1},\ell)
			\leq f(i_{\ell'-1},\ell)-\frac{\eps^{5k^2}}{2}
		\end{align*}
		and thus
		\begin{equation*}
			\chi_{\frakb'}\leq \chi_{\frakb}+m-k-\abs{V\setminus V'}-\frac{\eps^{-5k}}{2}
			\leq \chi_{\frakb}-1.
		\end{equation*}
		If
		\begin{equation}\label{equation: intersection upper bound}
			\abs{V_{i_{\ell'-1}}\cap V_{\ell-1}}\leq \abs{V_{i_{\ell'-1}}\cap V_{\ell}},
		\end{equation}
		then Lemma~\ref{lemma: copy vertex set intersections} entails
		\begin{equation}\label{equation: same intersection}
			V_{i_{\ell'-1}}\cap V_{\ell-1}= V_{i_{\ell'-1}}\cap V_{\ell},
		\end{equation}
		and thus
		\begin{equation*}
			\sum_{i_{\ell'-1}\leq i\leq \ell-2} f(i,i+1)
			=f(i_{\ell'-1},\ell-1)
			=f(i_{\ell'-1},\ell).
		\end{equation*}
		Due to Lemma~\ref{lemma: copy vertex set intersections}, a consequence of~\eqref{equation: same intersection} is
		\begin{equation*}
			V_{i_{\ell'-1}}\cap V_{\ell-1}\subseteq V_{\ell-1}\cap V_{\ell}.
		\end{equation*}
		Since we assume that~$i_{\ell'-1}= \ell-2$, this yields
		\begin{equation}\label{equation: what remains of ell-1}
			V_{\ell-1}\cap V'
			\subseteq (V_{i_{\ell'-1}}\cup V_\ell)\cap V_{\ell-1}
			\subseteq V_\ell\cap V_{\ell-1}
			=e'
		\end{equation}
		and so in particular~$\abs{V_{\ell-1}\cap V'}\leq k$ and thus~$\abs{V\setminus V'}\geq m-k$.
		Hence, if~\eqref{equation: intersection upper bound} holds, then
		\begin{equation}\label{equation: possible equality}
			\chi_{\frakb'}\leq\chi_{\frakb}+m-k-\abs{V\setminus V'}-\frac{\eps^{-5k}}{2}\Lambda \leq\chi_{\frakb}-\frac{\eps^{-5k}}{2}\Lambda
		\end{equation}
		and thus~$\chi_{\frakb'}\leq\chi_{\frakb}-1$ or~$\chi_{\frakb'}=\chi_{\frakb}$.
		
		Finally, suppose that~$i_{\ell'-1}=\ell-1$.
		Then,
		\begin{equation*}
			\chi_{\frakb'}
			\leq \chi_{\frakb}+m-k-\abs{V\setminus V'}-\eps^{-5k(k+1)}f(i_{\ell'-1},\ell)
			\leq \chi_{\frakb}+m-k-\abs{V\setminus V'} -\eps^{-5k}
			\leq \chi_{\frakb}-1.
		\end{equation*}
		This finishes the analysis of the three cases and the proof that we have~$\chi_{\frakb'}\leq\chi_{\frakb}-1$ or~$\chi_{\frakb'}=\chi_{\frakb}$.
		
		It remains to further investigate the case where~$\chi_{\frakb'}=\chi_{\frakb}$.
		Suppose that~$\chi_{\frakb'}=\chi_{\frakb}$.
		Note that by Lemma~\ref{lemma: branchings are branchings of support}, it suffices to obtain that~$\frakB^e_{\frakb|\om}$ and~$\frakB^{e'}_{\frakb'|\om}$ are template equivalent, so due to Lemma~\ref{lemma: branchings are copy invariant}, it suffices to show that~$(\cC_{\frakb'|\om},I)$ is a copy of~$(\cC_{\frakb|\om},I)$ with~$e'$ playing the role of~$e$.
		Our analysis of the three cases above shows that~$\chi_{\frakb'}=\chi_{\frakb}$ is only possible if~$i_{\ell'-1}=\ell-2$,~\eqref{equation: same intersection},~\eqref{equation: what remains of ell-1} and~\eqref{equation: possible equality} hold.
		Revisiting the first inequality in~\eqref{equation: possible equality}, we see that~$\Lambda=0$ and~$\abs{V\setminus V'}=m-k$ necessarily hold.
		Let~$\cS:=\cF_1+\ldots+\cF_{\ell-2}$, let~$\cE$ denote the~$k$-graph with vertex set~$e$ and edge set~$\set{e}$ and let~$\cE'$ denote the~$k$-graph with vertex set~$e'$ and edge set~$\set{e'}$.
		Note that~$\cC_{\frakb|\om}=\cS[V\cap V_\cS]+\cE$ and that as a consequence of~$\Lambda=0$, we have~$\cC_{\frakb'|\om}=\cS[V'\cap V_\cS]+\cE'$.
		Thus, to see that~$(\cC_{\frakb'|\om},I)$ is a copy of~$(\cC_{\frakb|\om},I)$ with~$e'$ playing the role of~$e$, it suffices to obtain~$V\cap V_\cS=V'\cap V_\cS$ and additionally~$e\cap V\cap V_\cS=e'\cap V'\cap V_\cS$.
		Since~\eqref{equation: what remains of ell-1} entails
		\begin{equation*}
			V_{\ell-1}\setminus e'
			\subseteq V_{\ell-1}\setminus (V_{\ell-1}\cap V')
			=V_{\ell-1}\setminus V'
			\subseteq V\setminus V',
		\end{equation*}
		from~$\abs{V\setminus V'}=m-k$, we obtain~$V\setminus V'=V_{\ell-1}\setminus e'$ and thus using~\eqref{equation: same intersection}, we have
		\begin{equation*}
			\begin{aligned}
				(V\cap V_\cS)\setminus (V'\cap V_\cS)
				&=(V\setminus V')\cap V_\cS
				=(V_{\ell-1}\setminus e')\cap V_\cS
				=(V_{\ell-1}\cap V_\cS)\setminus e'
				=(V_{\ell-1}\cap V_{\ell-2})\setminus e'\\
				&=(V_{\ell}\cap V_{\ell-1}\cap V_{\ell-2})\setminus e'
				=\emptyset.
			\end{aligned}
		\end{equation*}
		Since~$V'\cap V_\cS\subseteq V\cap V_\cS$, this yields
		\begin{equation}\label{equation: V and V' initial segment equal}
			V\cap V_\cS=V'\cap V_\cS.
		\end{equation}
		Furthermore, again using~\eqref{equation: same intersection}, we obtain
		\begin{equation*}
			e\cap V_{\ell-2}
			=V_{\ell-1}\cap V_{\ell-2}
			=V_\ell\cap V_{\ell-1}\cap V_{\ell-2}
			=e'\cap V_{\ell-2}.
		\end{equation*}
		Combining this with~\eqref{equation: V and V' initial segment equal} yields
		\begin{equation*}
			e\cap V\cap V_\cS
			=(e\cap V_{\ell-2})\cap V\cap V_\cS
			=(e'\cap V_{\ell-2})\cap V\cap V_\cS
			=(e'\cap V_{\ell-2})\cap V'\cap V_\cS
			=e'\cap V_{\ell-1}\cap V_{\ell-2}\cap V'\cap V_\cS.
		\end{equation*}
		Since Lemma~\ref{lemma: copy vertex set intersections} entails~$V_{\ell-1}\cap V_i\subseteq V_{\ell-1}\cap V_{\ell-2}$ for all~$1\leq i\leq \ell-2$, thus~$V_{\ell-1}\cap V_\cS\subseteq V_{\ell-1}\cap V_{\ell-2}$ and hence~$V_{\ell-1}\cap V_\cS= V_{\ell-1}\cap V_{\ell-2}\cap V_\cS$, this yields
		\begin{equation*}
			e\cap V\cap V_\cS=e'\cap V_{\ell-1}\cap V'\cap V_\cS=e'\cap V'\cap V_\cS,
		\end{equation*}
		which completes the proof.
	\end{proof}
	
	\subsection{Tracking branching families}\label{subsection: tracking families}
	Suppose that~$0\leq i\leq i^\star$, consider a chain~$\frakc=(F,V,I)\in\frakC$ and let~$\psi\colon I\injection V_\cH$.
	Let~$e\in\cC_\frakc\setminus\cC_{\frakc}[I]$.
	Similarly as in Section~\ref{subsection: tracking chains}, we show that~$\sum_{\frakb\in\frakB_\frakc^e}\Phi_{\cC_\frakb,\psi}$ is typically close to~$\sum_{\frakb\in\frakB_\frakc^e}\Phihat_{\frakb,\psi}$, that is that
	\begin{equation*}
		X_{\frakc,\psi}^e:=\sum_{\frakb\in\frakB_\frakc^e}\Phi_{\cC_\frakb,\psi}-\sum_{\frakb\in\frakB_\frakc^e}\Phihat_{\frakb,\psi}=\sum_{\frakb\in\frakB_{\frakc}^e} X_{\frakb,\psi}
	\end{equation*}\gladd{realRV}{Xcpsi}{$X_{\frakc,\psi}^e=\sum_{\frakb\in\frakB_\frakc^e}\Phi_{\cC_\frakb,\psi}-\sum_{\frakb\in\frakB_\frakc^e}\Phihat_{\frakb,\psi}$}%
	is typically small, where the quantification of the deviation we allow crucially relies on the insights from Section~\ref{subsection: error parameter}.
	Formally, we finally define the fifth stopping time mentioned in Section~\ref{section: heuristics} as
	\begin{equation*}
		\tau_\frakB:=\min\set*{\setlength\arraycolsep{0pt}\begin{array}{ll}
				i\geq 0 :~&\sum_{\frakb\in \frakB_{\frakc}^e}\Phi_{\cC_\frakb,\psi}\neq \sum_{\frakb\in \frakB_{\frakc}^e}\Phihat_{\frakb,\psi}\pm \eps^{-\chi_{\frakB_\frakc^e}}\zeta\phihat_{\frakb,I}\\&\quad\text{for some } \frakc=(F,V,I)\in\frakC, e\in\cC_\frakc\setminus\cC_\frakc[I], \psi\colon I\injection V_\cH
		\end{array}}
	\end{equation*}\gladd{stoppingtime}{tauB}{$\tau_\frakB=\min\set*{\setlength\arraycolsep{0pt}\begin{array}{ll}
				i\geq 0 :~&\sum_{\frakb\in \frakB_{\frakc}^e}\Phi_{\cC_\frakb,\psi}\neq \sum_{\frakb\in \frakB_{\frakc}^e}\Phihat_{\frakb,\psi}\pm \eps^{-\chi_{\frakB_\frakc^e}}\zeta\phihat_{\frakb,I}\\&\quad\text{for some } \frakc=(F,V,I)\in\frakC, e\in\cC_\frakc\setminus\cC_\frakc[I], \psi\colon I\injection V_\cH
		\end{array}}$}%
	and we show that the probability that~$\tau_\frakB\leq \tau^\star\wedge i^\star$ is small.
	The following Lemma~\ref{lemma: error parameter bound} shows that indeed~$\tautilde_\frakB\geq \tau_\frakB$.
	Similarly as in Section~\ref{subsection: tracking chains}, Lemma~\ref{lemma: finite branching collection} shows that it suffices to consider a collection of branching families that has size at most~$1/\delta$, which in turn allows us to restrict our attention to only one fixed branching family.
	To prove Lemma~\ref{lemma: finite branching collection}, we observe that there are only finitely many relevant error parameters (see Lemma~\ref{lemma: finite error parameters}).
	
	\begin{lemma}\label{lemma: error parameter bound}
		Let~$\frakc\in\frakC$.
		Then,~$\delta^{1/2}\leq \eps^{-\chi_\frakc}\leq \delta^{-1/2}$.
	\end{lemma}
	\begin{proof}
		Suppose that~$\frakc=(F,V,I)$.
		From Lemma~\ref{lemma: chain size}, we obtain~$\chi_\frakc\leq \abs{V}\leq \eps^{-3}$ and from Lemma~\ref{lemma: minimally sufficient length}, we obtain~$\chi_\frakc\geq -\eps^{-5k(k+1)}\cdot \eps^{-4k^2}$, so the statement follows.
	\end{proof}

	\begin{lemma}\label{lemma: finite error parameters}
		The set~$\cset{ \chi_\frakc }{ \frakc\in\frakC }$ is finite.
	\end{lemma}
	\begin{proof}
		As a consequence of Lemma~\ref{lemma: chain size}, it suffices to show that
		\begin{equation*}
			X:=\cset{ \chi_F }{ \text{ $F$ is minimally~$\frakc$-sufficient for some~$\frakc\in\frakC$ } }
		\end{equation*}
		is finite.
		By Lemma~\ref{lemma: minimally sufficient length}, every sequence that is minimally~$\frakc$-sufficient for some~$\frakc\in\frakC$ has length at most~$\eps^{-4k}$, which entails that~$X$ is indeed finite.
	\end{proof}
	\begin{lemma}\label{lemma: finite branching collection}
		There exists a collection~$\frakC_0\subseteq\frakC$ with~$\abs{\frakC_0}\leq 1/\delta$ such that for all~$\frakc=(F,V,I)\in\frakC$ and~$e\in\cC_\frakc\setminus\cC_\frakc[I]$, there exist~$\frakc_0=(F_0,V_0,I_0)\in\frakC_0$ and~$e_0\in\cC_{\frakc_0}\setminus\cC_{\frakc_0}[I]$ such that~$\frakB_{\frakc_0}^{e_0}$ and~$\frakB_\frakc^e$ are template equivalent with~$\chi_{\frakB_{\frakc_0}^{e_0}}\leq \chi_{\frakB_{\frakc}^{e}}$.
	\end{lemma}
	\begin{proof}
		Similarly as in the proof of Lemma~\ref{lemma: finite chain collection}, consider the set~$\ccT$ of all templates~$(\cA,I)$ where~$V_\cA\subseteq\set{1,\ldots,1/\eps^3}$.
		By Lemma~\ref{lemma: chain size}, for all~$\frakc=(F,V,I)\in\frakC$, we may choose a template~$\cT_\frakc\in\ccT$ that is a copy of~$(\cC_\frakc,I)$.
		For every chain~$\frakc=(F,V,I)$ and~$e\in\cC_\frakc\setminus\cC_\frakc[I]$, we may consider the unordered~$(\abs{\cF}k!)$-tuple~$\ctup{(\cT_\frakb,\cT_{\frakb|\om})}{\frakb\in\frakB_\frakc^e}$ whose components are the pairs~$(\cT_\frakb,\cT_{\frakb|\om})$ of templates where~$\frakb\in\frakB_\frakc^e$.
		We use~$\ccT_2$ to denote the set of such unordered tuples, that is we set
		\begin{equation*}
			\ccT_2:=\cset{ \ctup{ (\cT_\frakb,\cT_{\frakb|\om}) }{ \frakb\in \frakB_\frakc^e } }{ \frakc=(F,V,I)\in\frakC,e\in\cC_\frakc\setminus\cC_\frakc[I] }.
		\end{equation*}
		Note that~$\abs{\ccT_2}\leq \abs{\ccT}^{2\abs{\cF}k!}$.
		Consider an unordered tuple~$\ccP\in\ccT_2$.
		As a consequence of Lemma~\ref{lemma: finite error parameters}, among all pairs~$(\frakc,e)$ where~$\frakc=(F,V,I)\in\frakC$ and~$e\in\cC_{\frakc}\setminus\cC_{\frakc}[I]$ such that~$\ccP=\ctup{ (\cT_\frakb,\cT_{\frakb|-}) }{ \frakb\in \frakB_\frakc^e }$, we may choose a pair~$(\frakc_\ccP,e_\ccP)$ such that~$\chi_{\frakB_{\frakc_\ccP}^{e_\ccP}}$ is minimal.
		Then,~$\cset{\frakc_\ccP}{ \ccP\in\ccT_2 }$ is a collection as desired.
	\end{proof}
	
	\begin{observation}\label{observation: family individual}
		Suppose that~$\frakC_0\subseteq\frakC$ is a collection of chains as in Lemma~\ref{lemma: finite branching collection}.
		For~$\frakc=(F,V,I)\in\frakC$,~$e\in\cC_\frakc\setminus\cC_\frakc[I]$ and~$\psi\colon I\injection V_\cH$, let
		\begin{equation*}
			\tau_{\frakc,\psi}^e:=\min\cset[\Big]{ i\geq 0 }{ \sum_{\frakb\in\frakB_\frakc^e}\Phi_{\frakb,\psi}\neq \sum_{\frakb\in\frakB_\frakc^e}\Phihat_{\frakb,\psi}\pm \eps^{-\chi_{\frakB_\frakc^e}}\zeta\phihat_{\frakb,I} }.
		\end{equation*}
		Then,
		\begin{equation*}
			\pr{\tau_\frakB\leq \tau^\star\wedge i^\star}
			\leq \sum_{\substack{\frakc=(F,V,I)\in\frakC_0,\\e\in\cC_\frakc\setminus\cC_\frakc[I], \psi\colon I\injection V_\cH}}\pr{ \tau_{\frakc,\psi}^e\leq \tau^\star\wedge i^\star }.
		\end{equation*}
	\end{observation}
	
	Hence, fix~$\frakc=(F,V,I)\in\frakC$,~$e\in\cC_\frakc\setminus\cC_\frakc[I]$ and~$\psi\colon I\injection V_\cH$ and let~$\chi:=\chi_{\frakB_\frakc^e}$.
	Besides~$\frakc$ and~$\psi$, we redefine several other symbols from Section~\ref{section: chains}, for example~$\xi_0$,~$\xi_1$ and~$\tau$.
	However, we still use some symbols from previous sections that we do not redefine.
	Whenever we use a symbol, its most recent definition applies.
	For~$i\geq 0$, let
	\begin{equation*}
		\xi_1(i):=\sum_{\frakb\in\frakB_\frakc^e}\eps^{-\chi}\zeta\phihat_{\frakb,I},\quad
		\xi_0(i):=(1-\delta)\xi_1
	\end{equation*}
	and define the stopping time
	\begin{equation*}
		\tau:=\min\cset[\Big]{ i\geq 0 }{ \sum_{\frakb\in \frakB_\frakc^e}\Phi_{\cC_\frakb,\psi}\neq \paren[\Big]{\sum_{\frakb\in \frakB_\frakc^e}\Phihat_{\frakb,\psi}}\pm\xi_1 }.
	\end{equation*}
	Define the critical intervals
	\begin{equation*}
		I^-(i):=\brack[\Big]{\paren[\Big]{\sum_{\frakb\in \frakB_\frakc^e}\Phihat_{\frakb,\psi}}-\xi_1,\paren[\Big]{\sum_{\frakb\in \frakB_\frakc^e}\Phihat_{\frakb,\psi}}-\xi_0},\quad
		I^+(i):=\brack[\Big]{\paren[\Big]{\sum_{\frakb\in \frakB_\frakc^e}\Phihat_{\frakb,\psi}}+\xi_0,\paren[\Big]{\sum_{\frakb\in \frakB_\frakc^e}\Phihat_{\frakb,\psi}}+\xi_1}.
	\end{equation*}
	For~$\pom\in\set{-,+}$, let
	\begin{equation*}
		Y^\pom(i):=\pom X_{\frakc,\psi}^e-\xi_1.
	\end{equation*}
	For~$i_0\geq 0$, define the stopping time
	\begin{equation*}
		\tau^\pom_{i_0}:=\min\cset[\Big]{ i\geq i_0 }{ \sum_{\frakb\in \frakB_\frakc^e}\Phi_{\cC_\frakc,\psi}\notin I^\pom }
	\end{equation*}
	and for~$i\geq i_0$, let
	\begin{equation*}
		Z_{i_0}^\pom(i):=Y^\pom(i_0\vee (i\wedge \tau_{i_0}^\pom\wedge \tau^\star\wedge i^\star)).
	\end{equation*}
	Let
	\begin{equation*}
		\sigma^\pom:=\min\cset{ j\geq 0 }{ \pom X_{\frakc,\psi}^e\geq \xi_0 \stforall j\leq i< \tau^\star\wedge i^\star }\leq \tau^\star\wedge i^\star.
	\end{equation*}
	With this setup, similarly as in Section~\ref{subsection: tracking chains}, it in fact suffices to consider the evolution of~$Z^\pom_{\sigma^\pom}(\sigma^\pom),Z^\pom_{\sigma^\pom}(\sigma^\pom+1),\ldots$.
	\begin{observation}\label{observation: family critical times}
		$\set{\tau\leq \tau^\star\wedge i^\star}\subseteq \set{ Z^-_{\sigma^-}(i^\star)>0 }\cup \set{ Z^+_{\sigma^+}(i^\star)>0 }$.
	\end{observation}
	We again use Lemma~\ref{lemma: freedman} to show that the probabilities of the events on the right in Observation~\ref{observation: family critical times} are sufficiently small.
	
	\subsubsection{Trend}\label{subsubsection: family trend}
	Here, we prove that for all~$\pom\in\set{-,+}$ and~$i_0\geq 0$, the expected one-step changes of the process~$Z_{i_0}^\pom(i_0),Z_{i_0}^\pom(i_0+1),\ldots$ are non-positive.
	Branching families are closely related to individual chains, so we may use statements from Section~\ref{subsection: tracking chains} as a starting point for our arguments here.
	As a consequence of Lemma~\ref{lemma: delta phihat G J xi ladder}, we obtain Lemma~\ref{lemma: delta xi family} where we state estimates for the one-step changes of the error term that we use in this section.
	Using these estimates, we turn to proving that the process we consider here is indeed a supermartingale (see Lemma~\ref{lemma: expected change family}).
	We prove this by revisiting the expression for individual chains stated in Lemma~\ref{lemma: expected change chain deviation} where, since we are now in the setting of branching families, we may now exploit that one step-changes depend on branching families.
	This allows us to no longer differentiate between the different branchings as they always appear in complete families.
	This ultimately enables us to identify self-correcting behavior as desired as a consequence of our careful choice of error parameters crucially relying on the insights from Section~\ref{subsection: error parameter}.

	Note that for~$\frakb,\frakb'\in\frakB_\frakc^e$, we have~$\abs{\cC_\frakb}=\abs{\cC_{\frakb'}}$.
	Hence, we may choose~$b$ such that~$b=\abs{\cC_\frakb}$ for all~$\frakb\in\frakB_\frakc^e$.

	\begin{lemma}\label{lemma: delta xi family}
		Let~$0\leq i\leq i^\star$ and~$\cX:=\set{i\leq \tau_\emptyset}$.
		Then,
		\begin{equation*}
			\Delta\xi_1\Xeq-\paren[\bigg]{b-1-\frac{\rho_\cF}{2}}\frac{\abs{\cF}\xi_1}{H}\pm\frac{\zeta^2\xi_1}{H}.
		\end{equation*}
	\end{lemma}
	\begin{proof}
		For~$\frakb\in\frakB$, we may apply Lemma~\ref{lemma: delta phihat G J xi ladder} with~$\frakb$ playing the role of~$\frakc$ to obtain
		\begin{equation*}
			\Delta(\delta^{-1}\zeta\phihat_{\frakb,I})
			=-\paren[\bigg]{\abs{\cC_\frakb}-1-\frac{\rho_\cF}{2}}\frac{\abs{\cF}\delta^{-1}\zeta\phihat_{\frakb,I}}{H}\pm\frac{\zeta^2\delta^{-1}\zeta\phihat_{\frakb,I}}{H}.
		\end{equation*}
		This yields
		\begin{equation*}
			\Delta\xi_1
			=\sum_{\frakb\in\frakB} \eps^{-\chi}\delta\Delta(\delta^{-1}\zeta\phihat_{\frakb,I})
			\Xeq-\paren[\bigg]{b-1-\frac{\rho_\cF}{2}}\frac{\abs{\cF}\sum_{\frakb\in\frakB} \eps^{-\chi}\zeta\phihat_{\frakb,I}}{H}\pm\frac{\zeta^2\sum_{\frakb\in\frakB} \eps^{-\chi}\zeta\phihat_{\frakb,I}}{H},
		\end{equation*}
		which completes the proof.
	\end{proof}
	
	\begin{lemma}\label{lemma: family trend}
		Let~$0\leq i_0\leq i$ and~$\pom\in\set{-,+}$.
		Then,~$\exi{\Delta Z_{i_0}^\pom}\leq 0$.
	\end{lemma}
	\begin{proof}
		Suppose that~$i<i^\star$ and let~$\cX:=\set{i<\tau_{i_0}^\pom\wedge \tau^\star}$.
		We have~$\exi{\Delta Z_{i_0}^\pom}=_{\cX^\comp}0$ and~$\exi{\Delta Z_{i_0}^\pom}\Xeq \exi{\Delta Y^\pom}$, so it suffices to obtain~$\exi{\Delta Y^\pom}\Xleq 0$.
		From Lemma~\ref{lemma: ladder deviation trend}, using Lemma~\ref{lemma: error parameter bound}, we obtain
		\begin{equation*}
			\begin{aligned}
				\exi{\Delta (\pom X_{\frakc,\psi}^e)}\hspace{-1cm}&\\
				&=\pom\sum_{\frakb\in\frakB_{\frakc}^e}\exi{\Delta X_{\frakb,\psi}}\\
				&\Xleq \pom\sum_{\frakb\in\frakB_{\frakc}^e}\paren[\bigg]{-\frac{(\abs{\cC_\frakb}-1)\abs{\cF}}{H}X_{\frakb,\psi}-\paren[\bigg]{ \sum_{e'\in\cG_\frakb\setminus \set{J_\frakb}}\sum_{\frakb'\in \frakB_\frakb^{e'}} \frac{\phihat_{\frakb,I}}{k!\,H\phihat_{\frakb',I}}X_{\frakb',\psi} }+ \delta^2\frac{\delta^{-1}\zeta \phihat_{\frakb,I}}{H}}\\
				&\leq -\frac{\abs{\cF}}{H}\paren[\bigg]{(b-1)(\pom X_{\frakc,\psi}^e)+\frac{1}{\abs{\cF}k!}\paren[\bigg]{ \sum_{\frakb\in\frakB_{\frakc}^e}\sum_{e'\in\cG_\frakb\setminus \set{J_\frakb}} \sum_{\frakb'\in \frakB_\frakb^{e'}}\frac{\phihat_{\frakb,I}}{\phihat_{\frakb',I}}(\pom X_{\frakb',\psi}) }- \eps\xi_1}
			\end{aligned}
		\end{equation*}
		Note that for all~$\frakb\in\frakB_\frakc^e$,~$e'\in\cG_\frakb\setminus\set{J_\frakb}$ and~$\frakb'_1,\frakb'_2\in\frakB_\frakb^{e'}$, we have~$\phihat_{\frakb'_1,I}=\phihat_{\frakb'_2,I}$, so we may choose~$\phihat_{\frakb,I}^{e'}$ such that~$\phihat_{\frakb,I}^{e'}=\phihat_{\frakb',I}$ for all~$\frakb'\in\frakB_\frakb^{e'}$.
		With Lemma~\ref{lemma: chain size}, we obtain
		\begin{equation*}
			\begin{aligned}
				\exi{\Delta (\pom X_{\frakc,\psi}^e)}\hspace{-1cm}&\\
				&\Xleq -\frac{\abs{\cF}}{H}\paren[\bigg]{(b-1)(\pom X_{\frakc,\psi}^e)+\frac{1}{\abs{\cF}k!}\paren[\bigg]{ \sum_{\frakb\in\frakB_{\frakc}^e}\sum_{e'\in\cG_\frakb\setminus \set{J_\frakb}} \frac{\phihat_{\frakb,I}}{\phihat_{\frakb,I}^{e'}}\sum_{\frakb'\in \frakB_\frakb^{e'}}\pom X_{\frakb',\psi} }- \eps\xi_1}\\
				&\Xleq -\frac{\abs{\cF}}{H}\paren[\bigg]{(b-1)(1-\delta)\xi_1+\frac{1}{\abs{\cF}k!}\paren[\bigg]{ \sum_{\frakb\in\frakB_{\frakc}^e}\sum_{e'\in\cG_\frakb\setminus \set{J_\frakb}} \frac{\phihat_{\frakb,I}}{\phihat_{\frakb,I}^{e'}}(\pom X_{\frakb,\psi}^{e'}) }- \eps\xi_1}\\
				&\leq -\frac{\abs{\cF}}{H}\paren[\bigg]{(b-1)\xi_1+\frac{1}{\abs{\cF}k!}\paren[\bigg]{ \sum_{\frakb\in\frakB_{\frakc}^e}\sum_{e'\in\cG_\frakb\setminus \set{J_\frakb}} \frac{\phihat_{\frakb,I}}{\phihat_{\frakb,I}^{e'}}(\pom X_{\frakb,\psi}^{e'}) }- \eps^{1/2}\xi_1}.
			\end{aligned}
		\end{equation*}
		Thus, due to Lemma~\ref{lemma: delta xi family}, we have
		\begin{equation}\label{equation: family Y depends on others}
			\begin{aligned}
				\exi{\Delta Y^\pom}
				&\Xleq-\frac{\abs{\cF}}{H}\paren[\bigg]{\frac{\rho_\cF}{2}\xi_1+\frac{1}{\abs{\cF}k!}\paren[\bigg]{ \sum_{\frakb\in\frakB_{\frakc}^e}\sum_{e'\in\cG_\frakb\setminus \set{J_\frakb}} \frac{\phihat_{\frakb,I}}{\phihat_{\frakb,I}^{e'}}(\pom X_{\frakb,\psi}^{e'}) }- \eps^{1/3}\xi_1}.
			\end{aligned}
		\end{equation}
		Note that for all~$\frakb\in\frakB_\frakc^e$ and~$e'\in\cG_\frakb\setminus\set{J_\frakb}$, if~$\frakB_\frakc^e$ and~$\frakB_\frakb^{e'}$ are template equivalent, then
		\begin{equation*}
			\pom X_{\frakb,\psi}^{e'}=\pom X_{\frakc,\psi}^e\Xgeq \xi_0\geq 0
		\end{equation*}
		and otherwise, Lemma~\ref{lemma: error parameter property} implies
		\begin{equation*}
			\abs{\pom X_{\frakb,\psi}^{e'}}
			\Xleq \sum_{\frakb'\in\frakB_{\frakb}^{e'}}\eps^{-\chi_{\frakb'}}\zeta\phihat_{\frakb',I}
			\leq \eps\sum_{\frakb'\in\frakB_{\frakb}^{e'}}\eps^{-\chi_{\frakb}}\zeta\phihat_{\frakb',I}
			=\eps \abs{\cF}k!\cdot \eps^{-\chi}\zeta\phihat_{\frakb,I}^{e'}.
		\end{equation*}
		Hence, in any case,
		\begin{equation*}
			\begin{aligned}
				\sum_{\frakb\in\frakB_{\frakc}^e}\sum_{e'\in\cG_\frakb\setminus \set{J_\frakb}} \frac{\phihat_{\frakb,I}}{\phihat_{\frakb,I}^{e'}}(\pom X_{\frakb,\psi}^{e'})
				&\Xgeq -\eps\abs{\cF}k!\sum_{\frakb\in\frakB_{\frakc}^e}\sum_{e'\in\cG_\frakb\setminus \set{J_\frakb}}\eps^{-\chi}\zeta\phihat_{\frakb,I}
				=-\eps\abs{\cF}k!\,\abs{\cG_\frakb\setminus\set{J_\frakb}}\xi_1\\
				&\geq -\eps^{1/2}\xi_1.
			\end{aligned}
		\end{equation*}
		Consequently, returning to~\eqref{equation: family Y depends on others}, we obtain
		\begin{equation*}
			\exi{\Delta Y^\pom}
			\Xleq -\frac{\abs{\cF}}{H}\paren[\bigg]{\frac{\rho_\cF}{2}\xi_1-\eps^{1/2}\xi_1- \eps^{1/3}\xi_1}
			\leq 0,
		\end{equation*}
		which completes the proof.
	\end{proof}
	
	\subsubsection{Boundedness}\label{subsubsection: family boundedness}
	Here, we transfer the relevant results from Section~\ref{subsubsection: chain boundedness} for individual chains, namely Lemma~\ref{lemma: absolute change ladder not stopped}, Lemma~\ref{lemma: absolute change ladder} and Lemma~\ref{lemma: expected change ladder}, to branching families.
	
	\begin{lemma}\label{lemma: absolute change family not stopped}
		Let~$0\leq i_0\leq i\leq i^\star$,~$\pom\in\set{-,+}$  and~$\cX:=\set{i<\tau_\ccB\wedge\tau_{\ccB'}\wedge \tau_\frakC}$.
		Then,
		\begin{equation*}
			\abs{\Delta Y^\pom}\leq n^{\eps^3}\frac{\sum_{\frakb\in\frakB_\frakc^e} \phihat_{\frakb,I}(i_0)}{n\phat(i_0)^{\rho_\cF}}.
		\end{equation*}
	\end{lemma}
	\begin{proof}
		Combining Lemma~\ref{lemma: delta xi family} and Lemma~\ref{lemma: absolute change chain deviation}, we obtain
		\begin{equation*}
			\begin{aligned}
				\abs{\Delta Y^\pom }
				&\leq \paren[\Big]{\sum_{\frakb\in\frakB_\frakc^e}\abs{\Delta X_{\frakb,\psi}}}+\abs{\Delta\xi_1}
				\Xleq n^{\eps^4}\frac{\sum_{\frakb\in\frakB_\frakc^e}\phihat_{\frakb,I}(i_0)}{n\phat(i_0)^{\rho_\cF}}+\frac{\sum_{\frakb\in\frakB_\frakc^e}\phihat_{\frakb,I}}{H}\\
				&\leq n^{\eps^4}\frac{\sum_{\frakb\in\frakB_\frakc^e}\phihat_{\frakb,I}(i_0)}{n\phat(i_0)^{\rho_\cF}}+\frac{\sum_{\frakb\in\frakB_\frakc^e}\phihat_{\frakb,I}(i_0)}{H(i_0)}
			\end{aligned}
		\end{equation*} 
		With Lemma~\ref{lemma: zeta and H}, this completes the proof.
	\end{proof}
	
	\begin{lemma}\label{lemma: absolute change family}
		Let~$0\leq i_0\leq i\leq i^\star$ and~$\pom\in\set{-,+}$.
		Then,
		\begin{equation*}
			\abs{\Delta Z_{i_0}^\pom}\leq n^{\eps^3}\frac{\sum_{\frakb\in\frakB_\frakc^e} \phihat_{\frakb,I}(i_0)}{n\phat(i_0)^{\rho_\cF}}.
		\end{equation*}
	\end{lemma}
	\begin{proof}
		This is an immediate consequence of Lemma~\ref{lemma: absolute change family not stopped}.
	\end{proof}
	
	\begin{lemma}\label{lemma: expected change family}
		Let~$0\leq i_0\leq i^\star$ and~$\pom\in\set{-,+}$.
		Then,~$\sum_{i\geq i_0} \exi{ \abs{ \Delta Z_{i_0}^\pom } }\leq n^{\eps^3}\sum_{\frakb\in\frakB_\frakc^e}\phihat_{\frakb,I}$.
	\end{lemma}
	\begin{proof}
		Suppose that~$i_0\leq i< i^\star$ and let~$\cX:=\set{i<\tau_{\cH^*}\wedge\tau_\ccB\wedge\tau_{\ccB'}\wedge  \tau_{\frakC}}$.
		We have~$\exi{\abs{\Delta Z_{i_0}^\pom}}=_{\cX^\comp}0$ and with Lemma~\ref{lemma: delta xi family}, Lemma~\ref{lemma: expected change chain deviation} and Lemma~\ref{lemma: edges of H}, we obtain
		\begin{equation*}
			\begin{aligned}
				\exi{\abs{ \Delta Z_{i_0}^\pom }}
				&\leq \exi{ \abs{\Delta Y^\pom} }
				\leq \paren[\Big]{\sum_{\frakb\in\frakB_\frakc^e}\exi{\abs{ \Delta X_{\frakb,\psi} }}}+\abs{\Delta \xi_1}
				\Xleq n^{\eps^4}\frac{\sum_{\frakb\in\frakB_\frakc^e} \phihat_{\frakb,I}}{n^k\phat}+\frac{\sum_{\frakb\in\frakB_\frakc^e}\phihat_{\frakb,I}}{H}\\
				&\Xleq n^{\eps^3}\frac{\sum_{\frakb\in\frakB_\frakc^e} \phihat_{\frakb,I}}{n^k\phat}
				\leq n^{\eps^3}\frac{\sum_{\frakb\in\frakB_\frakc^e} \phihat_{\frakb,I}(i_0)}{n^k\phat(i_0)}.
			\end{aligned}
		\end{equation*}
		Thus,
		\begin{equation*}
			\sum_{i\geq i_0} \exi{ \abs{ \Delta Z_{i_0}^\pom } }
			=\sum_{i_0\leq i\leq i^\star-1} \exi{\abs{ \Delta Z_{i_0}^\pom }}
			\leq (i^\star-i_0)n^{\eps^3}\frac{\sum_{\frakb\in\frakB_\frakc^e} \phihat_{\frakb,I}(i_0)}{n^k\phat(i_0)}.
		\end{equation*}
		Since
		\begin{equation*}
			i^\star-i_0
			\leq \frac{\theta n^k}{\abs{\cF}k!}-i_0
			=\frac{n^k\phat(i_0)}{\abs{\cF}k!}
			\leq n^k\phat(i_0),
		\end{equation*}
		this completes the proof.
	\end{proof}

	\subsubsection{Supermartingale argument}\label{subsubsection: family concentration}
	This section follows a similar structure as Section~\ref{subsubsection: chain concentration}.
	Lemma~\ref{lemma: initial error family} is the final ingredient that we use for our application of Lemma~\ref{lemma: freedman} in the proof of Lemma~\ref{lemma: control family} where we show that the probabilities of the events on the right in Observation~\ref{observation: family critical times} are indeed small.
	
	\begin{lemma}\label{lemma: initial error family}
		Let~$\pom\in\set{-,+}$.
		Then,~$Z^\pom_{\sigma^\pom}(\sigma^\pom)\leq -\delta^2\xi_1(\sigma^\pom)$.
	\end{lemma}
	
	\begin{proof}
		Together with Lemma~\ref{lemma: ladder subextension density}, Lemma~\ref{lemma: initially good} implies~$\tau^\star\geq 1$ and~$\pom X_{\frakc,\psi}^e(0)< \xi_0(0)$, so we have~$\sigma^\pom\geq 1$.
		Thus, by definition of~$\sigma^\pom$, for~$i:=\sigma^\pom-1$, we have~$\pom X_{\frakc,\psi}^e\leq \xi_0$ and thus
		\begin{equation*}
			Z_i^\pom=\pom X_{\frakc,\psi}^e-\xi_1\leq -\delta \xi_1.
		\end{equation*}
		Furthermore, since~$\sigma^\pom\leq \tau_{\ccB}\wedge \tau_{\ccB'}\wedge \tau_\frakC$, we may apply Lemma~\ref{lemma: absolute change family not stopped} to obtain
		\begin{equation*}
			Z_{\sigma^\pom}^\pom(\sigma^\pom)
			=Z_i^\pom +\Delta Y^\pom
			\leq Z_i^\pom + \delta^2 \xi_1
			\leq -\delta \xi_1 + \delta^2 \xi_1
			\leq -\delta^2\xi_1.
		\end{equation*}
		Since Lemma~\ref{lemma: chains are not trivial} entails~$\Delta\xi_1\leq 0$, this completes the proof.
	\end{proof}
	
	\begin{lemma}\label{lemma: control family}
		$\pr{\tau_{\frakB}\leq \tau^\star\wedge i^\star}\leq \exp(-n^{\eps^3})$.
	\end{lemma}
	
	\begin{proof}
		Considering Observation~\ref{observation: family individual}, it suffices to show that
		\begin{equation*}
			\pr{\tau\leq \tau^\star\wedge i^\star}\leq \exp(-n^{2\eps^3}).
		\end{equation*}
		Hence, by Observation~\ref{observation: family critical times}, is suffices to show that for~$\pom\in\set{-,+}$, we have
		\begin{equation*}
			\pr{ Z^\pom_{\sigma^\pom}(i^\star)>0 }\leq \exp(-n^{3\eps^3}).
		\end{equation*}
		Due to Lemma~\ref{lemma: initial error ladder}, we have
		\begin{equation*}
			\pr{ Z_{\sigma^\pom}^\pom(i^\star)>0 }
			\leq \pr{ Z_{\sigma^\pom}^\pom(i^\star) -Z_{\sigma^\pom}^\pom(\sigma^\pom) > \delta^2\xi_1(\sigma^\pom) }
			\leq \sum_{0\leq i\leq i^\star}\pr{ Z_{i}^\pom(i^\star) -Z_{i}^\pom > \delta^2\xi_1 }.
		\end{equation*}
		Thus, for~$0\leq i\leq i^\star$, it suffices to obtain
		\begin{equation*}
			\pr{ Z_{i}^\pom(i^\star) -Z_{i}^\pom > \delta^2\xi_1 }\leq \exp(-n^{4\eps^3}).
		\end{equation*}
		We show that this bound is a consequence of Freedman's inequality for supermartingales.
		
		Let us turn to the details.
		Lemma~\ref{lemma: family trend} shows that~$Z^\pom_i(i),Z^\pom_i(i+1),\ldots$ is a supermartingale, while Lemma~\ref{lemma: absolute change family} provides the bound~$\abs{\Delta Z^\pom_i(j)}\leq n^{\eps^3}\paren{\sum_{\frakb\in\frakB_\frakc^e} \phihat_{\frakb,I}}/(n\phat^{\rho_\cF})$ for all~$j\geq i$ and Lemma~\ref{lemma: expected change family} provides the bound~$\sum_{j\geq i} \ex[][\bE_j]{ \abs{\Delta Z^\pom_i(j)} }\leq n^{\eps^3}\sum_{\frakb\in\frakB_\frakc^e} \phihat_{\frakb,I}$.
		Hence, we may apply Lemma~\ref{lemma: freedman} such that using Lemma~\ref{lemma: error parameter bound}, we obtain
		\begin{align*}
			\pri{Z^\pom_i(i^\star)>0}
			&\Xleq \exp\paren[\Bigg]{ -\frac{\delta^4\xi_1^2}{2n^{\eps^3}\frac{\sum_{\frakb\in\frakB_\frakc^e} \phihat_{\frakb,I}}{n\phat^{\rho_\cF}}(\xi_1+n^{\eps^3}\sum_{\frakb\in\frakB_\frakc^e} \phihat_{\frakb,I})} }\\
			&\leq \exp\paren[\bigg]{ -\frac{\delta^4\xi_1^2n\phat^{\rho_\cF}}{4n^{2\eps^3}\paren{\sum_{\frakb\in\frakB_\frakc^e} \phihat_{\frakb,I}}^2}}
			\leq \exp\paren[\bigg]{ -\frac{\delta^5 n^{2\eps^2}}{4n^{2\eps^3}}}
			\leq \exp(-n^{4\eps^3}),
		\end{align*}
		which completes the proof.
	\end{proof}

\section{Proof of Theorem~\ref{theorem: technical bounds}}\label{section: proof}
	In this section, we combine Lemma~\ref{lemma: auxiliary control} with Lemma~\ref{lemma: control ladder} and Lemma~\ref{lemma: control family} to conclude that typically, we have~$i^\star<\tau^\star$, see Lemma~\ref{lemma: control everything}, which in turn yields a proof for Theorem~\ref{theorem: technical bounds}.
	
	\begin{lemma}\label{lemma: control everything}
		$\pr{\tau^\star\leq i^\star}\leq \exp(-\log n)^{4/3})$.
	\end{lemma}
	
	\begin{proof}
		Using Lemma~\ref{lemma: auxiliary control}, Lemma~\ref{lemma: control ladder} and Lemma~\ref{lemma: control family}, we obtain
		\begin{equation*}
			\begin{aligned}
				\pr{\tau^\star\leq i^\star}
				&\leq \sum_{\tau\in \set{\tau_{\cH^*}, \tau_\ccB,\tau_{\ccB'},\tau_\frakC,\tau_\frakB}} \pr{\tau\leq \tau^\star\wedge i^\star}\\
				&\leq \paren[\Big]{\sum_{\tau\in\set{ \tau_{\cH^*}, \tau_\ccB,\tau_{\ccB'} }}\pr{ \tau\leq \tautilde^\star\wedge i^\star }}
				+\pr{ \tau_{\frakC}\leq \tautilde_\frakC^\star\wedge i^\star }
				+\pr{ \tau_{\frakB}\leq \tau^\star\wedge i^\star }\\
				&\leq 5\exp(-(\log n)^{3/2}),
			\end{aligned}
		\end{equation*}
		which completes the proof.
	\end{proof}

	\begin{proof}[Proof of Theorem~\ref{theorem: technical bounds}]
		Let~$\cX:=\set{i^\star <\tau^\star}$,~$i:=i^\star$ and~$\theta^\star:=\phat$.
		By Lemma~\ref{lemma: control everything}, it suffices to show that if~$\cX$ occurs, then~$\cH$ is~$(4m,n^\eps)$-bounded,~$\cF$-populated,~$k'$-populated for all~$1\leq k'\leq k-1/\rho_\cF$ and has~$n^{k-1/\rho_\cF+\eps}/k!$ edges.

		Due to~$\cX\subseteq\set{i^\star<\tau_{\ccB}\wedge\tau_{\ccB'}}$, for all strictly balanced~$k$-templates~$(\cA,I)$ with~$\abs{V_\cA}\leq 1/\eps^4$ and all~$\psi\colon I\injection V_\cH$, Lemma~\ref{lemma: weak strictly balanced bound} yields
		\begin{equation*}
			\Phi_{\cA,I}
			\Xleq (1+\log n)^{\alpha_{\cA,I}}\max\set{1,\phihat_{\cA,I}}
			\leq n^\eps\max\set{1,n^{\abs{V_\cA}-\abs{I}}(\theta^\star)^{\abs{\cA}-\abs{\cA[I]}}}
		\end{equation*}
		Thus,~$H$ is~$(4m,n^\eps)$-bounded if~$\cX$ occurs.
		
		Furthermore, due to~$\cX\subseteq\set{i^\star<\tau_\ccF}$, for all~$e\in\cH$, Lemma~\ref{lemma: star degrees} entails
		\begin{equation*}
			d_{\cH^*}(e)
			\Xgeq \frac{\abs{\cF}k!\,\phihat_{\cF,f}}{2\aut(\cF)}
			=\frac{\abs{\cF}k!\, n^{\eps(\abs{\cF}-1)}}{2\aut(\cF)}
			\geq n^{\eps^2},
		\end{equation*}
		which shows that~$\cH$ is~$\cF$-populated if~$\cX$ occurs.
		
		Let~$1\leq k'\leq k-1/\rho_\cF$ and let~$(\cA,I)$ denote a~$k$-template with~$\abs{V_\cA}=k$,~$\abs{\cA}=1$ and~$\abs{I}=k'$.
		Fix a~$k'$-set~$U\subseteq V_\cH$ and~$\psi\colon I\injection U$.
		We have~$\rho_{\cA,I}\leq \rho_\cF$, so for all~$j\leq i$, Lemma~\ref{lemma: bounds of zeta} implies
		\begin{equation*}
			\phihat_{\cA,I}(j)
			\geq n^{k-k'}\phat^{\rho_\cF(k-k')}
			=n^{\eps\rho_\cF(k-k')}
			\geq n^{\eps^2}
			> \zeta^{-\delta^{1/2}}
		\end{equation*}
		and hence~$i^\star< i_{\cA,I}^{\delta^{1/2}}$.
		Thus, due to~$\cX\subseteq\set{i^\star<\tau_\ccB}$, we obtain
		\begin{equation*}
			d_\cH(U)
			=\frac{\Phi_{\cA,\psi}}{(k-k')!}
			\Xgeq \eps \phihat_{\cA,I}
			\geq n^{\eps^2}, 
		\end{equation*}
		which shows that~$\cH$ is~$k'$-populated if~$\cX$-occurs.

		Finally, since~$\cX\subseteq\set{i^\star < \tau_\emptyset}$, Lemma~\ref{lemma: edges of H} yields~$H\Xeq \theta^\star n^k/k!=n^{k-1/\rho_\cF+\eps}/k!$.
	\end{proof}

	\section{The sparse setting}\label{section: sparse}
	\gladd{realII}{rhoAI}{$\rho_{\cA,I}=\frac{\abs{\cA}-\abs{\cA[I]}}{\abs{V_\cA}-\abs{I}}$}%
	The first part of our argumentation is now complete and as mentioned in Section~\ref{section: outline}, we now focus on the second part.
	We first describe the setting for this section and subsequent sections and remark that from now on, we redefine some symbols that appeared in the first part.
	Let~$k\geq 2$ and fix a~$k$-graph~$\cF$ on~$m$\gladd{realII}{m}{$m=\abs{V_\cF}$} vertices with~$\abs{\cF}\geq 2$ and~$k$-density~$\rho_\cF$\gladd{realII}{rhoF}{$\rho_\cF=\frac{\abs{\cF}-1}{\abs{V_\cF}-k}$} that is not a matching such that~$(\cF,f)$ is strictly balanced for all~$f\in\cF$.
	Suppose that~$0<\eps<1$ is sufficiently small in terms of~$1/m$ and that~$n$ is sufficiently large in terms of~$1/\eps$.
	Suppose that~$\cH(0)$ is a~$k$-graph on~$n$ vertices with~$n^{k-1/\rho_\cF-\eps^4}\leq \abs{\cH(0)}\leq n^{k-1/\rho_\cF+\eps^4}$ that is~$(4m,n^{\eps^4})$-bounded\footnote{%
		Note that for~$\cF$, besides strictly~$k$-balanced~$k$-graphs, this setup also allows~$k$-graphs as in Theorem~\ref{theorem: sparse cherries}.
		We choose this slightly more general setting as this makes many of the results we present available for a proof of Theorem~\ref{theorem: sparse cherries} while only requiring very minor adaptations.
	}.
	
	For the second part, that is for the proof of Theorems~\ref{theorem: sparse} and~\ref{theorem: sparse cherries}, one key idea is the identification of substructures in $\cH(0)$ whose existence enforces the existence of edges that are no longer contained in a copy of~$\cF$ with a substantial probability.
	We show that there is a sufficiently large subset of these substructures whose members are far apart from each other and hence act, to a large extent, independently.
	We employ a concentration inequality to verify that a substantial number of these substructures indeed give rise to edges that are no longer contained in a copy of~$\cF$ and hence remain until the termination of the process.
	
	On a very high level, similar ideas have also been utilized by Bohman, Frieze and Lubetzky for determining the number of remaining edges in the triangle-removal process (starting at $K_n$), see~\cite[Section~6]{BFL:15}.
	In our significantly more general setting however, we require additional insights concerning the distribution of copies of~$\cF$ in~$\cH(0)$.
	Notably, while in the special case where~$\cF$ is a triangle, two distinct copies of~$\cF$ that both contain an edge~$e$ cannot overlap outside~$e$, such overlaps can exist in general.
	However, since~$(\cF,f)$ is strictly balanced for all~$f\in\cF$, if two copies of $\cF$, both containing an edge $e$, overlap outside $e$, then their union forms a~$k$-graph with~$k$-density greater than~$\rho_\cF$.
	As a crucial step in our proof, we utilize this to show that certain substructures consisting of copies of~$\cF$ barely exist in the sense that we obtain a strong upper bound on the number of such structures.

	The remainder of the paper is organized as follows.
	In Section~\ref{section: gluing}, we prove several structural results which are important for the following parts.
	This includes properties of the aforementioned substructures that yield the edges that still remain at the end of the process.
	In Section~\ref{section: counting}, we obtain an upper bound on the number of remaining copies that holds well beyond the point where we would expect the process to terminate (this general idea is taken from~\cite{BFL:15}). 
	To this end, we again employ an approach that resembles the \emph{differential equation method} or more specifically the \emph{critical interval method}.
	
	Combining the structural results from Section~\ref{section: gluing} and the upper bound on the number of edges at a very late time in the process obtained in Section~\ref{section: counting}, we finally prove Theorem~\ref{theorem: sparse} in Section~\ref{section: isolation argument}.
	As mentioned above, here the idea is to identify certain configurations that have to appear frequently before the process terminates and that with sufficiently large probability lead to edges that remain in the hypergraph until termination.
	Compared to the (in spirit) similar argument in~\cite[Section~6]{BFL:15} here the (involved) insights from Section~\ref{section: gluing} replace properties that are obvious in the triangle case.
	
	For Theorem~\ref{theorem: sparse cherries}, one may argue very similarly, however, the structures that in the end enforce the existence of edges that remain until termination are different.
	In more detail, to obtain Theorem~\ref{theorem: sparse cherries}, parts of the argumentation in Section~\ref{section: gluing} and the key structures considered in Section~\ref{section: isolation argument} need to be replaced but the results from Section~\ref{section: counting} remain valid and the high level structure of the proof remains the same.
	For completeness, we provide a full proof of Theorem~\ref{theorem: sparse cherries} in Appendix~\ref{appendix: cherries}.

	\section{Unions of strictly balanced hypergraphs}\label{section: gluing}
	
	In this section, as preparation for the arguments in subsequent sections, we gather some lemmas that provide further insight into the distribution of the copies of~$\cF$ in~$\cH$.
	First, we state several lemmas concerning the densities of substructures obtained as unions of~$k$-balanced~$k$-graphs (see Lemmas~\ref{lemma: balanced gluing rest density}--\ref{lemma: strictly balanced gluing rest density}).
	In particular, we are interested in structures that are in a sense cyclic, where formally for~$\ell\geq 2$, we say that a sequence~$\cA_1,\ldots,\cA_\ell$ of distinct~$k$-graphs forms a \emph{self-avoiding cyclic walk} if there exist distinct~$e_1,\ldots,e_\ell$ such that~$e_i\in\cA_i\cap \cA_{i+1}$ for all~$1\leq i\leq\ell$ with indices taken modulo~$\ell$.
	
	From the~$(4m,n^{\eps^4})$-boundedness of~$\cH(0)$, we then deduce Lemma~\ref{lemma: dense template chain} where for all~$k$-graphs~$\cA$ that satisfy a suitable density property, we bound the number~$\Phi_\cA$\gladd{realII}{PhiA}{$\Phi_\cA=\abs{\Phi_\cA^\sim }$} of injections~$\phi\colon V_\cA\injection V_{\cH}$ with~$\phi(e)\in \cH(0)$ for all~$e\in\cA$ where we set~$V_\cH:=V_{\cH(0)}$.
	
	Using~$\rho_\cF\geq 1/(k-1)$ (see Lemma~\ref{lemma: lower bound density}), the aforementioned density observations allow us to apply Lemma~\ref{lemma: dense template chain} to then obtain Lemma~\ref{lemma: few cyclic walks general} as an intermediate result and subsequently Lemma~\ref{lemma: few cyclic walks} which states that~$\cH(0)$ contains only few cyclic structures formed by copies of~$\cF$.
	This turns out to be a crucial observation concerning the structure of~$\cH(0)$ that we require in two separate places in our argumentation (namely in the proofs of Lemma~\ref{lemma: sparse copy trend} and Lemma~\ref{lemma: remain configuration bound}).
	
	As these objects frequently appear in our proofs, we generalize the notation~$\Phi_\cA$ as follows.
	For a template~$(\cA,I)$ and~$\psi\colon I\injection V_\cH$, we use~$\Phi_{\cA,\psi}^\sim$\gladd{constructionII}{PhiApsitilde}{$\Phi_{\cA,\psi}^\sim=\cset{ \phi\colon V_\cA\injection V_{\cH} }{ \restr{\phi}{I}=\psi\stand \phi(e)\in \cH(0) \stforall e\in\cA }$} to denote the set of injections~$\phi\colon V_\cA\injection V_{\cH}$ with~$\restr{\phi}{I}=\psi$ and~$\phi(e)\in \cH(0)$ for all~$e\in\cA\setminus\cA[I]$ and we set~$\Phi_{\cA,\psi}:=\abs{\Phi_{\cA,\psi}^\sim}$\gladd{realII}{PhiApsi}{$\Phi_{\cA,\psi}=\abs{\Phi_{\cA,\psi}^\sim}$}.
	Additionally, we define~$\Phi_{\cA}^\sim:=\Phi_{\cA,\psi}^\sim$\gladd{constructionII}{PhiAtilde}{$\Phi_\cA^\sim=\Phi_{\cA,\psi}^\sim$, where~$\psi\colon\emptyset\to V_\cH$} where~$\psi$ denotes the unique function from~$\emptyset$ to~$V_\cH$.
	Note that~$\Phi_\cA=\abs{\Phi_\cA^\sim}$.
	
	The bounds on~$\abs{\cH(0)}$ and the numbers of embeddings of strictly balanced templates into~$\cH(0)$ yield the following lemma.
	\begin{lemma}\label{lemma: bounds for H}
		Suppose that~$(\cA,I)$ is a strictly balanced~$k$-template with~$\abs{V_\cA}\leq 4m$ and let~$\psi\colon I\injection V_\cH$.
		Then,~$\Phi_{\cA,\psi}\leq n^{\eps^3}\cdot\max\set{1, n^{\abs{V_\cA}-\abs{I}-(\abs{\cA}-\abs{\cA[I]})/\rho_\cF} }$.
	\end{lemma}
	\begin{proof}
		We have~$\abs{\cH(0)}\leq n^{-1/\rho_\cF+2\eps^4}\cdot n^k/k!$, so since~$\cH(0)$ is~$(4m,n^{\eps^4})$-bounded, we obtain
		\begin{equation*}
			\Phi_{\cA,\psi}
			\leq n^{\eps^4}\cdot \max\set{ 1,n^{\abs{V_\cA}-\abs{I}} n^{(-1/\rho_\cF+2\eps^4)(\abs{\cA}-\abs{\cA[I]})} }
			\leq n^{\eps^3}\cdot \max\set{ 1,n^{\abs{V_\cA}-\abs{I}-(\abs{\cA}-\abs{\cA[I]})/\rho_\cF} },
		\end{equation*}
		which completes the proof.
	\end{proof}

	\begin{lemma}\label{lemma: balanced gluing rest density}
		Let~$\ell\geq 1$.
		Suppose that~$\cA_1,\ldots,\cA_\ell$ is a sequence of~$k$-balanced~$k$-graphs with~$k$-density at least~$\rho$.
		For~$1\leq i\leq\ell$, let~$\cS_i:=\cA_1+\ldots+\cA_i$.
		Suppose that for all~$2\leq i\leq\ell$, we have~$\cS_{i-1}\cap \cA_i\neq\emptyset$.
		Let~$\cS:=\cS_\ell$ and~$J\subsetneq V_{\cS}$ with~$\cS[J]\neq\emptyset$.
		Then,~$\rho_{\cS,J}\geq \rho$.
	\end{lemma}
	\begin{proof}
		By rearranging the elements of~$\cA_1,\ldots,\cA_\ell$ if necessary, we may assume that~$\cA_1[J]\neq\emptyset$.
		For~$1\leq i\leq\ell$, let
		\begin{equation*}
			\begin{gathered}
				U:=V_\cS\setminus J,\quad
				E:=\cS\setminus \cS[J],\quad
				W_{i-1}:=V_{\cA_1}\cup \ldots\cup V_{\cA_{i-1}},\\
				J_i:=(J\cup W_{i-1})\cap V_{\cA_i},\quad
				U_i:=V_{\cA_i}\setminus J_i,\quad
				E_i:=\cA_i\setminus \cA_i[J_i].
			\end{gathered}
		\end{equation*}
		Note that~$U=\bigcup_{1\leq i\leq\ell} U_i$ and~$U_i\cap U_j=\emptyset$ for all~$1\leq i<j\leq\ell$.
		Hence,~$\abs{U}=\sum_{1\leq i\leq\ell}\abs{U_i}$.
		Similarly, we have~$E\supseteq \bigcup_{1\leq i\leq \ell} E_i$ and~$E_i\cap E_j=\emptyset$ for all~$1\leq i<j\leq\ell$ and thus~$\abs{E}\geq \sum_{1\leq i\leq\ell}\abs{E_i}$.
		This yields
		\begin{equation*}
			\rho_{\cS,J}
			=\frac{\abs{E}}{\abs{U}}
			\geq \frac{ \sum_{1\leq i\leq\ell}\abs{E_i} }{ \sum_{1\leq i\leq\ell } \abs{U_i} }.
		\end{equation*}
		Let~$e_1\in\cA_1[J]$ and for~$2\leq i\leq\ell$, let~$e_i\in\cA_i\cap \cS_{i-1}$.
		For all~$1\leq i\leq\ell$, the extension~$(\cA_i,e_i)$ is balanced and has density at least~$\rho$, so due to~$e_i\subseteq J_i$, we obtain
		\begin{equation*}
			\abs{E_i}
			=\rho_{\cA_i,e_i}(\abs{V_{\cA_i}}-k)-\rho_{\cA_i[J_i],e_i}(\abs{J_i}-k)
			\geq \rho_{\cA_i,e_i}(\abs{V_{\cA_i}}-k)-\rho_{\cA_i,e_i}(\abs{J_i}-k)
			=\rho_{\cA_i,e_i}\abs{U_i}
			\geq \rho\abs{U_i}.
		\end{equation*}
		Hence, we obtain
		\begin{equation*}
			\rho_{\cS,J}
			\geq \frac{ \sum_{1\leq i\leq\ell}\rho\abs{U_i} }{ \sum_{1\leq i\leq\ell } \abs{U_i} }
			=\rho,
		\end{equation*}
		which completes the proof.
	\end{proof}
	
	\begin{lemma}\label{lemma: balanced gluing dense of strictly balanced}
		Let~$\ell\geq 1$.
		Suppose that~$\cA_1,\ldots,\cA_\ell$ is a sequence of~$k$-balanced~$k$-graphs with~$k$-density at least~$\rho$.
		For~$1\leq i\leq\ell$, let~$\cS_i:=\cA_1+\ldots+\cA_i$.
		Suppose that for all~$2\leq i\leq\ell$, we have~$\cA_i\cap \cS_{i-1}\neq\emptyset$.
		Let~$\cS:=\cS_\ell$.
		Then,~$\max_{\cB\subseteq \cS} \rho_{\cB,\emptyset}\geq \rho$ or~$(\cS,\emptyset)$ is strictly balanced.
	\end{lemma}
	\begin{proof}
		Suppose that~$\max_{\cB\subseteq \cS} \rho_{\cB,\emptyset}<\rho$.
		We show that then~$(\cS,\emptyset)$ is strictly balanced.
		To this end, consider~$(\cC,\emptyset)\subseteq (\cS,\emptyset)$ with~$V_\cC\neq\emptyset$ and~$\cC\neq\cS$.
		It suffices to show that~$\rho_{\cC,\emptyset}<\rho_{\cS,\emptyset}$.
		
		First, note that we may assume that~$\cC$ is an induced subgraph of~$\cS$ with non-empty edge set.
		By Lemma~\ref{lemma: balanced gluing rest density}, we have~$\rho_{\cS,V_\cC}\geq \rho$ and due to~$\max_{\cB\subseteq \cS} \rho_{\cB,\emptyset}<\rho$ furthermore~$\rho_{\cC,\emptyset}<\rho$.
		Hence~$\rho_{\cS,V_\cC}>\rho_{\cC,\emptyset}$.
		Thus,
		\begin{equation*}
			\rho_{\cS,\emptyset}
			=\frac{\abs{\cS}-\abs{\cS[V_\cC]}+\abs{\cC}}{\abs{V_\cS}}
			=\frac{ \rho_{\cS,V_\cC}(\abs{V_\cS}-\abs{V_\cC})+\rho_{\cC,\emptyset}\abs{V_\cC}}{\abs{V_\cS}}
			>\frac{ \rho_{\cC,\emptyset}(\abs{V_\cS}-\abs{V_\cC})+\rho_{\cC,\emptyset}\abs{V_\cC} }{\abs{V_\cS}}
			=\rho_{\cC,\emptyset},
		\end{equation*}
		which completes the proof.
	\end{proof}
	
	\begin{lemma}\label{lemma: minimal strictly balanced gluing rest density}
		Suppose that~$\cA_1,\ldots,\cA_\ell$ is a sequence of strictly~$k$-balanced~$k$-graphs with~$k$-density~$\rho$ that forms a self-avoiding cyclic walk such that no proper subsequence forms a self-avoiding cyclic walk.
		Let~$\cS:=\cA_1+\ldots+\cA_\ell$.
		Then, there exists~$e\in\cS$ such that~$\rho_{\cS,e}>\rho$.
	\end{lemma}
	\begin{proof}
		First note that since no proper subsequence of~$\cA_1,\ldots,\cA_\ell$ forms a self-avoiding cyclic walk, we have~$\cA_\ell\cap \cA_i=\emptyset$ for all~$2\leq i\leq\ell-2$.
		Furthermore if~$\ell\geq 3$, then again since no proper subsequence of~$\cA_1,\ldots,\cA_\ell$ forms a self-avoiding cyclic walk, for all~$1\leq i\leq\ell$, we have~$\abs{\cA_{i}\cap \cA_{i+1}}=1$ with indices taken modulo~$\ell$ (otherwise~$\cA_i,\cA_{i+1}$ forms a self-avoiding cyclic walk).
		Hence if~$\ell\geq 3$, then~$\abs{\cA_{\ell-1}\cap\cA_\ell}=\abs{\cA_\ell\cap \cA_1}=1$.
		
		For~$1\leq i\leq\ell$, let~$\cS_i:=\cA_1+\ldots+\cA_i$.
		If~$\ell\geq 3$, then, as a consequence of the above observations, due to~$\abs{\cA_1},\abs{\cA_\ell}\geq 3$, we have~$\cA_\ell\setminus\cS_{\ell-1}\neq\emptyset$ as well as~$\cA_1\setminus\cA_\ell\neq\emptyset$.
		If~$\ell=2$, then~$\cA_1\not\subseteq\cA_2$ and~$\cA_2\not\subseteq \cA_1$ and hence~$\cA_\ell\setminus\cS_{\ell-1}\neq\emptyset$ and~$\cA_1\setminus\cA_\ell\neq\emptyset$ follow from the fact that~$\cA_1$ and~$\cA_2$ are distinct strictly~$k$-balanced~$k$-graphs with the same~$k$-density.
		Let~$e\in\cA_1\setminus\cA_\ell$.
		As a consequence of Lemma~\ref{lemma: balanced gluing rest density}, we have~$\rho_{\cS_{\ell-1},e}\geq \rho$.
		Hence,
		\begin{equation*}
			\rho_{\cS,e}
			=\frac{\rho_{\cS_{\ell-1},e}(\abs{V_{\cS_{\ell-1}}}-k)+\abs{\cA_\ell\setminus\cS_{\ell-1}} }{\abs{V_{\cS_{\ell-1}}}-k+\abs{V_{\cA_{\ell}}\setminus V_{\cS_{\ell-1}}}}
			\geq \frac{\rho(\abs{V_{\cS_{\ell-1}}}-k)+\abs{\cA_\ell\setminus\cS_{\ell-1}} }{\abs{V_{\cS_{\ell-1}}}-k+\abs{V_{\cA_{\ell}}\setminus V_{\cS_{\ell-1}}}}.
		\end{equation*}
		Thus, it suffices to show that~$\abs{\cA_\ell\setminus\cS_{\ell-1}}>\rho \abs{V_{\cA_{\ell}}\setminus V_{\cS_{\ell-1}}}$.
		Due to~$\cA_\ell\setminus\cS_{\ell-1}\neq\emptyset$, the inequality holds if~$\abs{V_{\cA_{\ell}}\setminus V_{\cS_{\ell-1}}}=0$, so we may assume that~$V_{\cA_{\ell}}\setminus V_{\cS_{\ell-1}}\neq\emptyset$.
		Since~$\cA_1,\ldots,\cA_\ell$ forms a self-avoiding cyclic walk, there exist distinct~$e_1,e_2\in\cS_{\ell-1}\cap\cA_{\ell}$, so in particular, we have~$e_1\subsetneq V_{\cS_{\ell-1}} \cap V_{\cA_\ell}\subsetneq V_{\cA_\ell}$.
		The template~$(\cA_\ell,e_1)$ is strictly balanced, so we obtain
		\begin{equation*}
			\begin{aligned}
				\abs{\cA_\ell\setminus\cS_{\ell-1}}
				&\geq \abs{\cA_\ell\setminus \cA_\ell[V_{\cS_{\ell-1}} \cap V_{\cA_\ell}] }
				=\rho(\abs{V_{\cA_\ell}}-k)-\rho_{\cA_\ell[V_{\cS_{\ell-1}}\cap V_{\cA_{\ell}}],e_1}(\abs{V_{\cS_{\ell-1}}\cap V_{\cA_{\ell}}}-k)\\
				&>\rho(\abs{V_{\cA_\ell}}-k)-\rho(\abs{V_{\cS_{\ell-1}}\cap V_{\cA_{\ell}}}-k)
				=\rho \abs{V_{\cA_\ell}\setminus V_{\cS_{\ell-1}}},
			\end{aligned}
		\end{equation*}
		which completes the proof.
	\end{proof}
	
	\begin{lemma}\label{lemma: strictly balanced gluing rest density}
		Suppose that~$\cA_1,\ldots,\cA_\ell$ is a sequence of strictly~$k$-balanced~$k$-graphs with~$k$-density~$\rho$ that forms a self-avoiding cyclic walk.
		Let~$\cS:=\cA_1+\ldots+\cA_\ell$.
		Then, there exists~$e\in\cS$ such that~$\rho_{\cS,e}>\rho$.
	\end{lemma}
	\begin{proof}
		Consider a subsequence~$\cA_{i_1},\ldots,\cA_{i_{\ell'}}$ of~$\cA_1,\ldots,\cA_\ell$ that forms a self-avoiding cyclic walk such that no proper subsequence forms a self-avoiding cyclic walk.
		Let~$\cS':=\cA_{i_1}+\ldots+\cA_{i_{\ell'}}$.
		By Lemma~\ref{lemma: minimal strictly balanced gluing rest density}, there exists~$e\in\cS'$ such that~$\rho_{\cS',e}>\rho$ and by Lemma~\ref{lemma: balanced gluing rest density}, if~$V_{\cS'}\subsetneq V_\cS$, then~$\rho_{\cS,V_{\cS'}}\geq\rho$.
		This yields
		\begin{equation*}
			\begin{aligned}
				\rho_{\cS,e}
				&=\frac{ \rho_{\cS',e}(\abs{ V_{\cS'} }-k)+\abs{\cS\setminus\cS'} }{\abs{V_{\cS}}-k}
				\geq \frac{ \rho_{\cS',e}(\abs{ V_{\cS'} }-k)+\rho_{\cS,V_{\cS'}}( \abs{V_{\cS}}-\abs{V_{\cS'}} ) }{\abs{V_{\cS}}-k}\\
				&>\frac{ \rho(\abs{ V_{\cS'} }-k)+\rho( \abs{V_{\cS}}-\abs{V_{\cS'}} ) }{\abs{V_{\cS}}-k}
				=\rho,
			\end{aligned}
		\end{equation*}
		which completes the proof.
	\end{proof}
	
	\begin{lemma}\label{lemma: dense template chain}
		Suppose that~$(\cA,I)$ is a~$k$-template with~$\abs{V_\cA}\leq 1/\eps$ and~$\rho_{\cA,J}\geq \rho_\cF$ for all~$I\subseteq J\subsetneq V_\cA$.
		Let~$\psi\colon I\injection V_\cH$.
		Then,~$\Phi_{\cA,\psi}\leq n^{\eps^2}$.
	\end{lemma}
	\begin{proof}
		We use induction on~$\abs{V_\cA}-\abs{I}$ to show that
		\begin{equation}\label{equation: high density induction}
			\Phi_{\cA,\psi}\leq n^{\eps^{3}(\abs{V_\cA}-\abs{I})}.
		\end{equation}
		Then, since~$\abs{V_\cA}\leq 1/\eps$, the statement follows.
		
		If~$\abs{V_\cA}-\abs{I}=0$, then~$\Phi_{\cA,\psi}=1=\phihat_{\cA,I}$.
		Let~$\ell\geq 1$ and suppose that~\eqref{equation: high density induction} holds if~$\abs{V_\cA}-\abs{I}\leq\ell-1$.
		Suppose that~$\abs{V_\cA}-\abs{I}=\ell$.
		Let~$I\subseteq U\subseteq V_\cA$ such that~$\rho_{\cA[U],I}$ is maximal and subject to this, that~$\abs{U}$ is minimal.
		Then,~$(\cA[U],I)$ is strictly balanced.
		Furthermore, we have~$\rho_{\cA[U],I}\geq\rho_{\cA,I}\geq\rho_\cF>0$ and hence~$U\neq I$.
		Note that
		\begin{equation}\label{equation: high density split}
			\Phi_{\cA,\psi}=\sum_{\phi\in\Phi_{\cA[U],\psi}^\sim} \Phi_{\cA,\phi}.
		\end{equation}
		We exploit the strict balancedness of~$(\cA[U],I)$ to bound~$\Phi_{\cA[U],\psi}$ and the induction hypothesis to bound~$\Phi_{\cA,\phi}$ for all~$\phi\in\Phi_{\cA[U],\psi}^\sim$.
		
		In detail, we argue as follows.
		Due to Lemma~\ref{lemma: bounds for H}, we have
		\begin{equation*}
			\Phi_{\cA[U],\psi}
			\leq n^{\eps^3}\cdot\max\set{1,n^{(1-\rho_{\cA[U],I}/\rho_\cF)(\abs{U}-\abs{I})}}
			= n^{\eps^3}.
		\end{equation*}
		Furthermore, for all~$\phi\in\Phi_{\cA[U],\psi}^\sim$, by induction hypothesis, we obtain
		\begin{equation*}
			\Phi_{\cA,\phi}\leq n^{\eps^3(\abs{V_\cA}-\abs{U})}.
		\end{equation*}
		Combining this with~\eqref{equation: high density split} yields
		\begin{equation*}
			\Phi_{\cA[U],\psi}
			\leq n^{\eps^{3}}\cdot n^{\eps^{3}(\abs{V_\cA}-\abs{U})}
			\leq n^{\eps^{3}(\abs{V_\cA}-\abs{I})},
		\end{equation*}
		which completes the proof.
	\end{proof}
	
	\begin{lemma}\label{lemma: lower bound density}
		$\rho_\cF\geq 1/(k-1)$.
	\end{lemma}
	\begin{proof}
		Since~$\cF$ is not a matching, there exist edges~$e_1,e_2\in\cF$ with~$e_1\cap e_2\neq\emptyset$.
		Let~$\cA$ denote the~$k$-graph with vertex-set~$e_1\cup e_2$ and edge-set~$\set{e_1,e_2}$.
		Since~$(\cF,e_1)$ is strictly balanced, we have
		\begin{equation*}
			\rho_\cF
			\geq\rho_{\cA,e_1}
			\geq \frac{1}{k-1},
		\end{equation*}
		which completes the proof
	\end{proof}
	
	\begin{lemma}\label{lemma: few cyclic walks general}
		Let~$\ell\leq 4$ and suppose that~$\cA_1,\ldots,\cA_\ell$ is a sequence of~$k$-balanced~$k$-graphs with~$k$-density at least~$\rho_\cF$ and at most~$m$ vertices each that forms a self-avoiding cyclic walk.
		Let~$\cS:=\cA_1+\ldots+\cA_\ell$ and suppose that there exists~$e\in\cS$ with~$\rho_{\cS,e}>\rho_\cF$.
		Then~$\Phi_{\cS}\leq n^{k-1/\rho_\cF-\eps^{1/7}}$.
	\end{lemma}
	\begin{proof}
		Based on Lemma~\ref{lemma: balanced gluing dense of strictly balanced}, we distinguish two cases: The first case, where~$\max_{\cB\subseteq \cS}\rho_{\cB,\emptyset}\geq \rho_\cF$ and the second case where~$(\cS,\emptyset)$ is strictly balanced.
		
		First, suppose that~$\max_{\cB\subseteq \cS}\rho_{\cB,\emptyset}\geq \rho_\cF$.
		From all~$(\cB',\emptyset)\subseteq(\cS,\emptyset)$ choose~$(\cB,\emptyset)$ such that~$\rho_{\cB,\emptyset}$ is maximal and subject to this, that~$\abs{V_\cB}$ is minimal.
		Then,~$(\cB,\emptyset)$ is strictly balanced and we have~$\rho_{\cB,\emptyset}\geq\rho_\cF$.
		Furthermore, we have
		\begin{equation}\label{equation: cyclic walk split}
			\Phi_{\cS}=\sum_{\phi\in \Phi_{\cB}^\sim} \Phi_{\cS,\phi}.
		\end{equation}
		For all~$\phi\in \Phi_{\cB}^\sim$, due to Lemma~\ref{lemma: balanced gluing rest density}, we may apply Lemma~\ref{lemma: dense template chain} to obtain~$\Phi_{\cS,\phi}\leq n^{\eps^2}$.
		Furthermore, due to Lemma~\ref{lemma: bounds for H}, we have
		\begin{equation*}
			\Phi_{\cB}
			\leq n^{\eps^3}\cdot\max\set{1,n^{(1-\rho_{\cB,\emptyset}/\rho_\cF)\abs{V_\cB}}}
			= n^{\eps^3}.
		\end{equation*}
		Returning to~\eqref{equation: cyclic walk split}, due to Lemma~\ref{lemma: lower bound density}, this yields
		\begin{equation*}
			\Phi_{\cS,\psi}
			\Xleq n^{\eps^2}\cdot n^{\eps^{3}}
			\leq n^{k-1/\rho_\cF-\eps^{1/7}},
		\end{equation*}
		and hence completes our analysis of the first case.
		
		We proceed with the second case.
		Hence, assume that~$(\cS,\emptyset)$ is strictly balanced and that~$\rho_{\cS,\emptyset}=\max_{\cB\subseteq \cS}\rho_{\cB,\emptyset}< \rho_\cF$.
		Then
		\begin{equation*}
			n^{\abs{V_\cS}-\abs{\cS}/\rho_\cF}
			=n^{(1-\rho_{\cS,\emptyset}/\rho_\cF)\abs{V_\cS}}
			\geq 1.
		\end{equation*}
		Thus, Lemma~\ref{lemma: bounds for H} entails
		\begin{equation*}
			\Phi_{\cS}
			\leq n^{\abs{V_\cS}-\abs{\cS}/\rho_\cF+\eps^3}
			= n^{k-1/\rho_\cF+\eps^3}\cdot n^{\abs{V_\cS}-k-(\abs{\cS}-1)/\rho_\cF}.
		\end{equation*}
		If there exists~$e\in\cS$ with~$\rho_{\cS,e}>\rho_\cF$, then since~$\ell\leq 4$ and~$\abs{V_{\cA_i}}\leq m$ for all~$1\leq i\leq\ell$ we have
		\begin{equation*}
			\rho_\cF+\eps^{1/8}<\rho_{\cS,e}=\frac{\abs{\cS}-1}{\abs{V_\cS}-k},
		\end{equation*}
		so we then obtain
		\begin{equation*}
			\Phi_{\cS}
			\leq n^{k-1/\rho_\cF+\eps^3}\cdot n^{\abs{V_\cS}-k-(\rho_\cF+\eps^{1/8})(\abs{V_\cS}-k)/\rho_\cF}
			<n^{k-1/\rho_\cF-\eps^{1/7}},
		\end{equation*}
		which completes the proof.
	\end{proof}
	
	\begin{lemma}\label{lemma: few cyclic walks}
		Let~$\ell\leq 4$ and suppose that~$\cF_1,\ldots,\cF_\ell$ is a sequence of copies of~$\cF$ that forms a self-avoiding cyclic walk.
		Let~$\cS:=\cA_1+\ldots+\cA_\ell$.
		If~$\abs{\cF}\geq 3$, that is if~$\cF$ is strictly~$k$-balanced, then~$\Phi_{\cS}\leq n^{k-1/\rho_\cF-\eps^{1/7}}$.
	\end{lemma}
	\begin{proof}
		Due to Lemma~\ref{lemma: strictly balanced gluing rest density}, this follows from Lemma~\ref{lemma: few cyclic walks general}.
	\end{proof}
	
	\section{Bounding the number of copies of~\texorpdfstring{$\cF$}{F}}\label{section: counting}
	We assume the setup described in Section~\ref{section: sparse} and, similarly as in Section~\ref{section: process}, we define~$\cH^*(0)$ to be the~$\abs{\cF}$-graph with vertex set~$\cH(0)$ whose edges are the edge sets of copies of~$\cF$ that are subgraphs of~$\cH(0)$.
	We now begin to analyze the~$\cF$-removal process formally again given by Algorithm~\ref{algorithm: removal}.
	Again, if the process fails to execute step~$i+1$ and instead terminates, that is if~$\cH^*(i)=\emptyset$, then, for~$j\geq i+1$, we set~$\cH^*(j):=\cH^*(i)$.
	For~$i\geq 1$, we define~$\cH(i)$,~$H^*(i)$,~$H(i)$\gladd{realRVII}{Hstar}{$H^*(i)=\abs{\cH^*(i)}$}\gladd{realRVII}{Hstar}{$H(i)=\abs{\cH(i)}$} and the filtration~$\frakF(0),\frakF(1),\ldots$ as in Section~\ref{section: process}.
	We again define the stopping time
	\begin{equation*}
		\tau_\emptyset:=\min\cset{ i\geq 0 }{ \cH^*(i)=\emptyset }.
	\end{equation*}\gladd{stoppingtimeII}{tauemptyset}{$\tau_\emptyset=\min\cset{ i\geq 0 }{ \cH^*(i)=\emptyset }$}%
	
	To prove Theorem~\ref{theorem: sparse}, in Section~\ref{section: isolation argument}, we show that the following theorem holds.
	
	\begin{theorem}\label{theorem: technical}
		If~$\abs{\cF}\geq 3$, then~$\pr{ H(\tau_\emptyset)\leq n^{k-1/\rho-\eps}  }\leq \exp(-n^{1/4})$.
	\end{theorem}
	
	For our proof of Theorem~\ref{theorem: sparse}, in addition to the structural insights about configurations consisting of copies that we may encounter in~$\cH(0)$, we crucially rely on an upper bound for the number of copies of~$\cF$ present in~$\cH(i)$ for~$i\geq 0$, which is the focus of this section.
	First, note that initially, we may bound the number of copies as follows.
	Let~$\theta:=k!\,H(0)/n^k$\gladd{realII}{theta}{$\theta=k!\,H(0)/n^k$}.
	
	\begin{lemma}\label{lemma: F embeddings bound}
		Let~$i\geq 0$ and~$e\in\cH$.
		Then,~$d_{\cH^*}(e)\leq n^{m-k+\eps^{7/2}}\theta^{\abs{\cF}-1}\leq n^{\eps^3}$.
	\end{lemma}
	\begin{proof}
		By our assumptions on~$\cH(0)$, we have~$n^{-1/\rho_\cF-\eps^4}\leq\theta\leq n^{-1/\rho_\cF+2\eps^4}$.
		Hence, arguing similarly as in the proof of Lemma~\ref{lemma: bounds for H}, we obtain
		\begin{equation*}
			\begin{aligned}
				d_{\cH^*}(e)
				&\leq d_{\cH^*(0)}(e)
				\leq \sum_{f\in\cF}\sum_{\psi\colon f\injection e} \Phi_{\cF,\psi}
				\leq \abs{\cF}k!\cdot n^{\eps^4}\cdot\max\set{ 1,n^{m-k}n^{(-1/\rho_\cF+2\eps^4)(\abs{\cF}-1)} }\\
				&=\abs{\cF}k!\cdot n^{\eps^4}\cdot n^{m-k} n^{(-1/\rho_\cF-\eps^4)(\abs{\cF}-1)}\cdot n^{3\eps^4(\abs{\cF}-1)}
				\leq n^{m-k+\eps^{7/2}}\theta^{\abs{\cF}-1}.
			\end{aligned}
		\end{equation*}
		Furthermore, again using~$\theta\leq n^{-1/\rho_\cF+2\eps^4}$, we obtain
		\begin{equation*}
			n^{m-k+\eps^{7/2}}\theta^{\abs{\cF}-1}
			\leq n^{\eps^{7/2}}\cdot n^{2\eps^4(\abs{\cF}-1)}
			\leq n^{\eps^3},
		\end{equation*}
		which completes the proof.
	\end{proof}
	
	\begin{lemma}\label{lemma: initial copy bound}
		$H^*(0)\leq n^{m+\eps^{7/2}}\theta^{\abs{\cF}}$.
	\end{lemma}
	\begin{proof}
		Using Lemma~\ref{lemma: F embeddings bound}, we obtain
		\begin{equation*}
			H^*(0)
			=\frac{1}{\abs{\cF}}\sum_{e\in\cH(0)} d_{\cH^*}(e)
			\leq \frac{\theta n^k}{\abs{\cF}k!} \cdot n^{m-k+\eps^{7/2}}\theta^{\abs{\cF}-1}
			\leq n^{m+\eps^{7/2}}\theta^{\abs{\cF}},
		\end{equation*}
		which completes the proof.
	\end{proof}

	\subsection{Heuristics}\label{subsection: heuristics}
	With the same justification as in Section~\ref{section: heuristics}, we again assume that typically, for all~$i\geq 0$, the edge set of~$\cH$ behaves essentially as if it was obtained by including every~$k$-set~$e\subseteq V_\cH$ independently at random with probability
	\begin{equation*}
		\phat(i):=\theta-\frac{\abs{\cF}k!\, i}{n^k}.
	\end{equation*}\gladd{trajectoryII}{phati}{$\phat(i)=\theta-\frac{\abs{\cF}k!\, i}{n^k}$}%
	We may guess deterministic upper bounds for these numbers of copies that we expect to typically hold as follows by considering the expected one-step changes of these numbers.
	Lemma~\ref{lemma: few cyclic walks} in particular shows that for almost all distinct edges~$e,f\in\cH(0)$, there exists at most one copy~$\cF'\subseteq \cH(0)$ of~$\cF$ with~$e,f\in\cF'$.
	Thus, for~$i\geq 0$, for the one-step change~$\Delta H^*$, we estimate
	\begin{equation*}
		\begin{aligned}
			\exi{\Delta H^*}
			&=-\sum_{\cF'\in\cH^*} \pr{\cF'\notin \cH^*(i+1)}
			\approx -\sum_{\cF'\in\cH^*}\frac{\paren[\big]{ \sum_{f\in \cF'} d_{\cH^*}(f) }-\abs{\cF}+1}{H^*}\\
			&=-\frac{\sum_{e\in \cH} d_{\cH^*}(e)^2}{H^*}+\abs{\cF}-1.
		\end{aligned}
	\end{equation*}
	Using convexity and~$H=\phat n^k/k!$, this leads us to expect
	\begin{equation*}
		\begin{aligned}
			\exi{\Delta H^*}
			&\leq -\frac{\paren[\big]{\sum_{e\in \cH} d_{\cH^*}(e)}^2}{H\cdot H^*}+\abs{\cF}-1
			=-\frac{\abs{\cF}^2 H^*}{H}+\abs{\cF}-1\\
			&= -\frac{\abs{\cF}^2k!\, H^*}{n^k\phat}+\abs{\cF}-1.
		\end{aligned}
	\end{equation*}
	Motivated by this, we aim to choose our deterministic upper bounds~$\hhat^*(0),\hhat^*(1),\ldots$ for the random variables~$H^*(0),H^*(1),\ldots$ such that, with some room to spare for estimation errors, they approximately satisfy
	\begin{equation*}
		\Delta\hhat^*\geq -\frac{\abs{\cF}^2k!\, \hhat^*}{n^k\phat}+\abs{\cF}-1.
	\end{equation*}
	By Lemma~\ref{lemma: initial copy bound}, initially, that is for~$i=0$, there are at most~$n^{m+\eps^3}\phat^{\abs{\cF}}$ copies of~$\cF$ in~$\cH$.
	With this initial condition, guided by the above intuition, for~$i\geq 0$, we set
	\begin{equation*}
		\hhat^*(i):=n^{m+\eps^3}\phat^{\abs{\cF}-\eps^3}+\frac{(\abs{\cF}-1)n^k\phat}{\abs{\cF}(\abs{\cF}-1-\eps^3)k!}.
	\end{equation*}\gladd{trajectoryII}{hhatstar}{$\hhat^*(i)=n^{m+\eps^3}\phat^{\abs{\cF}-\eps^3}+\frac{(\abs{\cF}-1)n^k\phat}{\abs{\cF}(\abs{\cF}-1-\eps^3)k!}$}%
	Observe that this expression is the sum of two parts where the second part is negligible up to step~$i$ where~$\phat\approx n^{-(m-k+\eps^3)/(\abs{\cF}-1-\eps^3)}$ and where then, the first part becomes negligible.
	For our argumentation, we focus our attention on the evolution of the process up to step~$i^\star$, where
	\begin{equation*}
		i^\star:=\frac{(\theta -n^{-1/\rho_\cF-\eps^2})n^k}{\abs{\cF}k!}.
	\end{equation*}\gladd{timeII}{istar}{$i^\star=\frac{(\theta -n^{-1/\rho_\cF-\eps^2})n^k}{\abs{\cF}k!}$}%
	Note that following the above heuristic, for all~$i\geq 0$ and~$e\in\cH$, up to constant factors, we would expect approximately~$n^{m-k}\phat(i)^{\abs{\cF}-1}$ copies of~$\cF$ in~$\cH$ that contain~$e$, which suggests that the process should terminate around the step~$i$ where~$\phat\approx n^{-1/\rho_\cF}$.
	Since~$i^\star$ lies beyond this step, an analysis up to step~$i^\star$ should suffice.
	
	\subsection{Formal setup}
	Formally, we argue similarly as in Sections~\ref{subsection: tracking chains} and~\ref{subsection: tracking families} and phrase our statement about the boundedness of~$H$ from above for~$0\leq i\leq i^\star$ in terms of the stopping time
	\begin{equation*}
		\tau^\star:=\min\cset{i\geq 0}{ H^*\geq \hhat^* }.
	\end{equation*}
	Our goal is to show that typically,~$i^\star <\tau^\star$.
	To this end, for a similar argumentation as in the aforementioned sections, for~$i\geq 0$, define the critical interval
	\begin{equation*}
		I(i):=[(1-\eps^4)\hhat^*,\hhat^*].
	\end{equation*}
	For~$i\geq 0$, let
	\begin{equation*}
		Y(i):=H^*-\hhat^*.
	\end{equation*}
	For~$i_0\geq 0$, define the stopping time
	\begin{equation*}
		\tau_{i_0}:=\min\cset{i\geq i_0}{H^*\notin I}
	\end{equation*}
	and for~$i\geq i_0$, let
	\begin{equation*}
		Z_{i_0}(i):=Y(i_0\vee (i\wedge \tau_{i_0}\wedge i^\star)).
	\end{equation*}
	Let
	\begin{equation*}
		\sigma:=\min\cset{j\geq 0}{ H^*\geq (1-\eps^4)\hhat^*\stforall j\leq i<\tau^\star\wedge i^\star }\leq \tau^\star\wedge i^\star.
	\end{equation*}
	With this setup, similarly as in Sections~\ref{subsection: tracking chains} and~\ref{subsection: tracking families}, it in fact suffices to consider the evolution of~$Z_\sigma(\sigma),Z_\sigma(\sigma+1),\ldots$.
	
	\begin{observation}\label{observation: critical times}
		$\set{\tau^\star\leq i^\star}\subseteq \set{Z_\sigma(i^\star)>0}$.
	\end{observation}
	
	We use Azuma's inequality below to show that the probability of the event on the right in Observation~\ref{observation: critical times} is sufficiently small.
	
	\begin{lemma}[Azuma's inequality]\label{lemma: azuma}
		Suppose that~$X(0),X(1),\ldots$ is a supermartingale with~$\abs{X(i+1)-X(i)}\leq a_i$ for all~$i\geq 0$.
		Then, for all~$i\geq 0$ and~$t>0$,
		\begin{equation*}
			\pr{X(i)- X(0)\geq t}\leq\exp\paren[\bigg]{-\frac{t^2}{2\sum_{0\leq j\leq i-1} a_j^2}}.
		\end{equation*}
	\end{lemma}
	
	Before we turn to verifying that the conditions for an application of Azuma's inequality in Sections~\ref{subsection: trend} and~\ref{subsection: boundedness} and applying the inequality in Section~\ref{subsection: supermartingale}, similarly as in Section~\ref{section: stopping times} however now for the sparse setting that we consider since  Section~\ref{section: sparse}, we gather some useful facts concerning key quantities defined up to this point.
	
	\begin{lemma}\label{lemma: bounds for phat}
		Let~$0\leq i\leq i^\star$.
		Then,~$n^{1-k-\eps^2}\leq n^{-1/\rho_\cF-\eps^2}\leq\phat\leq n^{-1/\rho_\cF+\eps^3}$
	\end{lemma}
	\begin{proof}
		We have~$n^{-1/\rho_\cF-\eps^2}=\phat(i^\star)\leq \phat\leq \phat(0)=\theta\leq n^{-1/\rho_\cF+\eps^3}$.
		With Lemma~\ref{lemma: lower bound density}, this completes the proof.
	\end{proof}
	
	\begin{lemma}\label{lemma: bounds for delta phat}
		Let~$0\leq i\leq i^\star$.
		Then,~$\phat(i+1)\geq (1-n^{-1/2})\phat$.
	\end{lemma}
	\begin{proof}
		Lemma~\ref{lemma: bounds for phat} implies%
		\begin{equation*}
			\phat(i+1)
			=\paren[\bigg]{ 1-\frac{\abs{\cF}k!}{n^k\phat} }\phat
			\geq \paren[\bigg]{ 1-\frac{\abs{\cF}k!}{n^{1-\eps^2}} }\phat
			\geq (1-n^{-1/2})\phat,
		\end{equation*}
		which completes the proof.
	\end{proof}
	
	\begin{lemma}\label{lemma: sparse edges of H}
		Let~$0\leq i\leq i^\star$ and~$\cX:=\set{i\leq \tau_\emptyset}$.
		Then,~$n^{1/2}\leq n^k\phat/k!\leq H\Xeq n^k\phat/k!$.
	\end{lemma}
	\begin{proof}
		Indeed, we have~$H\geq \theta n^k/k!-\abs{\cF}i\Xeq H$ and~$\theta n^k/k!-\abs{\cF}i=n^k\phat/k!$, so Lemma~\ref{lemma: bounds for phat} completes the proof.
	\end{proof}
	
	\subsection{Trend}\label{subsection: trend}
	Here, essentially following the argumentation in Section~\ref{subsection: heuristics}, we prove that for all~$i_0\geq 0$, the expected one-step changes of the process~$Z_{i_0}(i_0),Z_{i_0}(i_0+1),\ldots$ are non-positive.
	We bound the one-step changes of~$\hhat^*$ in Lemma~\ref{lemma: delta hhat star}, then we turn to the non-deterministic one-step changes of~$H^*$.
	Crucially, to see that for~$0\leq i\leq i^\star$, the expected one-step changes of~$H^*$ are at most those of~$\hhat^*$, which justifies our choice of~$\hhat^*$, we employ Lemma~\ref{lemma: few cyclic walks} in the proof of Lemma~\ref{lemma: sparse copy trend}.

	\begin{observation}\label{observation: hhat derivatives}
		Extend~$\phat$ and~$\hhat^*$ to continuous trajectories defined on the whole interval~$[0,i^\star+1]$ using the same expressions as above.
		Then, for~$x\in[0,i^\star+1]$,
		\begin{equation*}
			\begin{gathered}
				(\hhat^*)'(x)=-\frac{\abs{\cF}(\abs{\cF}-\eps^3)k!\,\hhat^*}{n^k\phat}+\abs{\cF}-1,\\	(\hhat^*)''(x)=\frac{\abs{\cF}^2(\abs{\cF}-\eps^3)(\abs{\cF}-1-\eps^3)(k!)^2n^{m+\eps^3}\phat^{\abs{\cF}-\eps^3}}{(n^k\phat)^2}.
			\end{gathered}
		\end{equation*}
	\end{observation}
	
	\begin{lemma}\label{lemma: delta hhat star}
		Let~$0\leq i\leq i^\star$ and~$\cX:=\set{i\leq \tau_\emptyset}$.
		Then,
		\begin{equation*}
			\Delta \hhat^*
			\Xgeq -\frac{\abs{\cF}(\abs{\cF}-\eps^3)\hhat^*}{H}+\abs{\cF}-1-\frac{n^{-\eps}\hhat^*}{H}.
		\end{equation*}
	\end{lemma}
	\begin{proof}
		This is a consequence of Taylor's theorem.
		
		In detail, we argue as follows.
		Together with Observation~\ref{observation: hhat derivatives}, Lemma~\ref{lemma: taylor} yields
		\begin{equation*}
			\Delta \hhat^*
			=-\frac{\abs{\cF}k!\,(\abs{\cF}-\eps^3)\hhat^*}{n^k\phat}+\abs{\cF}-1\pm \max_{x\in[i,i+1]} \frac{\abs{\cF}^4 (k!)^2 n^{m+\eps^3}\phat(x)^{\abs{\cF}-\eps^3}}{(n^k\phat(x))^2}.
		\end{equation*}
		We investigate the first term and the maximum separately.
		Lemma~\ref{lemma: sparse edges of H} yields
		\begin{equation*}
			\frac{\abs{\cF}k!\,(\abs{\cF}-\eps^3)\hhat^*}{n^k\phat}
			\Xeq \frac{\abs{\cF}(\abs{\cF}-\eps^3)\hhat^*}{H}.
		\end{equation*}
		Furthermore, using Lemma~\ref{lemma: bounds for delta phat} and Lemma~\ref{lemma: sparse edges of H}, we obtain
		\begin{equation*}
			\begin{aligned}
				\max_{x\in[i,i+1]} \frac{\abs{\cF}^4 (k!)^2 n^{m+\eps^3}\phat(x)^{\abs{\cF}-\eps^3}}{(n^k\phat(x))^2}
				&\leq \max_{x\in[i,i+1]} \frac{\abs{\cF}^4 (k!)^2 \hhat^*(x)}{(n^k\phat(x))^2}
				\leq \frac{2\abs{\cF}^4 (k!)^2 \hhat^*}{(n^k\phat)^2}\\
				&\Xeq \frac{2\abs{\cF}^4 \hhat^*}{H^2}
				\leq \frac{n^{-\eps}\hhat^*}{H},
			\end{aligned}
		\end{equation*}
		which completes the proof.
	\end{proof}

	\begin{lemma}\label{lemma: sparse copy trend}
		Let~$0\leq i\leq i^\star$.
		Let~$\cX:=\set{ i<\tau_i }$.
		Then,
		\begin{equation*}
			\exi{\Delta H^*}\Xleq -\frac{\abs{\cF}^2 H^*}{H}+\abs{\cF}-1+\frac{n^{-\eps}\hhat^*_+}{H}.
		\end{equation*}
	\end{lemma}
	
	\begin{proof}
		Let~$\ccF_2$ denote a collection of~$k$-graphs~$\cG$ with~$V_\cG\subseteq\set{1,\ldots,2m}$ such that for all copies~$\cF_1$ and~$\cF_2$ of~$\cF$ with~$2\leq \abs{\cF_1\cap\cF_2}\leq \abs{\cF}-1$, the collection~$\ccF_2$ contains a copy of~$\cF_1+\cF_2$ and that only contains copies of such~$k$-graphs.
		We have
		\begin{equation}\label{equation: expected one step}
			\begin{aligned}
				\exi{\Delta H^*}
				&\leq -\frac{1}{H^*}\sum_{\cF'\in\cH^*}\paren[\Big]{1+\sum_{e\in\cF'}(d_{\cH^*}(e)-1)-\sum_{\substack{e,f\in\cF'\\ e\neq f}}(d_{\cH^*}(ef)-1)}\\
				&=-\paren[\bigg]{\frac{1}{H^*}\sum_{e\in\cH} d_{\cH^*}(e)^2}+\paren[\bigg]{\frac{1}{H^*}\sum_{\cF'\in\cH^*}\sum_{\substack{ e,f\in\cF' \colon\\ e\neq f }}\sum_{\substack{ \cF''\in\cH^*\setminus\set{\cF'}\colon\\ e,f\in\cF'' }}1 }+\abs{\cF}-1\\
				&\leq -\paren[\bigg]{\frac{1}{H^*}\sum_{e\in\cH} d_{\cH^*}(e)^2}+\paren[\bigg]{\frac{2\abs{\cF}^2}{H^*}\sum_{\cG\in\ccF_2 } \Phi_{\cG} }+\abs{\cF}-1.
			\end{aligned}
		\end{equation}
		We investigate the first two terms separately.
		
		For the first term, using convexity, we obtain
		\begin{equation}\label{equation: expected one step first}
			\frac{1}{H^*}\sum_{e\in\cH} d_{\cH^*}(e)^2
			\geq \frac{1}{HH^*}\paren[\Big]{\sum_{e\in\cH} d_{\cH^*}(e)}^2
			= \frac{\abs{\cF}^2 H^*}{H}.
		\end{equation}
		
		Let us now consider the second term.
		If~$\abs{\cF}=2$, then~$\ccF_2=\emptyset$ and otherwise, for all~$\cG\in\ccF_2$,
		Lemma~\ref{lemma: few cyclic walks} together with Lemma~\ref{lemma: bounds for phat} and Lemma~\ref{lemma: sparse edges of H} entails
		\begin{equation*}
			\begin{aligned}
				\Phi_{\cG}
				&\leq n^{k-1/\rho_\cF-\eps^{1/7}}
				\leq n^{k-1/\rho_\cF-\eps^{1/6}}(n^{1/\rho_\cF}\phat)^{2(\abs{\cF}-1)}\\
				&\leq n^{-\eps^{1/5}}\cdot n^k\phat\cdot n^{2(m-k)}\phat^{2(\abs{\cF}-1)}
				\leq \frac{n^{-\eps^{1/4}} (\hhat^*)^2 }{n^k\phat}
				\Xleq \frac{n^{-\eps^{1/3}} (\hhat^*)^2 }{H}
				\Xleq \frac{n^{-\eps^{1/2}} H^*\hhat^* }{H}.
			\end{aligned}
		\end{equation*}
		Thus,
		\begin{equation}\label{equation: expected one step second}
			\frac{2\abs{\cF}^2}{H^*}\sum_{\cG\in\ccF_2 } \Phi_{\cG}
			\leq \frac{n^{-\eps}\hhat^*}{H}.
		\end{equation}
		Combining~\eqref{equation: expected one step first} and~\eqref{equation: expected one step second} with~\eqref{equation: expected one step} yields the desired upper bound for~$\exi{\Delta H^*}$.
	\end{proof}
		
	\begin{lemma}\label{lemma: copy deviation is supermartingale}
		Let~$0\leq i_0\leq i$.
		Then,~$\exi{\Delta Z_{i_0}}\leq 0$.
	\end{lemma}
	\begin{proof}
		Suppose that~$i< i^\star$ and let~$\cX:=\set{ i<\tau_{i_0} }$.
		We have~$\exi{\Delta Z_{i_0}}=_{\cX^\comp}0$ and~$\exi{\Delta Z_{i_0}}\Xeq \exi{\Delta Y}$, so it suffices to obtain~$\exi{\Delta Y}\Xleq 0$.
		Combining Lemma~\ref{lemma: delta hhat star} with Lemma~\ref{lemma: sparse copy trend}, we have
		\begin{equation*}
			\begin{aligned}
				\exi{\Delta Y}
				&\Xleq -\frac{\abs{\cF}}{H}(\abs{\cF}H^*-(\abs{\cF}-\eps^3)\hhat^*)+\frac{2n^{-\eps}\hhat^*}{H}\\
				&\Xleq -\frac{\abs{\cF}}{H}(\abs{\cF}(1-\eps^4)\hhat^*-(\abs{\cF}-\eps^3)\hhat^*)+\frac{2n^{-\eps}\hhat^*}{H}
				\leq -\frac{\eps^4\abs{\cF}\hhat^*}{H}+\frac{2n^{-\eps}\hhat^*}{H}
				\leq 0,
			\end{aligned}
		\end{equation*}
		which completes the proof.
	\end{proof}
	\subsection{Boundedness}\label{subsection: boundedness}
	For our application of Azuma's inequality, it suffices to obtain suitable bounds for the absolute one-step changes of the processes~$Y(0), Y(1),\ldots $ and~$Z_{i_0}(i_0),Z_{i_0}(i_0+1),\ldots$.
	Furthermore, crude upper bounds that we obtain as an immediate consequence of the previously gained insights concerning the distribution of the copies of~$\cF$ within~$\cH(0)$ suffice.
	
	\begin{lemma}\label{lemma: sparse absolute change copies not stopped}
		Let~$0 \leq i\leq i^\star$.
		Then,~$\abs{\Delta Y}\leq n^{\eps}$.
	\end{lemma}
	\begin{proof}
		From Lemma~\ref{lemma: F embeddings bound}, Lemma~\ref{lemma: delta hhat star}, Lemma~\ref{lemma: bounds for phat}, Lemma~\ref{lemma: sparse edges of H} and the second inequality in Lemma~\ref{lemma: F embeddings bound}, we obtain
		\begin{equation*}
			\begin{aligned}
				\abs{\Delta Y}
				&\leq \abs{\Delta H^*}+\abs{\Delta \hhat^*}
				\leq \paren[\Big]{\sum_{e\in \cF_0(i+1)}d_{\cH^*}(e)}-\Delta\hhat^*
				\leq \abs{\cF}n^{\eps^3}+\abs{\cF}+\frac{2\abs{\cF}^2\hhat^*}{H}\\
				&\leq 2\abs{\cF}n^{\eps^3}+2\abs{\cF}^2k!\,n^{\eps^2}\cdot n^{m-k}\theta^{\abs{\cF}-1}
				\leq n^\eps.
			\end{aligned}
		\end{equation*}
		which completes the proof.
	\end{proof}
	
	\begin{lemma}\label{lemma: sparse absolute change copies}
		Let~$0 \leq i_0\leq i\leq i^\star$.
		Then,~$\abs{\Delta Z_{i_0}}\leq n^{\eps}$.
	\end{lemma}
	\begin{proof}
		This is an immediate consequence of Lemma~\ref{lemma: sparse absolute change copies not stopped}.
	\end{proof}
	
	\subsection{Supermartingale argument}\label{subsection: supermartingale}
	Lemma~\ref{lemma: sparse initial error copies} is the final ingredient that we use for our application of Azuma's inequality in the proof of Lemma~\ref{lemma: copy control} where we show that the probabilities of the events on the right in Observation~\ref{observation: critical times} are indeed small.
	
	\begin{lemma}\label{lemma: sparse initial error copies}
		$Z_{\sigma}(\sigma)\leq -\eps^5 \hhat^*(\sigma)$.
	\end{lemma}
	\begin{proof}
		Lemma~\ref{lemma: initial copy bound} implies~$\tau^\star\geq 1$ and~$H^*(0)<(1-\eps^4)\hhat^*(0)$, so we have~$\sigma\geq 1$.
		Thus, by definition of~$\sigma$, for~$i:=\sigma-1$, we have~$H^*\leq (1-\eps^4)\hhat^*$ and thus
		\begin{equation*}
			Z_i=H^*-\hhat^*\leq -\eps^4\hhat^*.
		\end{equation*}
		With Lemma~\ref{lemma: sparse absolute change copies not stopped} and Lemma~\ref{lemma: bounds for phat}, this then yields
		\begin{equation*}
			Z_\sigma(\sigma)
			\leq Z_i+\Delta Y
			\leq -\eps^4\hhat^*+n^\eps
			\leq -\eps^4\hhat^*+n^{-2\eps}n^k\phat
			\leq -\eps^4\hhat^*+n^{-\eps}\hhat^*
			\leq -\eps^5\hhat^*.
		\end{equation*}
		Since~$\Delta\hhat^*\leq 0$, this completes the proof.
	\end{proof}
	
	\begin{lemma}\label{lemma: copy control}
		$\pr{\tau^\star\leq i^\star}\leq \exp(-n^{1/3})$.
	\end{lemma}
	\begin{proof}
		Considering Observation~\ref{observation: critical times}, it suffices to show that
		\begin{equation*}
			\pr{Z_\sigma(i^\star)>0}\leq \exp(-n^{1/3}).
		\end{equation*}
		Due to Lemma~\ref{lemma: sparse initial error copies}, we have
		\begin{equation*}
			\begin{aligned}
				\pr{Z_\sigma(i^\star)>0}
				\leq \pr{Z_\sigma(i^\star)-Z_\sigma(\sigma)>\eps^5\hhat^*}
				\leq \sum_{0\leq i\leq i^\star}\pr{Z_i(i^\star)-Z_i>\eps^5\hhat^*}.
			\end{aligned}
		\end{equation*}
		Thus it suffices to show that for~$0\leq i\leq i^\star$, we have
		\begin{equation*}
			\pr{Z_i(i^\star)-Z_i>\eps^5\hhat^*}\leq \exp(-n^{1/2}).
		\end{equation*}
		We show that this bound is a consequence of Azuma's inequality.
		
		Let us turn to the details.
		Lemma~\ref{lemma: sparse copy trend} shows that~$Z_i(i),Z_i(i+1),\ldots$ is a supermartingale, while Lemma~\ref{lemma: sparse absolute change copies} provides the bound~$\abs{\Delta Z_i(j)}\leq n^\eps$ for all~$j\geq i$.
		Hence, we may apply Lemma~\ref{lemma: azuma} to obtain
		\begin{equation*}
			\pr{Z_i(i^\star)-Z_i>\eps^5\hhat^*}
			\leq \exp\paren[\bigg]{-\frac{\eps^{10}(\hhat^*)^2}{2(i^\star-i) n^{2\eps}}}.
		\end{equation*}
		Since
		\begin{equation*}
			i^\star-i
			\leq \frac{\theta n^k}{\abs{\cF}k!}-i
			=\frac{n^k\phat}{\abs{\cF}k!},
		\end{equation*}
		with Lemma~\ref{lemma: bounds for phat}, this yields
		\begin{equation*}
			\begin{aligned}
				\pr{Z_i(i^\star)-Z_i>\eps^5\hhat^*}
				&\leq \exp\paren[\bigg]{-\frac{\eps^{11}(\hhat^*)^2}{n^{k+2\eps}\phat}}
				\leq \exp\paren{-\eps^{11} n^{k-3\eps}\phat \cdot n^{2(m-k)}\phat^{2(\abs{\cF}-1)}}\\
				&\leq \exp\paren{-\eps^{11} n^{k-3\eps-2\eps^2(\abs{\cF}-1)}\phat}
				\leq \exp(-n^{1/2}),
			\end{aligned}
		\end{equation*}
		which completes the proof.
	\end{proof}
	
	\section{The isolation argument}\label{section: isolation argument}
	
	In this section, we show that~$H(\tau_\emptyset)\geq n^{k-1/\rho-\eps}$ with high probability if~$\cF$ is strictly~$k$-balanced.
	For this section, in addition to the setup described in Section~\ref{section: gluing}, assume that~$\cF$ is strictly~$k$-balanced, so in particular that~$\abs{\cF}\geq 3$, and that~$\cH$ is~$\cF$-populated.
	Overall, our approach is inspired by~\cite[Proof of Theorem~6.1]{BFL:15}; however, whenever~$\cF$ is not a triangle, copies of~$\cF$ can form substructures that may prevent a direct translation of the argument.
	For our more general setting, we rely on the insights gained in Section~\ref{section: gluing} to control these substructures in our analysis.
	
	\subsection{Overview}
	Instead of choosing the edge sets of copies~$\cF_0(i)$ with~$i\geq 1$ uniformly at random in Algorithm~\ref{algorithm: removal}, we may assume that during the initialization, a linear order~$\cleq$ on~$\cH^*$ is chosen uniformly at random and that for all~$i\geq 1$, the edge set~$\cF_0(i)$ is the minimum of~$\cH^*(i-1)$.
	Clearly, this yields the same random process.
	
	Our argument that typically, sufficiently many edges of~$\cH(0)$ remain when Algorithm~\ref{algorithm: removal} terminates may be summarized as follows.
	We crucially rely on identifying edges of~$\cH(0)$ that for some~$i\geq 0$ become isolated vertices of~$\cH^*$ and hence remain at the end of the process.
	We say that \emph{almost-isolation} occurs at a copy~$\cF'\in\cH^*(0)$ if for some edge~$e\in\cF'$ at some step, the copy~$\cF'$ is the only remaining copy that contains~$e$ and we say that \emph{isolation} occurs at~$\cF'$ if additionally at a later step, a copy~$\cF''\neq\cF'$ with~$e\notin \cF'\cap\cF''\neq\emptyset$ is selected for removal hence causing~$e$ to become an isolated vertex in~$\cH^*$.
	
	Initially, that is at step~$i=0$, for every edge~$e\in \cH$, there exist at least two copies of~$\cF$ that have~$e$ as one of their edges.
	If at step~$i=i^\star$ we do not already have sufficiently many edges of~$\cH$ that are isolated vertices of~$\cH^*$, then since by Lemma~\ref{lemma: copy control} we may assume that there is essentially not more than one copy of~$\cF$ for every~$\abs{\cF}$ edges that remain, we are in a situation where most of the remaining copies form a matching within~$\cH^*$.
	Thus, almost-isolation must have occurred many times.
	
	If it is the removal of~$\cF_0$ during step~$i$ that causes almost-isolation at a copy~$\cF'$, then before this removal, for all edges~$e\in\cF'$, there was a copy~$\cF''\neq\cF'$ with~$e\in\cF''$ and hence as a consequence of Lemma~\ref{lemma: few cyclic walks}, it only rarely happens that the removal of~$\cF_0$ destroys all copies~$\cF''\neq\cF'$ that previously shared an edge with~$\cF'$.
	Thus, in almost all cases where almost-isolation occurs, it is possible that isolation occurs.
	Furthermore, it turns out that the probability that this happens is not too small.
	
	We ensure that the copies at which we look for almost-isolation are spaced out as this allows us to assume that at these copies, almost-isolation turns into isolation independently of the development at the other copies.
	
	\subsection{Formal setup}\label{subsection: isolation formal}
	Formally, our setup is as follows.
	For~$\ell\geq 1$, a hypergraph~$\cA$ and~$e\in\cA$, inductively define~$\cN^\ell_{\cA}(e)$ as follows.
	Let~$\cN^1_{\cA}(e):=\cset{ f\in\cA }{ e\cap f\neq\emptyset }$ denote the set of edges of~$\cA$ that intersect with~$e$ and for~$\ell\geq 2$, let
	\begin{equation*}
		\cN^\ell_{\cA}(e):=\bigcup_{f\in N^{\ell-1}_{\cA}(e)} \cN^1_{\cA}(f).
	\end{equation*}
	For~$\ell\geq 1$, let~$N^\ell_{\cA}(e):=\abs{\cN^\ell_{\cA}(e)}$.
	During the random removal process, we additionally construct random subsets~$\emptyset=:\cR(0)\subseteq\ldots\subseteq \cR(i^\star)\subseteq \cH^*(0)$ where we collect copies of~$\cF$ at which almost-isolation occurs.
	We inductively define~$\cR(i)$ with~$1\leq i\leq i^\star$ as described by the following procedure.
	\par\bigskip
	\begin{algorithm}[H]
		\SetAlgoLined
		\DontPrintSemicolon
		$\cR(i)\gets \cR(i-1)$\;
		consider an arbitrary ordering~$\cF_1,\ldots,\cF_\ell$ of~$\cH^*(i)$\;
		\For{$\ell'\gets 1$ \KwTo $\ell$}{
			\If{$i=\min\cset{ j\geq 0 }{ d_{\cH^*(j)}(e)=1\stforsome e\in\cF_{\ell'} }$ and $\cN^4_{\cH^*(0)}(\cF_{\ell'})\cap \cR(i)=\emptyset$}{
				$\cR(i)\gets \cR(i)\cup\set{\cF_{\ell'}}$\;
			}
		}
		\caption{Construction of~$\cR(i)$.}
		\label{algorithm: further edge sets}
	\end{algorithm}
	\par\bigskip
	To exclude the copies at which almost-isolation occurs without the option that isolation occurs, we define subsets~$\cR'(i)\subseteq \cR(i)$ as follows.
	For~$\cF'\in\cR(i^\star)$, let
	\begin{equation*}
		i_{\cF'}:=\min\cset{ i\geq 0 }{ \cF'\in\cR(i) }
	\end{equation*}
	be the step where~$\cF'$ is added as an element for the eventually generated~$\cR(i^\star)$ and for~$i\geq 0$, let
	\begin{equation*}
		\cR'(i):=\cset{ \cF'\in\cR(i) }{ \cN^1_{\cH^*(i_{\cF'})}(\cF')\neq\set{\cF'} }
	\end{equation*}
	be the elements~$\cF'\in\cR(i)$ where at step~$i_{\cF'}$, the copy~$\cF'$ shared at least one edge with another copy of~$\cF$.
	Finally, we define events that entail almost-isolation becoming isolation. 
	For~$\cF'\in\cR'(i)$, fix an arbitrary~$\cG_{\cF'}\in \cN^1_{\cH^*(i_{\cF'})}(\cF')\setminus \set{\cF'}$ and let
	\begin{equation*}
		\cE_{\cF'}:=\set{ \cG_{\cF'}\cleq \cG \stforall \cG\in \cN^1_{\cH^*(0)}(\cG_{\cF'}) }.
	\end{equation*}
	
	\subsection{Proof of Theorem~\ref{theorem: technical}}\label{subsection: isolation proof}
	Since every almost-isolation that turns into isolation causes an edge of~$\cH(0)$ to become an isolated vertex of~$\cH^*$ for some~$i\geq 0$ and hence an edge that remains at the end of the removal process, we obtain the following statement.
	\begin{observation}\label{observation: remain configuration effect}
		$H(\tau_\emptyset)\geq  \sum_{\cF'\in\cR'(i^\star)} \ind_{\cE_{\cF'}}$.
	\end{observation}
	
	We organize the formal presentation of the arguments outlined above in two lemmas.
	At the end of the section, using the above observation together with these two lemmas, we prove Theorem~\ref{theorem: technical}. 
	
	Define the event
	\begin{equation*}
		\cE_0:=\set{ \abs{\cset{ e\in \cH(i^\star) }{ d_{\cH^*(i^\star)}(e)=0 }}<\eps H(i^\star)}
	\end{equation*}
	that occurs if and only if the number of isolated vertices of~$\cH^*(i^\star)$ is only a small fraction of all present vertices.
	\begin{lemma}\label{lemma: remain configuration bound}
		Let~$\cX:=\set{i^\star<\tau^\star}\cap \cE_0$.
		Then, $\abs{\cR'(i^\star)}\Xgeq n^{k-1/\rho-4\eps^2}$.
	\end{lemma}
	\begin{proof}
		Let~$i:=i^\star$ and consider the set
		\begin{equation*}
			\cI^*:=\cset{ \cF'\in\cH^* }{ \cN^1_{\cH^*}(\cF')=\set{\cF'}  }
		\end{equation*}
		of edge sets of copies of~$\cF$ in~$\cH$ that are isolated in the sense that they do not share an edge with another copy of~$\cF$.
		Since~$\cH(0)$ is~$\cF$-populated, by construction of~$\cR$, for every~$\cF'\in\cI^*$, either~$\cF'$ itself is an element of~$\cR$ or there exists some~$\cF''\in \cN^4_{\cH^*(0)}(\cF')\cap \cR$ that prevented the inclusion of~$\cF'$ in~$\cR$.
		Hence, there exists a function~$\pi\colon \cI^*\rightarrow \cR$ that for every~$\cF'\in\cI^*$ chooses a witness~$\pi(\cF')$ with~$\pi(\cF')\in \cN^4_{\cH^*(0)}(\cF')$ or equivalently~$\cF'\in \cN^4_{\cH^*(0)}(\pi(\cF'))$.
		If~$\cF'\in\cR$ and~$\cF''\in\pi^{-1}(\cF')$, we have~$\cF''\in \cN^4_{\cH^*(0)}(\cF')$ and hence~$\pi^{-1}(\cF')\subseteq \cN^4_{\cH^*(0)}(\cF')$.
		Thus, Lemma~\ref{lemma: F embeddings bound} entails~$\abs{\pi^{-1}(\cF')}\leq N^4_{\cH^*(0)}(\cF')\leq n^{\eps^{2}}$ and so we have
		\begin{equation}\label{equation: isolated copies and remain configurations}
			\abs{\cI^*}\leq \sum_{\cF'\in\cR} \abs{\pi^{-1}(\cF')}\leq \abs{\cR}n^{\eps^{2}}.
		\end{equation}
		First, we obtain a suitable lower bound for~$\abs{\cI^*}$ which, by the above inequality, yields a lower bound for~$\abs{\cR}$, then we show that~$\abs{\cR}$ is essentially as large as~$\abs{\cR'}$.
		
		Let us proceed with the first step.
		Using Lemma~\ref{lemma: sparse edges of H}, we have
		\begin{equation}\label{equation: upper bound on H^*}
			\begin{aligned}
				H^*
				&\Xleq \hhat^*
				\leq \paren[\bigg]{\frac{\abs{\cF}-1}{\abs{\cF}-1-\eps^2}+n^{m-k+2\eps^3}\phat^{\abs{\cF}-1-\eps^3}}\frac{n^k\phat}{\abs{\cF}k!}\\
				&\leq(1+\eps^{3/2}+n^{2\eps^3-\eps^2(\abs{\cF}-1-\eps^3)+\eps^3/\rho_\cF})\frac{n^k\phat}{\abs{\cF}k!}
				\leq (1+\eps)\frac{H}{\abs{\cF}}.
			\end{aligned}
		\end{equation}
		From this, we obtain
		\begin{equation}\label{equation: edges and isolated copies}
			\begin{aligned}
				H&=\abs{\cset{ e\in\cH }{ d_{\cH^*}(e)=0 }}+\sum_{\cF'\in\cH^*}\sum_{e\in\cF'}\frac{1}{d_{\cH^*}(e)}
				\Xleq \eps H+\abs{\cF}\abs{\cI^*}+\paren[\bigg]{\abs{\cF}-\frac{1}{2}}\abs{\cH^*\setminus\cI^*}\\
				&= \eps H+\paren[\bigg]{\abs{\cF}-\frac{1}{2}}H^* +\frac{1}{2}\abs{\cI^*}
				\Xleq \eps H+\paren[\bigg]{\abs{\cF}-\frac{1}{2}}(1+\eps)\frac{H}{\abs{\cF}} +\frac{1}{2}\abs{\cI^*}\\
				&= H -\frac{1+\eps-4\eps\abs{\cF}}{2\abs{\cF}}H+\frac{1}{2}\abs{\cI^*}
				\leq H-\frac{1}{4\abs{\cF}} H+\frac{1}{2}\abs{\cI^*}.
			\end{aligned}
		\end{equation}
		With Lemma~\ref{lemma: sparse edges of H}, this implies
		\begin{equation}\label{equation: lower bound isolation}
			\abs{\cI^*}
			\Xgeq \frac{1}{2\abs{\cF}} H
			\geq \frac{n^k\phat}{2\abs{\cF}k!}
			\geq n^{k-1/\rho_\cF-2\eps^2}.
		\end{equation}
		Combining this with~\eqref{equation: isolated copies and remain configurations}, we conclude that~$\abs{\cR}\geq n^{k-1/\rho_\cF-3\eps^{2}}$, which completes the first step.
		
		Consider a copy~$\cF'$ of~$\cF$ with~$\cF'\in\cR\setminus\cR'$.
		Let~$e_1\in\cF'$.
		There exists a copy~$\cF_1\neq\cF'$ of~$\cF$ with~$\cF_1\in\cH^*(i_{\cF'}-1)$ such that~$e_1\in\cF_1$.
		Furthermore, there exists an edge~$e_2\in\cF'\setminus\cF_1$ and a copy~$\cF_2\neq\cF'$ of~$\cF$ with~$\cF_2\in\cH^*(i_{\cF'}-1)$ such that~$e_2\in\cF_2$.
		By choice of~$\cF'$ and~$i_{\cF'}$, both copies~$\cF_1$ and~$\cF_2$ have an edge that is contained in~$\cF_0(i_{\cF'})$.
		Hence, if~$\cF_1,\cF',\cF_2$ does not form a self-avoiding cyclic walk, then, using~$\cF''$ to denote the copy of~$\cF$ with edge set~$\cF_0(i_{\cF'})$, the sequence~$\cF_1,\cF',\cF_2,\cF''$ forms a self-avoiding cyclic walk.
		Thus, for every copy~$\cF'$ of~$\cF$ with~$\cF'\in\cR\setminus\cR'$, there exist copies of~$\cF$ whose edge sets are elements of~$\cH^*(0)$ and that together with~$\cF'$ form a self-avoiding cyclic walk of length~$3$ or~$4$.
		
		Let~$\ccF_4$ denote a collection of~$k$-graphs~$\cG$ with~$V_\cG\subseteq \set{1,\ldots,4m}$ that for every self-avoiding walk~$\cF_1,\ldots,\cF_\ell$ of copies of~$\cF$ with~$3\leq\ell\leq 4$ contains a copy of~$\cF_1+\ldots+\cF_\ell$ and that only contains copies of such~$k$-graphs.
		Then, we have~$\abs{\cR'}\geq \abs{\cR}-\sum_{\cG\in\ccF_4}4\Phi_{\cG}$, so it suffices to show that~$\Phi_{\cG}\leq n^{k-1/\rho_\cF-4\eps^{2}}$ for all~$\cG\in\ccF_4$.
		This is a consequence of Lemma~\ref{lemma: few cyclic walks}.
	\end{proof}
	
	\begin{lemma}\label{lemma: remain configuration dominance}
		Suppose that~$X$ is a binomial random variable with parameters~$n^{k-1/\rho_\cF-4\eps^{2}}$ and~$n^{-\eps^{2}}$ and let~$Y:=(n^{k-1/\rho_\cF-4\eps^{2}}-\abs{\cR'(i^\star)})\vee 0$.
		Let
		\begin{equation*}
			Z:=Y+\sum_{\cF'\in\cR'(i^\star)} \ind_{\cE_{\cF'}}.
		\end{equation*}
		Then,~$Z$ stochastically dominates~$X$.
	\end{lemma}
	\begin{proof}
		First, observe that by Lemma~\ref{lemma: F embeddings bound}, whenever~$\cF'\in\cR'(i^\star)$, for~$i:=0$, we have
		\begin{equation}\label{equation: few neighbors}
			N^1_{\cH^*}(\cG_{\cF'})
			\leq \sum_{f\in\cG_{\cF'}} d_{\cH^*}(f)
			\leq n^{\eps^{2}}.
		\end{equation}
		
		Consider distinct~$\cF',\cF''\in\cH^*(0)$.
		By construction of~$\cR(i^\star)$, whenever~$\cF',\cF''\in\cR(i^\star)$, then, for all~$\cG'\in\cN^1_{\cH^*(i_{\cF'})}(\cF')$ and~$\cG''\in\cN^1_{\cH^*(i_{\cF''})}(\cF'')$, we have
		\begin{equation*}
			\cN^1_{\cH^*(0)}(\cG')\cap \cN^1_{\cH^*(0)}(\cG'')=\emptyset.
		\end{equation*}
		Thus, for all distinct~$\cF_1,\ldots,\cF_{\ell}\in\cR'(i^\star_+)$ and all~$z_1,\ldots,z_{\ell-1}\in\set{0,1}$, from~\eqref{equation: few neighbors}, we obtain
		\begin{equation*}
			\cpr{ \ind_{\cE_{\cF_{\ell}}}=1 }{ \ind_{\cE_{\cF_{\ell'}}}=z_{\ell'}\stforall 1\leq \ell'<\ell }
			= \pr{ \cE_{\cF_{\ell}} }
			\geq n^{-\eps^{2}},
		\end{equation*}
		which completes the proof.
	\end{proof}
	
	\begin{lemma}[Chernoff's inequality]\label{lemma: chernoff}
		Suppose~$X_1,\ldots,X_n$ are independent Bernoulli random variables and let~$X:= \sum_{1\leq i\leq n} X_i$.
		Then, for all~$\delta\in(0,1)$,
		\begin{equation*}
			\pr{X\neq (1\pm\delta)\ex{X}}\leq 2\exp\paren[\bigg]{-\frac{\delta^2\ex{X}}{3}}.
		\end{equation*}
	\end{lemma}
	
	\begin{proof}[Proof of Theorem~\ref{theorem: technical}]
		Define the events
		\begin{equation*}
			\cB:=\set{ H(\tau_\emptyset)\leq n^{k-1/\rho_\cF-\eps} }\qtand
			\cX:=\set{i^\star<\tau^\star}\cap \cE_0.
		\end{equation*}
		We need to show that~$\pr{\cB}$ is sufficiently small.
		Choose~$X$,~$Y$ and~$Z$ as in Lemma~\ref{lemma: remain configuration dominance}.
		Lemma~\ref{lemma: remain configuration bound} entails~$\cX\subseteq\set{Y=0}$ and hence~$\set{Y\neq 0}\subseteq \cX^\comp$.
		Thus, from Observation~\ref{observation: remain configuration effect} and Lemma~\ref{lemma: remain configuration dominance}, we obtain
		\begin{equation*}
			\begin{aligned}
				\cB
				&= \set[\Big]{\sum_{\cF'\in\cR'(i^\star)} \ind_{\cE_{\cF'}}\leq n^{k-1/\rho_\cF-\eps}}\cap\cB
				\subseteq \paren{\set{Z\leq n^{k-1/\rho_\cF-\eps}}\cup\set{Y\neq 0}}\cap\cB\\
				&\subseteq\set{Z\leq n^{k-1/\rho_\cF-\eps}}\cup(\cX^\comp\cap \cB)
				\subseteq \set{Z\leq n^{k-1/\rho_\cF-\eps}}\cup\set{\tau^\star\leq i^\star}\cup (\cE_0^\comp\cap \cB).
			\end{aligned}
		\end{equation*}
		By Lemma~\ref{lemma: sparse edges of H}, we have
		\begin{equation*}
			H(\tau_\emptyset)\geq_{\cE_0^\comp} \eps H(i^\star)\geq \eps^2 n^k\phat(i^\star)\geq n^{k-1/\rho_\cF-2\eps^2}
		\end{equation*}
		and hence~$\cE_0^\comp\cap \cB=\emptyset$.
		Thus, using Lemma~\ref{lemma: copy control}, we obtain
		\begin{equation*}
			\pr{\cB}\leq \pr{Z\leq n^{k-1/\rho_\cF-\eps}}+\exp(-n^{1/3}).
		\end{equation*}
		With Lemma~\ref{lemma: remain configuration dominance} and Chernoff's inequality (see Lemma~\ref{lemma: chernoff}), this completes the proof.
	\end{proof}
	
	\section{Proofs for the main theorems}\label{section: main proofs}
	In this section, we show how to obtain Theorems~\ref{theorem: main}--\ref{theorem: sparse} from Theorems~\ref{theorem: technical bounds} and~\ref{theorem: technical}.
	Proofs for Theorems~\ref{theorem: cherries} and~\ref{theorem: sparse cherries} can be found in Appendix~\ref{appendix: cherries}.

	\begin{proof}[Proof of Theorem~\ref{theorem: only upper bound}]
		This is an immediate consequence of Theorem~\ref{theorem: technical bounds}.
	\end{proof}

	\begin{proof}[Proof of Theorem~\ref{theorem: sparse}]
		By definition of~$\tau_\emptyset$ in Section~\ref{section: counting}, this is an immediate consequence of Theorem~\ref{theorem: technical}.
	\end{proof}

	\begin{proof}[Proof of Theorem~\ref{theorem: pseudorandom}]
		Let~$m:=\abs{V_\cF}$.
		Suppose that~$0<\eps<1$ is sufficiently small in terms of~$1/m$, that~$0<\delta<1$ is sufficiently small in terms of~$\eps$ and that~$n$ is sufficiently large in terms of~$1/\delta$.
		Suppose that~$\cH$ is an~$(\eps^{20},\delta,\rho)$-pseudorandom~$k$-graph on~$n$ vertices with~$\abs{\cH}\geq n^{k-1/\rho+\eps^5}$.
		Let
		\begin{equation*}
			\theta:=\frac{k!\,\abs{\cH}}{n^k}\geq n^{-1/\rho+\eps^5}.
		\end{equation*}
		We consider the~$\cF$-removal process starting at~$\cH$ where we assume the generated hypergraphs to remain constant if the process normally terminated due to the absence of copies of~$\cF$.
		Let~$\cH'$ denote the~$k$-graph generated after~$i^\star$ iterations, where
		\begin{equation*}
				i^\star:=\frac{(\theta-n^{-1/\rho+\eps^5})n^k}{\abs{\cF}k!}.
		\end{equation*}
		Let~$\cH''$ denote the~$k$-graph eventually generated by the process that contains no copies of~$\cF$ as subgraphs.
		Let~$\cX'$ denote the event that~$\cH'$ is~$(4m,n^{\eps^4})$-bounded,~$\cF$-populated and has~$n^{k-1/\rho+\eps^5}/k!$ edges.
		Let
		\begin{equation*}
			\cX'':=\set{\abs{\cH''}\leq n^{k-1/\rho+\eps}}\qtand
			\cY'':=\set{n^{k-1/\rho-\eps}\leq \abs{\cH''}}.
		\end{equation*}
		We need to show that
		\begin{equation*}
			\pr{\cX''\cap \cY''}\geq 1-\exp(-(\log n)^{5/4}).
		\end{equation*}
		Since~$\cX'\subseteq \cX''$, we have~$\pr{\cX''\cap \cY''}\geq \pr{\cX'\cap \cY''}$, so it suffices to obtain sufficiently large lower bounds for~$\pr{\cX'}$ and~$\pr{\cY''}$.
		We may apply Theorem~\ref{theorem: technical bounds} with~$\eps^5$ playing the role of~$\eps$ to obtain~$\pr{\cX'}\geq 1-\exp(-(\log n)^{4/3})$ and Theorem~\ref{theorem: technical} shows that~$\cpr{\cY''}{\cX'}\geq 1-\exp(-n^{1/4})$.
		Using~$\pr{\cY''}=\cpr{\cY''}{\cX'}\pr{\cX'}$, this yields suitable lower bounds for~$\pr{\cX'}$ and~$\pr{\cY''}$.
	\end{proof}

	\begin{proof}[Proof of Theorem~\ref{theorem: main}]
		This is an immediate consequence of Theorem~\ref{theorem: pseudorandom}.
	\end{proof}
	
	\section{Concluding remarks}\label{section: concluding remarks}
	For both, the~$\cF$-free process and the~$\cF$-removal process, the number of edges present at step~$i$ of the process, that is, after~$i$ iterations, is a deterministic quantity.
	Heuristically, intuition suggests that the set of edges present at step~$i$ behaves as if it was obtained by including every~$k$-set of vertices independently at random with an appropriate probability~$p$.
	
	For the~$\cF$-free process on~$n$ vertices, we have~$p\approx k!\,i/n^k$.
	There are approximately~$(1-p)n^k/k!$ potential edges that are not yet present.
	Using~$v(\cF)$ to denote the number of vertices,~$e(\cF)$ to denote the number of edges, and~$\aut(\cF)$ to denote the number of automorphisms of~$\cF$, for every such edge~$e$, the expected number of copies of~$\cF$ that that would be generated by adding~$e$ is~$e(\cF)k!\,n^{v(\cF)-k}p^{e(\cF)-1}/\aut(\cF)$.
	Hence, the Poisson paradigm suggests that the number of potential edges that are available for addition in a later step is approximately
	\begin{equation*}
		(1-p)\exp\paren[\bigg]{-\frac{e(\cF)k!\, n^{v(\cF)-k}p^{e(\cF)-1}}{\aut(\cF)}}\frac{n^k}{k!}.
	\end{equation*}
	This number becomes negligible compared to the approximate number~$n^kp/k!$ of present edges when
	\begin{equation*}
		p=\paren[\Bigg]{ \paren[\bigg]{ \frac{\aut(\cF)(v(\cF)-k)}{e(\cF)(e(\cF)-1)k!}}^{\frac{1}{e(\cF)-1}} \pm o(1) }(\log n)^{\frac{1}{e(\cF)-1}}  n^{-\frac{v(\cF)-k}{e(\cF)-1}}.
	\end{equation*}
	Hence, we conjecture the following.
	\begin{conjecture}\label{conjecture: free}
		Let~$k\geq 2$ and consider a strictly~$k$-balanced~$k$-uniform hypergraph~$\cF$ with~$k$-density~$\rho$.
		Then, for all~$\eps>0$, there exists~$n_0\geq 0$ such that for all~$n\geq n_0$, with probability at least~$1-\eps$, we have
		\begin{equation*}
			F(n,\cF)=\paren[\Bigg]{ \frac{1}{k!}\paren[\bigg]{ \frac{\aut(\cF)(v(\cF)-k)}{e(\cF)(e(\cF)-1)k!}}^{\frac{1}{e(\cF)-1}} \pm \eps }(\log n)^{\frac{1}{e(\cF)-1}}  n^{k-\frac{v(\cF)-k}{e(\cF)-1}}.
		\end{equation*}
	\end{conjecture}
	The known bounds for the case where~$\cF$ is a triangle, see~\cite{BK:21,PGM:20}, match this prediction and it would be interesting to further investigate other cases.
	Conjecture~\ref{conjecture: free} is closely related to~\cite[Conjecture~13.1]{BK:10}.

	Again following the above heuristic, for the~$\cF$-removal process we have~$p\approx 1-e(\cF)k!\, i/n^k$ such that again, there are approximately~$n^kp/k!$ edges present.
	Let~$\cH^*$ denote the auxiliary hypergraph where the present edges are the vertices and where the edges sets of present copies are the edges.
	Let~$H^*$ denote the number of edges of~$\cH^*$, that is the number of remaining copies of~$\cF$.
	We expect the~$2$-degrees in~$\cH^*$, that is the number of edges in~$\cH^*$ that contain two fixed vertices of~$\cH$, to be generally negligible compared to the vertex degrees in~$\cH^*$.
	Hence for the probability that a fixed present copy~$\cF'$ of~$\cF$ is no longer present in the next step, we estimate
	\begin{equation*}
		\frac{(\sum_{e\in E(\cF')} d_{\cH^*}(e))-e(\cF)+1}{H^*}.
	\end{equation*}
	Then, using~$\ccF_0,\ccF_1,\ldots$ to denote the natural filtration associated with the process, for the expected one-step change~$\cex{\Delta H^*}{\ccF_i}$ of~$H^*$, we obtain
	\begin{equation*}
		\cex{\Delta H^*}{\ccF_i}\approx-\sum_{\cF'\in E(\cH^*)}\frac{(\sum_{e\in E(\cF')} d_{\cH^*}(e))-e(\cF)+1}{H^*}
		=-\frac{1}{H^*}\paren[\Big]{ \sum_{d\geq 0} d_{\cH^*}(e)^2}+e(\cF)-1.
	\end{equation*}
	We expect the degrees in~$\cH^*$ to be Poisson distributed and mutually independent.
	Thus, since the average degree in~$\cH^*$ is approximately~$\lambda:=e(\cF)k!\,H^*/(n^k p)$, we expect that for all~$d\geq 0$, the random variable~$\abs{\cset{ e\in\cH }{ d_{\cH^*}(e)=d }}$ is concentrated around
	\begin{equation*}
		\frac{n^kp}{k!}\cdot \frac{\lambda^d \exp(-\lambda)}{d!}.
	\end{equation*} 
	Thus, we estimate
	\begin{equation*}
		\begin{aligned}
			\cex{\Delta H^*}{\ccF_i}&\approx-\frac{1}{H^*}\paren[\Big]{ \sum_{d\geq 0} d^2\abs{\cset{ e\in\cH }{ d_{\cH^*}(e)=d }} }+e(\cF)-1\\
			&\approx-\frac{n^kp}{k!\, H^*}\paren[\Big]{ \sum_{d\geq 0} d^2\cdot \frac{\lambda^d \exp(-\lambda)}{d!} }+e(\cF)-1
			=-\frac{n^kp}{k!\, H^*}(\lambda^2+\lambda)+e(\cF)-1\\
			&=-\frac{e(\cF)^2k! H^*}{n^k p}-1.
		\end{aligned}
	\end{equation*}
	We expect the number of present copies to typically closely follow a deterministic trajectory~$\hhat^*_0,\hhat^*_1,\ldots$ which by our above argument should satisfy
	\begin{equation*}
		\hhat^*_{i+1}-\hhat^*_{i}\approx -\frac{e(\cF)^2k! \hhat^*_i}{n^k p}-1.
	\end{equation*}
	Guided by this intuition, for~$i\geq 0$, we obtain an expression for~$\hhat^*_i$ by solving the corresponding differential equation.
	Specifically, since initially there are approximately~$n^{v(\cF)}/\aut(\cF)$ copies of~$\cF$ in~$K_n^{(k)}$, we set
	\begin{equation*}
		\hhat^*_i:=\frac{n^{v(\cF)}p^{e(\cF)}}{\aut(\cF)}-\frac{n^k p}{e(\cF)(e(\cF)-1)k!}.
	\end{equation*}
	This quantity becomes zero when
	\begin{equation*}
		p=\paren[\Bigg]{\paren[\bigg]{ \frac{\aut(\cF)}{e(\cF)(e(\cF)-1)k!} }^{\frac{1}{e(\cF)-1}}\pm o(1)}n^{-\frac{v(\cF)-k}{e(\cF)-1}}.
	\end{equation*}
	Hence, for the~$\cF$-removal process, we conjecture the following.
	\begin{conjecture}\label{conjecture: removal}
		Let~$k\geq 2$ and consider a strictly~$k$-balanced~$k$-uniform hypergraph~$\cF$ with~$k$-density~$\rho$.
		Then, for all~$\eps>0$, there exists~$n_0\geq 0$ such that for all~$n\geq n_0$, with probability at least~$1-\eps$, we have
		\begin{equation*}
			R(n,\cF)=\paren[\Bigg]{\frac{1}{k!}\paren[\bigg]{ \frac{\aut(\cF)}{e(\cF)(e(\cF)-1)k!} }^{\frac{1}{e(\cF)-1}}\pm \eps}n^{k-\frac{v(\cF)-k}{e(\cF)-1}}.
		\end{equation*}
	\end{conjecture}
	
	Theorem~\ref{theorem: main} confirms the order of magnitude in this conjecture whenever~$\cH$ is strictly~$k$-balanced.
	It would be interesting to obtain more precise results and to confirm the asymptotic value of the constant factor.
	
	The~$\cF$-free process where~$\cF$ is a diamond, which is a graph that is not strictly~$2$-balanced, typically terminates with a final number of edges that has a different exponent for the logarithmic factor compared to Conjecture~\ref{conjecture: free}, see~\cite{P:14a}.
	Hence, for the~$\cF$-free process as well as the~$\cF$-removal process, it could be interesting to further investigate the situation for graphs or hypergraphs that are not (strictly) balanced.
	
	In terms of applications, the conjectures above suggest that the~$\cF$-free process is more suitable for generating dense~$\cF$-free graphs, however, the~$\cF$-removal process might prove to be a useful tool for decomposition and packing problems since it carefully constructs a maximal collection of edge-disjoint copies of~$\cF$.
	For such applications, we believe that the fact that we do not require the initial hypergraph to be complete might be crucial.
	
	Additionally, as we believe that such an extension could be useful for applications, we remark that directly using Lemma~\ref{lemma: control everything} instead of one of the theorems makes it possible to easily amend our analysis as follows if the goal is to show that the random graphs generated by the process typically exhibit further properties that we did not consider in our analysis.
	
	Similarly to how we organized our analysis by using stopping times, one may define a stopping time~$\tau$ that measures when the desired property is first violated.
	Then for~$\tau^\star$ and~$i^\star$ as defined in Lemma~\ref{lemma: control everything}, it suffices show that~$\pr{\tau\leq \tau^\star\wedge i^\star}$ is small as this entails that~$\pr{\tau\wedge\tau^\star\leq i^\star}$ is small and hence that the process typically runs for at least~$i^\star$ steps while maintaining the desired property.
	For example, it is easy to see that in fact, typically a more precise estimate for the number of copies of~$\cF$ in every step holds provided that the guarantees concerning the initial hypergraph are more precise.
	This might be useful for counting the number of choices available for every deletion which can in turn be useful for counting the number of packings of edge-disjoint copies of~$\cF$.
	Specifically, instead of only obtaining~$\hhat^*(i)\pm \zeta(i)^{1+\eps^3}$ as an estimate for the number of copies present after~$i$ deletions as in our first part of the proof, it is possible to instead obtain~$\hhat^*(i)\pm \delta^{-6}\zeta(i)^2$ if a slightly more precise estimate holds for~$i=0$.
	To obtain this refinement following an approach as mentioned above, it suffices use the same argumentation that proves Lemma~\ref{lemma: auxiliary control}~\ref{item: control copies} with only minor adaptations.

	\bibliographystyle{amsplain}

\begin{thebibliography}{10}

\bibitem{AKS:97}
N.~Alon, J.-H. Kim, and J.~Spencer, \emph{Nearly perfect matchings in regular simple hypergraphs}, Israel J. Math. \textbf{100} (1997), 171--187.

\bibitem{BB:16}
P.~Bennett and T.~Bohman, \emph{A note on the random greedy independent set algorithm}, Random Structures Algorithms \textbf{49} (2016), 479--502.

\bibitem{BB:19}
\bysame, \emph{A natural barrier in random greedy hypergraph matching}, Combin. Probab. Comput. \textbf{28} (2019), 816–825.

\bibitem{B:09}
T.~Bohman, \emph{The triangle-free process}, Adv. Math. \textbf{221} (2009), 1653--1677.

\bibitem{BFL:18}
T.~Bohman, A.~Frieze, and E.~Lubetzky, \emph{A note on the random greedy triangle-packing algorithm}, J. Comb. \textbf{1} (2010), 477--488.

\bibitem{BFL:15}
\bysame, \emph{Random triangle removal}, Adv. Math. \textbf{280} (2015), 379--438.

\bibitem{BK:10}
T.~Bohman and P.~Keevash, \emph{The early evolution of the {$H$}-free process}, Invent. Math. \textbf{181} (2010), 291--336.

\bibitem{BK:21}
\bysame, \emph{Dynamic concentration of the triangle-free process}, Random Structures Algorithms \textbf{58} (2021), 221--293.

\bibitem{BP:12}
T.~Bohman and M.~E. Picollelli, \emph{S{IR} epidemics on random graphs with a fixed degree sequence}, Random Structures Algorithms \textbf{41} (2012), 179--214.

\bibitem{BR:00}
B.~Bollob\'{a}s and O.~Riordan, \emph{Constrained graph processes}, Electron. J. Combin. \textbf{7} (2000), Research Paper 18, 20.

\bibitem{ESW:95}
P.~Erd\H{o}s, S.~Suen, and P.~Winkler, \emph{On the size of a random maximal graph}, Proceedings of the {S}ixth {I}nternational {S}eminar on {R}andom {G}raphs and {P}robabilistic {M}ethods in {C}ombinatorics and {C}omputer {S}cience, ``{R}andom {G}raphs '93'' ({P}ozna\'{n}, 1993), vol.~6, 1995, pp.~309--318.

\bibitem{PGM:20}
G.~Fiz~Pontiveros, S.~Griffiths, and R.~Morris, \emph{The triangle-free process and the {R}amsey number {$R(3,k)$}}, Mem. Amer. Math. Soc. \textbf{263} (2020), v+125.

\bibitem{freedman:75}
D.~A. Freedman, \emph{On tail probabilities for martingales}, Ann. Probab. \textbf{3} (1975), 100--118.

\bibitem{GJKKL:24}
S.~Glock, F.~Joos, J.~Kim, M.~K\"{u}hn, and L.~Lichev, \emph{Conflict-free hypergraph matchings}, J. Lond. Math. Soc. (2) \textbf{109} (2024), Paper No. e12899, 78.

\bibitem{G:97}
D.~A. Grable, \emph{On random greedy triangle packing}, Electron. J. Combin. \textbf{4} (1997), Research Paper 11, 19.

\bibitem{H:24}
J.~Hofstad, \emph{Behaviour of the minimum degree throughout the {$d$}-process}, Combin. Probab. Comput. \textbf{33} (2024), 564--582.

\bibitem{J:90}
S.~Janson, \emph{Poisson approximation for large deviations}, Random Structures Algorithms \textbf{1} (1990), 221--229.

\bibitem{JR:04}
S.~Janson and A.~Ruci\'{n}ski, \emph{The deletion method for upper tail estimates}, Combinatorica \textbf{24} (2004), 615--640.

\bibitem{KOT:16}
D.~K\"uhn, D.~Osthus, and A.~Taylor, \emph{On the random greedy {$F$}-free hypergraph process}, SIAM J. Discrete Math. \textbf{30} (2016), 1343--1350.

\bibitem{OT:01}
D.~Osthus and A.~Taraz, \emph{Random maximal {$H$}-free graphs}, Random Structures Algorithms \textbf{18} (2001), 61--82.

\bibitem{P:11}
M.~E. Picollelli, \emph{The final size of the {$C_4$}-free process}, Combin. Probab. Comput. \textbf{20} (2011), 939--955.

\bibitem{P:14a}
\bysame, \emph{The diamond-free process}, Random Structures Algorithms \textbf{45} (2014), 513--551.

\bibitem{P:14b}
\bysame, \emph{The final size of the {$C_\ell$}-free process}, SIAM J. Discrete Math. \textbf{28} (2014), 1276--1305.

\bibitem{RT:96}
V.~R\"{o}dl and L.~Thoma, \emph{Asymptotic packing and the random greedy algorithm}, Random Structures Algorithms \textbf{8} (1996), 161--177.

\bibitem{RW:92}
A.~Ruci\'{n}ski and N.~C. Wormald, \emph{Random graph processes with degree restrictions}, Combin. Probab. Comput. \textbf{1} (1992), 169--180.

\bibitem{S:95}
J.~Spencer, \emph{Asymptotic packing via a branching process}, Random Structures Algorithms \textbf{7} (1995), 167--172.

\bibitem{S:95b}
\bysame, \emph{Maximal trianglefree graphs and {R}amsey~{$R(3,k)$}}, unpublished manuscript, 1995.

\bibitem{TWZ:07}
A.~Telcs, N.~Wormald, and S.~Zhou, \emph{Hamiltonicity of random graphs produced by 2-processes}, Random Structures Algorithms \textbf{31} (2007), 450--481.

\bibitem{W:14a}
L.~Warnke, \emph{The {$C_\ell$}-free process}, Random Structures Algorithms \textbf{44} (2014), 490--526.

\bibitem{W:14b}
\bysame, \emph{When does the {$K_4$}-free process stop?}, Random Structures Algorithms \textbf{44} (2014), 355--397.

\bibitem{W:99}
N.~C. Wormald, \emph{The differential equation method for random graph processes and greedy algorithms}, In M. Karo\'{n}ski and H. Pr\"{o}mel, Eds., Lectures on Approximation and Randomized Algorithms (1999), 73--155.

\end{thebibliography}
\providecommand{\bysame}{\leavevmode\hbox to3em{\hrulefill}\thinspace}
\providecommand{\MR}{\relax\ifhmode\unskip\space\fi MR }
\providecommand{\MRhref}[2]{%
  \href{http://www.ams.org/mathscinet-getitem?mr=#1}{#2}
}
\providecommand{\href}[2]{#2}

	\appendix
	
	\section{Counting copies of~\texorpdfstring{$\cF$}{F}}\label{appendix: copies}
	In this section, our goal is to prove Lemma~\ref{lemma: auxiliary control}~\ref{item: control copies}.
	Hence, for this section, we assume the setup that we used in Section~\ref{section: stopping times} to state Lemma~\ref{lemma: auxiliary control}.
	Our approach is similar as in Sections~\ref{subsection: tracking chains} and~\ref{subsection: tracking families}.
	
	For~$i\geq 0$, let
	\begin{equation*}
		\eta_1(i):=\zeta^{1+\eps^3}\hhat^*\qtand \eta_0(i):=\paren{1-\eps}\eta_1(i).
	\end{equation*}
	Define the critical intervals
	\begin{equation*}
		I^-(i):=[\hhat^*-\eta_1,\hhat^*-\eta_0]\qtand
		I^+(i):=[\hhat^*+\eta_0,\hhat^*+\eta_1].
	\end{equation*}
	For~$\pom\in\set{-,+}$, let
	\begin{equation*}
		Y^\pom(i):=\pom(H^*-\hhat^*)-\eta_1
	\end{equation*}
	For~$i_0\geq 0$, define the stopping time
	\begin{equation*}
		\tau_{i_0}^\pom:=\min\cset{i\geq i_0}{H^*\notin I^\pom}
	\end{equation*}
	and for~$i\geq i_0$, let
	\begin{equation*}
		Z^\pom_{i_0}(i):=Y^\pom(i_0\vee (i\wedge \tau^\pom_{i_0} \wedge \tautilde^\star\wedge i^\star)).
	\end{equation*}
	Let
	\begin{equation*}
		\sigma^\pom:=\min\cset{j\geq 0}{ \pom(H^*- \hhat^*)\geq \eta_0 \stforall j\leq i<  \tautilde^\star\wedge i^\star }\leq \tautilde^\star\wedge i^\star
	\end{equation*}
	With this setup, similarly as in Sections~\ref{subsection: tracking chains} and~\ref{subsection: tracking families}, it in fact suffices to consider the evolution of~$Z^\pom_{\sigma^\pom}(\sigma^\pom),Z^\pom_{\sigma^\pom}(\sigma^\pom+1),\ldots$.
	
	\begin{observation}\label{observation: copies critical times}
		$\set{\tau_{\cH^*}\leq \tautilde^\star\wedge i^\star}
		\subseteq \set{Z^-_{\sigma^-}(i^\star)>0}\cup\set{Z^+_{\sigma^+}(i^\star)>0}$.
	\end{observation}
	
	We again use supermartingale concentration techniques to show that the probabilities of the events on the right in Observation~\ref{observation: copies critical times} are sufficiently small.
	However, instead of relying on Freedman's inequality, here, similarly as in Section~\ref{section: counting}, we instead use Azuma's inequality.

	\subsection{Trend}\label{subsection: copy trend}
	Here, we prove that for all~$\pom\in\set{-,+}$ and~$i_0\geq 0$, the expected one-step changes of the process~$Z^\pom_{i_0}(i_0),Z^\pom_{i_0}(i_0+1),\ldots$ are non-positive.
	We begin with estimating the one-step changes of the deterministic parts of this random process in Lemma~\ref{lemma: delta hhat, delta eta}.
	Using Lemma~\ref{lemma: averaging}, we obtain Lemma~\ref{lemma: copy trend} where we provide a precise estimate for the expected one-step change of the non-deterministic part that holds whenever the removal process was well-behaved up to the step we consider.
	Finally, we combine our estimates for the deterministic and non-deterministic parts to see that the above process is indeed a supermartingale (see Lemma~\ref{lemma: copy deviations are supermartingales}).
	
	\begin{observation}\label{observation: derivative hhat}
		Extend~$\phat$,~$\hhat^*$ and~$\eta_1$ to continuous trajectories defined on the whole interval~$[0,i^\star+1]$ using the same expressions as above.
		Then, for~$x\in[0,i^\star+1]$,
		\begin{equation*}
			\begin{gathered}
				(\hhat^*)'(x)=-\frac{\abs{\cF}^2k!\,\hhat^*(x)}{n^k\phat(x)},\quad
				(\hhat^*)''(x)=\frac{\abs{\cF}^3(\abs{\cF}-1)(k!)^2\hhat^*(x)}{n^{2k}\phat(x)^2},\\
				\eta_1'(x)=-\frac{\paren[\big]{\abs{\cF}-\frac{(1+\eps^3)\rho_\cF}{2}}\abs{\cF}k!\,\eta_1(x)}{n^k\phat(x)},\\
				\eta_1''(x)=-\frac{\paren[\big]{\abs{\cF}-\frac{(1+\eps^3)\rho_\cF}{2}}\paren[\big]{\abs{\cF}-\frac{(1+\eps^3)\rho_\cF}{2}-1}\abs{\cF}^2(k!)^2\eta_1(x)}{n^{2k}\phat(x)^2}.
			\end{gathered}
		\end{equation*} 
	\end{observation}
	
	\begin{lemma}\label{lemma: delta hhat, delta eta}
		Let~$0\leq i\leq i^\star$ and~$\cX:=\set{i\leq \tau_\emptyset}$.
		Then,
		\begin{equation*}
			\Delta\hhat^*\Xeq -\frac{\abs{\cF}^2\hhat^*}{H}\pm \frac{\zeta^{2+\eps^2}\hhat^*}{H},\quad
			\Delta \eta_1\Xeq -\paren[\bigg]{\abs{\cF}-\frac{(1+\eps^3)\rho_\cF}{2}} \frac{\abs{\cF}\eta_1}{H}\pm \frac{\zeta^{2+\eps^2}\eta_1}{H}.
		\end{equation*}
	\end{lemma}
	\begin{proof}
		This is a consequence of Taylor's theorem.
		In detail, we argue as follows.
		
		Together with Observation~\ref{observation: derivative hhat}, Lemma~\ref{lemma: taylor} yields
		\begin{equation*}
			\Delta\hhat^*
			=-\frac{\abs{\cF}^2k!\,\hhat^*}{n^k\phat}\pm \max_{x\in[i,i+1]} \frac{\hhat^*(x)}{\delta n^{2k}\phat(x)^2}.
		\end{equation*}
		We investigate the first term and the maximum separately.
		Using Lemma~\ref{lemma: edges of H}, we have
		\begin{equation*}
			-\frac{\abs{\cF}^2k!\,\hhat^*}{n^k\phat}
			\Xeq-\frac{\abs{\cF}^2\hhat^*}{H}.
		\end{equation*}
		Furthermore, since~$\hhat^*(x)/\phat(x)^2$ is non-decreasing in~$x$ for~$x\in[i,i+1]$, Lemma~\ref{lemma: edges of H} together with Lemma~\ref{lemma: zeta and H} yields
		\begin{equation*}
			\max_{x\in[i,i+1]} \frac{\hhat^*(x)}{\delta n^{2k}\phat(x)^2}
			\leq \frac{\hhat^*}{\delta n^{2k}\phat^2}
			\Xleq \frac{\hhat^*}{\delta H^2}
			\leq \frac{\zeta^{2+2\eps^2}\hhat^*}{\delta H}
			\leq \frac{\zeta^{2+\eps^2}\hhat^*}{H}.
		\end{equation*}
		Thus we obtain the desired expression for~$\Delta\hhat^*$.
		
		We argue similarly for~$\Delta\eta_1$.
		Again together with Observation~\ref{observation: derivative hhat}, Lemma~\ref{lemma: taylor} yields
		\begin{equation*}
			\Delta\eta_1
			=-\paren[\bigg]{\abs{\cF}-\frac{(1+\eps^3)\rho_\cF}{2}}\frac{\abs{\cF}k!\,\eta_1}{n^k\phat}\pm \max_{x\in[i,i+1]}\frac{\eta_1(x)}{\delta n^{2k}\phat(x)^2}.
		\end{equation*}
		We again investigate the first term and the maximum separately.
		Using Lemma~\ref{lemma: edges of H}, we have
		\begin{equation*}
			-\paren[\bigg]{\abs{\cF}-\frac{(1+\eps^3)\rho_\cF}{2}}\frac{\abs{\cF}k!\,\eta_1}{n^k\phat}
			\Xeq-\paren[\bigg]{\abs{\cF}-\frac{(1+\eps^3)\rho_\cF}{2}}\frac{\abs{\cF}\eta_1}{H}.
		\end{equation*}
		Furthermore, using Lemma~\ref{lemma: bounds of delta phat}, Lemma~\ref{lemma: edges of H} and Lemma~\ref{lemma: zeta and H} yields
		\begin{equation*}
			\max_{x\in[i,i+1]}\frac{\eta_1(x)}{\delta n^{2k}\phat(x)^2}
			\leq \frac{\eta_1}{\delta n^{2k}\phat(i+1) ^2}
			\leq \frac{\eta_1}{\delta^2 n^{2k}\phat^2}
			\Xleq \frac{\eta_1}{\delta^2 H^2}
			\leq \frac{\zeta^{2+2\eps^3}\eta_1}{\delta^2 H}
			\leq \frac{\zeta^{2+\eps^3}\eta_1}{H}.
		\end{equation*}
		Thus we also obtain the desired expression for~$\Delta\eta_1$.
	\end{proof}

	\begin{lemma}\label{lemma: copy trend}
		Let~$0\leq i\leq i^\star$ and~$\cX:=\set{i< \tautilde^\star}$.
		Then,
		\begin{equation*}
			\exi{\Delta H^*}\Xeq -\frac{\abs{\cF}^2H^*}{H}\pm \frac{\zeta^{2}\hhat^*}{\delta^5 H}.
		\end{equation*}
	\end{lemma}
	
	\begin{proof}
		Fix~$f\in\cF$.
		Lemma~\ref{lemma: star codegrees} entails
		\begin{equation*}
			\begin{aligned}
				\exi{\Delta H^*}
				&\Xeq -\frac{1}{H^*}\sum_{\cF'\in \cH^*}\paren[\Big]{\paren[\Big]{\sum_{e\in \cF'} d_{\cH^*}(e)}\pm\abs{\cF}^2\zeta^{2+\eps^2}\phihat_{\cF,f}}\\
				&=-\frac{1}{H^*}\paren[\Big]{\sum_{e\in \cH} d_{\cH^*}(e)^2}\pm \abs{\cF}^2\zeta^{2+\eps^2}\phihat_{\cF,f}.
			\end{aligned}
		\end{equation*}
		For all~$e\in\cH$, from Lemma~\ref{lemma: star degrees}, we obtain
		\begin{equation*}
			d_{\cH^*}(e)\Xeq \frac{\abs{\cF}k!\,\phihat_{\cF,f}}{\aut(\cF)}\pm \frac{1}{\delta}\abs{\cF}k!\,\zeta\phihat_{\cF,f}.
		\end{equation*}
		Thus, Lemma~\ref{lemma: averaging} yields
		\begin{align*}
			\exi{\Delta H^*}&\Xeq -\frac{1}{H^*}\frac{\paren{\sum_{e\in\cH} d_{\cH^*}(e)}^2}{H} \pm \frac{2\abs{\cF}^2(k!)^2\zeta^{2}\phihat_{\cF,f}^2H}{\delta^2 H^*}\pm \abs{\cF}^2\zeta^{2+\eps^2}\phihat_{\cF,f}\\
			&=-\frac{\abs{\cF}^2 H^*}{H} \pm \frac{\phihat_{\cF,f}H}{H^*}\frac{\zeta^{2}\phihat_{\cF,f} H}{\delta^3 H}\pm \frac{\zeta^{2+\eps^2}\phihat_{\cF,f}H}{\delta H}.
		\end{align*}
		Since Lemma~\ref{lemma: edges of H} implies~$\phihat_{\cF,f} H\leq \hhat^*/\eps$, we obtain
		\begin{equation*}
			\exi{\Delta H^*}
			\Xeq-\frac{\abs{\cF}^2H^*}{H}\pm \frac{\zeta^{2}\hhat^*}{\delta^5 H},
		\end{equation*}
		which completes the proof.
	\end{proof}
	
	\begin{lemma}\label{lemma: copy deviations are supermartingales}
		Let~$0\leq i_0\leq i$ and~$\pom\in\set{-,+}$.
		Then,~$\exi{\Delta Z^\pom_{i_0}}\leq 0$.
	\end{lemma}
	
	\begin{proof}
		Suppose that~$i< i^\star$ and let~$\cX:=\set{ i< \tau_{i_0}^\pom\wedge \tautilde^\star}$.
		We have~$\exi{\Delta Z^\pom_{i_0}}=_{\cX^\comp}0$ and~$\exi{\Delta Z^\pom_{i_0}}\Xeq \exi{\Delta Y^\pom}$, so it suffices to obtain~$\exi{\Delta Y^\pom}\Xleq 0$.
		Combining Lemma~\ref{lemma: delta hhat, delta eta} with Lemma~\ref{lemma: copy trend}, we obtain
		\begin{equation*}
			\begin{aligned}
				\exi{\Delta Y^\pom}
				&= \pom(\exi{\Delta H^*}-\Delta\hhat^*)-\Delta\eta_1\\
				&\Xleq \pom\paren[\bigg]{ -\frac{\abs{\cF}^2}{H}H^*+\frac{\abs{\cF}^2}{H}\hhat^* }+\paren[\bigg]{\abs{\cF}-\frac{(1+\eps^3)\rho_\cF}{2}}\frac{\abs{\cF}}{H}\eta_1+\frac{\zeta^2}{\delta^5 H}\hhat^*+\frac{2\zeta^{2+\eps^2}}{H}\hhat^*\\
				&\leq -\frac{\abs{\cF}}{H}\paren[\bigg]{\pom\abs{\cF}(H^*-\hhat^*)-\paren[\bigg]{\abs{\cF}-\frac{\rho_\cF}{2}}\eta_1-\eps^2\eta_1}\\
				&\Xleq -\frac{\abs{\cF}}{H}\paren[\bigg]{\abs{\cF}(1-\eps)\eta_1-\paren[\bigg]{\abs{\cF}-\frac{\rho_\cF}{2}}\eta_1-\eps^2\eta_1}\\
				&=-\frac{\abs{\cF}\eta_1}{H}\paren[\bigg]{\frac{\rho_\cF}{2}-\eps\abs{\cF}-\eps^2}
				\leq 0,
			\end{aligned}
		\end{equation*}
		which completes the proof.
	\end{proof}

	\subsection{Boundedness}\label{subsection: copy boundedness}
	As we intend to apply Azuma's inequality, it suffices to obtain suitable bounds for the absolute one-step changes of the processes~$Y^\pom(0),Y^\pom(1),\ldots$ and~$Z^\pom_{i_0}(i_0),Z^\pom_{i_0}(i_0+1),\ldots$.
	
	\begin{lemma}\label{lemma: absolute change copies not stopped}
		Let~$0\leq i_0\leq i\leq i^\star$,~$\pom\in\set{-,+}$,~$f\in\cF$ and~$\cX:=\set{i<\tau_\ccF}$.
		Then,~$\abs{\Delta Y^\pom}\Xleq \phihat_{\cF,f}(i_0)/\delta$.
	\end{lemma}
	\begin{proof}
		From Lemma~\ref{lemma: star degrees}, we obtain
		\begin{equation*}
			\abs{\Delta H^*}
			\leq\sum_{e\in\cF_0(i+1)}d_{\cH^*}(e)
			\leq \sum_{e\in\cF_0(i+1)}\sum_{f'\in\cF}\sum_{\psi\colon f'\bijection e}\Phi_{\cF,\psi}
			\Xleq 2\abs{\cF}^2k!\,\phihat_{\cF,f}.
		\end{equation*}
		Hence, using Lemma~\ref{lemma: delta hhat, delta eta}, we have
		\begin{equation*}
			\abs{\Delta Y^\pom}\leq \abs{\Delta H^*}+\abs{\Delta \hhat^*}+\abs{\Delta \eta_1}
			\Xleq 2\abs{\cF}^2k!\,\phihat_{\cF,f}+\frac{2\abs{\cF}^2\hhat^*}{H}+\frac{2\abs{\cF}^2\eta_1}{H}.
		\end{equation*}
		With Lemma~\ref{lemma: edges of H} and~$\phihat_{\cF,f}\leq \phihat_{\cF,f}(i_0)$, this completes the proof.
	\end{proof}
	
	\begin{lemma}\label{lemma: absolute change copies}
		Let~$0\leq i_0\leq i$,~$\pom\in\set{-,+}$ and~$f\in\cF$.
		Then,~$\abs{\Delta Z^\pom_{i_0}}\leq \phihat_{\cF,f}(i_0)/\delta$.
	\end{lemma}

	\begin{proof}
		This is an immediate consequence of Lemma~\ref{lemma: absolute change copies not stopped}.
	\end{proof}
	
	\subsection{Supermartingale concentration}\label{subsection: copy concentration}
	This section follows a similar structure as Sections~\ref{subsubsection: chain concentration} and~\ref{subsubsection: family concentration}.
	Lemma~\ref{lemma: initial error copies} is the final ingredient that we use for our application of Azuma's inequality in the proof of Lemma~\ref{lemma: control copies} where we show that the probabilities of the events on the right in Observation~\ref{observation: copies critical times} are indeed small.
	
	\begin{lemma}\label{lemma: initial error copies}
		Let~$\pom\in\set{-,+}$.
		Then,~$Z^\pom_{\sigma^\pom}(\sigma^\pom)\leq -\eps^2\eta_1(\sigma^\pom)$.
	\end{lemma}
	
	\begin{proof}
		Lemma~\ref{lemma: initially good} implies~$\tautilde^\star\geq 1$ and~$\pom(H^*(0)-\hhat^*(0))< \eta_0(0)$, so we have~$\sigma^\pom\geq 1$. 
		Thus, by definition of~$\sigma^\pom$, for~$i:=\sigma^\pom-1$, we have~$\pom(H^*-\hhat^*)\leq \eta_0$ and thus
		\begin{equation*}
			Z^\pom_i=\pom(H^*-\hhat^*)-\eta_1\leq -\eps\eta_1.
		\end{equation*}
		Furthermore, since~$\sigma^\pom\leq \tau_\ccF$, we may apply Lemma~\ref{lemma: absolute change copies not stopped} such that with Lemma~\ref{lemma: edges of H} and Lemma~\ref{lemma: zeta and H}, for~$f\in\cF$, we obtain
		\begin{equation*}
			Z^\pom_{\sigma^\pom}(\sigma^\pom)
			=Z^\pom_i+\Delta Y^\pom
			\leq -\eps \eta_1+\frac{\phihat_{\cF,f}}{\delta}
			\leq -\eps \eta_1+\frac{\hhat^*}{\delta^2 H}
			\leq -\eps \eta_1+\zeta^{2+\eps^2}\hhat^*
			\leq -\eps^2 \eta_1.
		\end{equation*}
		Since~$\Delta\eta_1\leq 0$, this completes the proof.
	\end{proof}
	
	\begin{lemma}\label{lemma: control copies}
		$\pr{\tau_{\cH^*}\leq\tautilde^\star\wedge i^\star}\leq \exp(-n^{\eps^2})$.
	\end{lemma}
	
	\begin{proof}
		Fix~$\pom\in\set{-,+}$.
		By Observation~\ref{observation: copies critical times}, is suffices to show that
		\begin{equation*}
			\pr{Z^\pom_{\sigma^\pom}(i^\star)>0}\leq \exp(-n^{2\eps^2}).
		\end{equation*}
		Due to Lemma~\ref{lemma: initial error copies}, we have
		\begin{equation*}
			\pr{Z^\pom_{\sigma^\pom}(i^\star)>0}
			\leq \pr{Z^\pom_{\sigma^\pom}(i^\star)-Z^\pom_{\sigma^\pom}(\sigma^\pom)\geq \eps^2\eta_{1}(\sigma^\pom)}
			\leq \sum_{0\leq i\leq i^\star} \pr{Z^\pom_{i}(i^\star)-Z^\pom_{i}\geq \eps^2\eta_{1}}.
		\end{equation*}
		Thus, for~$0\leq i\leq i^\star$, it suffices to obtain
		\begin{equation*}
			\pr{Z^\pom_{i}(i^\star)-Z^\pom_{i}\geq \eps^2\eta_{1}}\leq \exp(-n^{3\eps^2}).
		\end{equation*}
		We show that this bound is a consequence of Azuma's inequality.
		
		Fix~$f\in\cF$.
		Lemma~\ref{lemma: copy trend} shows that~$Z^\pom_i(i),Z^\pom_i(i+1),\ldots$ is a supermartingale, while Lemma~\ref{lemma: absolute change copies} provides the bound~$\abs{\Delta Z^\pom_i(j)}\leq \phihat_{\cF,f}/\delta$ for all~$j\geq i$.
		Hence, we may apply Lemma~\ref{lemma: azuma} to obtain
		\begin{equation*}
			\pr{Z^\pom_{i}(i^\star)-Z^\pom_{i}\geq \eps^2\eta_{1}}
			\leq \exp\paren[\bigg]{ -\frac{ \eps^4\delta^2\eta_{1}^2 }{2(i^\star-i)\phihat_{\cF,f}^2 }}.
		\end{equation*}
		Since
		\begin{equation*}
			i^\star-i\leq \frac{\theta n^k}{\abs{\cF}k!}-i=\frac{n^k\phat}{\abs{\cF}k!},
		\end{equation*}
		this yields
		\begin{equation*}
			\begin{aligned}
				\pr{Z^\pom_{i}(i^\star)-Z^\pom_{i}\geq \eps^2\eta_{1}}
				&\leq \exp\paren[\bigg]{ -\frac{ \eps^5\delta^2\eta_{1}^2 }{n^k\phat \phihat_{\cF,f}^2 }}
				= \exp\paren[\bigg]{ -\frac{\eps^5\delta^{2}\zeta^{2+2\eps^3}(\hhat^*)^2}{n^k\phat\phihat_{\cF,f}^2} }
				\leq  \exp\paren{ -\delta^{3}\zeta^{2+2\eps^3}n^k\phat}\\
				&\leq  \exp\paren{ -\delta^{3}\zeta^{2+2\eps^3}(n\phat^{\rho_\cF})^k}
				= \exp\paren{ -\delta^{3}n^{2k\eps^2}\zeta^{2+2\eps^3-2k}}
				\leq \exp(-n^{4\eps^2}),
			\end{aligned}
		\end{equation*}
		which completes the proof.
	\end{proof}

	\section{Counting balanced templates}\label{appendix: balanced}
	In this section, our goal is to prove Lemma~\ref{lemma: auxiliary control}~\ref{item: control balanced}.
	Hence, for this section, we assume the setup that we used in Section~\ref{section: stopping times} to state Lemma~\ref{lemma: auxiliary control}.
	Similarly as in Sections~\ref{subsection: tracking chains} and~\ref{subsection: tracking families}, this requires us to consider several balanced templates, however, it again suffices to essentially only consider a fixed balanced template~$(\cA,I)$, see Observation~\ref{observation: balanced individual} below.
	Moreover, we may assume that~$\cA\setminus\cA[I]\neq\emptyset$ as otherwise, for all~$\psi\colon I\injection V_\cH$ and~$0\leq i\leq i^\star$, we have~$\Phi_{\cA,\psi}=(1\pm\zeta^\delta)\phihat_{\cA,I}$ as a consequence of Lemma~\ref{lemma: bounds of zeta}.
	Overall, our approach is similar as in Sections~\ref{subsection: tracking chains} and~\ref{subsection: tracking families}.
	
	\begin{observation}\label{observation: balanced individual}
		For~$(\cA,I)\in\ccB$ and~$\psi\colon I\injection V_\cH$, let
		\begin{equation*}
			\tau_{\cA,\psi}:=\min\cset{i\geq 0}{ \Phi_{\cA,\psi}\neq (1\pm\zeta^\delta)\phihat_{\cA,I} }.
		\end{equation*}
		Then,
		\begin{equation*}
			\pr{\tau_{\ccB}\leq \tautilde^\star \wedge i^\star}\leq \sum_{\substack{(\cA,I)\in\ccB\colon \cA\setminus\cA[I]\neq\emptyset,\\ \psi\colon I\injection V_\cH}} \pr{ \tau_{\cA,\psi}\leq \tautilde^\star \wedge i^{\delta^{1/2}}_{\cA,I}\wedge i^\star }.
		\end{equation*}
	\end{observation}
	
	Fix~$(\cA,I)\in\ccB$ with~$\cA\setminus\cA[I]\neq\emptyset$ and~$\psi\colon I\injection V_\cH$ and for~$i\geq 0$, let
	\begin{equation*}
		\xi_1(i):=\zeta^\delta\phihat_{\cA,I}\qtand
		\xi_0(i):=(1-\delta^2)\xi_1
	\end{equation*}
	and define the stopping time
	\begin{equation*}
		\tau:=\min\cset{i\geq 0}{ \Phi_{\cA,\psi}\neq \phihat_{\cA,I}\pm\xi_1 }.
	\end{equation*}
	We only expect tight concentration of~$\Phi_{\cA,\psi}$ around~$\phihat_{\cA,I}$ as long as we expect~$\Phi_{\cA,\psi}$ to be sufficiently large, that is up to step~$i_{\cA,I}^{\delta^{1/2}}$.
	Formally, in this section it is our goal to obtain an upper bound for the probability that~$\tau\leq \tautilde^\star\wedge i_{\cA,I}^{\delta^{1/2}}\wedge i^\star$ and hence the minimum~$i_{\cA,I}^{\delta^{1/2}}\wedge i^\star$ often plays the role that~$i^\star$ plays in Sections~\ref{subsection: tracking chains} and~\ref{subsection: tracking families}.
	
	Define the critical intervals
	\begin{equation*}
		I^-(i):=[\phihat_{\cA,I}-\xi_1,\phihat_{\cA,I}-\xi_0]\qtand
		I^+(i):=[\phihat_{\cA,I}+\xi_0,\phihat_{\cA,I}+\xi_1].
	\end{equation*}
	For~$\pom\in\set{-,+}$, let
	\begin{equation*}
		Y^\pom(i):=\pom(\Phi_{\cA,\psi}-\phihat_{\cA,I})-\xi_1.
	\end{equation*}
	For~$i_0\geq 0$ define the stopping time
	\begin{equation*}
		\tau^\pom_{i_0}:=\min\cset{i\geq i_0}{\Phi_{\cA,\psi}\notin I^\pom}
	\end{equation*}
	and for~$i\geq i_0$, let
	\begin{equation*}
		Z^\pom_{i_0}(i):=Y^\pom(i_0\vee (i\wedge \tau^\pom_{i_0}\wedge \tautilde^\star \wedge i^{\delta^{1/2}}_{\cA,I}\wedge i^\star)).
	\end{equation*}
	Let
	\begin{equation*}
		\sigma^\pom:=\min\cset{j\geq 0}{\pom(\Phi_{\cA,\psi}-\phihat_{\cA,I})\geq \xi_0 \stforall j\leq i< \tautilde^\star\wedge i^{\delta^{1/2}}_{\cA,I}\wedge i^\star}\leq \tautilde^\star\wedge i^{\delta^{1/2}}_{\cA,I}\wedge i^\star.
	\end{equation*}
	With this setup, similarly as in Sections~\ref{subsection: tracking chains} and~\ref{subsection: tracking families}, it in fact suffices to consider the evolution of~$Z^\pom_{\sigma^\pom}(\sigma^\pom),Z^\pom_{\sigma^\pom}(\sigma^\pom+1),\ldots$.
	\begin{observation}\label{observation: balanced critical times}
		$\set{\tau\leq \tautilde^\star\wedge i^{\delta^{1/2}}_{\cA,I}\wedge i^\star}\subseteq\set{Z^-_{\sigma^-}(i^\star)>0}\cup \set{Z^+_{\sigma^+}(i^\star)>0}$.
	\end{observation}
	
	We again use supermartingale concentration techniques to show that the probabilities of the events on the right in Observation~\ref{observation: balanced critical times} are sufficiently small.
	More specifically, for this section, we use Lemma~\ref{lemma: freedman}.
	
	\subsection{Trend}\label{subsection: balanced trend}
	Here, we prove that for all~$\pom\in\set{-,+}$ and~$i_0\geq 0$, the expected one-step changes of the process~$Z^\pom_{i_0}(i_0),Z^\pom_{i_0}(i_0+1),\ldots$ are non-positive.
	Lemma~\ref{lemma: delta phihat} already yields estimates for the one-step changes of the relevant deterministic trajectory, in Lemma~\ref{lemma: delta xi balanced} we estimate the one-step changes of the error term that we use in this section.
	Then we state Lemma~\ref{lemma: balanced change} where we provide a precise estimate for the expected one-step change of the non-deterministic part that holds whenever the removal process was well-behaved up to the step we consider.
	Finally, combining these estimates shows that the above process is indeed a supermartingale (see Lemma~\ref{lemma: balanced trend}).
	
	\begin{observation}\label{observation: xi balanced}
		Extend~$\phat$ and~$\xi_1$ to continuous trajectories defined on the whole interval~$[0,i^\star+1]$ using the same expressions as above.
		Then, for~$x\in [0,i^\star+1]$,
		\begin{equation*}
			\begin{gathered}
				\xi_1'(x)=-\frac{(\abs{\cA}-\abs{\cA[I]}-\frac{\delta\rho_\cF}{2})\abs{\cF}k!\,\xi_1(x)}{n^k\phat(x)},\\
				\xi_1''(x)=-\frac{(\abs{\cA}-\abs{\cA[I]}-\frac{\delta\rho_\cF}{2})(\abs{\cA}-\abs{\cA[I]}-\frac{\delta\rho_\cF}{2}-1)\abs{\cF}^2(k!)^2\xi_1(x)}{n^{2k}\phat(x)^2}.
			\end{gathered}
		\end{equation*} 
	\end{observation}
	
	\begin{lemma}\label{lemma: delta xi balanced}
		Let~$0\leq i\leq i^\star$ and~$\cX:=\set{i\leq \tau_\emptyset}$.
		Then,
		\begin{equation*}
			\Delta \xi_1\Xeq-\paren[\bigg]{\abs{\cA}-\abs{\cA[I]}-\frac{\delta\rho_\cF}{2}}\frac{\abs{\cF}\xi_1}{H}\pm \frac{\zeta \xi_1}{H}.
		\end{equation*}
	\end{lemma}
	\begin{proof}
		This is a consequence of Taylor's theorem.
		In detail, we argue as follows.
		
		Together with Observation~\ref{observation: xi balanced}, Lemma~\ref{lemma: taylor} yields
		\begin{equation*}
			\Delta\xi_1 = -\paren[\bigg]{\abs{\cA}-\abs{\cA[I]}-\frac{\delta\rho_\cF}{2}}\frac{\abs{\cF}k!\,\xi_1}{n^k\phat}\pm \max_{x\in[i,i+1]} \frac{\xi_1(x)}{\delta n^{2k}\phat(x)^2}.
		\end{equation*}
		We investigate the first term and the maximum separately.
		Using Lemma~\ref{lemma: edges of H}, we have
		\begin{equation*}
			-\paren[\bigg]{\abs{\cA}-\abs{\cA[I]}-\frac{\delta\rho_\cF}{2}}\frac{\abs{\cF}k!\,\xi_1}{n^k\phat}
			\Xeq-\paren[\bigg]{\abs{\cA}-\abs{\cA[I]}-\frac{\delta\rho_\cF}{2}}\frac{\abs{\cF}\xi_1}{H}.
		\end{equation*}
		Furthermore, using Lemma~\ref{lemma: bounds of delta phat}, Lemma~\ref{lemma: edges of H} and Lemma~\ref{lemma: zeta and H} yields
		\begin{equation*}
			\max_{x\in[i,i+1]} \frac{\xi_1(x)}{\delta n^{2k}\phat(x)^2}
			\leq \frac{\xi_1}{\delta n^{2k}\phat(i+1)^2}
			\leq \frac{\xi_1}{\delta^2 n^{2k}\phat^2}
			\Xleq \frac{\xi_1}{\delta^2 H^2}
			\leq \frac{\zeta^{2+2\eps^2}\xi_1}{\delta^2 H}
			\leq \frac{\zeta^{2+\eps^2}\xi_1}{H}.
		\end{equation*}
		Thus we obtain the desired expression for~$\Delta\xi_1$.
	\end{proof}
	
	\begin{lemma}\label{lemma: balanced change}
		Let~$0\leq i\leq i_{\cA,I}^{\delta^{1/2}}\wedge i^\star$ and~$\cX:=\set{i<\tautilde^\star}$.
		Then,
		\begin{equation*}
			\exi{\Delta\Phi_{\cA,\psi}}\Xeq -(\abs{\cA}-\abs{\cA[I]})\frac{\abs{\cF}}{H}\Phi_{\cA,\psi}\pm \frac{\zeta^{1/2}\phihat_{\cA,I}}{H}.
		\end{equation*}
	\end{lemma}
	
	\begin{proof}
		Fix~$f\in\cF$.
		Lemma~\ref{lemma: star codegrees} entails
		\begin{equation*}
			\exi{\Delta\Phi_{\cA,\psi}}\Xeq-\frac{1}{H^*}\sum_{\phi\in\Phi_{\cA,\psi}^\sim}\paren[\Big]{\paren[\Big]{\sum_{e\in\cA\setminus\cA[I]} d_{\cH^*}(\phi(e))}\pm \abs{\cA}^2\zeta\phihat_{\cF,f} }.
		\end{equation*}
		From Lemma~\ref{lemma: star degrees}, for all~$e\in \cH$, we obtain
		\begin{equation*}
			d_{\cH^*}(e)\Xeq \frac{\abs{\cF}k!\,\phihat_{\cF,f}}{\aut(\cF)}\pm\frac{1}{\delta}\abs{\cF}k!\,\zeta\phihat_{\cF,f}.
		\end{equation*}
		Thus, due to Lemma~\ref{lemma: edges of H}, we have
		\begin{align*}
			\exi{\Delta\Phi_{\cA,\psi}}
			&\Xeq -\frac{1}{H^*}\Phi_{\cA,\psi}\paren[\bigg]{(\abs{\cA}-\abs{\cA[I]}) \frac{\abs{\cF}k!\,\phihat_{\cF,f}}{\aut(\cF)}\pm \frac{1}{\delta^2}\zeta\phihat_{\cF,f} }\\
			&=-\frac{\abs{\cF}k!\,\phihat_{\cF,f}}{\aut(\cF)H^*}\paren[\bigg]{(\abs{\cA}-\abs{\cA[I]}) \Phi_{\cA,\psi}\pm \frac{1}{\delta^3}\zeta\Phi_{\cA,\psi} }\\
			&\Xeq-(1\pm \zeta^{1+\eps^4})\frac{\abs{\cF}}{H}\paren[\bigg]{(\abs{\cA}-\abs{\cA[I]}) \Phi_{\cA,\psi}\pm \frac{1}{\delta^3}\zeta\Phi_{\cA,\psi} }\\
			&=-(\abs{\cA}-\abs{\cA[I]})\frac{\abs{\cF}}{H} \Phi_{\cA,\psi}\pm \frac{\zeta\Phi_{\cA,\psi}}{\delta^4 H} 
			\Xeq-(\abs{\cA}-\abs{\cA[I]})\frac{\abs{\cF}}{H}\Phi_{\cA,\psi}\pm\frac{\zeta^{1/2}\phihat_{\cA,I}}{H} ,
		\end{align*}
		which completes the proof.
	\end{proof}
	
	\begin{lemma}\label{lemma: balanced trend}
		Let~$0\leq i_0\leq i$ and~$\pom\in\set{-,+}$. Then,~$\exi{\Delta Z^\pom_{i_0}}\leq 0$.
	\end{lemma}
	
	\begin{proof}
		Suppose that~$i\leq i_{\cA,I}^{\delta^{1/2}}\wedge i^\star$ and let~$\cX:=\set{i<\tau^\pom_{i_0}\wedge \tautilde^\star}$.
		We have~$\exi{\Delta Z_{i_0}^\pom}=_{\cX^\comp}0$ and~$\exi{\Delta Z_{i_0}^\pom }\Xeq \exi{\Delta Y^\pom}$, so it suffices to obtain~$\exi{\Delta Y^\pom}\Xleq 0$.
		Combining Lemma~\ref{lemma: delta phihat}, Lemma~\ref{lemma: delta xi balanced} and Lemma~\ref{lemma: balanced change}, we obtain
		\begin{equation*}
			\begin{aligned}
				\exi{\Delta Y^\pom}
				&= \pom(\exi{\Delta \Phi_{\cA,\psi}}-\Delta\phihat_{\cA,I})-\Delta\xi_1\\
				&\Xleq \pom\paren[\bigg]{ -(\abs{\cA}-\abs{\cA[I]})\frac{\abs{\cF}}{H}\Phi_{\cA,\psi}+(\abs{\cA}-\abs{\cA[I]})\frac{\abs{\cF}}{H}\phihat_{\cA,I} }\\&\hphantom{\Xleq}\mathrel{}\quad+\paren[\bigg]{ \abs{\cA}-\abs{\cA[I]}-\frac{\delta\rho_\cF}{2} }\frac{\abs{\cF}\xi_1}{H}+ \frac{\zeta^{1/3}\phihat_{\cA,I}}{H}\\
				&\leq -\frac{\abs{\cF}}{H}\paren[\bigg]{\pom(\abs{\cA}-\abs{\cA[I]})(\Phi_{\cA,\psi}-\phihat_{\cA,I})-\paren[\bigg]{\abs{\cA}-\abs{\cA[I]}-\frac{\delta\rho_\cF}{2}}\zeta^\delta\phihat_{\cA,I}-\zeta^{1/3}\phihat_{\cA,I}}\\
				&\Xleq -\frac{\abs{\cF}\phihat_{\cA,I}}{H}\paren[\bigg]{(\abs{\cA}-\abs{\cA[I]})(1-\delta^2)\zeta^{\delta}-\paren[\bigg]{\abs{\cA}-\abs{\cA[I]}-\frac{\delta\rho_\cF}{2}}\zeta^\delta-\zeta^{1/3}}\\
				&= -\frac{\abs{\cF}\phihat_{\cA,I}}{H}\paren[\bigg]{-(\abs{\cA}-\abs{\cA[I]})\delta^2\zeta^{\delta}+\frac{\delta\rho_\cF}{2}\zeta^\delta-\zeta^{1/3}}
				\leq 0,
			\end{aligned}
		\end{equation*}
		which completes the proof.
	\end{proof}
	
	\subsection{Boundedness}\label{subsection: balanced boundedness}
	Here, similarly as in Sections~\ref{subsubsection: chain boundedness} and~\ref{subsubsection: family boundedness}, we obtain suitable bounds for the absolute one-step changes of the processes~$Y^\pom(0),Y^\pom(1),\ldots$ and~$Z_{i_0}^\pom(i_0),Z_{i_0}^\pom(i_0+1),\ldots$ (see Lemma~\ref{lemma: absolute change balanced not stopped} and Lemma~\ref{lemma: absolute change balanced}) as well as for the expected absolute one-step changes of these processes (see Lemma~\ref{lemma: expected change balanced not stopped} and Lemma~\ref{lemma: expected change balanced}).
	To obtain these bounds, we argue similarly as in Section~\ref{subsection: tracking chains}.

	\begin{lemma}\label{lemma: absolute change balanced not stopped}
		Let~$0\leq i_0\leq i\leq i^{\delta^{1/2}}_{\cA,I}\wedge i^\star$,~$\pom\in\set{-,+}$ and~$\cX:=\set{i<\tau_{\ccB}\wedge\tau_{\ccB'}}$.
		Then,~$\abs{\Delta Y^\pom}\Xleq \zeta(i_0)^{8\delta}\phihat_{\cA,I}(i_0)$.
	\end{lemma}
	
	\begin{proof}
		From Lemma~\ref{lemma: delta phihat} and Lemma~\ref{lemma: delta xi balanced}, we obtain
		\begin{equation*}
			\abs{\Delta Y^\pom}
			\leq \abs{\Delta \Phi_{\cA,\psi}}+\abs{\Delta \phihat_{\cA,I}}+\abs{\Delta\xi_1}
			\leq \abs{\Delta \Phi_{\cA,\psi}}+2\frac{\abs{\cA}\abs{\cF}\phihat_{\cA,I}}{H}+2\frac{\abs{\cA}\abs{\cF}\xi_1}{H}.
		\end{equation*}
		Hence, since~$\cA\setminus\cA[I]\neq\emptyset$ implies~$\zeta^{8\delta}\phihat_{\cA,I}\leq \zeta(i_0)^{8\delta}\phihat_{\cA,I}(i_0)$, by Lemma~\ref{lemma: zeta and H} it suffices to show that
		\begin{equation*}
			\abs{\Delta \Phi_{\cA,\psi}}\Xleq \zeta^{9\delta}\phihat_{\cA,I},
		\end{equation*}
		which we obtain as a consequence of Lemma~\ref{lemma: loss at one edge}.
		
		To this end, note that for all~$(\cB,I)\subseteq (\cA,I)$ with~$V_\cB\neq I$, since~$(\cA,I)$ is balanced, we have~$\rho_{\cB,I}\leq\rho_{\cA,I}$ and thus using Lemma~\ref{lemma: trajectory at cutoff}, we obtain
		\begin{equation*}
			\phihat_{\cB,I}
			=(n\phat^{\rho_{\cB,I}})^{\abs{V_\cB}-\abs{I}}
			\geq (n\phat^{\rho_{\cA,I}})^{\abs{V_\cB}-\abs{I}}
			= \phihat_{\cA,I}^{(\abs{V_\cB}-\abs{I})/(\abs{V_\cA}-\abs{I})}
			\geq \phihat_{\cA,I}^{10\delta^{1/2}}
			\geq \frac{\zeta^{-10\delta}}{2}.
		\end{equation*}
		Hence, Lemma~\ref{lemma: loss at one edge} implies
		\begin{equation*}
			\abs{\Delta \Phi_{\cA,\psi}}
			\leq \sum_{e\in\cA\setminus\cA[{I}]}\abs{\cset{ \phi\in\Phi_{\cA,\psi}^\sim }{ \phi(e)\in\cF_0(i+1) }}
			\Xleq \abs{\cA}\cdot 4k!\,\abs{\cF}(\log n)^{\alpha_{\cA,I}}\zeta^{10\delta}\phihat_{\cA,I}
			\leq \zeta^{9\delta}\phihat_{\cA,I},
		\end{equation*}
		which completes the proof.
	\end{proof}
	
	\begin{lemma}\label{lemma: absolute change balanced}
		Let~$0\leq i_0\leq i$ and~$\pom\in\set{-,+}$.
		Then,~$\abs{\Delta Z^\pom_{i_0}}\leq \zeta(i_0)^{8\delta}\phihat_{\cA,I}(i_0)$.
	\end{lemma}
	
	\begin{proof}
		This is an immediate consequence of Lemma~\ref{lemma: absolute change balanced not stopped}.
	\end{proof}
	
	\begin{lemma}\label{lemma: expected change balanced not stopped}
		Let~$0\leq i\leq i^{\delta^{1/2}}_{\cA,I}\wedge i^\star$,~$\pom\in\set{-,+}$ and~$\cX:=\set{i<\tautilde^\star}$.
		Then,
		\begin{equation*}
			\exi{\abs{\Delta Y^\pom}}\Xleq \frac{\phihat_{\cA,I}}{\zeta^{5\delta}n^k\phat}.
		\end{equation*}
	\end{lemma}
	
	\begin{proof}
		From Lemma~\ref{lemma: delta phihat} and Lemma~\ref{lemma: delta xi balanced}, we obtain
		\begin{equation*}
			\exi{\abs{\Delta Y^\pom}}
			\leq \exi{\abs{\Delta \Phi_{\cA,I}}}+\abs{\Delta \phihat_{\cA,I}}+\abs{\Delta\xi_1}
			\leq \exi{\abs{\Delta \Phi_{\cA,I}}}+2\frac{\abs{\cA}\abs{\cF}\phihat_{\cA,I}}{H}+2\frac{\abs{\cA}\abs{\cF}\xi_1}{H}.
		\end{equation*}
		Hence, since~$\cA\setminus\cA[I]\neq\emptyset$ implies
		\begin{equation*}
			\frac{\phihat_{\cA,I}}{\zeta^{5\delta}n^k\phat}\leq \frac{\phihat_{\cA,I}(i_0)}{\zeta(i_0)^{5\delta}n^k\phat(i_0)},
		\end{equation*}
		by Lemma~\ref{lemma: edges of H} implies that it suffices to show that
		\begin{equation*}
			\exi{\abs{\Delta \Phi_{\cA,I}}}\leq \frac{\phihat_{\cA,I}}{\zeta^{4\delta}n^k\phat}.
		\end{equation*}
		We obtain this as a consequence of Lemma~\ref{lemma: arbitrary embedding} and Lemma~\ref{lemma: loss at one edge}.
		
		We argue similarly as in the proof of Lemma~\ref{lemma: expected change chain deviation}.
		For~$e\in\cA\setminus\cA[I]$, from all subtemplates~$(\cB,I)\subseteq(\cA,I)$ with~$e\in\cB$ choose~$(\cB_e,I)$ such that~$\phihat_{\cB_e,I}$ is minimal.
		Furthermore, for every subtemplate~$(\cB,I)\subseteq (\cA,I)$, let
		\begin{equation*}
			\Phi^e_{\cB,\psi}:=\abs{\cset{ \phi\in \Phi^\sim_{\cB,\psi} }{ \phi(e)\in \cF_0(i+1) }}.
		\end{equation*}
		Lemma~\ref{lemma: loss at one edge} yields
		\begin{equation*}
			\Phi^e_{\cA,I}
			\Xleq 2k!\,\abs{\cF}(\log n)^{\alpha_{\cA,I\cup e}} \frac{\phihat_{\cA,I}}{\phihat_{\cB_e,I}},
		\end{equation*}
		so we obtain
		\begin{equation}\label{equation: random upper bound balanced change}
			\begin{aligned}
				\abs{\Delta \Phi_{\cA,\psi}}
				&\leq \sum_{e\in\cA\setminus\cA[I]}\Phi^e_{\cA,\psi}
				= \sum_{e\in\cA\setminus\cA[I]}\ind_{\set{\Phi^e_{\cA,\psi}\geq 1}}\Phi^e_{\cA,\psi}\\
				&\Xleq 2k!\,\abs{\cF}(\log n)^{\alpha_{\cA,I\cup e}}\phihat_{\cA,I}\sum_{e\in\cA\setminus\cA[I]} \frac{\ind_{\set{\Phi^e_{\cA,\psi}\geq 1}}}{\phihat_{\cB_e,I}}
				\leq \frac{\phihat_{\cA,I}}{\zeta^\delta}\sum_{e\in\cA\setminus\cA[I]} \frac{\ind_{\set{\Phi^e_{\cA,\psi}\geq 1}}}{\phihat_{\cB_e,I}}\\
				&\leq \frac{\phihat_{\cA,I}}{\zeta^\delta}\sum_{e\in\cA\setminus\cA[I]} \frac{\ind_{\set{\Phi^e_{\cB_e,\psi}\geq 1}}}{\phihat_{\cB_e,I}}
				\leq \frac{\phihat_{\cA,I}}{\zeta^\delta}\sum_{e\in\cA\setminus\cA[I]}\sum_{\phi\in\Phi_{\cB_e,\psi}^\sim} \frac{\ind_{\set{\phi(e)\in\cF_0(i+1)}}}{\phihat_{\cB_e,I}}.
			\end{aligned}
		\end{equation}
		For all~$e\in \cH$,~$f\in\cF$ and~$\psi'\colon f\bijection e$, we have~$\Phi_{\cF,\psi'}\Xeq (1\pm \delta^{-1}\zeta)\phihat_{\cF,f}$.
		Furthermore, we have~$H^*\Xeq (1\pm \zeta^{1+\eps^3})\hhat^*$.
		Thus, using Lemma~\ref{lemma: star degrees}, for all~$e\in\cA\setminus\cA[I]$ and~$\phi\in\Phi_{\cB_e,\psi}^\sim$, we obtain
		\begin{equation*}
			\pri{\phi(e)\in\cF_0(i+1)}
			= \frac{ d_{\cH^*}(\phi(e)) }{H^*}
			\Xleq \frac{ 2\abs{\cF}k!\,\phihat_{\cF,f}  }{H^*}
			\Xleq \frac{ 4\abs{\cF}k!\,\phihat_{\cF,f}  }{\hhat^*}
			\leq \frac{1}{\zeta^\delta n^k\phat}.
		\end{equation*}
		Combining this with~\eqref{equation: random upper bound balanced change} yields
		\begin{equation*}
			\exi{\abs{\Delta \Phi_{\cA,I}}}
			\Xleq \frac{\phihat_{\cA,I}}{\zeta^\delta}\sum_{e\in\cA\setminus\cA[I]}\sum_{\phi\in\Phi_{\cB_e,I}^\sim} \frac{\pri{\phi(e)\in\cF_0(i+1)}}{\phihat_{\cB_e,I}}
			\Xleq \frac{\phihat_{\cA,I}}{\zeta^{2\delta}n^k\phat}\sum_{e\in\cA\setminus\cA[I]} \frac{\Phi_{\cB_e,I}}{\phihat_{\cB_e,I}}.
		\end{equation*}
		This shows that it suffices to prove that~$\Phi_{\cB_e,I}\Xleq \phihat_{\cB_e,I}/\zeta^\delta$, which we obtain as a consequence of Lemma~\ref{lemma: arbitrary embedding} as follows.
		First, note that since~$(\cA,I)$ is balanced, for all~$e\in\cA\setminus\cA[I]$ and~$(\cC,I)\subseteq(\cB_e,I)\subseteq(\cA,I)$, we have~$\rho_{\cC,I}\leq \rho_{\cA,I}$ and thus
		\begin{equation*}
			\phihat_{\cC,I}
			=(n\phat^{\rho_{\cC,I}})^{\abs{V_\cC}-\abs{I}}
			\geq (n\phat^{\rho_{\cA,I}})^{\abs{V_\cC}-\abs{I}}
			=\phihat_{\cA,I}^{(\abs{V_\cC}-\abs{I})/(\abs{V_\cA}-\abs{I})}.
		\end{equation*}
		As Lemma~\ref{lemma: trajectory at cutoff} implies~$\phihat_{\cA,I}\geq (1-n^{-\eps^3})\zeta^{-\delta^{1/2}}\geq 1$, this entails~$\phihat_{\cC,I}\geq 1$ and so Lemma~\ref{lemma: arbitrary embedding} indeed yields
		\begin{equation*}
			\Phi_{\cB_e,I}\Xleq 2(\log n)^{\alpha_{\cB_e,I}}\phihat_{\cB_e,I}\leq \frac{1}{\zeta^{\delta}}\phihat_{\cB_e,I},
		\end{equation*}
		which completes the proof.
	\end{proof}
	
	\begin{lemma}\label{lemma: expected change balanced}
		Let~$0\leq i_0\leq i^\star$ and~$\pom\in\set{-,+}$.
		Then,~$\sum_{i\geq i_0}\exi{\abs{\Delta Z^\pom_{i_0}}}\leq \phihat_{\cA,I}(i_0)/\zeta(i_0)^{5\delta}$.
	\end{lemma}
	
	\begin{proof}
		Lemma~\ref{lemma: expected change balanced not stopped} entails
		\begin{equation*}
			\sum_{i\geq i_0}\exi{\abs{\Delta Z^\pom_{i_0}}}=\sum_{i_0\leq i\leq i^\star-1}\exi{\abs{\Delta Z^\pom_{i_0}}}\leq (i^\star-i_0)\frac{\phihat_{\cA,I}(i_0)}{\zeta(i_0)^{5\delta}n^k\phat(i_0)}.
		\end{equation*}
		Since
		\begin{equation*}
			i^\star-i_0
			\leq \frac{\theta n^k}{\abs{\cF}k!}-i_0
			=\frac{n^k\phat(i_0)}{\abs{\cF}k!}
			\leq n^k\phat(i_0),
		\end{equation*}
		this completes the proof.
	\end{proof}
	
	\subsection{Supermartingale concentration}\label{subsection: balanced concentration}
	This section follows a similar structure as Sections~\ref{subsubsection: chain concentration} and~\ref{subsubsection: family concentration}.
	Lemma~\ref{lemma: initial error balanced} is the final ingredient that we use for our application of Lemma~\ref{lemma: freedman} in the proof of Lemma~\ref{lemma: control balanced} where we show that the probabilities of the events on the right in Observation~\ref{observation: balanced critical times} are indeed small. 
	
	\begin{lemma}\label{lemma: initial error balanced}
		Let~$\pom\in\set{-,+}$.
		Then,~$Z^\pom_{\sigma^\pom}(\sigma^\pom)\leq -\delta^3\xi_1(\sigma^\pom)$.
	\end{lemma}
	
	\begin{proof}
		Lemma~\ref{lemma: initially good} implies~$\tautilde^\star\geq 1$ and we have~$i_{\cA,I}^{\delta^{1/2}}\geq 1$.
		Hence, we have~$\tautilde^\star\wedge i_{\cA,I}^{\delta^{1/2}}\wedge i^\star\geq 1$ and since for~$i:=0$, Lemma~\ref{lemma: initially good} also implies~$\pom(\Phi_{\cA,\psi}-\phihat_{\cA,I})\leq \xi_0$, we have~$\sigma^\pom\geq 1$.
		Thus, by definition of~$\sigma^\pom$, for~$i:=\sigma^\pom-1$, we have~$\pom(\Phi_{\cA,\psi}-\phihat_{\cA,I})\leq \xi_0$ and thus
		\begin{equation*}
			Z^\pom_i=\pom(\Phi_{\cA,\psi}-\phihat_{\cA,I})-\xi_1\leq -\delta^2\xi_1.
		\end{equation*}
		Furthermore, since~$\sigma^\pom\leq \tau_\ccB\wedge \tau_{\ccB'}$, we may apply Lemma~\ref{lemma: absolute change balanced not stopped} to obtain 
		\begin{equation*}
			Z^\pom_{\sigma^\pom}(\sigma^\pom)
			=Z^\pom_i+\Delta Y^\pom
			\leq Z^\pom_i+\zeta^{8\delta}\phihat_{\cA,I}
			\leq -\delta^2 \xi_1+\zeta^{8\delta}\phihat_{\cA,I}
			\leq -\delta^3 \xi_1.
		\end{equation*}
		Since~$\Delta\xi_1\leq 0$, this completes the proof.
	\end{proof}
	
	\begin{lemma}\label{lemma: control balanced}
		$\pr{\tau_{\ccB}\leq \tautilde^\star \wedge i^\star}\leq \exp(-n^{\delta^2})$.
	\end{lemma}
	
	\begin{proof}
		Considering Observation~\ref{observation: balanced individual}, it suffices to show that
		\begin{equation*}
			\pr{\tau\leq \tautilde^\star \wedge i^{\delta^{1/2}}_{\cA,I}\wedge i^\star}\leq \exp(-n^{2\delta^2}).
		\end{equation*}
		Hence, by Observation~\ref{observation: balanced critical times}, it suffices to show that for~$\pom\in\set{-,+}$, we have
		\begin{equation*}
			\pr{Z^\pom_{\sigma^\pom}(i^\star)>0}\leq \exp(-n^{3\delta^2}).
		\end{equation*}
		Due to Lemma~\ref{lemma: initial error balanced}, we have
		\begin{equation*}
			\pr{Z^\pom_{\sigma^\pom}(i^\star)>0}
			\leq \pr{Z^\pom_{\sigma^\pom}(i^\star)-Z^\pom_{\sigma^\pom}(\sigma^\pom)\geq \delta^3\xi_1(\sigma^\pom)}
			\leq \sum_{0\leq i\leq i^\star} \pr{Z^\pom_{i}(i^\star)-Z^\pom_{i}\geq \delta^3\xi_1}.
		\end{equation*}
		Thus, for~$0\leq i\leq i^\star$, it suffices to obtain
		\begin{equation*}
			\pr{Z^\pom_{i}(i^\star)-Z^\pom_{i}\geq \delta^3\xi_1}\leq \exp(-n^{4\delta^2}).
		\end{equation*}
		We show that this bound is a consequence of Lemma~\ref{lemma: freedman}.
		
		Let us turn to the details.
		Lemma~\ref{lemma: balanced trend} shows that~$Z^\pom_i(i),Z^\pom_i(i+1),\ldots$ is a supermartingale,
		while Lemma~\ref{lemma: absolute change balanced} provides the bound~$\abs{\Delta Z^\pom_i(j)}\leq \zeta^{8\delta}\phihat_{\cA,I}$ for all~$j\geq i$ and Lemma~\ref{lemma: expected change balanced} provides the bound~$\sum_{j\geq i}\ex[][\bE_j]{\abs{\Delta Z^\pom_i(j)}}\leq \zeta^{-5\delta}\phihat_{\cA,I}$.
		Hence, we may apply Lemma~\ref{lemma: freedman} such that using Lemma~\ref{lemma: bounds of zeta}, we obtain
		\begin{equation*}
			\pri{Z_i^\pom(i^\star)-Z_i^\pom>\delta^3\xi_1}\leq\exp\paren[\bigg]{-\frac{\delta^6\zeta^{2\delta}\phihat_{\cA,I}^2}{2\zeta^{8\delta}\phihat_{\cA,I}(\delta^3\zeta^\delta\phihat_{\cA,I}+\zeta^{-5\delta}\phihat_{\cA,I})}}
			\leq \exp(\delta^7\zeta^{-\delta})
			\leq \exp(-n^{4\delta^2}),
		\end{equation*}
		which completes the proof.
	\end{proof}

	\section{Counting strictly balanced templates}\label{appendix: strictly balanced}
	Lemma~\ref{lemma: auxiliary control}~\ref{item: control balanced} states that for a balanced template~$(\cA,I)\in\ccB$ and~$0\leq i\leq i^\star$, the number~$\Phi_{\cA,I}$ behaves as expected as long as the corresponding trajectory still suggests a significant number of embeddings in the sense that~$i\leq i_{\cA,I}^{\delta^{1/2}}$.
	In this section, our goal is to extend this guarantee that the number of embeddings is typically concentrated around the trajectory also beyond step~$i_{\cA,I}^{\delta^{1/2}}$ up to step~$i_{\cA,I}^0$ and also if~$i_{\cA,I}^{\delta^{1/2}}=0$ subject to the following two restrictions.
	First, we obtain this guarantee only for strictly balanced templates~$(\cA,I)$ with~$i_{\cA,I}^{0}\geq 1$ and second, we allow larger relative deviations from the trajectory compared to Lemma~\ref{lemma: auxiliary control}~\ref{item: control balanced}.
	Formally, for this section, we assume the setup that we used in Section~\ref{section: stopping times} to state Lemma~\ref{lemma: auxiliary control} and show that the probability that~$\tau_{\ccB'}\leq \tautilde^\star\wedge i^\star$ is small.
	Similarly as in Sections~\ref{subsection: tracking chains} and~\ref{subsection: tracking families} we may again restrict our attention to only one fixed strictly balanced template~$(\cA,I)$ with~$I\neq V_\cA$ and~$i_{\cA,I}^{\delta^{1/2}}\leq i^\star$, see Observation~\ref{observation: strictly balanced individual}.
	Note that~$I\neq V_\cA$ together with~$i_{\cA,I}^{\delta^{1/2}}\leq i^\star$ in particular entails~$\cA\setminus\cA[I]\neq\emptyset$.
	Overall, our approach is similar as in Sections~\ref{subsection: tracking chains} and~\ref{subsection: tracking families}, however, the fact that here we are only interested in steps~$i\geq i_{\cA,I}^{\delta^{1/2}}$ leads to a slightly different setup where we intuitively shift the beginning of our considerations from step~$0$ to step~$i_{\cA,I}^{\delta^{1/2}}$.
	To control the initial situation at this shifted start, we rely on Lemma~\ref{lemma: auxiliary control}~\ref{item: control balanced}.
	
	\begin{observation}\label{observation: strictly balanced individual}
		For~$(\cA,I)\in\ccB'$ and~$\psi\colon I\injection V_\cH$, let
		\begin{equation*}
			\tau_{\cA,\psi}:=\min\cset{ i\geq i^{\delta^{1/2}}_{\cA,I} }{ \Phi_{\cA,\psi}\neq (1\pm (\log n)^{\alpha_{\cA,I}}\phihat_{\cA,I}^{-\delta^{1/2}})\phihat_{\cA,I} }.
		\end{equation*}
		Then,
		\begin{equation*}
			\pr{\tau_{\ccB'}\leq \tautilde^\star\wedge i^\star}\leq  \sum_{\substack{(\cA,I)\in\ccB'\colon I\neq V_\cA\stand i_{\cA,I}^{\delta^{1/2}}\leq i^\star\\ \psi\colon I\injection V_\cH}} \pr{\tau_{\cA,\psi}\leq \tautilde^\star\wedge i^0_{\cA,I}\wedge i^\star}.
		\end{equation*}
	\end{observation}
	
	Fix~$(\cA,I)\in\ccB'$ with~$I\neq V_\cA$ and~$i_{\cA,I}^{\delta^{1/2}}\leq i^\star$ and hence~$\cA\setminus\cA[I]\neq\emptyset$.
	Let~$\psi\colon I\injection V_\cH$ and for~$i\geq 0$, let
	\begin{equation*}
		\xi_1(i):=(\log n)^{\alpha_{\cA,I}}\phihat_{\cA,I}^{1-\delta^{1/2}},\quad
		\xi_0(i):=(1-\delta)\xi_1%
	\end{equation*}
	and define the stopping time
	\begin{equation*}
		\tau:=\min\cset{i\geq i^{\delta^{1/2}}_{\cA,I}}{\Phi_{\cA,\psi}\neq \phihat_{\cA,I}\pm\xi_1}.
	\end{equation*}
	Define the critical intervals
	\begin{equation*}
		I^-(i):=[\phihat_{\cA,I}-\xi_1,\phihat_{\cA,I}-\xi_0],\quad
		I^+(i):=[\phihat_{\cA,I}+\xi_0,\phihat_{\cA,I}+\xi_1].
	\end{equation*}
	For~$\pom\in\set{-,+}$, let
	\begin{equation*}
		Y^\pom(i):=\pom(\Phi_{\cA,\psi}-\phihat_{\cA,I})-\xi_1.
	\end{equation*}
	For~$i_0\geq i^{\delta^{1/2}}_{\cA,I}$, define the stopping time
	\begin{equation*}
		\tau^\pom_{i_0}:=\min\cset{i\geq i_0}{\Phi_{\cA,\psi}\notin I^\pom}
	\end{equation*}
	and for~$i\geq i_0$, let
	\begin{equation*}
		Z^\pom_{i_0}(i):=\ind_{\set{i_{\cA,I}^{\delta^{1/2}}<\tau_\ccB}}Y^\pom(i_0\vee(i\wedge \tau^\pom_{i_0}\wedge \tautilde^\star\wedge i^0_{\cA,I}\wedge i^\star)).
	\end{equation*}
	Let
	\begin{equation*}
		\sigma^\pom := \min\cset{j\geq i_{\cA,I}^{\delta^{1/2}}}{ \pom(\Phi_{\cA,\psi}-\phihat_{\cA,I})\geq \xi_0 \stforall j\leq i< \tautilde^\star\wedge i^0_{\cA,I}\wedge i^\star }\leq\tautilde^\star\wedge i^0_{\cA,I}\wedge i^\star.
	\end{equation*}
	With this setup, similarly as in Section~\ref{subsection: tracking chains} and Section~\ref{subsection: tracking families}, it in fact again suffices to consider the evolution of~$Z^\pom_{\sigma^\pom}(\sigma^\pom),Z^\pom_{\sigma^\pom}(\sigma^\pom+1),\ldots$.
	Indeed, we have
	\begin{equation*}
		\begin{aligned}
			\set{\tau\leq \tautilde^\star \wedge i^{0}_{\cA,I}\wedge i^\star}
			&\subseteq \set{\tau_\ccB\leq i_{\cA,I}^{\delta^{1/2}}\stand \tau\leq \tautilde^\star \wedge i^{0}_{\cA,I}\wedge i^\star }\cup\set{Z^-_{\sigma^-}(i^\star)>0}\cup \set{Z^+_{\sigma^+}(i^\star)>0}\\
			&\subseteq \set{i_{\cA,I}^{\delta^{1/2}}\leq \tau\leq \tautilde^\star\leq \tau_\ccB\leq i_{\cA,I}^{\delta^{1/2}}}\cup\set{Z^-_{\sigma^-}(i^\star)>0}\cup \set{Z^+_{\sigma^+}(i^\star)>0}\\
			&\subseteq \set{\tau_\ccB\leq \tautilde^\star\wedge i_{\cA,I}^{\delta^{1/2}}}\cup\set{Z^-_{\sigma^-}(i^\star)>0}\cup \set{Z^+_{\sigma^+}(i^\star)>0}
		\end{aligned}
	\end{equation*}
	and due to~$i_{\cA,I}^{\delta^{1/2}}\leq i^\star$, this leads to the following observation.
	\begin{observation}\label{observation: strictly balanced critical times}
		$\set{\tau\leq \tautilde^\star \wedge i^{0}_{\cA,I}\wedge i^\star}\subseteq \set{\tau_\ccB\leq \tautilde^\star\wedge i^\star}\cup\set{Z^-_{\sigma^-}(i^\star)>0}\cup \set{Z^+_{\sigma^+}(i^\star)>0}$.
	\end{observation}
	
	Lemma~\ref{lemma: auxiliary control}~\ref{item: control balanced} shows that the probability of the first event on the right in Observation~\ref{observation: strictly balanced critical times} is sufficiently small and we again use supermartingale concentration techniques to show that the probabilities of the other two events are also sufficiently small.
	More specifically, for this section, we use Lemma~\ref{lemma: freedman}.
	
	\subsection{Trend}
	Here, we prove that for all~$\pom\in\set{-,+}$ and~$i_0\geq i^{\delta^{1/2}}_{\cA,I}$, the expected one-step changes of the process~$Z^\pom_{i_0}(i_0),Z^\pom_{i_0}(i_0+1),\ldots$ are non-positive.
	Lemma~\ref{lemma: delta phihat} already yields estimates for the one-step changes of the relevant deterministic trajectory, in Lemma~\ref{lemma: delta xi strictly balanced} we estimate the one-step changes of the error term that we use in this section.
	Furthermore, Lemma~\ref{lemma: balanced change} provides a precise estimate for the expected one-step change of the non-deterministic part of the random process.
	Combining these estimates shows that the process indeed is a supermartingale (see Lemma~\ref{lemma: strictly balanced trend}).
	
	\begin{observation}\label{observation: xi strictly balanced}
		Extend~$\phat$ and~$\xi_1$ to continuous trajectories defined on the whole interval~$[0,i^\star+1]$ using the same expressions as above.
		Then, for~$x\in [0,i^\star+1]$,
		\begin{equation*}
			\begin{gathered}
				\xi_1'(x)=-\frac{(1-\delta^{1/2})(\abs{\cA}-\abs{\cA[I]})\abs{\cF}k!\,\xi_1(x)}{n^k\phat(x)},\\
				\xi_1''(x)=-\frac{(1-\delta^{1/2})(\abs{\cA}-\abs{\cA[I]})((1-\delta^{1/2})(\abs{\cA}-\abs{\cA[I]})-1)\abs{\cF}^2(k!)^2\xi_1(x)}{n^{2k}\phat(x)^2},
			\end{gathered}
		\end{equation*} 
	\end{observation}
	
	\begin{lemma}\label{lemma: delta xi strictly balanced}
		Let~$0\leq i\leq i^\star$ and~$\cX:=\set{i\leq \tau_\emptyset}$.
		Then,
		\begin{equation*}
			\Delta\xi_1\Xeq-(1-\delta^{1/2})(\abs{\cA}-\abs{\cA[I]})\frac{\abs{\cF}\xi_1}{H}\pm \frac{\phihat_{\cA,I}^{1-\delta^{1/2}}}{H}.
		\end{equation*}
	\end{lemma}
	\begin{proof}
		This is a consequence of Taylor's theorem.
		In detail, we argue as follows.
		
		Together with Observation~\ref{observation: xi strictly balanced}, Lemma~\ref{lemma: taylor} yields
		\begin{equation*}
			\Delta \xi_1
			=-\frac{(1-\delta^{1/2})(\abs{\cA}-\abs{\cA[I]})\abs{\cF}k!\,\xi_1}{n^k\phat}\pm\max_{x\in[i,i+1]}\frac{\xi_1(x)}{\delta n^{2k}\phat(x)^2}.
		\end{equation*}
		We investigate the first term and the maximum separately.
		Using Lemma~\ref{lemma: edges of H}, we have
		\begin{equation*}
			-\frac{(1-\delta^{1/2})(\abs{\cA}-\abs{\cA[I]})\abs{\cF}k!\,\xi_1}{n^k\phat}
			\Xeq-(1-\delta^{1/2})(\abs{\cA}-\abs{\cA[I]})\frac{\abs{\cF}\xi_1}{H}.
		\end{equation*}
		Furthermore, precisely as at the end of the proof of Lemma~\ref{lemma: delta xi balanced}, we obtain
		\begin{equation*}
			\max_{x\in[i,i+1]}\frac{\xi_1(x)}{\delta n^{2k}\phat(x)^2}
			\Xleq \frac{\zeta^{2+\eps^2}\xi_1}{H}.
		\end{equation*}
		With Lemma~\ref{lemma: bounds of zeta}, this completes the proof.
	\end{proof}

	\begin{lemma}\label{lemma: strictly balanced trend}
		Let~$i_{\cA,I}^{\delta^{1/2}}\leq i_0\leq i$ and~$\pom\in\set{-,+}$.
		Then,~$\exi{\Delta Z^\pom_{i_0}}\leq 0$.
	\end{lemma}
	
	\begin{proof}
		Suppose that~$i\leq i^\star$ and let~$\cX:=\set{i<\tau^\pom_{i_0}\wedge\tautilde^\star}$.
		We have~$\exi{\Delta Z_{i_0}^\pom}=_{\cX^\comp}0$ and~$\exi{\Delta Z_{i_0}^\pom}\Xeq\exi{\Delta Y^\pom}$, so it suffices to obtain~$\exi{\Delta Y^\pom }\Xleq 0$.
		Due to Lemma~\ref{lemma: bounds of zeta}, we have~$\zeta^{1/2}\leq n^{-\delta^{1/2}\abs{V_\cA}}\leq \phihat_{\cA,I}^{-\delta^{1/2}}$.
		Hence, Lemma~\ref{lemma: delta phihat} yields (with room to spare)
		\begin{equation*}
			\Delta\phihat_{\cA,I}=-(\abs{\cA}-\abs{\cA[I]})\frac{\abs{\cF}\phihat_{\cA,I}}{H}\pm\frac{\phihat_{\cA,I}^{1-\delta^{1/2}}}{H}.
		\end{equation*}
		Arguing precisely as in the proof of Lemma~\ref{lemma: balanced change} for the first equality, we obtain
		\begin{equation*}
			\exi{ \Delta\Phi_{\cA,\psi} }
			\Xeq -(\abs{\cA}-\abs{\cA[I]})\frac{\abs{\cF}}{H}\Phi_{\cA,\psi}\pm\frac{\zeta^{1/2}\phihat_{\cA,I}}{H}
			= -(\abs{\cA}-\abs{\cA[I]})\frac{\abs{\cF}}{H}\Phi_{\cA,\psi}\pm\frac{\phihat_{\cA,I}^{1-\delta^{1/2}}}{H}.
		\end{equation*}
		Combining these two estimates with Lemma~\ref{lemma: delta xi strictly balanced}, we obtain
		\begin{equation*}
			\begin{aligned}
				\exi{\Delta Y^\pom}
				&=\pom(\exi{\Delta\Phi_{\cA,\psi}}-\Delta\phihat_{\cA,I})-\Delta\xi_1\\
				&\Xleq \pom\paren[\bigg]{ -(\abs{\cA}-\abs{\cA[I]})\frac{\abs{\cF}}{H}\Phi_{\cA,\psi}+(\abs{\cA}-\abs{\cA[I]})\frac{\abs{\cF}}{H}\phihat_{\cA,I} }\\&\hphantom{\Xleq}\mathrel{}\quad +(1-\delta^{1/2})(\abs{\cA}-\abs{\cA[I]})\frac{\abs{\cF}\xi_1}{H}  +\frac{3\phihat_{\cA,I}^{1-\delta^{1/2}}}{H}\\
				&\leq -\frac{\abs{\cF}(\abs{\cA}-\abs{\cA[I]})}{H}( \pom(\Phi_{\cA,\psi}-\phihat_{\cA,I})-(1-\delta^{1/2})\xi_1-3\phihat_{\cA,I}^{1-\delta^{1/2}} )\\
				&\Xleq -\frac{\abs{\cF}(\abs{\cA}-\abs{\cA[I]})\xi_1}{H}( (1-\delta)-(1-\delta^{1/2})-\delta )
				\leq 0,
			\end{aligned}
		\end{equation*}
		which completes the proof.
	\end{proof}
	
	\subsection{Boundedness}\label{subsection: strictly balanced boundedness}
	Here, similarly as in Sections~\ref{subsubsection: chain boundedness} and~\ref{subsubsection: family boundedness}, we obtain suitable bounds for the absolute one-step changes of the processes~$Y^\pom(0),Y^\pom(1),\ldots$ and~$Z^\pom_{i_0}(i_0),Z^\pom_{i_0}(i_0+1),\ldots$ (see Lemma~\ref{lemma: absolute change strictly balanced not stopped} and Lemma~\ref{lemma: absolute change strictly balanced}) as well as the expected absolute one-step changes of these processes (see Lemma~\ref{lemma: expected change strictly balanced not stopped} and Lemma~\ref{lemma: expected change strictly balanced}).
	The fact that we analyze the evolution potentially even until~$\phihat_{\cA,I}$ is essentially~$1$ often plays an important role in this section. 
	Furthermore, we crucially exploit that~$(\cA,I)$ is strictly balanced and not just balanced.
	\begin{lemma}\label{lemma: has smallest trajectory}
		Let~$i_{\cA,I}^{\delta^{1/2}}\leq i\leq i^\star$.
		Fix~$e\in\cA\setminus\cA[I]$ and~$(\cB,I)\subseteq(\cA,I)$ with~$e\in\cB$. Then,~$\phihat_{\cA,I}\leq\phihat_{\cB,I}$.
	\end{lemma}
	
	\begin{proof}
		If~$\phihat_{\cA,I}\leq 1$, then
		\begin{equation*}
			\phihat_{\cA,I}
			\leq \phihat_{\cA,I}^{(\abs{V_\cB}-\abs{I})/(\abs{V_\cA}-\abs{I})}
			= (n\phat^{\rho_{\cA,I}})^{\abs{V_\cB}-\abs{I}}
			\leq (n\phat^{\rho_{\cB,I}})^{\abs{V_\cB}-\abs{I}}
			=\phihat_{\cB,I}.
		\end{equation*}
		Hence, we may assume~$\phihat_{\cA,I}\geq 1$.
		Furthermore, we may assume that~$\cB\neq\cA$.
		
		Since~$(\cA,I)$ is strictly balanced, we have~$\rho_{\cB,I}+\delta^{1/4}\leq \rho_{\cA,I}$.
		This allows us to obtain
		\begin{equation*}
			\begin{aligned}
				\phihat_{\cA,I}
				&=(n\phat^{\rho_{\cA,I}})^{\abs{V_\cB}-\abs{I}}(n\phat^{\rho_{\cA,I}})^{\abs{V_\cA}-\abs{V_\cB}}
				=(n\phat^{\rho_{\cA,I}})^{\abs{V_\cB}-\abs{I}}\phihat_{\cA,I}^{(\abs{V_\cA}-\abs{V_\cB})/(\abs{V_\cA}-\abs{I})}\\
				&\leq (n\phat^{\rho_{\cB,I}+\delta^{1/4}})^{\abs{V_\cB}-\abs{I}}\phihat_{\cA,I}
				\leq \phihat_{\cB,I} \cdot\phat^{\delta^{1/4}}\phihat_{\cA,I}.
			\end{aligned}
		\end{equation*}
		Hence, it suffices to show that~$\phat^{\delta^{1/4}}\leq 1/\phihat_{\cA,I}$.
		Indeed, using Lemma~\ref{lemma: bounds of zeta} and the fact that~$\phihat_{\cA,I}\leq \zeta^{-\delta^{1/2}}$, we obtain
		\begin{equation*}
			\begin{aligned}
				\phat^{\delta^{1/4}}
				&\leq \phat^{\delta^{1/3}(\abs{\cA}-\abs{\cA[I]})}
				=n^{-\delta^{1/3}(\abs{V_\cA}-\abs{I})}\phihat_{\cA,I}^{\delta^{1/3}}
				\leq n^{-\delta^{1/3}}\phihat_{\cA,I}
				\leq \zeta^{\delta^{1/3}}\phihat_{\cA,I}
				= (\zeta^{\delta^{1/2}})^{\delta^{-1/6}}\phihat_{\cA,I}\\
				&\leq \phihat_{\cA,I}^{1-\delta^{-1/6}}
				\leq \phihat_{\cA,I}^{-1},
			\end{aligned}
		\end{equation*}
		which completes the proof.
	\end{proof}
	
	\begin{lemma}\label{lemma: absolute change strictly balanced not stopped}
		Let~$i^{\delta^{1/2}}_{\cA,I}\leq i\leq i^\star$,~$\pom\in\set{-,+}$ and~$\cX:=\set{i<\tau_{\ccB}\wedge\tau_{\ccB'}}$.
		Then,
		\begin{equation*}
			\abs{\Delta Y^\pom}\Xleq \frac{(\log n)^{\alpha_{\cA,I}/2}}{\delta^2 \log n}.
		\end{equation*}
	\end{lemma}
	
	\begin{proof}
		From Lemma~\ref{lemma: delta phihat} and Lemma~\ref{lemma: delta xi strictly balanced}, using the fact that~$\phihat_{\cA,I}\leq \zeta^{-\delta^{1/2}}$, we obtain
		\begin{equation*}
			\begin{aligned}
				\abs{\Delta Y^\pom}
				&\leq \abs{\Delta\Phi_{\cA,\psi}}+\abs{\Delta\phihat_{\cA,I}}+\abs{\Delta\xi_1}
				\leq \abs{\Delta\Phi_{\cA,\psi}}+2\frac{\abs{\cA}\abs{\cF}\phihat_{\cA,I}}{H}+2\frac{\abs{\cA}\abs{\cF}\xi_1}{H}\\
				&\leq \abs{\Delta\Phi_{\cA,\psi}}+2\frac{\abs{\cA}\abs{\cF}\zeta^{-\delta^{1/2}}}{H}+2\frac{\abs{\cA}\abs{\cF}(\log n)^{\alpha_{\cA,I}}\zeta^{-\delta^{1/2}(1-\delta^{1/2})}}{H}\\
				&\leq \abs{\Delta\Phi_{\cA,\psi}}+3\frac{\abs{\cA}\abs{\cF}}{\zeta^{\delta^{1/2}}H}.
			\end{aligned}
		\end{equation*}
		Hence, Lemma~\ref{lemma: zeta and H} implies that it suffices to show that
		\begin{equation*}
			\abs{\Delta\Phi_{\cA,\psi}}\Xleq \frac{(\log n)^{\alpha_{\cA,I}/2}}{\delta \log n}
		\end{equation*}
		which we obtain as a consequence of Lemma~\ref{lemma: loss at one edge} and Lemma~\ref{lemma: has smallest trajectory}.
		Indeed, these two lemmas together with Observation~\ref{observation: choice of alpha} imply
		\begin{equation*}
			\abs{\Delta\Phi_{\cA,\psi}}
			\leq \sum_{e\in\cA\setminus\cA[{I}]}\abs{\cset{ \phi\in\Phi_{\cA,\psi}^\sim }{ \phi(e)\in\cF_0(i+1) }}
			\Xleq \abs{\cA}\cdot 2k!\,\abs{\cF}(\log n)^{\alpha_{\cA,I\cup e}}
			\leq \frac{(\log n)^{\alpha_{\cA,I}/2}}{\delta \log n },
		\end{equation*}
		which completes the proof.
	\end{proof}
	
	\begin{lemma}\label{lemma: absolute change strictly balanced}
		Let~$i^{\delta^{1/2}}_{\cA,I}\leq i_0\leq i$ and~$\pom\in\set{-,+}$.
		Then,
		\begin{equation*}
			\abs{\Delta Z^\pom_{i_0}}\leq \frac{(\log n)^{\alpha_{\cA,I}/2}}{\delta^2 \log n }.
		\end{equation*}
	\end{lemma}
	
	\begin{proof}
		This is an immediate consequence of Lemma~\ref{lemma: absolute change strictly balanced not stopped}.
	\end{proof}
	
	\begin{lemma}\label{lemma: expected change strictly balanced not stopped}
		Let~$i^{\delta^{1/2}}_{\cA,I}\leq i_0\leq i\leq i^0_{\cA,I}\wedge i^\star$,~$\pom\in\set{-,+}$ and~$\cX:=\set{i<\tautilde^\star}$.
		Then,
		\begin{equation*}
			\exi{\abs{\Delta Y^\pom}}\Xleq \frac{(\log n)^{3\alpha_{\cA,I}/2}\phihat_{\cA,I}(i_0)}{\delta^5  n^k\phat(i_0)\log n}.
		\end{equation*}
	\end{lemma}
	
	\begin{proof}
		From Lemma~\ref{lemma: delta phihat} and Lemma~\ref{lemma: delta xi strictly balanced}, we obtain
		\begin{equation*}
			\exi{\abs{\Delta Y^\pom}}
			\leq \exi{\abs{\Delta\Phi_{\cA,I}}}+\abs{\Delta\phihat_{\cA,I}}+\abs{\Delta\xi_1}
			\leq \exi{\abs{\Delta\Phi_{\cA,I}}}+2\frac{\abs{\cA}\abs{\cF}\phihat_{\cA,I}}{H}+2\frac{\abs{\cA}\abs{\cF}\xi_1}{H}.
		\end{equation*}
		Since~$\cA\setminus\cA[I]\neq\emptyset$ implies~$\phihat_{\cA,I}/\phat\leq \phihat_{\cA,I}(i_0)/\phat(i_0)$, due to Lemma~\ref{lemma: trajectory at cutoff} and Lemma~\ref{lemma: edges of H}, it suffices to show that
		\begin{equation*}
			\exi{\abs{\Delta\Phi_{\cA,I}}}\leq \frac{(\log n)^{3\alpha_{\cA,I}/2}\phihat_{\cA,I}}{\delta^4 n^k\phat\log n}.
		\end{equation*}
		Arguing similarly as in the proof of~\ref{lemma: expected change chain deviation}, we obtain this as a consequence of Lemma~\ref{lemma: loss at one edge} and Lemma~\ref{lemma: has smallest trajectory}.
		
		To this end, for~$e\in\cA\setminus\cA[I]$, let
		\begin{equation*}
			\Phi^e_{\cA,I}:=\abs{\cset{ \phi\in\Phi_{\cA,I}^\sim }{ \phi(e)\in \cF_0(i+1) }}.
		\end{equation*}
		Using Observation~\ref{observation: choice of alpha}, Lemma~\ref{lemma: loss at one edge} together with Lemma~\ref{lemma: has smallest trajectory} yields
		\begin{equation*}
			\Phi^e_{\cA,I}
			\Xleq 2k!\,\abs{\cF}(\log n)^{\alpha_{\cA,I\cup e}}
			\leq \frac{(\log n)^{\alpha_{\cA,I}/2}}{\delta\log n},
		\end{equation*}
		so we obtain
		\begin{equation}\label{equation: random upper bound strictly balanced change}
			\begin{aligned}
				\abs{\Delta\Phi_{\cA,I}}
				&\leq \sum_{e\in\cA\setminus\cA[I]} \Phi_{\cA,I}^e
				=\sum_{e\in\cA\setminus\cA[I]} \ind_{\set{\Phi_{\cA,I}^e\geq 1}}\Phi_{\cA,I}^e
				\Xleq \frac{(\log n)^{\alpha_{\cA,I}/2}}{\delta\log n}\sum_{e\in\cA\setminus\cA[I]}\ind_{\set{\Phi_{\cA,I}^e\geq 1}}\\
				&\leq \frac{(\log n)^{\alpha_{\cA,I}/2}}{\delta\log n}\sum_{e\in\cA\setminus\cA[I]}\sum_{\phi\in\Phi_{\cA,I}^\sim}\ind_{\set{\phi(e)\in\cF_0(i+1)}}.
			\end{aligned}
		\end{equation}
		For all~$e\in\cH$,~$f\in\cF$ and~$\psi'\colon f\injection e$, we have~$\Phi_{\cF,\psi'}\Xeq (1\pm\delta^{-1}\zeta)\phihat_{\cF,f}$.
		Furthermore, we have~$H^*\Xeq (1\pm\zeta^{1+\eps^3})\hhat^*$.
		Thus, using Lemma~\ref{lemma: star degrees}, for all~$e\in\cA\setminus\cA[I]$ and~$\varphi\in\Phi_{\cA,I}^\sim$, we obtain
		\begin{equation*}
			\pri{ \phi(e)\in\cF_0(i+1) }
			=\frac{d_{\cH^*}(\phi(e))}{H^*}
			\Xleq \frac{2\abs{\cF}k!\, \phihat_{\cF,f}}{H^*}
			\Xleq \frac{4\abs{\cF}k!\,\phihat_{\cF,f}}{\hhat^*}
			\leq \frac{1}{\delta n^k\phat}.
		\end{equation*}
		Combining this with~\eqref{equation: random upper bound strictly balanced change} and using the fact that~$\Phi_{\cA,I}\Xeq (1\pm (\log n)^{\alpha_{\cA,I}}\phihat_{\cA,I}^{-\delta^{1/2}})\phihat_{\cA,I}$ as well as Lemma~\ref{lemma: trajectory at cutoff} yields
		\begin{equation*}
			\begin{aligned}
				\exi{\abs{\Delta \Phi_{\cA,I}}}
				&\Xleq \frac{(\log n)^{\alpha_{\cA,I}/2}}{\delta^2n^k\phat \log n}\sum_{e\in\cA\setminus\cA[I]} \Phi_{\cA,I}
				\leq \frac{(\log n)^{\alpha_{\cA,I}/2}}{\delta^3n^k\phat \log n}\Phi_{\cA,I}\\
				&\Xleq \frac{(\log n)^{\alpha_{\cA,I}/2}}{\delta^3n^k\phat \log n}(1+ (\log n)^{\alpha_{\cA,I}}\phihat_{\cA,I}^{-\delta^{1/2}})\phihat_{\cA,I}
				\leq \frac{(\log n)^{\alpha_{\cA,I}/2}}{\delta^3n^k\phat \log n}(1+ 2(\log n)^{\alpha_{\cA,I}})\phihat_{\cA,I}\\
				&\leq \frac{(\log n)^{3\alpha_{\cA,I}/2}\phihat_{\cA,I}}{\delta^4  n^k\phat\log n},
			\end{aligned}
		\end{equation*}
		which completes the proof.
	\end{proof}
	
	\begin{lemma}\label{lemma: expected change strictly balanced}
		Let~$i^{\delta^{1/2}}_{\cA,I}\leq i_0\leq i^\star$ and~$\pom\in\set{-,+}$.
		Then,
		\begin{equation*}
			\sum_{i\geq i_0}\exi{\abs{\Delta Z^\pom_{i_0}}}\leq \frac{(\log n)^{3\alpha_{\cA,I}/2}\phihat_{\cA,I}(i_0)}{\delta^5 \log n}.
		\end{equation*}
	\end{lemma}
	
	\begin{proof}
		Lemma~\ref{lemma: expected change strictly balanced not stopped} entails
		\begin{equation*}
			\sum_{i\geq i_0}\exi{\abs{\Delta Z^\pom_{i_0}}}
			=\sum_{i_0\leq i\leq i^\star-1}\exi{\abs{\Delta Z^\pom_{i_0}}}
			\leq (i^\star-i_0)\frac{(\log n)^{3\alpha_{\cA,I}/2}\phihat_{\cA,I}(i_0)}{\delta^5  n^k\phat(i_0)\log n}.
		\end{equation*}
		Since
		\begin{equation*}
			i^\star-i_0
			\leq \frac{\theta n^k}{\abs{\cF}k!}-i_0
			=\frac{n^k\phat(i_0)}{\abs{\cF}k!}
			\leq n^k\phat(i_0),
		\end{equation*}
		this completes the proof.
	\end{proof}
	
	\subsection{Supermartingale concentration}\label{subsection: strictly balanced concentration}
	This section follows a similar structure as Sections~\ref{subsubsection: chain concentration} and~\ref{subsubsection: family concentration}.
	Lemma~\ref{lemma: initial error strictly balanced} is the final ingredient that we use for our application of Lemma~\ref{lemma: freedman} in the proof of Lemma~\ref{lemma: control strictly balanced} where we show that the probabilities of the events on the right in Observation~\ref{observation: strictly balanced critical times} are indeed small.
	One notable difference compared to the aforementioned sections is the fact that here, our analysis does not start at step~$0$ but instead at step~$i_{\cA,I}^{\delta^{1/2}}$.

	\begin{lemma}\label{lemma: initial error strictly balanced}
		Let~$\pom\in\set{-,+}$ and~$\cX:=\set{ i_{\cA,I}^{\delta^{1/2}}<\tau_\ccB }$.
		Then,~$Z^\pom_{\sigma^\pom}(\sigma^\pom)\Xleq -\delta^2\xi_1(\sigma^\pom)$.
	\end{lemma}
	\begin{proof}
		If~$i=i_{\cA,I}^{\delta^{1/2}}=0$, then Lemma~\ref{lemma: initially good} implies~$\pom(\Phi_{\cA,\psi}-\phihat_{\cA,I})\Xleq \xi_0$.
		If~$i=i_{\cA,I}^{\delta^{1/2}}\geq 1$, then due to~$\phihat_{\cA,I}\leq \zeta^{-\delta^{1/2}}$, we have
		\begin{equation*}
			\pom(\Phi_{\cA,I}-\phihat_{\cA,I})
			\Xleq \zeta^\delta\phihat_{\cA,I}
			\leq \phihat_{\cA,I}^{1-\delta^{1/2}}
			\leq \xi_0
		\end{equation*}
		Hence, if~$\sigma^\pom =i_{\cA,I}^{\delta^{1/2}}$, then~$Z_{\sigma^\pom}^\pom\Xleq \xi_0(\sigma^\pom)-\xi_1(\sigma^\pom)=-\delta\xi_1(\sigma^\pom)$, so we may assume~$\sigma^\pom\geq i_{\cA,I}^{\delta^{1/2}}+1$.
		Then, by definition of~$\sigma^\pom$, for~$i:=\sigma^\pom-1$, we have~$\pom(\Phi_{\cA,\psi}-\phihat_{\cA,I})\leq \xi_0$
		and thus
		\begin{equation*}
			Z_i^\pom
			= \pom(\Phi_{\cA,\psi}-\phihat_{\cA,I})-\xi_1
			\leq -\delta\xi_1.
		\end{equation*}
		Furthermore, since~$\sigma^\pom\leq \tau_\ccB\wedge\tau_{\ccB'}\wedge i_{\cA,I}^0$, we may apply Lemma~\ref{lemma: absolute change strictly balanced not stopped} and Lemma~\ref{lemma: trajectory at cutoff} to obtain
		\begin{equation*}
			Z_{\sigma^\pom}^\pom(\sigma^\pom)
			=Z_i^\pom+\Delta Y^\pom
			\leq Z^\pom_i+\frac{(\log n)^{\alpha_{\cA,I}/2}}{\delta^2 \log n}
			\leq  -\delta\xi_1+\frac{2(\log n)^{\alpha_{\cA,I}/2}\phihat_{\cA,I}^{1-\delta^{1/2}}}{\delta^2 \log n}
			\leq -\delta^2\xi_1.
		\end{equation*}
		Since~$\Delta\xi_1\leq 0$, this completes the proof.
	\end{proof}
	
	\begin{lemma}\label{lemma: control strictly balanced}
		$\pr{\tau_{\ccB'}\leq\tautilde^\star \wedge i^\star}\leq \exp(-(\log n)^{3/2})$.
	\end{lemma}
	\begin{proof}
		Considering Observation~\ref{observation: strictly balanced individual}, it suffices to obtain
		\begin{equation*}
			\pr{\tau\leq \tautilde^\star\wedge i^0_{\cA,I} \wedge i^\star}\leq \exp(-(\log n)^{5/3}).
		\end{equation*}
		Hence, by Observation~\ref{observation: strictly balanced critical times} and Lemma~\ref{lemma: auxiliary control}~\ref{item: control balanced}, it suffices to show that for~$\pom\in\set{-,+}$, we have
		\begin{equation*}
			\pr{Z^\pom_{\sigma^\pom}(i^\star)>0}\leq \exp(-(\log n)^{7/4}).
		\end{equation*}
		Using Lemma~\ref{lemma: initial error strictly balanced}, we obtain
		\begin{equation*}
				\pr{Z^\pom_{\sigma^\pom}(i^\star)>0}
				\leq \pr{Z^\pom_{\sigma^\pom}(i^\star)-Z^\pom_{\sigma^\pom}(\sigma^\pom)\geq \delta^2\xi_1(\sigma^\pom)}
				\leq \sum_{i_{\cA,I}^{\delta^{1/2}}\leq i\leq i^\star} \pr{Z^\pom_{i}(i^\star)-Z^\pom_{i}\geq \delta^2\xi_1}.
		\end{equation*}
		Thus, for~$i_{\cA,I}^{\delta^{1/2}}\le i\leq i^\star$, it suffices to obtain
		\begin{equation*}
			\pr{Z^\pom_{i}(i^\star)-Z^\pom_{i}\geq \delta^2\xi_1}\leq \exp(-(\log n)^{9/5}).
		\end{equation*}
		We show that this bound is a consequence of Lemma~\ref{lemma: freedman}.
		
		Lemma~\ref{lemma: strictly balanced trend} shows that~$Z^\pom_i(i),Z^\pom_i(i+1),\ldots$ is a supermartingale, while Lemma~\ref{lemma: absolute change strictly balanced} provides the bound
		\begin{equation*}
			\abs{\Delta Z^\pom_i(j)}\leq \frac{(\log n)^{\alpha_{\cA,I}/2}}{\delta^2\log n}
		\end{equation*}
		for all~$j\geq i$ and Lemma~\ref{lemma: expected change strictly balanced} provides the bound
		\begin{equation*}
			\sum_{j\geq 0}\ex[][\bE_j]{\abs{\Delta Z^\pom_i(j)}}\leq \frac{(\log n)^{3\alpha_{\cA,I}/2}\phihat_{\cA,I}}{\delta^5\log n}.
		\end{equation*}
		Observe that due to Lemma~\ref{lemma: trajectory at cutoff}, we have
		\begin{equation*}
			\frac{(\log n)^{3\alpha_{\cA,I}/2}\phihat_{\cA,I}}{\delta^5\log n}+\delta^2\xi_1
			\leq  \frac{(\log n)^{3\alpha_{\cA,I}/2}\phihat_{\cA,I}}{\delta^5\log n}+(\log n)^{\alpha_{\cA,I}}\phihat_{\cA,I}
			\leq \frac{(\log n)^{3\alpha_{\cA,I}/2}\phihat_{\cA,I}}{\delta^6\log n}.
		\end{equation*}
		Hence, we may apply Lemma~\ref{lemma: freedman} to obtain
		\begin{equation*}
			\begin{aligned}
				\pr{Z^\pom_{i}(i^\star)-Z^\pom_{i}\geq \delta^2\xi_1}
				&\leq\exp\paren[\bigg]{-\frac{\delta^4(\log n)^{2\alpha_{\cA,I}}\phihat_{\cA,I}^{2-2\delta^{1/2}}}{2\delta^{-2}(\log n)^{\alpha_{\cA,I}/2-1}\cdot \delta^{-6}(\log n)^{3\alpha_{\cA,I}/2-1}\phihat_{\cA,I}}}\\
				&=\exp\paren{-\delta^{13}(\log n)^{2}\phihat_{\cA,I}^{1-2\delta^{1/2}}}.
			\end{aligned}
		\end{equation*}
		Another application of Lemma~\ref{lemma: trajectory at cutoff} shows that~$\phihat_{\cA,I}^{1-2\delta^{1/2}}\geq 1/2$ and hence completes the proof.
	\end{proof}

	\section{Cherries}\label{appendix: cherries}
	In this section, we prove Theorems~\ref{theorem: cherries} and~\ref{theorem: sparse cherries}.
	We argue similarly as for Theorem~\ref{theorem: pseudorandom} and~\ref{theorem: sparse} in the sense that we obtain Theorem~\ref{theorem: sparse cherries} as a consequence of Theorem~\ref{theorem: technical cherries} below which plays a similar role as Theorem~\ref{theorem: technical} and which we then apply together with Theorem~\ref{theorem: technical bounds} to obtain Theorem~\ref{theorem: cherries}, see Section~\ref{subsection: cherry proofs}.
	To state Theorem~\ref{theorem: technical cherries}, we assume the setup described in Section~\ref{section: sparse} and again consider the~$\cF$-removal process formally given by Algorithm~\ref{algorithm: removal} as in Section~\ref{section: counting}.
	In particular, we define~$\cF_0(i)$,~$\cH(i)$,~$H(i)$,~$\cH^*(i)$ and~$H^*(i)$ for~$i\geq 0$ as well as~$\tau_\emptyset$ as in Section~\ref{section: counting}. 
	Furthermore, we introduce the following terminology.
	For a~$k$-graph~$\cA$ and~$1\leq k'\leq k-1$, we say that~$\cA$ is a~\emph{$k'$-cherry} if~$\cA$ has no isolated vertices and exactly two edges such that the two edges of~$\cA$ share~$k'$ vertices.
	We say that~$\cA$ is a \emph{cherry} if~$\cA$ is a~$k'$-cherry for some~$1\leq k'\leq k-1$.
	
	\begin{theorem}\label{theorem: technical cherries}
		If~$\cF$ is a cherry, then~$\pr{H(\tau_\emptyset)\leq n^{k-1/\rho_\cF-\eps}}\leq \exp(-n^{1/4})$.
	\end{theorem}
	
	\subsection{Unions of cherries}
	To prove Theorem~\ref{theorem: technical cherries}, we argue similarly as for Theorem~\ref{theorem: technical}.
	However, some of the key results in Section~\ref{section: gluing} only hold for hypergraphs with at least three edges since self-avoiding cyclic walks of cherries can form stars, that is hypergraphs where the intersection of any distinct edges is the same vertex set.
	This forces us to slightly adapt the corresponding arguments for the cherry case.
	More specifically, we employ the following two results that replace Lemma~\ref{lemma: strictly balanced gluing rest density} and Lemma~\ref{lemma: few cyclic walks}.
	
	For~$\ell\geq 2$, we say that a sequence~$e_1,\ldots,e_\ell$ of distinct~$k$-sets forms a~\emph{$k'$-tight self-avoiding cyclic walk} if there exist distinct~$k'$-sets~$U_1,\ldots,U_\ell$ with~$U_i\subseteq e_{i}\cap e_{i+1}$ for all~$1\leq i\leq \ell$ with indices taken modulo~$\ell$.
	Note that the~$k$-graph~$\cS$ with no isolated vertices and edge set~$\set{e_1,\ldots,e_\ell}$ is a union of cherries.
	Indeed, for~$1\leq i\leq\ell$, the~$k$-graph~$\cA_i$ with no isolated vertices and edge set~$\set{e_i,e_{i+1}}$ with indices taken modulo~$\ell$ is a~$k''$-cherry for some~$k''\geq k'$ and we have~$\cS=\cA_1+\ldots+\cA_\ell$.
	Furthermore, the~$k$-graphs~$\cA_1,\ldots,\cA_\ell$ form a self-avoiding cyclic walk as defined in Section~\ref{section: gluing}.
	
	\begin{lemma}\label{lemma: strictly balanced gluing rest density cherry}
		Let~$1\leq k'\leq k-1$.
		Suppose that~$e_1,\ldots,e_\ell$ forms a~$k'$-tight self-avoiding cyclic walk and let~$\cS$ denote the~$k$-graph without isolated vertices and edge set~$\set{e_1,\ldots,e_\ell}$.
		Then, there exists~$e\in\cS$ such that~$\rho_{\cS,e}>1/(k-k')$.
	\end{lemma}
	\begin{proof}
		For~$0\leq i\leq\ell$, let~$V_i:=e_1\cup\ldots\cup e_i$ and for~$1\leq i\leq\ell$, let~$W_i:=e_i\setminus V_{i-1}$.
		Note that~$V_\cS=\bigcup_{1\leq i\leq \ell} W_i$ and that for all~$1\leq i<j\leq \ell$, we have~$W_i\cap W_j=\emptyset$.
		Hence,~$\abs{V_\cS}=\sum_{1\leq i\leq\ell} \abs{W_i}$.
		Since~$e_1,\ldots,e_\ell$ forms a~$k'$-tight self-avoiding cyclic walk, there exist distinct~$k'$-sets~$U_1,\ldots,U_\ell$ with~$U_i\subseteq e_{i}\cap e_{i+1}$ for all~$1\leq i\leq \ell$ with indices taken modulo~$\ell$.
		Hence, for all~$2\leq i\leq\ell$, we have~$\abs{e_{i-1}\cap e_i}\geq k'$ and thus~$\abs{W_i}\leq k-k'$.
		Furthermore, we have
		\begin{equation*}
			\abs{ (e_1\cup e_{\ell-1})\cap e_\ell }
			\geq\abs{ U_{\ell-1}\cup U_\ell }
			\geq k'+1
		\end{equation*}
		and thus~$\abs{W_\ell}\leq k-k'-1$.
		We conclude that
		\begin{equation*}
			\rho_{\cS,e_1}
			=\frac{\ell-1}{\paren{\sum_{1\leq i\leq\ell} \abs{W_i}}-k}
			\geq \frac{\ell-1}{(\ell-2)(k-k')+k-k'-1}
			>\frac{\ell-1}{(\ell-1)(k-k')}
			=\frac{1}{k-k'},
		\end{equation*}
		which completes the proof.
	\end{proof}

	\begin{lemma}\label{lemma: few cyclic walks cherry}
		Let~$1\leq k'\leq k-1$ and~$\ell\leq 4$ and suppose that~$e_1,\ldots,e_\ell$ forms a~$k'$-tight self-avoiding cyclic walk.
		Let~$\cS$ denote the~$k$-graph without isolated vertices and edge set~$\set{e_1,\ldots,e_\ell}$.
		If~$\cF$ is a~$k'$-cherry, then~$\Phi_{\cS}\leq n^{k-1/\rho_\cF-\eps^{1/7}}$.
	\end{lemma}
	\begin{proof}
		Suppose that~$\cF$ is a~$k'$-cherry.
		For~$1\leq i\leq\ell$, let~$\cA_i$ denote the~$k$-graph with no isolated vertices and edge set~$\set{e_i,e_{i+1}}$ with indices taken modulo~$\ell$.
		Then~$\cA_i$ has~$k$-density at least~$\rho_\cF$ and~$\cA_i$ is strictly~$k$-balanced.
		Furthermore,~$\cA_1,\ldots,\cA_\ell$ forms a self-avoiding cyclic walk and we have~$\cS=\cA_1+\ldots+\cA_\ell$.
		Hence, due to Lemma~\ref{lemma: strictly balanced gluing rest density cherry}, the statement follows from Lemma~\ref{lemma: few cyclic walks general}.
	\end{proof}
	
	\subsection{Overview of the argument}
	In this section, we show that~$H(\tau_\emptyset)\geq n^{k-1/\rho_\cF-\eps}$ with high probability if~$\cF$ is a cherry.
	To this end, from now on, for this section, in addition to the setup described in Section~\ref{section: sparse}, we assume that~$\cF$ is a~$k'$-cherry for some~$1\leq k'\leq k-1$ and that~$\cH$ is~$k'$-populated.
	Furthermore, we define~$i^\star$,~$\tau^\star$ and~$V_\cH$ as in Section~\ref{section: counting}.
	Overall, we argue similarly as in Section~\ref{section: isolation argument} based on isolation, however, the structures we focus on here are different.
	
	Still, instead of choosing the edge sets~$\cF_0(i)$ of copies with~$i\geq 1$ uniformly at random in Algorithm~\ref{algorithm: removal}, we again assume that during the initialization, a linear order~$\preceq$ on~$\cH^*$ is chosen uniformly at random and that for all~$i\geq 1$, the edge set~$\cF_0(i)$ is the minimum of~$\cH^*(i-1)$.
	
	For a~$k'$-set~$U\subseteq V_\cH$ and~$i\geq 0$, we 
	use
	\begin{equation*}
		\cD_{\cH(i)}(U):=\cset{ e\in\cH }{ U\subseteq e }
	\end{equation*}
	to denote the set of edges~$e\in\cH$ that contain~$U$ as a subset and we use
	\begin{equation*}
		\cD_{\cH(i)}^*(U):=\cset{ e\in\cD_{\cH}(U) }{ \abs{e\cap f}=k' \stforsome f\in\cH\setminus \cD_{\cH}(U) }
	\end{equation*}
	to denote the set of edges~$e\in\cH$ that contain~$U$ as a subset and that are an edge of a~$k'$-cherry where not both edges contain~$U$ as a subset.
	Note that~$d_{\cH}(U)=\abs{\cD_\cH(U)}$.
	We set~$d_{\cH}^*(U):=\abs{\cD_\cH^*(U)}$.
	We say that a~$k'$-set~$U\subseteq V_\cH$ is \emph{suitable} if there exists no sequence~$e_1,\ldots,e_\ell$ of edges of~$\cH(0)$ that forms a~$k'$-tight self-avoiding cyclic walk with~$2\leq \ell\leq 4$ such that~$U\subseteq e_1$.
	We use~$\cU$ to denote the set of suitable~$k'$-sets.
	Density considerations show that~$\cU$ includes almost all~$k'$-sets~$U\subseteq V_\cH$.
	We say that \emph{almost-isolation} occurs at~$U\in\cU$ if at some step~$i\geq 0$, we have~$1\leq d^*_{\cH}(U)\leq 2$ and~$d_{\cH}(U)\geq d^*_{\cH}(U)+1$.
	We say that \emph{isolation} occurs at~$U$ if additionally at a later step~$j> i$, we have~$d^*_{\cH(j)}(U)=0$ while~$d_{\cH(j)}(U)$ is odd hence causing at least one of the edges~$e\in\cH(j)$ to eventually become an isolated vertex of~$\cH^*(j')$ for some~$j'\geq j$.
	
	If at step~$i=i^\star$, we do not already have sufficiently many edges of~$\cH$ that are isolated vertices of~$\cH^*$, then by Lemma~\ref{lemma: copy control}, we may assume that there is essentially not more than one copy of~$\cF$ for every~$\abs{\cF}$ edges that remain.
	Hence, we are then in a situation where most of the remaining copies form a matching within~$\cH^*$.
	We claim that for these copies that form a matching, almost-isolation must have occurred at the set~$U$ of vertices that both edges of the copy share if~$U\in\cU$.
	This follows from Lemma~\ref{lemma: almost isolation must occur} below.
	Indeed, the lemma guarantees that for such~$U$, there exists~$0\leq i\leq i^\star$ with~$d^*_{\cH}(U)=1$ or there exists~$0\leq i\leq i^\star-1$ with~$d^*_{\cH}(U)=2$,~$d^*_{\cH(i+1)}(U)=0$ and~$d_{\cH}(U)-d_{\cH(i+1)}(U)\geq 1$.
	Almost-isolation at~$U$ occurs in both cases.
	
	\begin{lemma}\label{lemma: almost isolation must occur}
		Let~$U\in\cU$ and~$0\leq i\leq i^\star$.
		Then~$\Delta d^*_{\cH}(U):= d^*_{\cH}(U)-d^*_{\cH(i+1)}(U)\leq 2$.
		Furthermore, if~$\Delta d^*_{\cH}(U)= 2$, then~$U\subseteq f$ for all~$f\in\cF_0$.
	\end{lemma}
	\begin{proof}
		We only assume that~$U\subseteq V_\cH$ is a~$k'$-set and show that~$\Delta d^*_{\cH}(U)\geq 3$ entails~$U\notin\cU$ and furthermore that if~$\Delta d^*_{\cH}(U)=2$ and~$U\not\subseteq f$ for some~$f\in\cF_0$, then again~$U\notin\cU$.
		We distinguish three cases.
		
		For the first case, assume that~$U\subseteq f$ for all~$f\in\cF_0$.
		Then, only the edges of~$\cF_0$ can potentially be elements in~$\Delta\cD^*:=\cD^*_{\cH(i-1)}(U)\setminus \cD^*_\cH(U)$, so we have~$\abs{\Delta\cD^*}\leq 2$.
		
		For the second case, assume that there is exactly one~$f\in\cF_0$ with~$U\subseteq f$.
		Then if~$\abs{\Delta\cD^*}\geq 2$, there exists~$e\in\Delta\cD^*\setminus\cF_0$.
		For~$e$ to be in~$\Delta\cD^*$, it is necessary that there exists~$f\in\cF_0$ with~$U\not\subseteq f$ and~$\abs{e\cap f}=k'$.
		There is only one possible choice for~$f$ and for this edge~$f$, using~$f'$ to denote the other edge in~$\cF_0$, if~$e,f'$ does not form a~$k'$-tight self-avoiding cyclic walk, then~$e,f,f'$ forms a~$k'$-tight self-avoiding cyclic walk.
		
		For the third case, assume that~$U\not\subseteq f$ for all~$f\in\cF_0$.
		Then if~$\abs{\Delta\cD^*}\geq 2$, there exist distinct~$e_1,e_2\in\Delta\cD^*\setminus\cF_0$ such that for~$e\in\set{e_1,e_2}$, there exists~$f\in\cF_0$ with~$\abs{e\cap f}=k'$.
		If~$e_1,e_2$ does not form a~$k'$-tight self-avoiding cyclic walk and if for all~$f\in\cF_0$, the sequence~$e_1,e_2,f$ does not form a~$k'$-tight self-avoiding cyclic walk, then, using~$f$ and~$f'$ to denote the edges of~$\cF_0$, the sequence~$e_1,e_2,f,f'$ forms a~$k'$-tight self-avoiding cyclic walk.
		
		Furthermore, our above arguments show that if~$\Delta d^*_{\cH}(U)=2$ and~$U\not\subseteq f$ for some~$f\in\cF_0$, then~$U\notin \cU$.
	\end{proof}
	
	Overall, our argument shows that, if eventually most of the remaining copies form a matching within~$\cH^*$, almost-isolation must have occurred many times.
	In all cases where almost-isolation occurs, it is possible that this turns into isolation and the probability that this happens is not too small.
	We ensure that the~$k'$-sets at which we look for almost isolation are spaced out as this allows us to assume that at these sets, almost-isolation turns into isolation independently of the development at the other sets.
	
	\subsection{Formal setup}
	Formally, our setup is as follows.
	For~$\ell\geq 1$, a~$k$-graph~$\cA$ and a~$k'$-set~$U\subseteq V_\cA$, we inductively define~$\cW^\ell_\cA(U)$ as follows.
	Let
	\begin{equation*}
		\cW^1_\cA(U):=\cset[\bigg]{ U'\in\binom{V_\cA}{k'} }{ d_{\cA}(U\cup U')\geq 1 }
	\end{equation*}
	and for~$\ell\geq 2$, let
	\begin{equation*}
		\cW^\ell_\cA(U):=\bigcup_{U'\in \cW^{\ell-1}_\cA(U)}\cW^1_\cA(U').
	\end{equation*}
	For~$\ell\geq 1$, let~$W^\ell_\cA(U):=\abs{\cW^\ell_\cA(U)}$.
	Similarly as in Section~\ref{subsection: isolation formal}, during the random removal process, starting at step~$i^\star$, we additionally construct random subsets~$\emptyset=:\cR(0)\subseteq\ldots\subseteq \cR(i^\star)\subseteq \cU$ where we collect~$k'$-sets at which almost isolation occurs.
	We inductively define~$\cR(i)$ with~$1\leq i\leq i^\star$ as described by the following procedure.
	\par\bigskip
	\begin{algorithm}[H]
		\SetAlgoLined
		\DontPrintSemicolon
		$\cR(i)\gets \cR(i-1)$\;
		consider an arbitrary ordering~$U_1,\ldots,U_\ell$ of~$\cU$\;
		\For{$\ell'\gets 1$ \KwTo $\ell$}{
			\If{$i=\min\cset{ j\geq 0 }{ 1\leq d^*_{\cH(j)}(U_{\ell'})\leq 2\stand d_{\cH(j)}(U_{\ell'})\geq d^*_{\cH(j)}(U_{\ell'})+1}$ and $W_{\cH(0)}^4(U_{\ell'})\cap \cR(i)=\emptyset$}{
				$\cR(i)\gets \cR(i)\cup\set{U_{\ell'}}$\;
			}
		}
		\caption{Construction of~$\cR(i)$.}
		\label{algorithm: further edge sets cherry}
	\end{algorithm}
	\par\bigskip
	
	For~$U\in\cR(i^\star)$, let~$i_U:=\min\cset{ i\geq 0 }{ U\in \cR(i) }$.
	To define events that entail almost-isolation becoming isolation, for~$U\in\cR(i^\star)$ choose possibly non-distinct copies~$\cG_U,\cG_U'\in\cH^*(i_U)$ of~$\cF$ whose vertex sets contain~$U$ as a subset as follows.
	\begin{enumerate}[label=(\roman*)]
		\item If~$d^*_{\cH(i_U)}(U)=1$ and~$d_{\cH(i_U)}(U)$ is even, choose~$\cG_U=\cG_U'$ such that one edge of~$\cG_U$ is in~$\cD^*_{\cH(i_U)}(U)$ while the other edge of~$\cG_U$ is not in~$\cD_{\cH(i_U)}(U)$.
		\item If~$d^*_{\cH(i_U)}(U)=1$ and~$d_{\cH(i_U)}(U)$ is odd, choose~$\cG_U=\cG_U'$ such that one edge of~$\cG_U$ is in~$\cD^*_{\cH(i_U)}(U)$ while the other edge of~$\cG_U$ is in~$\cD_{\cH(i_U)}(U)$.
		\item If~$d^*_{\cH(j)}(U)=2$ and~$d_{\cH(i_U)}(U)$ is even, choose~$\cG_U\neq\cG_U'$ with~$\cG_U\cap \cG_U'=\emptyset$ such that one edge of~$\cG_U$ is in~$\cD^*_{\cH(i_U)}(U)$ while the other edge of~$\cG_U$ is not in~$\cD_{\cH(i_U)}(U)$ and such that one edge of~$\cG_U'$ is in~$\cD^*_{\cH(i_U)}(U)$ while the other edge of~$\cG_U'$ is in~$\cD_{\cH(i_U)}(U)$
		\item If~$d^*_{\cH(j)}(U)=2$ and~$d_{\cH(i_U)}(U)$ is odd, choose~$\cG_U=\cG_U'$ such that both edges of~$\cG_U$ are in~$\cD^*_{\cH(i_U)}$.
	\end{enumerate}
	Let
	\begin{equation*}
		\cE_U:=\set{ \cG_U\preceq \cG\stforall \cG\in\cN^1_{\cH^*(0)}(\cG_U)\stand \cG_U'\preceq \cG\stforall \cG\in\cN^1_{\cH^*(0)}(\cG_U') }.
	\end{equation*}

	\ifimages
	\begin{figure}
		\includegraphics{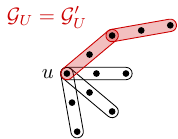}
		\hspace{0.5cm}
		\includegraphics{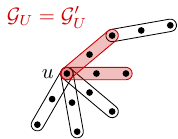}
		\hspace{0.5cm}
		\includegraphics{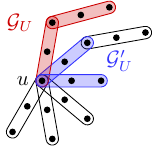}
		\hspace{0.5cm}
		\includegraphics{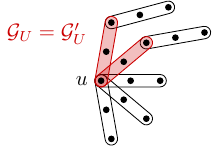}

		\caption{Examples for choices of the copies~$\cG_U$ and~$\cG_U'$ for the special case where~$\cF$ is a~$3$-uniform~$1$-cherry.
		Each example shows the situation of the edges containing~$U=\protect\set{u}$ as a subset at step~$i_U$.}
	\end{figure}
	\fi
	
	\subsection{Proof of Theorem~\ref{theorem: technical cherries}}\label{subsection: cherry proofs}
	As in Section~\ref{subsection: isolation proof}, since every almost-isolation that turns into isolation causes an edge of~$\cH(0)$ to become an isolated vertex of~$\cH^*$ at some step~$i\geq 0$ and hence an edge that remains at the end of the removal process, we obtain the following statement.
	
	\begin{observation}\label{observation: remain configuration effect cherry}
		$H(\tau_\emptyset)\geq \sum_{U\in \cR(i^\star)}\ind_{\cE_U}$.
	\end{observation}
	
	We again organize the formal presentation of the arguments outlined above into suitable lemmas.
	Some of these are similar to those in Section~\ref{subsection: isolation proof}.
	Combining the lemmas with the above observation, we then obtain Theorem~\ref{theorem: technical cherries}.
	We define the event~$\cE_0$ as in Section~\ref{subsection: isolation proof}.
	
	\begin{lemma}\label{lemma: remain configuration bound cherry}
		Let~$\cX:=\set{i^\star <\tau^\star}\cap \cE_0$.
		Then,~$\abs{\cR(i^\star)}\Xgeq n^{k-1/\rho_\cF-5\eps^2}$.
	\end{lemma}
	\begin{proof}
		We argue similarly as in the proof of Lemma~\ref{lemma: remain configuration bound}.
		Let~$\cA$ denote the~$k$-graph with no isolated vertices and exactly one edge and fix a~$k'$-set~$I\subseteq V_\cA$.
		Consider a~$k'$-set~$U\subseteq V_\cH$ and~$\psi\colon I\injection U$.
		Combining the fact that~$\cH(0)$ is~$k'$-populated and Lemma~\ref{lemma: bounds for H}, we have
		\begin{equation}\label{equation: degrees bound cherry}
			2\leq d_{\cH(0)}(U)\leq \Phi_{\cA,\psi}\leq n^{\eps^3}.
		\end{equation}
		Let~$i:=i^\star$ and consider the set
		\begin{equation*}
			\cI^*:=\cset{ \cF'\in\cH^* }{ \cN_{\cH^*}^1(\cF')=\set{\cF'} }
		\end{equation*}
		of edge sets of copies of~$\cF$ in~$\cH$ that are isolated in the sense that they do not share an edge with another copy of~$\cF$.
		Let~$\iota\colon\cI^*\to 2^{V_\cH}$ denote the function such that~$\iota(\cF')$ is the intersection of the two edges of~$\cF'$ for all~$\cF'\in\cI^*$.
		As a consequence of the lower bound in~\eqref{equation: degrees bound cherry}, for all~$U\in \cJ:=\iota(\cI^*)\cap \cU$ almost isolation must have occurred at~$U$ due to Lemma~\ref{lemma: almost isolation must occur} (see discussion in the paragraph before Lemma~\ref{lemma: almost isolation must occur}).
		Thus, either~$U$ itself is an element of~$\cR$ or there exists some~$U'\in\cW_{\cH^*(0)}^4(U')\cap \cR$ that prevented the inclusion of~$U$ in~$\cR$.
		Hence, we may choose a function~$\pi\colon\cJ\to \cR$ that for every~$U\in\cJ$ chooses a witness~$\pi(U)$ with~$\pi(U)\in\cW_{\cH^*(0)}^4(U)$ or equivalently~$U\in\cW_{\cH^*(0)}^4(\pi(U))$.
		If~$U\in\cR$ and~$U'\in\pi^{-1}(U)$, we have~$U'\in\cW_{\cH^*(0)}^4(U)$ and hence~$\pi^{-1}(U)\subseteq \cW_{\cH^*(0)}^4(U)$.
		The upper bound in~\eqref{equation: degrees bound cherry} entails~$W_{\cH^*(0)}^{1}(U)\leq n^{\eps^{3}}\cdot k^{k'}$ and for all~$\ell\geq 1$ furthermore~$W_{\cH^*(0)}^{\ell+1}(U)\leq W_{\cH^*(0)}^{\ell}(U)\cdot n^{\eps^{3}}\cdot k^{k'}$.
		Hence, we inductively obtain~$W_{\cH^*(0)}^\ell(U)\leq k^{\ell k'} n^{\ell \eps^{3}}$ and in particular~$W_{\cH^*(0)}^4(U)\leq n^{\eps^{2}}$.
		Thus,
		\begin{equation*}
			\abs{\cJ}\leq \sum_{U\in\cR}\abs{\pi^{-1}(U)}\leq \abs{\cR}n^{\eps^{2}}.
		\end{equation*}
		As a consequence of Lemma~\ref{lemma: few cyclic walks cherry}, the number of~$k'$-sets~$U\subseteq V_\cH$ that are not suitable is at most~$3\cdot n^{k-1/\rho_\cF-\eps^{1/7}}\cdot (4k)^{k'}\leq n^{k-1/\rho_\cF-\eps^{1/6}}$.
		Hence,~$\abs{\cJ}\geq \abs{\iota(\cI^*)}-n^{k-1/\rho_\cF-\eps^{1/6}}$ and thus
		\begin{equation*}
			\abs{\cR}\geq n^{-\eps^{2}}(\abs{\iota(\cI^*)}-n^{k-1/\rho_\cF-\eps^{1/6}}).
		\end{equation*}
		Furthermore, if~$U\subseteq V_\cH$ is a~$k'$-set and~$e\in\cF'$ for some~$\cF'\in \iota^{-1}(U)$, then for all~$\cF''\in \iota^{-1}(U)\setminus\set{\cF'}$ and~$f\in\cF''$, we have~$\abs{e\cap f}\geq k'+1$ and for all distinct~$\cF'',\cF'''\in \iota^{-1}(U)\setminus\set{\cF'}$,~$f\in\cF''$ and~$g\in\cF'''$, we have~$f\neq g$.
		Thus, there exists~$k'+1\leq k''\leq k$ and at least~$(\abs{\iota^{-1}(U)}-1)/k$ distinct edges~$f_1,\ldots,f_\ell\in\cH$ with~$\abs{e\cap f_{\ell'}}=k''$ for all~$1\leq \ell'\leq\ell$.
		Now, let~$\cA$ denote a~$k''$-cherry, let~$I\in\cA$ and fix~$\psi\colon I\injection e$.
		By Lemma~\ref{lemma: bounds for H}, we have
		\begin{equation*}
			\frac{\abs{\iota^{-1}(U)}-1}{k}
			\leq\ell
			\leq \Phi_{\cA,\psi}
			\leq n^{\eps^{3}}
		\end{equation*}
		and thus
		\begin{equation*}
			\abs{\cI^*}\leq \sum_{U\in\iota(\cI^*)}\abs{\iota^{-1}(U)}\leq \abs{\iota(\cI^*)}n^{\eps^{2}}.
		\end{equation*}
		Overall, this yields
		\begin{equation}\label{equation: isolated copies and remain configurations cherry}
			\abs{\cR}\geq n^{-\eps^{2}}(n^{-\eps^{2}}\abs{\cI^*}-n^{k-1/\rho_\cF-\eps^{1/6}}),
		\end{equation}
		so it suffices to find an appropriate lower bound for~$\cI^*$.
		Similarly as in the proof of Lemma~\ref{lemma: remain configuration bound}, we may again rely on Lemma~\ref{lemma: sparse edges of H} to obtain~$H^*\Xleq (1+\eps)H/\abs{\cF}$ precisely as in~\eqref{equation: upper bound on H^*} and then
		\begin{equation*}
			H\Xleq H-\frac{1}{4\abs{\cF}}H+\frac{1}{2}\abs{\cI^*}
		\end{equation*}
		precisely as in~\eqref{equation: edges and isolated copies}.
		With Lemma~\ref{lemma: sparse edges of H}, precisely as in~\eqref{equation: lower bound isolation}, this implies~$\abs{\cI^*}\Xgeq n^{k-1/\rho_\cF-2\eps^2}$. 
		Combining this with~\eqref{equation: isolated copies and remain configurations cherry} yields~$\abs{\cR}\Xgeq n^{k-1/\rho_\cF-5\eps^2}$.
	\end{proof}
	
	\begin{lemma}\label{lemma: remain configuration dominance cherry}
		Suppose that~$X$ is a binomial random variable with parameters~$n^{k-1/\rho_\cF-5\eps^2}$ and~$n^{-2\eps^{2}}$ and let~$Y:=(n^{k-1/\rho_\cF-5\eps^2}-\abs{\cR(i^\star)})\vee 0$.
		Let
		\begin{equation*}
			Z:=Y+\sum_{U\in\cR(i^\star)} \ind_{\cE_U}.
		\end{equation*}
		Then,~$Z$ stochastically dominates~$X$.
	\end{lemma}
	\begin{proof}
		We argue similarly as in the proof of Lemma~\ref{lemma: remain configuration dominance}.
		First, observe that by Lemma~\ref{lemma: F embeddings bound}, whenever~$U\in\cR(i^\star)$, for~$i:=0$, we have
		\begin{equation*}
			N_{\cH^*}^1(\cG_U)
			\leq \sum_{f\in\cG_U} d_{\cH^*}(f)
			\leq n^{\eps^{2}}
		\end{equation*}
		and thus
		\begin{equation}\label{equation: few neighbors cherry}
			\abs{\cN_{\cH^*}^1(\cG_U)\cup \cN_{\cH^*}^1(\cG_U')}
			\leq n^{2\eps^{2}}.
		\end{equation}
		Consider distinct~$k'$-sets~$U,U'\subseteq V_\cH$.
		By construction of~$\cR(i^\star)$, whenever~$U,U'\in\cR(i^\star)$, then
		\begin{equation*}
			(\cN^1_{\cH(0)}(\cG_U)\cup \cN^1_{\cH(0)}(\cG_U'))\cap (\cN^1_{\cH(0)}(\cG_{U'})\cup \cN^1_{\cH(0)}(\cG_{U'}'))=\emptyset.
		\end{equation*}
		Thus, for all distinct~$U_1,\ldots,U_\ell\in\cR(i^\star)$ and all~$z_1,\ldots,z_{\ell-1}\in\set{0,1}$, from~\eqref{equation: few neighbors cherry}, we obtain
		\begin{equation*}
			\cpr{\ind_{\cE_{U_\ell}}=1}{ \ind_{\cE_{U_{\ell'}}}=z_{\ell'}\stforall 1\leq\ell'\leq\ell }
			=\pr{\cE_{U_\ell}}
			\geq n^{-2\eps^{2}},
		\end{equation*}
		which completes the proof.
	\end{proof}
	
	\begin{proof}[Proof of Theorem~\ref{theorem: technical cherries}]
		The proof is almost exactly the same as for Theorem~\ref{theorem: technical} with the key difference that we replace objects and references with the appropriate analogous constructions and arguments form this section.
		Define the events
		\begin{equation*}
			\cB:=\set{ H(\tau_\emptyset)\leq n^{k-1/\rho_\cF-\eps} }\qtand
			\cX:=\set{i^\star <\tau^\star}\cap \cE_0.
		\end{equation*}
		We need to show that~$\pr{\cB}$ is sufficiently small.
		Choose~$X$,~$Y$ and~$Z$ as in Lemma~\ref{lemma: remain configuration dominance cherry}.
		Lemma~\ref{lemma: remain configuration bound cherry} entails~$\cX\subseteq\set{Y=0}$ and hence~$\set{Y\neq 0}\subseteq\cX^\comp$.
		Thus, from Observation~\ref{observation: remain configuration effect cherry} and Lemma~\ref{lemma: remain configuration dominance cherry}, we obtain
		\begin{equation*}
			\begin{aligned}
				\cB
				&=\set[\Big]{ \sum_{U\in\cR(i^\star)} \ind_{\cE_U}\leq n^{k-1/\rho_\cF-\eps} }\cap\cB
				\subseteq \paren{ \set{ Z\leq n^{k-1/\rho_\cF-\eps} }\cup \set{Y\neq 0} }\cap\cB\\
				&\subseteq \set{ Z\leq n^{k-1/\rho_\cF-\eps} } \cup (\cX^\comp\cap \cB)
				\subseteq \set{ Z\leq n^{k-1/\rho_\cF-\eps} } \cup \set{\tau^\star\leq i^\star}\cup (\cE_0^\comp\cap \cB).
			\end{aligned}
		\end{equation*}
		By Lemma~\ref{lemma: sparse edges of H}, we have
		\begin{equation*}
			H(\tau_\emptyset)
			\geq_{\cE_0^\comp} \eps H(i^\star)
			\geq \eps^2 n^k\phat(i^\star)
			\geq n^{k-1/\rho_\cF-2\eps^2}
		\end{equation*}
		and hence~$\cE_0^\comp\cap\cB=\emptyset$.
		Thus, using Lemma~\ref{lemma: copy control}, we obtain
		\begin{equation*}
			\pr{\cB}
			\leq \pr{Z\leq n^{k-1/\rho_\cF-\eps}}+\exp(-n^{1/3}).
		\end{equation*}
		With Lemma~\ref{lemma: remain configuration dominance cherry} and Chernoff's inequality (see Lemma~\ref{lemma: chernoff}), this completes the proof.
	\end{proof}

	\subsection{Proofs for Theorems~\ref{theorem: cherries} and~\ref{theorem: sparse cherries}}
	In this section, we show how to obtain Theorems~\ref{theorem: cherries} and~\ref{theorem: sparse cherries} from Theorems~\ref{theorem: technical bounds} and~\ref{theorem: technical cherries}.

	\begin{proof}[Proof of Theorem~\ref{theorem: sparse cherries}]
		By definition of~$\tau_\emptyset$ in Section~\ref{section: counting}, this is an immediate consequence of Theorem~\ref{theorem: technical cherries}.
	\end{proof}

	\begin{proof}[Proof of Theorem~\ref{theorem: cherries}]
		The argumentation is essentially the same as in the proof of Theorem~\ref{theorem: pseudorandom} except that we use Theorem~\ref{theorem: technical cherries} instead of Theorem~\ref{theorem: technical}.

		We define the constants~$m$,~$\eps$,~$\delta$,~$n$ and the~$k$-graphs~$\cH$~$\cH'$ and~$\cH''$ precisely as in the proof of Theorem~\ref{theorem: pseudorandom}.
		Let~$\cX'$ denote the event that~$\cH'$ is~$(4m,n^{\eps^4})$-bounded,~$k'$-populated and has~$n^{k-1/\rho+\eps^5}/k!$ edges and let~$\cX''$ denote the event that
		\begin{equation*}
			\cX'':=\set{\abs{\cH''}\leq n^{k-1/\rho+\eps}}\qtand
			\cY'':=\set{n^{k-1/\rho-\eps}\leq \abs{\cH''}}.
		\end{equation*}
		We need to show that
		\begin{equation*}
			\pr{\cX''\cap \cY''}\geq 1-\exp(-(\log n)^{5/4}).
		\end{equation*}
		Since~$\cX'\subseteq \cX''$, we have~$\pr{\cX''\cap \cY''}\geq \pr{\cX'\cap \cY''}$, so it suffices to obtain sufficiently large lower bounds for~$\pr{\cX'}$ and~$\pr{\cY''}$.
		Due to~$k'= k-1/\rho$, we may apply Theorem~\ref{theorem: technical bounds} with~$\eps^5$ playing the role of~$\eps$ to obtain~$\pr{\cX'}\geq 1-\exp(-(\log n)^{4/3})$ and Theorem~\ref{theorem: technical cherries} shows that~$\cpr{\cY''}{\cX'}\geq 1-\exp(-n^{1/4})$.
		Using~$\pr{\cY''}=\cpr{\cY''}{\cX'}\pr{\cX'}$, this yields suitable lower bounds for~$\pr{\cX'}$ and~$\pr{\cY''}$.
	\end{proof}
	
\end{document}